\documentclass[a4paper,11pt]{amsart}
\usepackage[a4paper, margin=2.5cm]{geometry}

\usepackage{rotating}
\usepackage{amsmath}
\usepackage{amsfonts}
\usepackage{graphicx}

\usepackage[bb=dsserif]{mathalpha}
\usepackage{bm}

\usepackage{stmaryrd}{}
\usepackage{graphicx} 
\usepackage{geometry}
\usepackage[foot]{amsaddr}
\usepackage[english]{babel}
\usepackage{amsthm}

\usepackage[all]{xy}
\usepackage{url}
\usepackage{breakurl}
\usepackage[breaklinks]{hyperref}
\usepackage[capitalize]{cleveref}

\usepackage{xfrac}

\newcommand{\lu}[1]{{\color{magenta}#1}}
\renewcommand{\emph}[1]{\textbf{#1}} 
\newcommand{\introthmname}{}
\newtheorem{introthminn}{\introthmname}

\usepackage{amssymb}
\usepackage{mathtools}
\usepackage{thmtools}
\usepackage{xfrac}
\usepackage{thm-restate}
\usepackage{enumerate}
\usepackage{hyperref}
\usepackage{xfrac}
\usepackage{cleveref}
\usepackage{verbatim}
\usepackage{amsmath}
\usepackage{amssymb}
\usepackage{mathabx}
\usepackage{graphicx}
\usepackage{amsfonts}
\usepackage{pb-diagram}
\usepackage{tikz-cd}
\usepackage{graphicx}
\usepackage{xcolor}
\usepackage{adjustbox}
\usepackage{mathtools}
\usepackage{todonotes}
\usepackage{amsthm}

\usepackage{mathrsfs}
\usepackage{bm}
\usepackage{enumerate}

\newcommand{\onto}[0]{\twoheadrightarrow}
\newtheorem{thmx}{Theorem}
\newtheorem{corx}[thmx]{Corollary}

\usepackage{amssymb}
\usepackage{mathtools}
\usepackage{thmtools}
\usepackage{xfrac}
\usepackage{thm-restate}
\usepackage{enumitem}
\usepackage{hyperref}
\usepackage{xfrac}
\usepackage{cleveref}
\usepackage{verbatim}
\usepackage{amsmath}
\usepackage{amssymb}
\usepackage{mathabx}
\usepackage{graphicx}
\usepackage{amsfonts}
\usepackage{pb-diagram}
\usepackage{tikz-cd}
\usepackage{graphicx}
\usepackage{xcolor}
\usepackage{adjustbox}
\usepackage{mathtools}
\usepackage{todonotes}

\usepackage{mathrsfs}
\usepackage{bm}

\DeclareMathOperator{\diam}{diam}
\DeclareMathOperator{\Gate}{Gate}

\makeatletter
\newcommand{\tpitchfork}{%
  \vbox{
    \baselineskip\z@skip
    \lineskip-.52ex
    \lineskiplimit\maxdimen
    \m@th
    \ialign{##\crcr\hidewidth\smash{$-$}\hidewidth\crcr$\pitchfork$\crcr}
  }%
}

\DeclareMathOperator{\NFr}{NFr}

\DeclareMathOperator{\SStab}{SStab}

\newcommand{\nb}[0]{\mathcal{O}}

\DeclareMathOperator{\ccl}{ccl}
\DeclareMathOperator{\dom}{Dom}
\DeclareMathOperator{\im}{Im}
\DeclareMathOperator{\Conn}{Conn}
\DeclareMathOperator{\T}{\mathsf{T}}
\DeclareMathOperator{\Ed}{\mathsf{E}}
\let\O\relax
\DeclareMathOperator{\O}{\mathsf{O}}

\newcommand{\actson}{\,\raisebox{2.5ex}[0pt][0pt]{\begin{turn}{-90}\ensuremath{\
				\circlearrowright}\end{turn}}\,}

\newcommand{\actsonr}{\,\raisebox{2.5ex}[0pt][0pt]{\begin{turn}{-90}\ensuremath{\
				\circlearrowleft}\end{turn}}\,}

\newtheorem{theorem}{Theorem}
\newtheorem{definition}[theorem]{Definition}

\newtheorem{corollary}[theorem]{Corollary}
\newtheorem{lemma}[theorem]{Lemma}
\newtheorem{remark}[theorem]{Remark}
\newtheorem*{remark*}{Remark}
\newtheorem{observation}[theorem]{Observation}

\newtheorem*{question*}{Question}

\newtheorem{claim}[theorem]{Claim}

\newtheorem{notation}[theorem]{Notation}
\newtheorem{assumption}[theorem]{Assumption}
\newtheorem*{assumption*}{Assumption}
\newtheorem*{examples*}{Examples}

\newtheorem{setting}[theorem]{Setting}
\newtheorem*{setting*}{Setting}

\theoremstyle{plain} 

\newcommand{\thistheoremname}{}
\newtheorem{genericthm}[theorem]{\thistheoremname}

\newtheorem*{genericthm*}{\thistheoremname}
\newenvironment{namedthm*}[1]
{\renewcommand{\thistheoremname}{#1}%
	\begin{genericthm*}}
	{\end{genericthm*}}

\newcommand{\D}[0]{\mathrm{D}}
\newcommand{\R}[0]{\mathbb{R}}

\newcommand{\ik}[0]{\in k}

\newenvironment{subproof}[1][\proofname]{%
\begin{proof}[#1]%
	}{%
\end{proof}%
}

\newenvironment{subsubproof}[1][\proofname]{%
\begin{proof}[#1]%
}{%
\end{proof}%
}

\renewcommand{\L}[0]{\ll}

\newcommand{\nin}[0]{\notin}

\newcommand{\nn}[0]{\mathsf{N}}
\newcommand{\N}[0]{\mathbb{N}}

\newcommand{\Z}[0]{\mathbb{Z}}

\DeclareMathOperator{\Lk}{\mathsf{Lk}}
\DeclareMathOperator{\Cc}{\mathsf{CC}}

\renewcommand{\nn}[1]{\lVert #1 \rVert}

\newcommand{\pa}[0]{\mathfrak{p}_{p}}
\newcommand{\ax}[0]{\mathfrak{a}_{p}}

\numberwithin{theorem}{section}
\numberwithin{lemma}{section}
\numberwithin{corollary}{section}
\numberwithin{observation}{section}
\numberwithin{claim}{section}
\numberwithin{definition}{section}
\numberwithin{proposition}{section}
\numberwithin{fact}{section}
\numberwithin{remark}{section}
\numberwithin{notation}{section}
\numberwithin{examples}{section}
\numberwithin{assumption}{section}
\numberwithin{setting}{section}

 \newcommand{\Id}[0]{\mathrm{Id}}
\newcommand{\subg}[1]{\langle #1 \rangle}

  \newcommand{\F}[0]{\mathbb{F}}

\def\Ind#1#2{#1\setbox0=\hbox{$#1x$}\kern\wd0\hbox to 0pt{\hss$#1\mid$\hss}
\lower.9\ht0\hbox to 0pt{\hss$#1\smile$\hss}\kern\wd0}

\def\Notind#1#2{#1\setbox0=\hbox{$#1x$}\kern\wd0\hbox to 0pt{\mathchardef
\nn="3236\hss$#1\nn$\kern1.4\wd0\hss}\hbox to
0pt{\hss$#1\mid$\hss}\lower.9\ht0
\hbox to 0pt{\hss$#1\smile$\hss}\kern\wd0}

\newcommand{\dih}[0]{d_{\mathcal{H}}}

\renewcommand{\t}[0]{\mathrm{T}}

\usepackage{hyperref}

\DeclareMathOperator{\supp}{Supp}

\DeclareMathOperator{\E}{E}
\DeclareMathOperator{\K}{K}

\newcommand{\myitem}[1]{%
	\item[#1]\protected@edef\@currentlabel{#1}%
}

 \newcommand{\sym}[0]{\mathtt{S}}
 \newcommand{\alt}[0]{\mathtt{A}}

\renewcommand{\t}[0]{\mathrm{T}}

\DeclareMathOperator{\Stab}{Stab}

 \newcommand{\dg}[0]{\D(\Theta)}
 \newcommand{\oid}[0]{\Omega\in\dg}

    \newcommand{\G}[0]{G_{\Omega_{0}}}

   \newcommand{\cyc}[0]{\mathtt{C}}
 \newcommand{\di}[0]{\mathtt{D}}
 \newcommand{\aon}[0]{\actson}
 \newcommand{\noa}[0]{\actsonr}

\makeatletter
\newcommand*\bigcdot{\mathpalette\bigcdot@{.7}}
\newcommand*\bigcdot@[2]{\mathbin{\vcenter{\hbox{\scalebox{#2}{$\m@th#1\bullet$}}}}}
\makeatother

\renewcommand{\setminus}{\smallsetminus}
\begin{document}

\title[Acylindricity and sharp multiple transitivity]{A simple acylindrical recipe for non-split characteristic $2$ sharply $k$-transitive actions and their generalizations.}

\author{J. de la Nuez Gonz\'alez}
\email{jnuezgonzalez@gmail.com}
\address{Korea Institute for Advanced Study (KIAS)}       
 
\begin{abstract}
We provide a fairly simple list of conditions on a group $G$ acting acylindrically on a $\delta$-hyperbolic metric space sufficient for the group to admit an action on a set that is $k$-sharp, transitive on $k$-sets and has the property that for any $k$-set $A$ the setwise stabilizer of $A$ acts on $A$ as prescribed by some fixed subgroup $\Theta\leq\sym_{k}$. In particular, this yields a large class of non-split sharply $2$ and $3$-transitive actions containing many finitely presented groups, and allows us to characterize sharp $2$-transitivity within the class of hyperbolic groups that do not contain a copy of $C_{2}\times C_{2}$.
\end{abstract}

\thanks{The author has been supported by the Mid-Career Researcher Program (RS-2023-00278510) through the National Research Foundation funded by the government of Korea and by the KIAS individual grant $SP084001$. }

\maketitle

\section{Introduction}
  
  \subsection{Background}  Given a set $U$ and an integer $k>1$ we write $(U)^{k}$ for the collection of $k$-tuples of distinct elements of $U$ and $[U]^{k}$ for the collection of $k$-sets of $U$. Clearly, any action of a group $G$ on $U$ induces a diagonal action of $G$ on both $(U)^{k}$ and $[U]^{k}$. Recall that an action $U\noa G$ is called:
  \begin{itemize}
  	\item \emph{$k$-sharp} if the stabilizer of any $\underline{a}\in(U)^{k}$ is trivial;
  	\item \emph{$k$-transitive} if the induced action $(U)^{k}\noa G$  is transitive on $(U)^{k}$, i.e. for any $\underline{a},\underline{b}\in (U)^{k}$ there is $g\in G$ such that $\underline{a}\cdot g=\underline{b}$;
  	\item \emph{sharply $k$-transitive} if it is both $k$-sharp and $k$-transitive.   
  \end{itemize}
   A group $G$ is said to be sharply $k$-transitive if it admits a sharply $k$-transitive action on a set.  
   A sharply $1$-transitive action of a group $G$ on a set, i.e., what is often called a \emph{regular action}, is isomorphic to the action by multiplication of $G$ on itself, so the problem is only interesting for $k \geq 2$. For technical reasons, we will look at the problem from the point of view of actions on the right, which does not incur any loss of generality.
   
   There are several well-known classical examples of sharply $k$-transitive groups (in each case below, showing sharp $k$-transitivity is straightforward):
   \begin{enumerate}[label=(\roman*)]
   	\item \label{i: sf AGL} for a skew-field $K$, the affine group $\mathrm{AGL}\lambda K)$ acts sharply $2$-transitively on $K$;
   	\item \label{i: f PGL} for a field $K$, the projective group $\mathrm{PGL}(2, K)$ acts sharply $3$-transitively on the projective line over $K$ (a particularly well-known example of this is that the group of M\"{o}bius transformations acts sharply $3$-transitively on the Riemann sphere);
   	\item for each $k \geq 4$, the symmetric groups $\sym_k$, $\sym_{k+1}$ and the alternating group $\alt_{k+2}$ are sharply $k$-transitive.
   \end{enumerate}
   
   Much is known about the classification of sharply $k$-transitive actions on finite sets. An introduction to the classification of sharply $k$-transitive groups can be found in \cite[Section 7.6]{DM96}, while the full classification can be found in \cite{Pas68}. 

   For infinite actions much less is known. It has been known for a long time that sharply $k$-transitive actions on infinite sets do not exist for $k\geq 4$, see  \cite{Tit49}, \cite[Chapitre IV, Th\'{e}or\`eme 1]{Tit52}.
   A stronger version of this result showing the non-existence of $4$-transitive actions where the stabilizer of a $4$-tuple of distinct elements is allowed to be a finite non-trivial group was later established by Hall \cite{Hal54} under the restriction that said stabilizers have finite order, and later by Yoshizawa for general finite stabilizers \cite{Yos79}.

   A distinctive feature of finite sharply $2$-transitive actions, as well as of the more elementary examples of sharply $2$-transitive actions is that they are \emph{split}, namely there is a normal subgroup $N\unlhd G$ so that $U\noa N$ is regular. In this case we say that the action (or $G$) is \emph{split}. It can be shown that in this case $N$ must be abelian and $G$ can be written as the semidirect product of $N$ with a point-stabiliser. One says that a sharply $3$-transitive action $U\noa G$ is split if so is the action of the stabilizer of any point $p\in U$ on $U\setminus\{p\}$. 
    
    
   Although many classes of infinite split sharply $2$-transitive actions have been known for a long time, the existence of non-split sharply $2$ transitive groups was established only relatively recently by Rips, Segev and Tent (\cite{RST17}). This led to further interest in the area in recent years: see, for example, \cite{RT19}, \cite{TZ16}, \cite{Ten16I}, \cite{Ten16}, \cite{ABW18}, \cite{GG21}, \cite{AAT23}, \cite{AG24}, \cite{Ame25}, \cite{AA26}. A relatively recent survey on the topic can be found in \cite{tent2016infinite}.
   
   The goal of this work is to provide a simple geometric construction producing non-split sharply $2$ and $3$-transitive actions in the context of negatively curved groups, by making explicit some of the geometric ideas in  \cite{gonzalez2025two}. Our results here also extend in the direction of the generalization of sharp $k$-transitivity introduced in said work.    
   
   One important invariant of a sharply $2$-transitive action on a set is its characteristic. A sharply $2$-transitive group $G$ can be shown to have infinitely many involutions (any two points in the set being acted on are exchanged by one such involution), and one can easily show they are all conjugate. We say that a sharply $2$-transitive action is of characteristic $2$ if all involutions act freely. This is the most flexible case and the one we treat and generalize in this work. In fact, an argument by Simon Andr\'e shows that a hyperbolic group cannot admit a sharply $2$-transitive action of characteristic different from $2$, in contrast to \cref{c:hyperbolic case} below.

 \subsection{Hyperbolic sharply $2$ and $3$-transitive groups.} Before introducing our most general result, \cref{t: main}, which will involve  recalling a certain amount of background and new definitions, we will state two consequences of it that are both directly legible in terms of the above discussion and require less preparation. 
 
 For now we assume the reader is familiar with the notion of a Gromov-hyperbolic metric space and its boundary, although we will review most relevant details in \cref{sus:Gromov hyperbolic spaces}.  
 Recall that a finitely generated group $G$ is said to be \emph{hyperbolic} if for some (equivalently, any) finite set of generators of $G$ the Cayley graph of $G$ with respect to those generators is Gromov-hyperbolic. The notion of a hyperbolic group is a cornerstone of geometric group theory and has been the subject of study from a variety of many different points of view, as well as generalized in a number of different directions, one of which we will also touch upon in this work.

  \begin{definition}\label{quasiconvex}
  	A subset $B$ of a geodesic metric space is said to be \emph{quasiconvex} if there is some constant $K>0$ 
  	such that any geodesic $\gamma$ between two points in $B$ is contained in the collection of points at distance at most $K$ from $B$. A subgroup $H$ of a hyperbolic group $G$ is said to be \emph{quasiconvex} if it is quasiconvex as a subset of vertices of the Cayley graph of $G$ with respect to some (equivalently any) finite set of generators. 
  \end{definition}
 
 We also have the following corollary of our main theorem, to be proved in \cref{s:proof of main}.
 \begin{thmx}\label{c:hyperbolic case}
 		Let $G$ be a non-elementary hyperbolic group.
 	\begin{enumerate}[label=(\Alph*)]
 		\item \label{2str}Assume that $G$ admits a unique infinite class of involutions $\sigma^{G}$, the centralizer $C_{G}(\sigma)$ is either
 		\begin{itemize}
 			\item non-elementary,
 			\item of the form $\subg{\sigma}\times \Z$,
 		\end{itemize}
 		 and $\subg{\sigma}$ is the largest finite group normalized by $C_{G}(\sigma)$.
 		 Then $G$ is sharply $2$-transitive.
 		\item \label{3str}Assume that $G$ contains a copy of $\sym_{3}$, henceforth identified with $\sym_{3}$, such that any finite subgroup of $G$ that embeds into $\sym_{3}$ must be conjugate into it. 
 		Let $s\in\sym_{3}\leq G$ be the involution fixing $0\in 3$, and write
 			 $G_{\sigma}=C_{G}(\sigma),G_{\tau}=N_{G}(\tau)$. 
 		Assume that the following holds:
 		\begin{enumerate}[label=(\roman*)]
		  	\item \label{normalizer condition0} the normalizer of any $\{1\}\leq\Omega<\sym_{3}$ is infinite and normalizes no finite group other than $\Omega$;
 		  	\item \label{Gsigma malnormal0}$G_{\sigma}$ is malnormal.
 				\item \label{Gsigma conjugates0} if $G_{\sigma}^{g}\cap\Theta\neq\{1\}$, then $g\in G_{\sigma}\theta$ for some $\theta\in\sym_{3}$; 
 				\item \label{Gsigma convex0}$G_{\sigma}$ is a quasiconvex subgroup of $G$;
 				\item \label{Gsigma alternative0}one of the following holds: 
 				\begin{itemize}
 					\item  $G_{\sigma}$ is non-elementary
 					\item $G_{\sigma}\cong \cyc_{2}\times\mathbb{Z}$ and $G$  splits as an amalgamated free product
 				\begin{equation*}
 				G=G_{\sigma}\ast_{\subg{\sigma}}G'.
 				\end{equation*}
 				\end{itemize} 
 			\end{enumerate}
 		\end{enumerate}
 \end{thmx}
 
 The result above leaves us reasonably close to being able to characterize those hyperbolic groups admitting a sharply $2$-transitive action, and we hope that the latter can be achieved by refining the techniques presented here. We record the following partial result, which already covers a vast family of hyperbolic groups.   
  
 \begin{corx}
 	For a hyperbolic group $G$ that does not contain a copy of $\cyc_{2}\times\cyc_{2}$ the following are equivalent:
   \begin{itemize}
   	\item $G$ is sharply $2$-transitive.
   	\item The following conditions hold:
   	\begin{enumerate}[label=(\roman*)]
   		\item \label{ch: non elementary}it is non-elementary;
   		\item \label{ch: no finite normal subgroup} it admits no finite normal subgroup;
   		\item \label{ch: unique conjugacy}it has a unique conjugacy class of involutions;
   		\item \label{ch: centralizer infinite} the centralizer of any involution is infinite;
   		\item \label{ch:centralizer does not normalize} the centralizer of any involution does not normalize any finite subgroup of $G$ other than the one generated by the involution itself. 
   	\end{enumerate}
  \end{itemize}
 \end{corx} 
 \begin{proof}
 	The if part follows from \cref{c:hyperbolic case}. As for the only if part, we have already seen that 
 	$G$ cannot be elementary in \cref{non-elementary for free}, and thus that \ref{ch: non elementary} is necessary. 
 	The remaining conditions are well-known necessary conditions for an infinite group to admit a sharply $2$-transitive action, see \cite{tent2016infinite}. 
 \end{proof}

 \subsection{Robust groups of finite permutations and sharply $\Theta$-transitive actions}
  
  Our main result generalizes \cref{c:hyperbolic case} in two different directions. The first direction is the generalization of sharp $k$-transitivity studied in \cite{gonzalez2025two}, which we recall here. 
  Note that we will often think of positive integers set theoretically, i.e. as $k=\{0,\dots, k-1\}$.  
  Given an action $\lambda: U\noa G$ of a group $G$ on the right and some subset $A\subseteq U$ we write $\SStab_{\lambda}(A)$ for the setwise stabilizer of $A$ in $\lambda$.  
  
 \begin{definition} \label{d:theta-transitive}
 	Let $k \geq 1$ and let $U \noa G$ be a group action on a set $U$ with $|U| \geq k$. For $\Theta \leq \sym_k$, we say that the action $\lambda :U \noa G$ is \emph{sharply $\Theta$-transitive} if it is $k$-sharp, it is transitive on $k$-sets and if, for each $k$-set $A \subseteq U$, the action $A\noa \SStab_{\lambda}(A)$ is isomorphic to the permutation action $ k \noa \Theta$.
 	We say that $G$ is sharply $\Theta$-transitive, if it admits a sharply $\Theta$-transitive action.
 \end{definition}

 Recall that a $1$-transitive action of a group $G$ on a set is $k$-sharp if and only if the stabilizer $H$ of a point in the set is \emph{$k$-malnormal}, i.e. whenever an intersection of the form $\bigcap_{i\in k}H^{g_{i}}$ is non-trivial for $g_{0},\dots, g_{k-1}\in G$ we must have $Hg_{i}=Hg_{j}$ for at least one pair of distinct $i,j\in k$.
 Here we use the convention $Z^{g}=g^{-1}Zg$ for an element or subgroup $Z$ of $G$, as well as the identification of $k$ with the set $\{0,\dots, k-1\}$. 
 
 Note that saying that an action of a group $G$ on a set $U$ is sharply $\sym_k$-transitive, where $\sym_{k}$ is regarded as a subgroup of itself, is the same as saying that it is sharply $k$-transitive.

 \begin{notation} \label{d:non-free part}
 	Given an action on the right $\mu : U \noa G$ of a group on a set, and let $\Omega \leq G$. We define $\NFr_\mu(\Omega)$ to be the union of the non-free orbits of the action given by the restriction of $\mu$ to $\Omega$. 
 \end{notation}
 
 \begin{definition}\label{d:robust permutation group}
 	Let $\Theta \leq \sym_k$. We write $\pi:  k \noa \Theta$ for the permutation action of $\Theta$ on $k$. We say that $\Theta$ is \emph{robust} if for each $\Omega \leq \Theta$ and each non-empty $\Omega$-invariant subset $X \subseteq \NFr_\pi(\Omega)$, we have that $|\Omega|$ does not divide $|X|$. We say that $\Theta$ is \emph{transitive} if its action on $k$ is transitive.
 	
 	We say that the action is of \emph{generalized characteristic $2$} if any element $g\in G\setminus \{1\}$ which preserves some $k$-set does not have any fixed points outside this set.  
 \end{definition}
 Note that for $\Theta=\sym_{k}$ sharp $\Theta$-transitivity is just the same as sharp $k$-transitivity. 
 
 In order to try to not let notational complexity get in the way of conveying the core idea, we will content ourselves with stating our results for transitive robust groups of finite permutations. Note that robustness depends only on the orbit of $\Theta$ under the action of $\sym_{k}$ by conjugation, and that if $\Theta$ is robust, so is any subgroup of $\Theta$.
  
 \begin{definition} \label{d: docile and unruly}
 	Let $\Theta \leq \sym_k$ be robust. We say that a subgroup $\Omega \leq \Theta$ is \emph{docile} if its permutation action on $k$ has at least one free orbit. Otherwise we say that $\Omega$ is \emph{unruly}.
 \end{definition}
  Note also that these are properties of $\Omega$ regarded as a concrete permutation group rather than an abstract group, and that if $\Omega$ is docile, then so is any conjugate of $\Omega$ in $\sym_k$.  
  
  \begin{notation}\label{ass SOmega}
  	Henceforth, when given a robust subgroup $\Theta\leq\sym_{k}$, we also assume that we have chosen some set $\D(\Theta)$ of representatives of the conjugacy classes in $\Theta$ of the docile subgroups of $\Theta$. We may and will assume that if $\Omega$ fixes a point in $k$, then $\Omega$ fixes $0\in k$, and given $\oid$ we write $\mathcal{S}_{\Omega}$ for some system of representatives of the cosets of $\Stab_{\pi}(0)\backslash\Theta$, so that $(0\cdot s)_{s\in\mathcal{S}_{\Omega}}$ is an enumeration of $\NFr_{\pi}(\Omega)$ without repetitions.   
  \end{notation}
 
 The following is an easy verification, see \cite[Lem. 2.9]{gonzalez2025two}. 
 \begin{lemma}
 	\label{l:which groups are robust}The following groups of permutations are robust: 
 	\begin{itemize}
 		\item The cyclic subgroup group $\cyc_{k}$ as a subgroup of $\sym_{k}$. 
 		\item The dihedral subgroup $\di_{m}\leq\sym_{m}$ for any odd number $m\geq 3$. 
 		\item The full symmetric group $\sym_{k}$ for $ k\leq 3$.
 		\item The alternating groups $\alt_{k}\leq\sym_{k}$ for $k\in\{4,5\}$. 
 	\end{itemize}
 \end{lemma}


One consequence of our main result \cref{t: main} is that the recipe for the construction of a finitely generated virtually free group given in \cite{gonzalez2025two} can be optimized by further reducing the number of additional generators one needs to add to the centralizer of any $\Omega\in\D(\Theta)$. 

 
 \begin{thmx}\label{thm:virtually free examples}
 	Assume that $\Theta$ is one of the following transitive robust permutation groups:
   \begin{enumerate}[label=(\roman*)]
   	\item  $\Theta=\Omega_{0}=\cyc_{p}\leq\sym_{p}$ for some prime $p$;
   	\item  $\Theta=\sym_{k}$ for $k\leq 3$;
   	\item  $\Theta=\alt_{k}\leq\sym_{k}$ for $k\in\{4,5\}$.
   \end{enumerate}
 		Then the group $G$ which is the result of taking the group $\Theta$ and amalgamating the group $\Omega\times\Z$ for each $\Omega\in\mathcal{D}(\Theta)$ admits a sharply $\Theta$-transitive action. 
 	\end{thmx}
 
 All of these examples should also admit a purely combinatorial proof, and in fact we provide such a proof that the group $\Z\ast(\cyc_{2}\times \Z)$ is sharply $2$-transitive in \cref{s: simple example}. Many other examples, including of finitely generated virtually free groups, result from \cref{t: main}, but in the absence of a good criterion as for which the resulting examples are the most interesting or representatives, let us finally move the our main result.

\subsection{Acylindrical actions on hyperbolic spaces}
 \label{ss: acylindrical actions}

  Recall that given some action $G\aon X$ of a group $G$ on a $\delta$-hyperbolic metric space $(X,d)$, an element $g\in G$ is called \emph{elliptic} if one (equivalently all) the orbits of $\subg{g}$ in $X$ are bounded, and that it is \emph{loxodromic} if the induced map of $\partial X$ fixes exactly two points in $\partial X$, denoted by $g^{+}$ and $g^{-}$, and for any point $p\in X$ the sequence of points $(g^{n}p)_{n\geq 0}$ converges to $g^{+}$ and the sequence of points $(g^{-n}p)_{n\in\N}$ converges to $g^{-}$. Two loxodromic elements $h_{1},h_{2}$ are said to be \emph{independent} if $\{h_{1}^{+},h_{1}^{-}\}\cap\{h_{2}^{+},h_{2}^{-}\}=\emptyset$. 
  
  The study of acylindrical actions on negatively curved spaces, whose roots can be traced to the work of Bestvina-Fujiwara \cite{bestvina2002bounded} and Sela's notion of $k$-acylindricity for actions on trees \cite{sela1997acylindrical}, has become a cornerstone of contemporary geometric group theory. The following definition was introduced in \cite{osin2016acylindrically}, one of the two seminal works on the subject, together with \cite{dahmani2017hyperbolically}.
 	\begin{definition}\label{d:acylindrical}
	 	We say that the action of a group $G$ on a Gromov-hyperbolic metric space $(X,d)$ is \emph{acylindrical} if for every	$B>0$ there are $L=L(B)$, $N=N(B)>0$ such that for all $x,y\in X$ with $d(x,y)\geq L$, we have
	   	\begin{equation*}
	   		|\{\,g\in G \,|\, d(x, gx)\leq B,\,d(y,gy)\leq B\,\}|\leq N.
	   	\end{equation*}  
 	\end{definition}
    
  The following is Theorem 1.1 in \cite{osin2016acylindrically}.
  \begin{theorem}\label{osin trichotomy}
  	If an action $G\aon X$ on a Gromov-hyperbolic space is acylindrical, then one of the following three conditions must hold:
  	\begin{enumerate}[label=(\alph*)]
  		\item $G$ has bounded orbits. 
  		\item $G$ is virtually cyclic and contains a loxodromic element. 
  		\item  $G$ contains infinitely many pairwise independent loxodromic elements.
  	\end{enumerate}
  \end{theorem}
  In the last case we say that the action of $G$ is \emph{non-elementary}. 
  A subgroup $H\leq G$ (resp. $g\in G$) such that the restriction of the action to $H$ (resp. $\subg{g}$) satisfies the first condition is called \emph{elliptic}.

  \cref{t: main} provides a set of conditions on a non-elementary acylindrical action of a group $G$ on a Gromov-hyperbolic metric space that ensures that $G$ is sharply $\Theta$-transitive in generalized characteristic $2$. Beyond the existence of such actions, we believe that there is a sense in which some of the actions constructed here are ``generic" from the point of view of logic, and, more precisely of their positive theory with respect to some natural signature encoding the action.  \\

  \subsection{Acylindrically hyperbolic groups and our main result } 
  
   Given a $\delta$-hyperbolic $(X,d)$, we found it useful for us to adopt the following convention. For any geodesic arc $J$ between points $x$ and $y$ we also fix a map $\pi_{J}:X\to J$ such that 
   \begin{itemize}
    	\item $(\pi_{J})_{\restriction J}=Id_{J}$,
    	\item for all $z\in X$ we have $d(z,J)=d(z,\pi_{J}(z))$. 
   \end{itemize}
   To simplify the calculations, we will assume the distance is always minimized by some point, as it is the case for a graph with the unit length distance. We write $d_{J}(z,z')$ as an abbreviation for $d(\pi_{J}(z),\pi_{J}(z'))$. Similarly, we write $\diam_{J}(A)$ as an abbreviation for $\diam(\pi_{J}(A))$. We will fix an assignment of a geodesic arc $[p,q]$ between $p$ and $q$ for any two given points $p,q$. 
   
  
  \begin{definition}
  	\label{d:transverse} Fix an acylindrical action $G\aon X$ on a Gromov-hyperbolic metric space and some loxodromic element $\alpha$.  We say that $\alpha$ is \emph{transverse} to $H$ relative to some subset $\mathcal{S}\subseteq G$ with the property that $H\cdot\mathcal{S}\cdot\E(\alpha)=H\cdot\mathcal{S}$ if for some (equivalently, for all) $p\in X$ there is some constant $C=C(p)$ such that
  	\begin{equation*}
  	 \sup\{\,\diam_{[p,\alpha^{m} p]}(g^{-1}[p,hp])\,|\,m\geq 1,\,h\in H\,\}\geq C
  	\end{equation*}
  	implies that $g\in H \cdot\mathcal{S}$. \footnote{To avoid potential puzzlement here, let us point out that this condition is meant to make sense only when $H\cdot\mathcal{S}\cdot\E(\alpha)=H\cdot \mathcal{S}$, where $\E(\alpha)$ is the largest elementary subgroup of $G$ containing $\alpha$. }
  \end{definition}

  \newcommand{\datum}[0]{(G,H_{0},h_{\Omega},(X,d),\mu)_{\oid}}
  \begin{definition}
 	\label{d:nice datum} Fix a transitive robust finite group of permutations $\Theta\leq\sym_{k}$. By an \emph{ample hyperbolic $\Theta$-seed} we mean a tuple $\datum$, where 
 	\begin{itemize}
 		\item $(X,d)$ is a Gromov-hyperbolic metric space;
 		\item $G$ is a countable group;
 		\item $\mu:G\actson X$ is an acylindrical non-elementary action by isometries;  
 		\item  $\Theta\leq G$;
 	\end{itemize}
 subject to the following additional conditions. 
 	\begin{enumerate}[label=(\Roman*)]
 		\item \label{conditions on finite subgruops}Every finite subgroup of $G$ that can be embedded into $\sym_{k}$ is conjugate to a subgroup of $\Theta$.
 		\item \label{conditions on Omega} For every $\oid$ we have: 
 		\begin{enumerate}[label=(\alph*),ref=(\Roman{enumi}\alph*)]
 			\item \label{GOmega non elementary}the action of $N_{G}(\Omega)$ on $X$ is non-elementary;
 			\item \label{GOmega K} $\Omega$ is the maximal finite subgroup of $G$ normalized by $N_{G}(\Omega)$;
 			\item \label{hOmega loxodromic}$h_{\Omega}$ is a loxodromic element of $N_{G}(\Omega)$.  
 		\end{enumerate}
 	 \item \label{existence of H0} There is $H_{0}\leq G$ which is \emph{geometrically taut with respect to $(h_{\Omega})_{\oid}$}, by which we mean:
 	  \begin{enumerate}[label=(\alph*),ref=(\Roman{enumi}\alph*)]
      \item \label{h0 set tautness} the action by multiplication $H_{0}\backslash G\noa \Theta$ agrees with the standard action of $\Theta$ on $ k$ on $\{H_{0}\theta\}_{\theta\in\Theta}$ and is free elsewhere;
      \item \label{h0 malnormal}   $H_{0}$ is $k$-malnormal in $G$;
 	  	\item \label{h0 transversality} for each $\oid$ the element $h_{\Omega}\in G_{\Omega}$  is transverse to $H_{0}$ relative to $\mathcal{S}_{\Omega}$.
 	  \end{enumerate} 
 	\end{enumerate}
 \end{definition}

  \begin{observation}\label{taut implies infinite index}
 	 Note that if $H\leq G$ is geometrically taut in the sense of the definition above, then $|G:H|=\infty$, since $\{Hh_{\{1\}}^{m}\}_{m\in\Z}$ is an infinite set of distinct cosets. 
 \end{observation}
  
  The notion of a geometrically taut subgroup of $G$ can be seen as a geometric version of that of a taut partial action introduced in \cite{gonzalez2025two}. Let us emphasize that although the given action of $G$ on $X$ is on the left, the actions of $G$ to be constructed here will be on the right. This difference plays an important role in our argument.
 
    In many cases, the techniques developed here extend to a variant of the definition above in which the $N_{G}(\Omega)$ is elementary for one $\Omega\in\D(\Theta)$. We describe the relevant differences in \cref{d:flat preseeds} and \cref{d:flat seed} below. We start with the following preliminary definition.
   \begin{definition}\label{d:flat preseeds}
    Suppose that we are given a transitive robust $\Theta\leq\sym_{k}$, as well as some distinguished
    $\Omega_{0}\in\D(\Theta)$ which is \emph{good}, by which we mean
    \begin{enumerate}[label=(\arabic*)]
    	\item \label{omega0 cyclic} $k\noa \Omega_{0}$ has a single free orbit;
      \item \label{omega0} $\Omega_{0}$ is abelian;
    	\item \label{self normalized} $N_{\Theta}(\Omega_{0})=\Omega_{0}$;
    	\item \label{no subgroups} $\Omega_{0}\cong\cyc_{p}$ for some prime $p$. 
    \end{enumerate}   
    
     A \emph{flat hyperbolic $\Theta$-preseed} is defined just as an ample hyperbolic $\Theta$-seed, except for condition \ref{GOmega non elementary}, which for $\Omega_{0}$ is replaced by the following: 
    \begin{enumerate}
    	\myitem{$(IIa')$} \label{no-dihedrals}The group $N_{G}(\Omega_{0})$ is non-elliptic and isomorphic to $\Omega_{0}\times\Z$.
    \end{enumerate}
    We also require the following condition: 
    \begin{enumerate}
    	\myitem{$(IIc)$} \label{omega narrow} For any $p\in X$ there is some constant $C(p)>0$ such that 
    	  for any $q\in X$ and any  $\omega\in\Omega_{0}\setminus\{1\}$ we have 
    	  \begin{equation*}
    	  	d(q,\omega q)\geq 2d(q,N_{G}(\Omega_{0})\cdot p)-C.
    	  \end{equation*}
    \end{enumerate}
  \end{definition}
   
  \begin{remark}
  		In the definition above the function of \ref{self normalized} on $\Omega_{0}$ is expository rather than mathematical, i.e., not a restriction needed for the construction to work, but rather a necessary condition for $N_{G}(\Omega_{0})$ being virtually cyclic to be a possibility to begin with. We elaborate on this in \cref{ss: remarks on definition}. 
  \end{remark} 
   
   To complete the definition above  we first need to introduce a number of auxiliary objects. 
   
    \newcommand{\so}[0]{\mathcal{S}_{\Omega_{0}}} 
   
    \begin{definition}\label{d:auxiliary sets}
    	Let $\Theta$, $\Omega_{0}$ and $\datum$ be as in \cref{d:flat preseeds}, and let $H\leq G$ be a geometrically taut. We write $\G=N_{G}(\Omega_{0})$ and
    	\begin{align*}
         \mathcal{Z}(g)&=\so\cup gG_{\Omega_{0}},\quad\quad  &\mathcal{Y}_{H}(g)&=H\backslash [H\cdot \mathcal{Z}(g)],\\\mathcal{Z}^{k}(g)&=[\mathcal{Z}(g)]^{k},\quad\quad 
         &\mathcal{Y}_{H}^{k}(g)&=[\mathcal{Y}_{H}(g)]^{k}, \\
         \mathcal{Z}^{k}_{H}&=\bigcup_{g\in G\setminus (H\cdot\mathcal{S}_{\Omega_{0}})}[\mathcal{Z}(g)]^{k},&\mathcal{Y}_{H}^{k}&=\bigcup_{g\in G\setminus(H\cdot\so)}\mathcal{Y}_{H}^{k}(g).  
    	\end{align*}
    	Given $C\in[H\backslash G]^{k}$, we write 
    	 \begin{equation*}
    	 	\O_{H}(C)=\mathcal{Z}_{H}^{k}\cap\{\,Cg\,|\,g\in G\,\}.
    	 \end{equation*}
    	
    	We define an equivalence relation $\equiv_{H}$ on $\mathcal{Z}_{H}^{k}$ by writing 
    	$C\equiv_{H} D$ if there are $\gamma\in G$, $\rho\in\G$ such that     
    	\begin{equation*}
    		D=(C\cap\so)\cup(\gamma(C\setminus \so)\rho).
    	\end{equation*}
    	(Note that here $H$ does not play any role other than in the definition of the universe of the relation.)
    	A relation $\equiv_{H}^{*}$ refining $\equiv_{H}$ is defined in a similar way, but with the additional restriction that $\gamma\in H$. We will use the term $\T_{H}(C)$ to denote the equivalence class of $C$ under $\equiv_{H}$ and $\Ed_{H}(C)$ to denote its equivalence class under $\equiv_{H}^{*}$.  
    \end{definition}
  
    \begin{definition}\label{auxiliary graph}
    	 In the setting of the previous definition we define a bipartite multi graph (that is, we allow a priori self-edges and multiple edges between the same two points):
    	 	\begin{equation*}
    	 		\mathtt{G}(H)=(\mathtt{V}(H),\mathtt{E}(H))
    	 	\end{equation*}
    	 as follows.
    	\begin{enumerate}[label=(\Alph*)]
    		\item The set of vertices is partitioned as $\mathtt{V}(H)=\mathtt{V}_{0}(H)\coprod\mathtt{V}_{1}(H)$, 
    		where
    		\begin{equation*}
    			\mathtt{V}_{0}(H)=\{\,\T_{H}(C)\,|\,C\in\mathcal{Z}^{k}_{H}\,\},\quad\quad \mathtt{V}_{1}(H)=\{\O_{H}(C)\,|\,C\in\mathcal{Z}^{k}_{H}\}.
    		\end{equation*}
    		\item For every $T\in\mathtt{V}_{0}(H)$ and $O\in\mathtt{V}_{1}(H)$ the set of edges between $T$ and $O$ is identified with the collection of $E=\Ed_{H}(C)$, $C\in\mathcal{Z}_{H}^{k}$, such that 
    		 \begin{equation*}
    		 	E\subseteq T, \quad\quad H\backslash C\in O.
    		 \end{equation*}
    	\end{enumerate}
      We say that $H$ is \emph{supported on $\G$} if the graph $\mathtt{G}(H)$ does not contain any cycles (including double edges). We say that it is \emph{strongly supported on $\G$} if no two vertices in $\mathtt{V}_{0}(H)$ belong to the same connected component. 
    \end{definition}
 
    The following observation helps justify the above definition. 
    \begin{observation}
    	In the setting of the two definitions above, for any $H\leq H'\leq G$ such that the action $H'\backslash G\noa G$ is sharply $\Theta$-transitive of generalized characteristic $2$ the action of $\G$ on the right on $G\setminus(H'\cdot\so)$ is transitive. In particular, if $\G$ is abelian, then for any $C,C'\in\mathcal{Z}_{H'}^{k}$ such that $\T_{H}(C)=\T_{H}(C')$ the quotients $H'\backslash C$ and $H'\backslash C$ must be in the same orbit under the action of $G$ on the right. 
    \end{observation}
 

   \begin{definition}\label{d:flat seed}
    Given a robust transitive group $\Theta\leq \sym_{k}$ with a distinguished good $\Omega_{0}\in\D(\Theta)$, as in \cref{d:flat preseeds}. We say that a flat hyperbolic $\Theta$-preseed  $\datum$ is a \emph{hyperbolic $\Theta$-seed} if $H$ is supported on $\G=N_{G}(\Omega_{0})$. 
   \end{definition}


 \newcommand{\bodymainthm}[0]{
 Let $k>1$ be an integer and $\Theta\leq\sym_{k}$ a transitive robust subgroup as in \cref{d:robust permutation group}. If a group $G$ admits a hyperbolic $\Theta$-seed, then $G$ admits a sharply $\Theta$-transitive action of generalized characteristic $2$. 
}

\begin{thmx}
	\label{t: main}\bodymainthm
\end{thmx}

		We show in \cref{normalizer of Omega not virtually abelian} that $N_{G}(\Omega)$ cannot be virtually abelian for $\Omega\in\mathcal{D}(\Theta)$ whose action on $k$ has at least two free orbits, so the result is optimal in this particular direction.

An easy corollary of \cref{t: main} is the following. 

\begin{corx}\label{c:simple cases} Let $G$ be a group with no finite normal subgroups which admits a non-elementary acylindrical action on a Gromov hyperbolic space $X$. Then the following holds: 
 	\begin{enumerate}[label=(\Alph*)]
 		\item \label{theta1} $G$ admits an action on a set which is $k$-sharp and transitive on $k$-sets for all $k$.
 		\item \label{s2t}Assume that $G$ has a unique conjugacy class of involutions, $\sigma^{G}$, and writing $G_{\sigma}=C_{G}(\sigma)$ we have the following:
 		    \begin{itemize}
 		    	\item $\subg{\sigma}$ is the maximal finite subgroup of $G$ normalized by $G_{\sigma}$,
 		    	\item $G_{\sigma}$ is not elliptic and either $G_{\sigma}\aon X$ is non elementary or $G_{\sigma}\cong\subg{\sigma}\times \Z$.
 		    \end{itemize}
 	   Then $G$ is sharply $2$-transitive.
 	\end{enumerate} 	
 \end{corx}
 \begin{proof}
 	In case \ref{theta1} we take $\Theta=\{1\}\leq\sym_{k}$, while in case \ref{s2t} we take $\Theta=\sym_{2}$. Either way one may satisfy condition \ref{h0 set tautness} of the definition by letting $H_{0}=\{1\}$, since the action of $\Theta$ is itself free. Likewise, $H_{0}=\{1\}$ is trivially $k$-malnormal and the weak transversality condition \ref{h0 transversality} for $h_{\Omega}$ is vacuously satisfied by any loxodromic $h_{\Omega}$.
 \end{proof}
 
 Note that $G$ as in the assumptions of the theorem does not have any finite normal subgroup, and its action on $(X,d)$ is non-elementary, by \cite[Cor 1.5]{osin2016acylindrically}. It follows that $G$ contains no non-trivial proper normal abelian subgroup and thus in the case $\Theta=\sym_{2}$ all the examples of sharply $2$-transitive actions provided by the theorem are non-split. The same argument applies to the stabilizer of a point in the case $\Theta=\sym_{3}$, so the sharply $3$-transitive examples resulting from the theorem are likewise non-split. 
 

 \subsection{Potential lines of generalization}
We believe \cref{t: main} is not optimal as far as negative curvature conditions implying the existence of sharp $\Theta$-transitive actions. We propose some possible lines of generalization below. 
 
\begin{question*}
	Does our main theorem generalize in one or more of the following directions?
	\begin{itemize}
		\item By removing condition \ref{no subgroups} in \cref{d:flat preseeds}, which we suspect to be in no way essential and more of an artifact of an incomplete digestion of the arguments in \cref{s:preservation of supportedness}.
		\item Allowing $N_{G}(\Omega_{0})$ to be elementary for more than one value of $\Omega_{0}$ as in \cref{d:flat preseeds}. An issue might be the lack of natural examples providing motivation.
		\item Allowing $N_{G}(\Omega_{0})$ to contain dihedral groups. 
		\item Replacing the role of global acylindricity of the action by the existence of WPD elements with certain properties or perhaps some even weaker condition, like WWPD, see \cite{bestvina2016stable}.
		\item Replacing the requirement for the space to be Gromov-hyperbolic with some milder negative curvature condition, such as $CAT(0)$. 
	\end{itemize}
\end{question*}

Regarding the last two points, note that, indeed, it should not be difficult to generalize the argument given in the proof of \cref{almost free case} to show that the group $(\cyc_{2}\times \mathbb{Z}^{2})\ast\mathbb{Z}$ admits a sharply $2$-transitive action. However, the latter does not admit an acylindrical hyperbolic action in which the copy of $\mathbb{Z}^{2}$ is not elliptic, since any virtually abelian subgroup of $G$ that is non-elliptic must be virtually cyclic.   
 
 \subsection{Outline of the paper}
  In \cref{s:geometric preliminaries} we recall a few basic facts about Gromov hyperbolic spaces and quasigeodesics and establish all the relevant notation. In \cref{s:transverality and small cancellation} we elaborate on the definitions of transversality and small cancellation with respect to a point that we will use throughout the rest of the work, describing some of their interactions and proving the basic existence result \cref{existence of alpha}, needed later on. 
  
  From this point on, we consider some group $G$ acting acylindrically on some hyperbolic space $(X,d)$ and some subgroup $H\leq G$, and we begin studying the properties of a subgroup $H'\leq G$ resulting from adding to $H$ a number of elements of the form $\alpha_{i}=g_{i,-1}\alpha g_{i,1}^{-1}$, $i\in k_{o}$ satisfying a number of conditions in relation to a finite group $\Omega$ centralized by $\alpha$, and $H$. Here $k_{o}$ stands for the number of free orbits of the action of $\Omega$ on $k$. Of particular importance among them will be a transversality condition for $\alpha$ in relation to $H$, as well as in some cases some small cancellation type condition for $\alpha$. 
  
  Our end goal will be to understand the case in which $G$ is as in \cref{t: main} and $H$ some subgroup containing $H_{0}$ which is geometrically taut with respect to our distinguished collection $\{h_{\Omega}\}_{\Omega\in\D(\Theta)}$, as well as supported on $\G$ in the case in which the seed is flat. Here  $H$ should be thought of as an approximation from below to a subgroup $H_{\infty}\leq G$ for which the action $H_{\infty}\noa G$ is sharply $\Theta$-transitive, $\Omega\in\D(\Theta)$, and the elements $g_{i,\epsilon}$ above will be assumed to be such that  
   \begin{equation}\label{two ksets}
      A=\{H s\}_{s\in\mathcal{S}_{\Omega}}\cup\{H g_{i,-1}\omega\}_{\omega\in\Omega},\quad\quad A'=\{H s\}_{s\in\mathcal{S}_{\Omega}}\cup\{H g_{i,1}\omega\}_{\omega\in\Omega},
   \end{equation}
  are $k$-subsets of $H\backslash G$ invariant under $\Omega$, and $\alpha$ will be assumed to commute with $\Omega$
  and satisfy $\alpha\in H^{s}$ for all $s\in\mathcal{S}_{\Omega}$ (recall \cref{ass SOmega}).
  The addition of the elements $\alpha_{i}$ above to $H$ forces the projection of $A$ and $A'$ on $H'\backslash G$ to land in the same orbit under the action of $G$ on the right in a straightforward manner (see \cref{algebraic consequence of basic extension}). The difficulty lies in showing that under the right assumptions on the element $\alpha$ (which should be thought of as ``generic" in $C_{G}(\Omega)$) we have:
  \begin{itemize}
  	\item $H'$ is geometrically taut and so, in particular, $k$-malnormal;
  	\item $H\backslash G\onto H'\backslash G$ is injective on a sufficiently large set;
  	\item $H'$ is supported on $\G$ in the case in which the given seed is flat.
  \end{itemize}
  In accordance with this, we will work in one of two possible sets of assumptions on the extension:
  \begin{itemize}
  	\item one corresponding to the case in which the seed is ample or it is flat but $\Omega$ differs from the distinguished good $\Omega_{0}\in\mathcal{D}(\Theta)$, under which we say that the extension is \textit{generous};
  	\item another one corresponding to the case in which the seed is flat and $\Omega=\Omega_{0}$, under which we say that the extension is \textit{linear}. 
  \end{itemize}
  
  For the sake of avoiding a heavy context to burden some of the technical discussion, and better emphasizing the assumptions relevant to the arguments in each step, we will not introduce the above context right away. In \cref{s:extending groups using tuples of elements with a common small cancellation part} we use some of the results in \cref{s:geometric preliminaries} to show that if $\alpha$ is sufficiently ``geometrically generic" with respect to the action on our hyperbolic space $X$, then the group $H'$ generated by $H$ and $(\alpha_{i})_{i\in k_{o}}$ is isomorphic to the free product of $H$ and the free group $\F(\alpha_{i})_{i\in k_{o}}$ and that for any $h\in H'$ there is an arc from $p$ to $hp$, whose geometry reflects the normal form of $h$ in said free product and which we will refer to as the $(p,h)$-extension arc. Such an arc will contain a translate of $[p,\alpha p]$ for each occurrence of some $\alpha_{i}$ in the normal form of $h$, and can be forced to be $(1+\mu,\mu L)$-quasigeodesic for any arbitrary $\mu>0$, where $L=d(p,\alpha p)$, see \cref{l:p-paths}. We refer to any such subarc, or its translate inside some translate of the $(p,h)$-extension arc $J$, as an $\alpha$-subarc of $\tau$, and to the element $g\in G$ given by the construction such that $J=g[p,\alpha p]\subseteq\tau$ as the translating element. 
  Similarly, any element $h_{j}\in H'$ that is not conjugate into $H$ in $H'$ is loxodromic and admits a bi-infinite line invariant under the action of $h_{j}$ with the same properties, referred to as the $p$-extension axis of $h_{j}$. 
  
  The next goal is to show that the group $H'$ can be forced to be  $k$-malnormal, and also supported on $\G$ if the given seed is flat. The former is done in \cref{s:preservation of kmalnormality}, whereas the second is done in \cref{s:preservation of supportedness}. The goal of section \cref{s:fellow traveling paths} is to prove an important tool of importance to both of these two tasks, namely \cref{p:intersection_many3}. The setting is that of a finite collection $\{\gamma_{j}^{-1}\sigma_{j}\}$ of translates of such special arcs $\sigma_{j}$, $j\in k$, where 
  $\gamma_{i}^{-1}\sigma_{i}$ and $\gamma_{j}^{-1}\sigma_{j}$ are closely fellow traveling for every pair of $i,j\in k$, and the conclusion is that the translating elements of the $\alpha$-subarcs of the different $\sigma_{j}$ are tightly related to each other and the $\gamma_{j}$. In particular, we get a sequence of tuples in $[H\backslash G]^{k}$ with features that will later on allow us to exploit all our induction assumptions on $H$. The algebraic conclusion of \cref{p:intersection_many3} results from a number of preparatory geometric results throughout the section. The most important of which is  \cref{many fellow traveling paths basic result}, which shows that the $\alpha$-subarcs of the different $g_{j}^{-1}\tau_{j}$ must cluster into families of tightly fellow-traveling arcs. 
  
  Let us now sketch how all of this is used in \cref{s:preservation of kmalnormality}.  Assume for the sake of contradiction that we are given $h_{j}\in H'$, $j\in k$, of which at least one is not conjugate in $H'$ to an element of $H$, and  $g_{j}\in G$, $j\in k$ such that the cosets $H'g_{j}$ are all distinct and the value of the element $h_{j}^{g_{j}}$ is independent from $j$. Assuming here for convenience that none of the $h_{j}$ are conjugate to $H$ in $H'$, and letting $\zeta_{j}$ be the $p$-extension axis provided by \cref{s:extending groups using tuples of elements with a common small cancellation part}, then it is a standard fact that the axes $g_{j}^{-1}\zeta_{j}$ must $C$-fellow travel for a suitable constant $C$. 
  The sequence of $k$-tuples of sets in $[H\backslash G]^{k}$ provided by \cref{p:intersection_many3} can be shown to be infinite and periodic. We show that if we assume certain mild conditions on the two $k$-sets in \eqref{two ksets}, which are related to the assumption that each lies in a different orbit under the action of $G$ on the right, or generalize it in some way, then it follows that length of the sequence must be bounded.  
  
  The proof of the fact that $H'$ may be taken to be supported on $\G$ if the seed is flat and $H$ satisfies the property is stated as \cref{preservation of supportedness} and occupies the entirety of \cref{s:preservation of supportedness}. Although the result ultimately boils down to a similar application of \cref{p:intersection_many3} as in \cref{s:preservation of kmalnormality}, to reach that point a substantial amount of technical geometrical work is needed. 
   
  The proof of \cref{t: main} takes place in \cref{s:proof of main}, the induction step for the iterative construction involved being handled by \cref{l:main induction step}. The proofs in question do not have much intrinsic mathematical content and consist for the most part in assembling together and organizing results from previous sections. After some general remarks about the statement of the previous Theorem, we then move on to prove 
  \cref{c:hyperbolic case} and \cref{thm:virtually free examples}, which can be seen as corollaries of \cref{t: main}, with only very modest mathematical content of their own. 
  
  We conclude the article with \cref{s: simple example}, which contains a short combinatorial proof that the group $\Z\ast(\cyc_{2}\times \Z)$ is sharply $2$-transitive, meant to serve as a source of motivation for some of the features of the flat case of \cref{t: main}, as well as an inspiration for future generalizations.

 \section{Geometric preliminaries} 
  \label{s:geometric preliminaries}
  \subsection{Gromov-hyperbolic spaces.}
   \label{sus:Gromov hyperbolic spaces}
     \newcommand{\hd}[0]{D}
    We will now recall some basic definitions and facts about Gromov hyperbolic spaces. A detailed discussion can be found, for instance, in \cite[III, Ch. 1]{bridson2013metric}.   
    
    Given a metric space $(X,d)$, by an \emph{arc} we mean a continuous map $\tau$ from some compact interval $[a,b]\subseteq\R$ into $X$. We will write $\partial\tau=\{\tau^{-},\tau^{+}\}$ for the set of endpoints of $\tau$, where $\tau^{-}=\tau(a)$ is the \emph{initial endpoint} and $\tau^{+}=\tau(b)$ the \emph{terminal endpoint} of $\tau$. 
    We often use the term arc to refer merely to the image of the map in question, in a way that the context will make unambiguous. We will talk about the parametrization in temporal terms, referring to larger values of the parameter as later times. Likewise, given $x,y\in \im(\tau)$, we say that $x$ \emph{appears before $y$} in $\tau$ if $x=\tau(t)$, $y=\tau(s)$ for $t<s$. Recall that $\tau$ is \emph{$1$-Lipschitz} if $d(\tau(s),\tau(t))\leq|t-s|$ for all $s,t\in[a,b]$, and it is \emph{geodesic} if it is an isometric embedding. For a geodesic arc $\tau$ we let $|\tau|=d(\tau^{-},\tau^{+})$.
    A \emph{line} is defined in the same way as an arc, but taking as domain the entire real line. In other words, for us a line will always be bi-infinite. A subarc of an arc or line $\gamma:I\to X$ is just the restriction $\gamma_{0}$ of $\gamma$ to some subinterval $I_{0}\subseteq I$.  
    For convenience, we will work in a \emph{geodesic metric space}, i.e., one in which any two points are connected by a geodesic arc. 
    
    A geodesic triangle between three points $x,y,z\in X$ is a collection of three geodesic arcs, one between each pair of points in $\{x,y,z\}$.
    Given a geodesic arc $I$ between $x$ and $y$ and a geodesic arc $J$ between $x$ and $z$ write $\sigma_{I,J}$ for the partial involution of the disjoint union $I\coprod J$ that exchanges every point $q\in I$ with $d(x,q)\leq (y,z)_{x}$ with the unique point in $J$ at the same distance from $x$. We say that a geodesic triangle $\Delta$ is \emph{$\delta$-slim} if for any pair of sides $I,J$ of $\Delta$ every orbit of $\sigma_{I,J}$ has diameter at most $\delta$. 
    
    Recall that in a metric space $(X,d)$ and points $x,y,z\in X$ one defines the \emph{Gromov product} of $x$ and $y$ over $z$ as 
    \begin{equation*}
    	(x,y)_{z}:=\frac{1}{2}(d(x,z)+d(y,z)-d(x,y)).
    \end{equation*}
    It is an easy calculation to show that $(y,z)_{x}+(x,z)_{y}=d(x,y)$ for all points $x,y,z$. Note that for any metric space $(X,d)$ and $y_{0},z_{0}\in X$ the map $x\mapsto (x,y_{0})_{z_{0}}$ is $1$-Lipschitz. 
    
     As is well-known, a geodesic space is Gromov-hyperbolic if and only if it satisfies either of the conditions in the definition below for some $\delta>0$, and the satisfaction of either of the conditions for some $\delta>0$ implies the satisfaction of the other for some other $\delta'>0$ depending only on $\delta$. Therefore, we lose no generality in adopting the following definition. 
    \begin{definition}
    	For convenience, we will say that a geodesic metric space $(X,d)$ is $\delta$-hyperbolic if the following conditions are satisfied:
    \begin{itemize}
    	\item All geodesic triangles are $\delta$-slim. 
    	\item For all points $x, y, z, w \in X$, we have
            \begin{align*}
             	(x,z)_{w}\geq \max\{(x,y)_{w}, (y,z)_{w}\}-\delta.
            \end{align*}
    \end{itemize}
    \end{definition}

     The following is a standard fact easily derived by decomposing the quadrilateral in the statement into two triangles. 
     \begin{lemma}\label{quadrilaterals geodesic}
     	Let $\sigma,\sigma'$ be geodesic arcs with sets of endpoints $\{p,q\}$ and $\{p',q'\}$ respectively.
     	Then any point $x\in\sigma$ at distance at least $\max\{d(p,q),d(p',q')\}+\delta$ from $\{p,q\}$ is at distance at most $2\delta$ from a point in $\sigma'$.  
     \end{lemma}

%

    The following is a simple standard fact:
    \begin{lemma}\label{product and projection}
     If $J$ is a geodesic arc between $x$ and $y$, then $|d_{J}(x,z)-(y,z)_{x}|\leq 2\delta$. A similar inequality holds after exchanging the role of the two endpoints $x,y$.  
    \end{lemma}
    \begin{proof}
    	Let $z'=\pi_{J}(z)$ and $R=(z,y)_{x}$ and $D=d(z,J)$. Notice that if we write $c$ for the point in $J$ at distance $R$ from $x$, then  $ d(z,c)\leq (x,y)_{z}+\delta$. If, for instance $d(z,z')>R+2\delta$, then $z'$ is at distance at most $\delta$ from some $q\in[z,y]$ with $d(z,q)>(x,y)_{z}+2\delta$, so that $d(z',z)>(x,y)_{z}+\delta\geq d(z,c)$; a contradiction. 
    \end{proof}
    In particular, the map $\pi_{J}$ is unique up to bounded distance (i.e., well-defined as a coarse map). 
    In line with this, we note the following easy observation. 
    \begin{lemma}\label{product and projection2}
       Let $x,y,z,w\in X$ and let $J$ be a geodesic arc between $x$ and $y$. Then 
       \begin{equation*}
       	(y,w)_{x}\leq d(z,x)+d_{J}(z,w)+4\delta.
       \end{equation*}
     \end{lemma} 
     \begin{proof}
       To begin with, $d(x,w)\leq d_{J}(x,w)+2\delta$ by \ref{product and projection}. Hence
        \begin{align*}
        	(y,w)_{x}\leq d_{J}(x,w)+2\delta\leq  d_{J}(x,z)+d_{J}(z,w)+2\delta\leq d(x,z)+d_{J}(z,w)+4\delta.
        \end{align*} 	
     \end{proof} 
    
    
     \subsection{Quasigeodesics}
     \label{sus:quasigeodesics}
     
     Fix some geodesic $\delta$-hyperbolic space $(X,d)$. For a pair of positive constants $(K,C)$, we say that an arc or line $\tau:I\to X$ is \emph{$(K,C)$-quasigeodesic} if for all $t,s\in I$ we have 
     \begin{equation*}
     	K^{-1}|t-s|-C\leq d(\tau(s),\tau(t))\leq K|t-s|+C.
     \end{equation*} 
     Any action of a group $G$ on $(X,d)$ by isometries induces an action on the collection of $(K,C)$-quasigeodesic arcs (lines) by postcomposition.

     Given a subset $Y\subseteq X$ and $M>0$, by the \emph{$M$-neighbourhood of $Y$}, denoted by $\nb_{M}(Y)$, we mean the collection of points in $X$ at distance at most $M$ from some point in $Y$. Given $Y,Z\subseteq X$, we say that $Y$ and $Z$ are \emph{$M$-close} if $Y\subseteq\nb_{M}(Z)$ and $Z\subseteq\nb_{M}(Y)$. Two proper subsets $A,B\subseteq X$ are $M$-close if and only if their Hausdorff distance, denoted by $\dih(A,B)$ is at most $M$. 
     
    \begin{definition}\label{d:fellow travel}
      Given $M>0$, we say that two arcs (proper lines) $\tau,\tau'$ are \emph{$M$-fellow traveling} if 
      \begin{itemize}
      	\item $\tau$ and $\tau'$ are $M$-close;
      	\item given $s,t\in\dom(\tau)$ such that $d(\tau(s),\tau(t))>2M$ and $s',t'\in\dom(\tau')$ with 
      	\begin{equation*}
      		d(\tau(s),\tau'(t')),d(\tau(t),\tau'(t'))\leq M
      	\end{equation*}
      	 we have $s<t$ if and only if $s'<t'$;   
      	\item the previous condition holds after exchanging the roles of $\tau$ and $\tau'$.
      \end{itemize}
    \end{definition}

     An important feature of $\delta$-hyperbolic spaces is the fact that for most practical purposes the behavior of quasigeodesics is just as good as that of geodesics. In particular, two $(K,C)$-quasigeodesic arcs between the same pair of points have to stay close to each other. The first step is to show that a $(K,C)$-quasigeodesic arc has to be $M$-close to any geodesic arc between the same pair of points for some $M(K,C,\delta)$, what is generally known as the Morse lemma. 
     
     For some applications the following quantitative version of Morse Lemma from \cite{shchur2013quantitative}, \cite{gouezel2019corrected} will be needed.
     \begin{lemma}
     	\label{morse lemma}Let $\tau$ be a $(K,C)$-quasigeodesic and
     	$\sigma$ be a geodesic arc connecting its endpoints. Then $\tau$ belongs to the $H$-neighborhood of
     	$\sigma$, where $H=K^{2}(A_{1}C+A_{2}\delta)$, for universal constants $A_{1}$ and $A_{2}$. 
     \end{lemma}
     The important thing for us is for the constant $H$ to be affine in $C$. An easy argument allows one to reverse the direction of the previous result with only marginally worse constants (rather suboptimally). 
     \begin{lemma}
      \label{anti morse lemma} Let $H$ be as in \cref{morse lemma}, and $\tau$ be a $(K,C)$-quasigeodesic and
     	$\sigma$ be a geodesic arc between the endpoints of $\tau$. Then $\sigma$ belongs to the $M=(H+\frac{C}{2}+1)$-neighborhood of
     	$\tau$.  
     \end{lemma}
     \begin{proof}
     	Assume for the sake of contradiction the existence of some $q\in\sigma$ such that $\nb_{H+\frac{C}{2}+1)}(q)\cap\tau=\emptyset$. 
     	Note that since $\tau$ has the same endpoints as $\sigma$ and $M>H$, it cannot be the case that 
     	$q$ is at distance at most $M$ from one of the endpoints of $\sigma$. 
     	If we write $\sigma_{0}$ and $\sigma_{1}$ be the two connected components of $\sigma \setminus B(q,M)$ in $\sigma$,
     	then we can find pairs of points $p_{0},p_{1}\in\tau$ at arbitrarily small distance from each other and where 
     	$d(p_{i},r_{i})\leq H$ for some $r_{i}\in\sigma_{i}$. This violates the $(K,C)$-quasigeodesic condition on $\tau$. 
     \end{proof}
     
      \begin{corollary}\label{quasigeodesic arcs are close}
      	Let $K,C>0$ be positive constants and $H$ as in \cref{morse lemma}. Then any two $(K,C)$-quasigeodesic arcs between the same pair of points are $(2H+\frac{C}{2}+\delta+1)$-close to each other. 
      \end{corollary}

     Just as it is the case for bi-infinite geodesics, a quasigeodesic line must converge at infinity to a pair of points in the boundary. The following follows, for instance, from \cite[Thm. 5.35]{vaisala2005gromov}
     \begin{lemma}
     	For any quasigeodesic line $\tau:\R\to X$ there are two distinct points $\tau^{+},\tau^{-}\in\partial X$ such that 
     	$\lim_{t\to +\infty}\tau(t)=\tau^{+}$, $\lim_{t\to -\infty}\tau(t)=\tau^{-}$. 
     \end{lemma}

     There is a generalization of \cref{quasigeodesic arcs are close} for $(K,C)$-quasigeodesic lines between the same two points in $\partial X$. The following is standard. 
     \begin{lemma}\label{bounded displacement on axis}
     	Let $(K,C)$ be positive constants. Then: 
     	\begin{itemize}
     	  \item There is some constant $M(K,C,\delta)$ such that any two $(K,C)$-quasigeodesic lines $\tau,\tau'$ between the same pair of points in $\partial X$ must be $M$-close. 
     	  \item There is some constant $M'(K,C,\delta)$ such that for any $(K,C)$-quasigeodesic line $\tau$ any elliptic isometry $f:X\to X$ fixing $\tau^{+}$ and $\tau^{-}$ must satisfy $d(q,f(q))\leq M'$ for all $q\in\tau$. 
     	\end{itemize}
     \end{lemma}
     \begin{proof}
     	The first bullet point follows from \cite[Thm. 6.3.2]{vaisala2005gromov}. For the second one recall the fact, part of the proof of \cite[Lem. 3.4]{osin2016acylindrically} that if an element $g\in G$ satisfies $2(g^{-1}x,gx)_{x}\leq d(x,gx)-19\delta$, for some $x\in X$, then $g$ must be loxodromic. If an element fixing $\tau^{+}$ and $\tau^{-}$ satisfies 
     	\begin{equation*}
     		d(x,gx)\geq 100\max\{M(K,C,\delta),H(K,C,\delta)\}
     	\end{equation*}
     	for some $x\in\tau$, then it is not hard to see that the three points $g^{-1}x,x,gx$ are at distance $M(K,C,\delta)+H(K,C,\delta)$ from three points lying on a geodesic arc, and so are easily shown to satisfy the criterion above for $g$ to be loxodromic.
     \end{proof}

    We will use \cref{morse lemma} and \cref{anti morse lemma} in the following form. 

    \begin{corollary}
  	\label{quadrilaterals general}  For every $\lambda>0$ there is some constant $\mu_{quad}(\lambda)>0$ such that the following holds for any $L>\frac{\delta}{\mu_{quad}(\lambda)}$.  
  	\begin{enumerate}[label=(\Roman*)]
      \item \label{quadrilateral three subarcs}Suppose that we are given points $p_{0},p'_{0},p_{1},p'_{1}$, as well as 
	    	    $(1+\mu,\mu L)$-quasigeodesic arcs for some $\mu<\mu_{quad}(\lambda)$ as follows: 
	    	    \begin{itemize}
	    	     	\item $\tau_{i}$ from $p_{i}$ to $p'_{i}$ for $i\in\{0,1\}$;
	    	    	\item $\rho$ from $p_{0}$ to $p_{1}$ and $\rho'$ from $p'_{0}$ to $p'_{1}$.   
	      	  \end{itemize}
	          Then for $i\in\{0,1\}$ the arc $\tau_{i}$ can be decomposed into three subarcs (possibly degenerate),  $\tau_{i,0}$, $\tau_{i,1}$, $\tau_{i,2}$, in that order, such that:
	         \begin{itemize}
	          	\item $\tau_{i,0}\subseteq\nb_{\lambda L}(\rho)$, 
	          	\item  $\tau_{i,2}\subseteq\nb_{\lambda L}(\rho')$,
	           	\item $\tau_{0,1}$ and $\tau_{1,1}$ are $\lambda L$-fellow traveling. 
	        \end{itemize}
  	  \item \label{qc lines}\label{part3} Any two $(1+\mu,\mu L)$-quasigeodesic lines between the same points in $\partial X$ are $\lambda L$-fellow traveling. 
  	  \item \label{distance estimate projection to qc} Let $\tau$ be a $(1+\mu,\mu L)$-quasigeodesic and assume that we are given 
  	  geodesic subarcs $J_{0}$, $J_{1}$ of $\tau$, which are the image of real intervals $I_{0}$ and $I_{1}$, respectively. Let $q_{0},q_{1}\in X$ be such that $r_{l}=\pi_{J_{l}}(q_{l})$, $r\in 2$ are such that 
  	  $d(r_{0},J_{0}^{+})>\lambda_{0}L$, $d(r_{1},J_{1}^{-})>\lambda_{0}L$. Then  we have: 
  	  \begin{equation*}
  	   	d(q_{0},q_{1})\geq d(r_{0},r_{1})+\sum_{l\in 2}d(q_{l},J_{l}).
  	  \end{equation*}
  	\end{enumerate}
   \end{corollary}
  \begin{proof}
  	  By \cref{morse lemma} and \cref{anti morse lemma} it is possible to find $\mu_{0}\in(0,10^{-3}\min\{\lambda,1\}$ such that every geodesic arc $\gamma$ and every $(1+\mu,\mu L)$-quasigeodesic between the same two points 
  	  are $10^{-1}\lambda L$-fellow traveling. Part \ref{quadrilateral three subarcs} can be proved using the argument we sketch below. First, a standard argument partitioning our quadrilateral into triangles shows that if $\tau_{0},\tau_{1},\rho,\rho'$ are geodesic, then $\tau_{i,l}$, $i\in 2$, $l\in 3$ can be found such that 
  	   \begin{itemize}
	         \item $\tau_{i,0}\subseteq\nb_{4\delta}(\rho)$, 
	         \item  $\tau_{0,2}\subseteq\nb_{4\delta}(\rho')$,
	         \item $\tau_{0,1}$ and $\tau_{1,1}$ are $4\delta$-fellow traveling. 
	      \end{itemize}
  	  In the general case the result then follows by comparing the given $(1+\mu,\mu L)$-quasigeodesic arcs with geodesics $\tilde{\tau}_{l}$, $l\in 2$, $\tilde{\rho},\tilde{\rho}'$ between the same endpoints, and applying the observation above. Part \ref{qc lines} follows easily from the statement of \ref{quadrilateral three subarcs}.  
  	  Part \ref{distance estimate projection to qc} follows from a similar statement in the case that all the arcs involved are geodesics using a similar argument.
  	  
   \end{proof}

   \begin{corollary}\label{quadrilateral projection}
   	For each $\lambda>0$ there is some $\mu_{proj}(\lambda)>0$ such that the following holds for $L>0$ large enough.  
    Let $p,q,p',q'\in X$, $J,J'$ geodesic arcs from $p$ to $q$ and from $p'$ to $q'$ respectively, and write 
    \begin{equation*}
    \diam_{J}(J')=D.
    \end{equation*}
    Assume also that $\tau,\tau'$ are $(1+\mu,\mu L)$-quasigeodesics between those same pairs of points, where $\mu<\mu_{proj}(\lambda)$ and $D\geq \dfrac{\delta}{\mu}$. 
     Then there are subarcs $\tau_{0}\subseteq\tau,\tau'_{0}\subseteq\tau'$ such that
    \begin{equation*}
    	\diam(\partial\tau_{0}),\diam(\partial\tau'_{0})>D-\lambda L
    \end{equation*}
     and $\tau_{0}$ and $\tau'_{0}$ are $\lambda L$-fellow traveling. 
   \end{corollary}

    The following result is \cite[Lem. 6.5]{osin2025classifying}, see also \cite[Lem. 2.1]{osin2016acylindrically} and \cite[Thm. 5.16]{ghys2013groupes}.
    \begin{lemma}
  	   \label{lower bound chains}Let $x_{0},\dots, x_{n}$ be a sequence of points in a $\delta$–hyperbolic space $X$. Suppose
       that there exists $C\geq 0$ such that $(x_{j-1},x_{j+1})_{x_{j}}\leq C$ for all $1\leq j\leq n-1$ and $d(x_{j-1},x_{j})>2C+16\delta$ for all $1\leq j\leq n$. Then 
       \begin{equation*}
   	     d(x_{0},x_{n})\geq \sum_{j=1}^{n}d(x_{j-1},x_{j}) -2(n-1)(C+8\delta).
       \end{equation*}
    \end{lemma}
    
   The previous result is useful in order to ensure that certain paths are quasigeodesics with a suitable constant. The next lemma focuses on the idiosyncratic case featuring in the proof of our main result.

   \begin{lemma}
  	\label{constructing quasigeodesics} For every $\mu>0$ there is some constant $\nu:=\nu_{piece}(\mu)>0$ such that the following holds. Let $\tau$ be a $1$-Lipschitz arc consisting of a concatenation of arcs $\tau_{0},\dots, \tau_{2m+1}$ and assume that there is some  $L>\frac{\delta}{\nu}$  such that: 
  	\begin{enumerate}[label=(\roman*)]
  		\item \label{piecewise cond 1}$\tau_{2l}$ is $1$-Lipschitz and a $(1,\mu L)$-quasigeodesic for $0\leq l\leq m$,
  		\item \label{piecewise cond 2}$\tau_{2l+1}$ is geodesic of length at least $L$ for $0\leq l\leq m$,
  		\item \label{piecewise cond 3}for every $0\leq l\leq m$ and $0\leq j\leq 2m+1$ with $|j-(2l+1)|\in\{1,2\}$ we have 
  		\begin{equation*}
  			\diam_{\tau_{2l+1}}(\partial\tau_{j})<\nu L.                             
  		\end{equation*}
  	 
  	\end{enumerate} 
    Then $\tau$ is $(1+\mu,\mu L)$-quasigeodesic. 
  \end{lemma}
  \begin{proof} 
  	 Let us say that $\tau:[a,b]\to X$. Clearly the conditions in the statement imply that $\tau$ is $1$-Lipschitz. It only remains to show that for $\nu$ small enough the lower bound 
  	\begin{equation}\label{desired lower bound}
  		d(\tau(t),\tau(s))\geq(1-\mu)(t-s)-\mu L 
  	\end{equation}
  	holds for all $a\leq s< t\leq b$ when $\tau$ has the structure in the statement. 
  	 
  	Let $\mathcal{I}$ be the collection of all $0\leq l\leq m$ such that the distance between the endpoints of $\tau_{2l}$ is smaller than $6\nu L$.  
     We consider the sequence $0=t_{0}<t_{1}<\dots t_{n}\in[a,b]$, with corresponding points $x_{k}=\tau(t_{j})\in X$ for $j\in(n+1)$, obtained as follows. We go over all values of $0\leq l\leq m$ in increasing order. If $l\notin \mathcal{I}$, then we extend our sequence by $c<d$, where $[c,d]\subseteq [a,b]$ is the subinterval during which $\tau$ traces $\tau_{2l+1}$. Otherwise, we include in our sequence only $d$. Notice that in either case the final point $x_{n}$ in our sequence coincides with $\tau^{+}$.
      
     It is an immediate consequence of the construction that the restriction of $\tau$ to any interval of the form $[t_{k},t_{k+1}]$ is a $(1,12\nu L)$-quasigeodesic. It follows that it suffices to check \eqref{desired lower bound} for values of $t$ and $s$ that do not simultaneously belong to a single such interval.
     
    \begin{lemma}
    	\label{xi xj distance}If $\nu$ is small enough, then for all $0\leq i<i+1<j\leq m$ we have: 
    	\begin{equation}\label{first lower bound distance}
    		d(x_{i},x_{j})\geq\sum_{k=i}^{j-1}d(x_{k},x_{k+1})-13(j-i-1)\nu L.
    	\end{equation}
    \end{lemma}
    \begin{subproof}
     We would like to bound $(x_{k-1},x_{k+1})_{x_{k}}$ from above for $i<k<j$ so as to be able to apply \cref{lower bound chains}. Consider first the case in which $x_{k}$ is the origin of $\tau_{2l+1}$. In particular, $l\notin \mathcal{I}$. In this case \cref{product and projection2} directly yields $(x_{k-1},x_{k+1})_{x_{k}}\leq\nu L+2\delta$. If $x_{k}$ is the terminal point of $\tau_{2l+1}$, then we might as well assume that $l+1\in \mathcal{I}$, since the other subcase is symmetrical to the case just covered. Here \cref{product and projection2} yields 
     \begin{equation*}
     	   (x_{k-1},x_{k+1})_{x_{k}}\leq 4\delta+2\nu L\leq 5\nu L.
     \end{equation*} 
     It is clear that for $\nu$ sufficiently small we have $d(x_{k-1},x_{k})>10\nu L+16\delta$ for all $1\leq j\leq n$, since by construction $d(x_{k-1},x_{k})\geq\max\{6\nu L, (1-5\nu)L\}$ for all $0\leq j<n$. 
     It follows from \cref{lower bound chains} that 
     \begin{equation*}
     	d(x_{i},x_{j})\geq\sum_{k=i}^{j-1}d(x_{k},x_{k+1})-2(j-i-1)(6L\nu+8\delta)\geq \sum_{k=i}^{j-1}d(x_{k},x_{k+1})-13(j-i-1)\nu L.
     \end{equation*}
     for all $0\leq i<j\leq n$ from which the conclusion is immediate. 
    \end{subproof}
  
     Assume now we are given $0\leq i\leq j<n$ and $t_{i}\leq s\leq t_{i+1}$, $t_{j}\leq t\leq t_{j+1}$ with $s\leq t$ and $i<j$, then $d(\tau(s),\tau(t))$ is greater or equal than $d(x_{i},x_{j+1})-d(\gamma(s),x_{i})-d(\gamma(t),x_{j+1})$ by the triangular inequality and thus greater or equal than the value 
     \begin{equation*}\label{lower bound general}
     	 	   A:=d(x_{i},x_{i+1})-d(x_{i},\tau(s))+d(x_{j},x_{j+1})-d(x_{j+1},\tau(t))+\sum_{k=i+1}^{j}d(x_{k},x_{k+1})-13(j-i)\nu L.
     \end{equation*}
     	 
     Our final goal is to bound $A$ from below by a suitable positive slope affine function of the time difference $t-s$. 
     Write 
     \begin{equation}
     	A_{1}:=d(x_{i},x_{i+1})-d(x_{i},\tau(s))+d(x_{j},x_{j+1})-d(x_{j+1},\tau(t)), \quad \quad A_{2}:=\sum_{k=i+1}^{j}d(x_{k},x_{k+1}).
     \end{equation}
     We prove this in several steps. 
     
     \begin{lemma} \label{distance to time A1}
     	  The following inequality holds: 
     	  \begin{equation}\label{A1 inequality}
     	  	A_{1}\geq(t_{i+1}-s)+(t_{j}-t)-24\nu L.
     	  \end{equation}
     \end{lemma}
     \begin{subproof}
     	Notice that $\tau_{\restriction [t_{k},t_{k+1}]}$ is a $1$-Lipschitz $(1,12\nu L)$-quasigeodesic for every $0\leq k<n$. It follows that
     	\begin{align*}
     		d(x_{i},x_{i+1})-d(x_{i},\tau(s))\geq& t_{i+1}-t_{i}+2\nu L-d(x_{i},\tau(s))\geq\\ &(t_{i+1}-t_{i})+(t_{i}-s)+2\nu L\geq t_{i+1}-s-12\nu L.
     	\end{align*}
     	and a similar inequality holds for $t$ and $j$, hence the result.   
     \end{subproof}
     
     It is also not hard to see that since for all $0\leq k<n$ the arc $\tau_{\restriction [t_{k},t_{k+1}]}$ is a $(1,12\nu L)$-quasigeodesic we must have 
     \begin{equation}\label{A2 inequality}
     	    A_{2}\geq\sum_{k=i+1}^{j}d(x_{k},x_{k+1})\geq t_{j}-t_{i+1}-12(j-i-1)\nu L.
     \end{equation}
     
     Adding the inequalities \eqref{A1 inequality} and \eqref{A2 inequality} we obtain 
     \begin{equation}
     	A\geq (t-s)-25(j-i)\nu L.
     \end{equation}
     Therefore, in order to infer the desired result, it suffices to bound $(i-j)$ from above by some affine function of $\frac{1}{L}$. Notice that there are at least $\lfloor\frac{(j-i-1)}{2}\rfloor$ many values of $k$ such that $\tau_{2l+1}$ is traversed in its entirety by the path $\tau$ between times $s$ and $t$. Since the domain of $\tau_{2l+1}$ has length at least $L$ by \ref{piecewise cond 2} and the $1$-Lipschitz assumption, we have 
     $(t-s)\geq L\lfloor\frac{(j-i-1)}{2}\rfloor$ and thus  
     $
     	 (j-i)\leq 2\lfloor\frac{(j-i-1)}{2}\rfloor+1\leq \frac{2(t-s)}{L}+1,
     $
     as needed.
  \end{proof}

\section{Transversality and small cancellation}
  \label{s:transverality and small cancellation}
  
  \subsection{More about acylindrical actions and their loxodromic elements}
  
  It follows from \cite[Thm. 6.14]{dahmani2017hyperbolically} that if
  the action $G\aon X$ is acylindrical and $G$ has no bounded orbits, then $G$ admits a (necessarily unique) maximal finite normal subgroup. We will denote this subgroup by $\K(G)$.
  
  Fix an action by isometries of a group $G$ on a $\delta$-hyperbolic metric space $(X,d)$.

    \begin{definition}
    	Let $K,C$ be two positive constants. By a $(K,C)$-quasiaxis for a loxodromic element $h\in G$ we mean some $(K,C)$-quasigeodesic line $\tau:\R\to X$ such that for any sequence $(t_{n})_{n}\subseteq\R$ we have 
    	$\lim_{n}\tau(t_{n})=h^{+}$ for $t_{n}\to+\infty$ and $\lim_{n}\tau(t_{n})=h^{-}$ for $t_{n}\to-\infty$.
      A quasiaxis is defined as a $(K,C)$-quasiaxis for some $K,C>0$.   
  \end{definition}
  
    	Some authors require a $(K,C)$-quasiaxis to satisfy the stronger property of being invariant under $h$. The advantage is that a $(K,C)$-quasiaxis for $h^{m}$, $m>0$ is also one for $h$.  
      
      The following is shown in \cite{bestvina2002bounded}.
      \begin{lemma}\label{existence of quasiaxes}
      	For any loxodromic element $h$ and any $p\in X$ the concatenation of $h^{m}[p,hp]$, $m\in\Z$ (in the obvious order) 
      	is a quasigeodesic axis for $h$. 
      \end{lemma}

    The following is well-known, see for instance \cite[Lem 6.5, Cor 6.6]{dahmani2017hyperbolically}. The last assertion follows from the classification of virtually cyclic groups.
    \begin{lemma}\label{l:Eg}
    	Let $G$ be a group equipped with an acylindrical action $\mu$ on a hyperbolic space and $g\in G$ a loxodromic element. Then there is a maximal virtually cyclic group $\E(g)$ containing $g$. The following are equivalent for an element $h\in G$: 
    	\begin{itemize}
    		\item $h\in \E(g)$,
    		\item $h$ preserves the two points in $\partial X$ fixed by $g$,
    		\item there is some $n\in\N\setminus \{0\}$ such that $hg^{n}h^{-1}=g^{\pm n}$.
    	\end{itemize}
       Moreover, every torsion free element of $\E(g)$ is loxodromic. The collection of all elliptic elements of $\E(g)$ which  do not exchange the two point in the boundary fixed by $g$ is a normal subgroup of $\E(g)$ which we denote by $\K(g)$.
       
       There is a semidirect product decomposition $\E(g)=\K(g)\rtimes Z$, where $Z$ is either cyclic or infinite dihedral, and we define $\K(g)=\K(\E(g))$ (note that $\E$ stands for two different definitions on each side). 
    \end{lemma}
  
    \begin{definition}\label{loxodromic equivalence}
    	Given a loxodromic element $h$, we write $E^{+}(h)$ for the subgroup of elements of $\E(h)$ which do not exchange $h^{+}$ and $h^{-}$. Given two loxodromic elements $h_{0}$ and $h_{1}$, we write $h_{0}\approx h_{1}$ if $h_{0}$ and $h_{1}$ have commuting powers, and $h_{0}\sim h_{1}$ if for some $g\in G$ we have $h_{0}^{g}\approx h_{1}$.  
    \end{definition}
    Note that by \cref{l:Eg} $h_{0}$ and $h_{1}$ are independent if $\E(h_{0})\neq \E(h_{1})$, i.e., if $h_{0}\napprox h_{1}^{\pm1}$. 
    
    The following follows from \cite[Lem. 6.7]{dahmani2017hyperbolically}. 
   \begin{lemma}\label{stable axes}
  	Let $G$ be a group acting acylindrically on a hyperbolic space $X$ and $h\in G$ a loxodromic element. Then for any constants $B,K,C>0$, and any $(K,C)$-quasiaxis $\tau$ of $h$, there exists $M>0$ with the following property. Let $g_{0},g_{1}\in G$ and assume that we are given subarcs
  	$\tau_{0},\tau_{1}\subseteq\tau$ such that  
  	\begin{itemize}
  		\item $d(g_{0}\tau_{0}^{+},g_{1} \tau_{1}^{+}),d(g_{0}\tau_{0}^{-},g_{1} \tau_{1}^{-})\leq B$,
      \item $\diam(\partial \tau_{i})\geq M$ for $i\in\{0,1\}$. 
  	\end{itemize}
     Then $g_{0}^{-1}g_{1}\in \E^{+}(h)$.  
  \end{lemma}
  From the lemma above, together with the second part of \cref{bounded displacement on axis}, one can deduce the following, see for instance \cite[p.7]{coulon2014theorie}. 
  \begin{lemma}\label{parallel axes}
     		Let $G$ be a group acting acylindrically on a hyperbolic space $X$ and $h_{0},h_{1}\in G$ be loxodromic elements. Then for any constants $B,K,C>0$, and $(K,C)$-quasiaxis $\tau_{i}$ of $h_{i}$, $i\in\{0,1\}$ there exists $M>0$ with the following property. Let $g\in G$ and assume that we are given subarcs $\rho_{i}\subseteq\tau_{i}$ such that  
  \begin{itemize}
  		\item $d(\rho_{0}^{+},g\rho_{1}^{+}),d(\rho_{0}^{-},g\rho_{1}^{-})\leq B$, 
      \item $\diam(\partial\rho_{i})\geq M$ for $i\in\{0,1\}$. 
  \end{itemize}
     Then $gh_{1}g^{-1}\in \E(h_{0})$ and thus $h_{0}\sim h_{1}^{\epsilon}$ for some $\epsilon\in\{\pm 1\}$.  
   \end{lemma}

  \begin{lemma}\label{intersection Ks}
    Let $G$ be a group admitting and acylindrical action on a Gromov-hyperbolic space and let $h_{0},h_{1}$ be loxodromic elements such that $h_{0}\napprox h_{1}$. Then there is $N>0$ such that for any two integers $n,n'\geq N$ the element $h=h_{0}^{n}h_{1}^{n'}$ is loxodromic with $\K(h)\subseteq\K(h_{0})\cap\K(h_{1})$.        	
  \end{lemma}
  \begin{proof}
  	The argument can be found in the proof of Proposition $6$, part (v) in \cite{bestvina2002bounded}.
  	For some fixed pair $(K,C)$ of positive constants independent of the choice of $n$ and $n'$ such that for any $M>0$ there is $N>0$ such that if $n,n'\geq N$ there is a $(K,C)$-quasiaxis for $h$ that contains a translate of subarcs of diameter at least $M$ of $(K,C)$-quasigeodesic axes for $h_{0}$ and $h_{1}$. The desired conclusion follows from \cref{stable axes} by taking $M$ large enough. 
  \end{proof}

  Non-elementary acylindrical actions always contain many non-equivalent loxodromic elements, as first shown in \cite{bestvina2002bounded}. 
  \begin{lemma}\label{many loxodromic}
  	Let $G$ be a group equipped with an acylindrical action on a hyperbolic space $(X,d)$, and let $G_{0}\leq G$ be a non-elementary subgroup. Write $\Omega=\K(G_{0})$ and assume that $\Omega$ is the largest finite subgroup of $G$ normalized by $G_{0}$. 
  	Then there are $\{f_{i}\}_{i\in\N}\subseteq G_{0}$ such that 
  	\begin{itemize}
  		\item $f_{i}^{-1}\nsim f_{i}\nsim  f_{j}$ for all distinct $i,j\in\N$;
  		\item $\K(f_{i})=\Omega$ for all $i\in\N$. 
  	\end{itemize}
  \end{lemma}
  \begin{proof}
  	 The existence of $\{f_{j}\}_{j\in\N}$ satisfying the first item follows from combining Proposition $2$ and Proposition $6$, part (v) in \cite{bestvina2002bounded}. The additional condition $\K(f_{i})=\Omega$ can be established by inspecting the construction in the proof of Proposition $2$. 
  	 
  	 There we have $f_{j}=g_{1}^{n_{j}}g_{2}^{m_{j}}g_{1}^{k_{j}}g_{2}^{-l_{j}}$ for suitable positive integers $n_{j},m_{j},k_{j},l_{j}$ which grow to infinity with $j$, and two loxodromic elements $g_{1},g_{2}$, of which it is only required that $g_{1}\nsim g_{2}$. This yields the condition $f_{j}\nsim f_{i}$ for distinct $i,j$, and the construction is then iterated in order to ensure that $f_{i}\nsim f_{i}^{-1}$ holds.
  	 
  	 Therefore, for the result to follow it suffices to be able to ensure that $\K(g_{1})=
  	 \Omega$. Such element $g_{1}$ can be easily constructed by applying \cref{intersection Ks} to some finite family $\{h^{g_{i}}\}_{i\in \mathcal{I}}$ of an arbitrary loxodromic element $h$, where $g_{i}\in G_{0}$ for $i\in\mathcal{I}$ and $\bigcap_{i\in\mathcal{I}}\K(h^{g_{i}})=\Omega$ (note that the existence of $h$ such that $\K(h)\cap G_{0}=\Omega$ can be ensured by applying \cite[Thm. 6.14]{dahmani2017hyperbolically} to $G_{0}$). 
  	 The existence of a suitable finite family $\{g_{i}\}_{i\in\mathcal{I}}$ follows easily from the fact that $\K(h)$ is finite and depends only on $\E(h)$, as well as the fact that $\Omega$ is the largest finite subgroup normalized by $G_{0}$.
   \end{proof}

   \begin{definition}\label{d:stable elements}
   	For $p\in X$ and $\nu>0$, let $h\in G$ be a loxodromic element and write $L=d(p,hp)$. 
   	We say that $h$ is \emph{$(p,\nu)$-stable} over some finite set $\mathcal{G}\subseteq G$ if letting $L=d(p,hp)$ we have: 
   	\begin{enumerate}[label=(\roman*)]
   		\item $10^{4}<\nu L$; 
   		\item for every $h',h''\in\E(h)$ and $g\in G$ the condition $\diam_{[p,h'p]}(g[p,h''p])>\nu L$ implies that $g\in\E(h)$; 
   		\item \label{concatenation quasigeodesic}the concatenation of the arcs $h^{m}[p,hp]$, $m\in\Z$ is a $(1+\nu,\nu L)$-quasigeodesic $\tau$;
   		\item any other point in $\E(h)\cdot p$ lies at distance at most $\nu L$ from $\tau$; 
   		\item for any $g\in \mathcal{G}\cup\K(h)$ and any $h'\in\E(h)$ we have $d(p,gp)\leq L$;
   	\end{enumerate}
   	 We say that it is \emph{positively $(p,\nu)$-stable (over some set)} if in addition to all of the above we have $\E(h)=\E^{+}(h)$.
   \end{definition}
   
   The following is a consequence of \cref{stable axes} and \cref{lower bound chains}. 
   \begin{observation}\label{asymptotically stable}
   	If the action $G\aon X$ is acylindrical, then for any loxodromic element $h$, any $p\in X$, and any $\nu>0$ there is $M=M(h,p,\nu)>0$ such that for all $m\geq M$ the element $h^{m}$ is strongly $(p,\nu)$-stable. 
   \end{observation}
   
    \begin{observation}\label{transversality is inherited}
     Assume that for some $p\in X$ and $\nu>0$ a loxodromic element $\alpha$ is $(p,\nu)$-stable and also $(p,\nu)$-transverse to some subgroup $H\leq G$ relative to some finite set $\mathcal{S}\subseteq G$ such that 
     \begin{equation*}
     	H\cdot \mathcal{S}\cdot \E(\alpha)=H\cdot\mathcal{S}.
     \end{equation*}
     Then any $\rho\in\E(\alpha)$ is $(p,5\nu)$-transverse to $H$ relative to $\mathcal{S}$.
    \end{observation}
     \begin{proof}
     	The result is immediate from the observation that $[p,\rho p]$ must $(\nu L+2\delta)$-fellow travels with some subarc of $[p,\alpha p]$.  
    \end{proof}
     
   \subsection{Existence of small cancellation transverse elements}
   \label{sus:loxodromic elements small cancellation}
    Fix some $G$ acting acylindrically and non-elementarily on a geodesic $\delta$-hyperbolic space $(X,d)$, as well as some point $p\in X$. 
   
      \begin{definition}
    	\label{d:transverse precise} Let $\nu>0$, $p\in X$ and $H$ some subgroup of $G$. We say that a loxodromic element $\alpha\in G$ is $(p,\nu)$-transverse to $H$ relative to a finite set of elements $\mathcal{S}\subseteq G$ if for any $g\in G$ and $h\in H$ the inequality   
    	\begin{equation*}
  	     \diam_{[p,\alpha p]}(g^{-1}[p,hp])\geq\nu d(p,\alpha p)
  	  \end{equation*}
      implies that $g\in H\cdot\mathcal{S}$.  
      
      We say that $\alpha$ is \emph{neatly transverse ($(p,\nu)$-transverse) to $H$ relative to $\mathcal{S}$} 
 		  if it is transverse (resp. $(p,\nu)$-transverse) to $H$ relative to $\mathcal{S}$ and, moreover, we have $\alpha\in H^{s}$ for all $s\in\mathcal{S}$.  
    \end{definition}
     
     The following observation is clear. 
     \begin{observation}
     	The following are equivalent for a loxodromic element $\alpha\in G$ and $H\leq G$:
     	\begin{itemize}
     		\item $\alpha\in G$ is transverse to $H$ relative to $\mathcal{S}$;
     		\item for all $p\in X$ and $\nu>0$ there is $m_{0}>0$ such that $\alpha^{m}$ is $(p,\nu)$-transverse to $H$ relative to $\mathcal{S}$ for all $m\geq m_{0}$;
     		\item for any quasiaxis $\tau$ for $\alpha$ and any $D>0$ we have :
     		\begin{equation*}
     			\sup\,\{\,\diam(\nb_{D}(\tau)\cap g[p,hp])\,\,|\,\,h\in H,\,g\in G\setminus (H\cdot\mathcal{S})\,\}<\infty.
     		\end{equation*}
     	\end{itemize}
     \end{observation}

   \begin{definition} \label{d:small cancellation pointed}Let $p\in X$ and $\nu\in\R_{>0}$. We say that a loxodromic element $\alpha$ is $(p,\nu)$-small cancellation over some finite set $\mathcal{G}\subseteq G$ if the following conditions are satisfied, where $L=d(p,\alpha p)$:
  	\begin{enumerate}[label=(\roman*)]
  	\item $\nu L \geq 10^{3}\max\{1,\delta\}$;
  	\item $d(g p,p)\leq \nu L$ for all $g\in\mathcal{G}\cup\K(\alpha)$; 
  	\item  for all $g\in G$ if $\diam_{[p,\alpha p]}(g[p,\alpha p])\geq\nu L$, then $g\in \K(\alpha)$.
  	\end{enumerate}
  \end{definition}
  Note how small cancellation is a strengthening of the notion of stability from \cref{d:stable elements} modulo adjusting the constant $\nu$.

  There is a great deal of literature on the small cancellation condition in acylindrically hyperbolic groups, see \cite{hull2017small}, and the existence result below follows a standard construction. In view of the fact that 
  we work directly on a given acylindrical action over which we have only limited control, the idiosyncrasy of the notion of small cancellation relative to a basepoint, as well as in order to account for the additional transversality statement in a way that is maximally transparent, we have chosen to give an account as self-contained as possible.   
  
  \begin{lemma}\label{small cancellation construction}
  	Suppose we are given $H,G_{0}\leq G$, some finite $\Omega=\K(G_{0})$, and $\nu>0$. Assume that we are given two loxodromic elements $h_{0},h_{1}\in G_{0}$, as well as an integer $m>0$
  	and a finite set $\mathcal{S}\subseteq G$ with the following properties:
  	\begin{enumerate}[label=(\roman*)]
  		\item \label{a G0 and S}$|G_{0}:G_{0}\cap\bigcap_{s\in\mathcal{S}}H^{s}|<\infty$;
  		\item \label{a different axes}$h_{0}\nsim h_{1}^{\pm1}$, $h_{1}\nsim h_{1}^{-1}$;
  		\item \label{a Omega}$\K(h_{0})\geq\Omega$ and $\K(h_{1})=\Omega$;
  		\item \label{a transv} $h_{0}$ is transverse to $H$ relative to $\mathcal{S}$.
  	\end{enumerate}
    Then for some integers $e_{0},e_{1},m\geq 2$ the element
    \begin{equation} \label{ee_word0}
    \alpha=h_{0}^{e_{0}2m}h_{1}^{e_{1}}h_{0}^{e_{0}(2m+1)}h_{1}^{e_{1}}\dots h_{0}^{e_{0}(3m-1)}h_{1}^{e_{1}}
    \end{equation}
    satisfies
    \begin{enumerate}[label=(\alph*)]
    	\item \label{c-alpha is small cancellation}$\alpha$ is $(p,\nu)$-small cancellation,
    	\item \label{c-K(alpha)}$\K(\alpha)=\Omega$,
    	\item \label{c-alpha commutes}$\alpha$ commutes with $\Omega$, 
    	\item \label{c-alpha transversality}$\alpha$ is neatly $(p,\nu)$-transverse to $H$ relative to $\mathcal{S}$.
    \end{enumerate}
  \end{lemma}
  \begin{proof}
  	Fix $\nu>0$. Write $L_{i}=d(h_{i}^{e_{i}}p,p)$ for $i\in\{0,1\}$. 
  	Note first that there is some integer $m_{0}>0$ such that if $m_{0}$ divides both $e_{0}$ and $e_{1}$, then for $i\in\{0,1\}$ we have 
  	\begin{equation}\label{commutation, neatness}
  		[h_{i}^{e_{i}},\Omega]=\{1\},\quad\quad h_{i}^{e_{i}}\in\bigcap_{s\in\mathcal{S}}H^{s}.
  	\end{equation}
    We will henceforth assume $e_{0},e_{1}$ satisfy said condition, so that \ref{c-alpha commutes} in the statement holds and for  \ref{c-alpha transversality} only the pure transversality condition remains to be established.
  	
  	Fix some $\zeta>0$, which will be determined a posteriori.
  	There is some constant $C$ such that 
  	$(h_{i}^{-me_{i}}p,h_{j}^{me_{j}}p)_{p}\leq C$ 
  	for all $m$. Therefore, as long as $e_{0},e_{1}$ and $\frac{e_{0}}{e_{1}}$ are large enough we have:
  	\begin{enumerate}[label=(\arabic*)]
  		\item \label{m1} $L_{0}\geq 3L_{1}\geq\frac{1}{\zeta}\max\{\delta,C\}$,
  		\item \label{m2} $(h_{i}^{-e_{i}}p,h_{j}^{e_{j}}p)_{p}\leq C$ for all $i,j\in\{0,1\}$,
  		\item \label{m3} $h_{0}^{e_{0}}$ is $(p,\frac{1}{3}L_{0})$-transverse to $H$.
  	\end{enumerate}

   	Using \cref{lower bound chains} it is easy to show the following: 
   	\begin{observation}\label{qg small cancellation}
   		There is a function $\zeta_{0}:\R_{>0}\to \R_{>0}$ such that for any $\mu>0$ 
   		conditions \ref{m1}-\ref{m3}, together with the inequality $\zeta<\zeta_{0}(\mu)$ imply that for any reduced product of the form $\prod_{j=1}^{M}h_{i_{j}}^{e_{i_{j}}}$
   	  if we write 
   	  \begin{equation*}
   	  	w_{j}=\prod_{l=1}^{j}h_{i_{l}}^{e_{i_{l}}},
   	  \end{equation*}
   	  then the arc $\tau$ given by the concatenation of the geodesic arcs 
   	\begin{equation*}
   		J_{j}=w_{j}[p,h_{i_{j+1}}^{e_{i_{j+1}}}p],\quad\quad 0\leq j\leq M-1
   	\end{equation*}
   	 is a $1$-Lipschitz, $(1+\mu,C)$-quasigeodesic with $C\leq\mu L_{1}$.
   	\end{observation}
   
    The $\tau$ as in the observation on which we will focus our attention will be the one obtained from the expression for $\alpha$ in \eqref{ee_word0}. Note that $\tau$ extends to an $\alpha$-invariant quasiaxis for $\alpha$ with the same quasigeodesic constants given by \cref{qg small cancellation}, consisting of the natural concatenation of translates of $\tau$ by all integer powers of $\alpha$. Write $u_{l}=h_{0}^{e_{0}m}h_{1}^{e_{1}}\cdots h_{0}^{e_{0}(m+l)}$ so that
   	\begin{equation}\label{def Il}
   		  I_{l}=u_{l}[p,h_{1}^{e_{1}}p]\subseteq\tau.
   	\end{equation}
    
    The following is an easy consequence of \cref{stable axes} and \cref{parallel axes} with $B=2\delta$, by making $L_{i}$ much larger than the constant $M$ involved in each of those results.
   	\begin{observation}\label{matching hs}
   		If $e_{0},e_{1}$ are large enough, we may additionally assume that 
     	no two points $q,q'\in [p,h_{1}p]$ at distance at least $\frac{L_{1}}{10}$ from each other can belong to the $2\delta$-neighbourhood of $g\tau$ for some $g\in G$. Likewise, if $q,q'\in [p,h_{i}p]$ for $i\in\{0,1\}$ satisfy $d(q,q')\geq \frac{L_{1}}{10}$ and both $q,q'$ belong to the $2\delta$-neighbourhood of $g[p,h_{i}p]$, then we must have $g\in \E(h_{1})$. 
   	\end{observation}
   	 
   	\cref{quadrilateral projection} implies the following. 
   	\begin{observation}\label{application quadrilaterals sc}
   	 For any $\lambda>0$ we can find $\zeta_{1}(\lambda)>0$ such that if $\zeta<\zeta_{1}(\lambda)$ and $\diam_{J}(g[q,q'])\geq D>L_{0}$ for some geodesic arc $J$, then there are subarcs $\tau_{0}\subseteq\tau,\rho_{0}\subseteq \rho$ such that 
   	 \begin{equation*}
   	 	\diam(\partial \tau_{0}),\diam(\partial \rho_{0})>D-\lambda L_{0}
   	 \end{equation*}
   	  and $d(\tau_{0}^{\pm},\rho_{0}^{\pm})\leq \lambda L_{0}$.
   	\end{observation}
    
    Pick now $\lambda_{0},\mu_{0}<10^{-3}$, and $e_{0}$ and $e_{1}$ such that in addition to \eqref{commutation, neatness}
    we have that \ref{m1}-\ref{m3} are satisfied for some $\zeta<\min\{\zeta_{0}(\mu_{0}),\zeta_{1}(\lambda_{0})\}$, and each $\omega\in\Omega$ displaces a point in $\tau$ by at most $\lambda_{0}L_{1}$ (this can be achieved by virtue of \cref{bounded displacement on axis}). 
    We claim that under all these restrictions $\alpha$ satisfies all the conditions in the statement provided $m$ is large enough. 
    This will conclude the proof. 
    
    For simplicity, from this point on we replace $h_{i}$ with $h_{i}^{e_{i}}$ and assume $e_{i}=1$. 
  	Write $L=d(p,\alpha p)$. By virtue of \cref{qg small cancellation}, one easily sees that $L>(1-\mu_{0})\frac{(5m-1)mL_{0}}{2}$, from which it follows that for $m$ large enough we have
  	\begin{equation*}
  		\nu L>\max\{\,4m(L_{0}+L_{1}),\,10^{3}\delta\,\}.
  	\end{equation*}

  	\emph{Transversality condition \ref{c-alpha transversality}.} Take $h\in H$, $g\in G$ and assume that $d_{[p,\alpha p]}(g^{-1}[p,hp])>\nu L$. It follows from \cref{application quadrilaterals sc} that there must be subarcs  $I\subseteq [p,hp]$ and $\tau_{0}\subseteq\tau$ such that 
  	\begin{equation*}
  		\diam(\partial \tau_{0}),\diam(\partial I)>L_{1}+2L_{0}, \quad\quad \dih(g^{-1}I,\tau_{0})\leq 10^{-2} L_{1}. 
  	\end{equation*}  	  
  	In particular, there is $0\leq l\leq M-1$ such that $i_{l}=0$ and $[p,h_{0}p]\subseteq \nb_{10^{-2} L_{0}}(w_{l}^{-1}g[p,hp])$. We conclude 
  	\begin{equation*}
  		\diam_{[p,h_{0}p]}(w_{l}^{-1}g^{-1}[p,hp])\geq (1-10^{-1})L_{0} 
  	\end{equation*}
  	and so $gw_{l}\in H s$ for some $s\in\mathcal{S}$ by assumption \ref{a transv}. On the other hand we know that $w_{l}\in G_{0}\leq H^{s}$ by \eqref{commutation, neatness}, so $g\in H s$ and so we are done. 
    
    \emph{Small cancellation condition \ref{c-alpha is small cancellation}, and $\K(\alpha)=\Omega$.} Assume that $d_{[p,\alpha p]}(g[p,\alpha p])>\nu L$ for some $g\in G$. Then for some $0\leq l<M-1$ and any $0\leq j\leq m$ each of the geodesic arcs $I_{l},I_{l+1}$ must be contained in the $\lambda_{0}L_{1}$ neighbourhood of $g\tau$. 
     Note that no arc of the form $I_{j}$ can simultaneously intersect the $\lambda_{0}$-neighbourhood of $gJ_{i}$ and $gJ_{j}$ if $|i-j|\geq 2$, by the quasigeodesic constraints on $\tau$. Neither can there be $q,q'\in I_{j}$ with $d(q,q')\geq\frac{L_{1}}{3}$ such that both $q$ and $q'$ are in the $\lambda_{0}$-neighbourhood of some translate $g'[p,h_{0}p]$, by \cref{matching hs}. Therefore, there must be $0\leq l'\leq M-1$ and $q,q'\in I_{l}$ such that 
     \begin{equation*}
     	q,q'\in\nb_{\lambda_{0}L_{1}}(gI_{l'}),\quad\quad d(q,q')>\frac{L_{1}}{2}. 
     \end{equation*}
     The arc $I_{l+1}$ must be paired in the same way with some arc of the form $gI_{l''}$. 
     
      Using \cref{quadrilaterals general} and  \cref{matching hs}, we conclude that 
     \begin{equation}\label{beta eq}
     	\beta=(u_{l'}^{-1}g^{-1})u_{l}\in \E(h_{1})=E^{+}(h_{1}).
     \end{equation}
     In particular, the two points $q,q'$ above are at distance at most $\lambda_{0}L_{1}$ from points of $r,r'\in g\tau$, respectively, which appear in the same order along $g\tau$. By virtue of the fact that $\tau$ is a $(1+\mu_{0},\mu_{0}L)$-quasigeodesic arc, the latter implies in turn  
     $l''>l'$. A simple consideration of the distance between $I_{l}$ and $I_{l+1}$ and $gI_{l'}$ and $gI_{l''}$ implies that in fact $l''=l'+1$.  
     
     Arguing in an entirely analogous manner we can show that the arc $u_{l}h_{1}[p,h_{0}p]\subseteq\tau$ must contain two points 
     at distance at least $\frac{L_{0}}{2}$ from each other and both of which lie in the $\lambda_{0}L_{0}$ of some arc of the form $u_{l'}h_{1}h_{0}^{e}[p,h_{0}p]$, $e>0$. Arguing as before, we deduce: 
     \begin{equation*}
     	h_{0}^{-e}h_{1}^{-1}\beta h_{1}h_{0}\in \E(h_{0})\quad \Rightarrow\quad h_{1}^{-1}\beta h_{1}\in \E(h_{0})\cap \E(h_{1})=\Omega\quad\Rightarrow\quad h_{1}^{-1}\beta h_{1}=\beta\in\Omega.
     \end{equation*} 
     Finally, applying the argument used for $I_{l},gI_{l'}$ to $I_{l+1},gI_{l'+1}$ and noting that
      \begin{equation*}
      	u_{l+1}=u_{l}h_{1}h_{0}^{m+l+1},\quad\quad u_{l'+1}=u_{l'}h_{1}h_{0}^{m+l'+1}
      \end{equation*} 
      we obtain: 
     \begin{equation*}
     h_{0}^{l'-l}\beta=	h_{0}^{-(m+l'+1)}h_{1}^{-1}\beta h_{1}h_{0}^{m+l+1}\in \E(h_{1})\quad\Rightarrow\quad h_{0}^{l'-l}\beta\in \E(h_{1})\cap \E(h_{0})=\Omega,
     \end{equation*}
     so that $l=l'$. Going back to \eqref{beta eq}, and since $[\beta,u_{l}]=1$ we get
     \begin{equation*}
     	\beta=u_{l}(u_{l}^{-1}g^{-1})=g^{-1}\in\Omega.
     \end{equation*}
     This does not only prove small cancellation, but also $\K(\alpha)=\Omega$, thereby concluding the proof of \cref{small cancellation construction}.
   \end{proof}

    The main application of \cref{small cancellation construction} is the following existence result, which is crucial to the induction step in the proof of \cref{t: main}. 
   \begin{corollary}\label{existence of alpha}
  	Let $\datum$ be a hyperbolic $\Theta$-seed, $H$ a subgroup of $G$ containing $H_{0}$  which is geometrically taut with respect to $(h_{\Omega})_{\oid}$, and $\mathcal{G}\subseteq G$ a finite subset.  
  	Fix $\oid$ and $p\in X$. Then for every $\nu>0$ there is some $\alpha\in N_{G}(\Omega)$ 
  	with the following properties: 
  	\begin{itemize}
  		\item $\K(\alpha)=\Omega$, 
  		\item $[\alpha,\Omega]=\{1\}$, 
  		\item $\alpha$ is $(p,\nu)$-small cancellation over $\mathcal{G}$;
  		\item $\alpha$ is neatly $(p,\nu)$-transverse to $H$ relative to $\mathcal{S}_{\Omega}$. 
  	\end{itemize}
  \end{corollary}
  \begin{proof}
  	Write $G_{\Omega}=N_{G}(\Omega)$. By assumption (\ref{GOmega non elementary} in \cref{d:nice datum}), the group $G_{\Omega}$ acts non-elementarily on $X$. By assumption, the loxodromic element $h_{0}=h_{\Omega}$ is transverse to $H$ relative to $\mathcal{S}_{\Omega}$. It follows from \cref{many loxodromic} that one can find $h_{1}$ with $h_{1}\nsim h_{1}^{-1}$ and $h_{1}\nsim h_{0}^{\pm1}$.   
  	
  	On the other hand, our assumption that $H$ is geometrically taut implies that the union of the non-free orbits of the action $H\backslash G\noa \Omega$ consist of the finitely many cosets $\{H s\}_{s\in\mathcal{S}_{\Omega}}$. This set must be preserved by $G_{\Omega}$ setwise, so that the pointwise stabilizer of the set in question has finite index in $G_{\Omega}$, i.e., 
  	\begin{equation*}
  		|G_{\Omega}:G_{\Omega}\cap\bigcap_{s\in\mathcal{S}_{\Omega}}H^{s}|<\infty .
  	\end{equation*}
  	The result now follows from applying \cref{small cancellation construction} to the elements $h_{0},h_{1}$ and $G_{\Omega}$ in place of $G_{0}$.
  	Note that the requirement that $\alpha$ is $(p,\nu)$-small cancellation over $\mathcal{G}$ involves only a lower bound on $d(p,\alpha p)$ and can be thus achieved by taking an even smaller $\nu$, if needed.
  \end{proof}

  Given a subgroup $H\leq G$ and a collection of loxodromic elements $\{h_{j}\}_{j\in\mathcal{J}}$ each of which is transverse to $H$ 
  relative to some finite set $\mathcal{S}_{j}$, we are interested in sufficient conditions on an extension $H'$ of $H$ by a finite number of small cancellation elements for each $h_{j}$ will still be transverse to $H'$ relative to $\mathcal{S}_{j}$. The next lemma does not fully solve this problem (we have not specified yet what this extension is supposed to look like), but is a key ingredient.

  \begin{lemma}\label{small cancellation and transversality}
  	Let $h_{0}$ be a loxodromic element and $\mu>0$. Then there is some positive constant $\nu_{sc,t}(\mu)>0$ such that if an element $\alpha$ is $(p,\nu_{sc,t}(\mu))$-small cancellation over $h_{0}$, then $\alpha$ is $(p,\mu)$-transverse to $\subg{h_{0}}$ relative to $\emptyset$. 
  \end{lemma}
  \begin{proof}
   We know that the concatenation of the arcs $h_{0}^{m}[p,h_{0}p]$, $m\in\Z$, is a $(K,C)$-quasiaxis for $h_{0}$ for some $K,C>0$ (this follows from \cref{constructing quasigeodesics} and is shown in \cite{bestvina2002bounded}). 
   Therefore, there is some constant $D>0$ such that 
   \begin{equation*}
   	\dih([p,h_{0}^{m}p],\bigcup_{l\in m}h_{0}^{l}[p,h_{0}p])\leq D
   \end{equation*}
   for all $m\in\N$.  
   It follows that for some constant $D'>0$ and any such $m$ the element $h_{0}$ displaces each point in $[p,h_{0}^{m}p]$ by at most $D'$.   
   
   Assume that $\alpha$ is $\nu$-small cancellation with 
   \begin{equation*}
   	\frac{1}{\nu}>10^{2}\max\{1,\frac{D'}{\mu}\}
   \end{equation*}
   and assume that $d_{[p,\alpha p]}(g^{-1}[p,h_{0}^{n}p])\geq\mu L$ for some $n\in\N$. It follows from the discussion in the previous paragraph that the loxodromic element $h_{0}^{g}$ displaces each point in some subarc $J\subseteq[p,\alpha p]$ with 
   $\diam(J)>\mu L-10\delta\geq 90\nu L$ by at most $D'+2\delta<\frac{3}{10^{2}}\nu$. 
   It is easy to verify that this violates the $\nu$-small cancellation condition.
  \end{proof}

 \subsection{A couple of technical lemmas around small cancellation}
 
 \label{s:some technical lemmas}
 This short subsection collects a few miscellaneous lemmas that will be used later on. We recommend the reader to skip their proof on a first reading, and only come back to this section after getting to the context in which they are used.
  
 \begin{lemma}\label{translates are transverse}
 	Let $\alpha$ be a $(p,\nu)$-small cancellation element with $\nu<\frac{1}{10}$. Then for any $g\in G$, if $g\notin\K(\alpha)$, at least one of the two points $gp,g\alpha p$ is at distance at least $\frac{L}{3}$ from $J=[p,\alpha p]$. The same statement holds after exchanging the role of $[p,\alpha p]$ and $g[p,\alpha p]$. 	
 \end{lemma}
 \begin{subproof}
  	Indeed, otherwise under the assumptions above we would have:
  	 \begin{equation*}
  	 	d(g p,g\alpha p)\leq d(gp,J)+d_{J}(gp,g\alpha p)+d(g\alpha p,J)<\left(\frac{2}{3}+2\nu\right)L<L,
  	 \end{equation*}
  	 a contradiction. The other statement and its proof are entirely analogous. 
 \end{subproof} 
   
 The following is an easy consequence of \cref{constructing quasigeodesics}.
  \begin{lemma}\label{sc and axis}
  	For every $\mu>0$ there is some $\nu_{qg}(\mu)$ such that for every element $\alpha$ which is $(p,\nu)$-small cancellation with $\nu<\nu_{qg}(\mu)$, if we write $L=d(p,\alpha p)$, and $\tau$ for the $1$-Lipschitz map $\tau$ consisting of the concatenation of the geodesic arcs $\alpha^{k}[p,\alpha p]$, then $\tau$ is a $(1+\mu,\mu L)$-quasigeodesic. 
  \end{lemma}

 Using \cref{bounded displacement on axis}, we get 
  \begin{corollary}\label{sc misc}
  	For every $\mu>0$ there is some $\nu_{K}(\mu)$ such that for every $(p,\nu)$-small cancellation element $\alpha$ with $\nu<\nu_{K}(\mu)$ such that if we write $J=[p,\alpha p]$, $L=d(p,\alpha p)$, then for all $k\in\Z$ and $\beta\in\K(\alpha)$ we have: 
  	\begin{itemize}
  		\item $d_{J}(\alpha^{k}\beta p,\alpha p)\leq \mu L$ if $k>0$ 
  		\item $d_{J}(\alpha^{k}\beta p, p)\leq \mu L$ if $k\leq 0$ 
  	\end{itemize}
  \end{corollary}
  \begin{proof}
  	Consider the $1$-Lipschitz map $\tau$ consisting of the concatenation of the geodesic arcs $\alpha^{k}J$, for $k\in\Z$, as in \cref{sc and axis}. Said result allows us to assume that $\tau:\R\to X$ is a $(1+\lambda,\lambda L)$-quasigeodesic for any given $\lambda>0$, provided we take $\nu$ to be small enough. The conclusion follows. For the first part we use \cref{bounded displacement on axis}.
  \end{proof}

   \begin{lemma}\label{sc and KG}
  	There is $\nu>0$, such that the following holds. 
  	Assume that $\alpha$ is $(p,\nu)$-small cancellation over $\{g_{1},g_{2}\}$, and $(p,\frac{1}{3})$-transverse to $H\leq G$ relative to some finite set $\mathcal{S}\subseteq G$ such that $g_{1}^{-1}\notin H\cdot\mathcal{S}$. 
  	Then for any $\eta\in H$ if $\chi=g_{1}\eta g_{2}\in\E(\alpha)$ holds, then $\chi\in\K(\alpha)$.
  \end{lemma}
  \begin{proof}
  	Pick $\nu_{0}=\min\{\nu_{K}(10^{-3}),10^{-3}\}$, where $\nu_{K}$ is given by \cref{sc misc}. We claim that this number satisfies the conditions of the statement. 
  	
  	Otherwise we can find $k\in\Z\setminus\{0\}$ and $\beta\in \K(\alpha)$ such that 
  	\begin{equation}\label{chi expression}
  		\chi=\alpha^{k}\beta. 
  	\end{equation}
  	Write $J=[p,\alpha p]$ and $L=d(p,\alpha p)$. Recall that the inequality $\delta\leq\nu L$ is part of the definition of the $(p,\nu)$-small cancellation condition.
  	For every $g\in G$ by \ref{product and projection} we have 
  	\begin{equation}\label{simple inequality}
  		d_{J}(g x,g y)\leq d(g x,g y)+4\delta\leq d(x,y)+\nu L.
  	\end{equation}
  	Therefore we can conclude:
  	\begin{align*}
  		d_{J}(g_{1}p, g_{1}\eta p)\geq & d_{J}(p,\chi p)-d_{J}(p,g_{1}p)-d_{J}(g_{1}\eta p,g_{1}\eta g_{2}p)\geq\\
  		 &d_{J}(p,\chi p)-d(p,g_{1}p)-d(p,g_{2}p)-3\nu L\geq\\
       &d_{J}(p,\alpha^{k}\beta p)-5\nu L\geq		 \\
       &L-d_{J}(\alpha p,\alpha^{k}p)-d(p,\beta p)-7\nu L\geq\\
       &L-(7\nu+2\cdot 10^{-3})L\geq\frac{1}{3}L. 
  	\end{align*}
    The first inequality above follows from the triangular inequality for $d_{J}$, while the second follows from \eqref{simple inequality}. The third one follows from \eqref{chi expression} and the assumption that $\alpha$ is $(p,\nu)$-small cancellation over $\{g_{1},g_{2}\}$.  The fourth one follows again from the triangle inequality for $d_{J}$ and \eqref{simple inequality}, while the last inequality is a result of our assumption that $\nu<\nu_{K}(10^{-3})$. 
    
    The inequality $d_{J}(g_{1}p,g_{1}\eta p)\geq\frac{1}{3}L$, together with the assumption that $\alpha$ is $(p,\frac{1}{3})$-transverse to $H$ relative to $\mathcal{S}$ implies that $g_{1}^{-1}\in H\cdot\mathcal{S}$, a contradiction. 
    
    The case in which $k<0$ can be handled similarly. First one notes that by using the fact that $\chi$ and $\alpha$ commutes and the triangle inequality one can deduce: 
    \begin{equation*}
    	d_{J}(\alpha g_{1}p,\alpha g_{1}\eta p)\geq d_{J}(\chi \alpha p,\alpha p)-d_{J}(\alpha g_{1}p,\alpha p)-d_{J}(\alpha g_{1}\eta p,\alpha g_{1}\eta g_{2}p).
    \end{equation*}
    One can then establish that the right-hand side is at least $\frac{1}{3}L$ using the same type of arguments as in the case $k>0$ and thereby reach a contradiction with the transversality assumption in the statement.  
  \end{proof}

  \section{Extending groups using tuples of elements with a common small cancellation part: the geometry of the extension} 
  \label{s:extending groups using tuples of elements with a common small cancellation part}

   For the entirety of this section, fix a group $G$, as well as an action $G\aon X$ by isometries and $p$ some point in $X$.

 \begin{definition}
 	\label{d:pairs of tuples}Suppose that we are given $H,\Omega\leq G$, of which $\Omega$ is finite, as well as some finite subset $\mathcal{S}\subseteq G$, of which we assume 
  \begin{enumerate}
  	\myitem{$(\diamond)$} $H\cdot \mathcal{S}\cdot N_{G}(\Omega)=H\cdot \mathcal{S}$.
  \end{enumerate}
 	 We say that a  $\underline{g}=(g_{i})_{i\in k_{o}}$ is a \emph{$(H,\Omega,\mathcal{S})$-generously admissible tuple} if the following conditions are satisfied: 
 	\begin{enumerate}[label=(\alph*)]
 		\item \label{hyp gi not F} $H^{g_{i}} \cap \Omega =\{1\}$ for any $i\in k_{o}$,  i.e., the orbit of $\{Hg_{i}\}_{i\in k_{o}}\subseteq H\backslash G$ under the action of $\Omega$ on the right is free;
 		\item \label{hyp away from fixed points} $g_{i}\notin H\cdot \mathcal{S}\cdot \Omega$ or, equivalently, $Hg_{i}\omega\neq H s$ for any $i\in k_{o}$, $s\in\mathcal{S}$ and $\omega\in\Omega$;
 		\item  \label{hyp_dist_cos}  $Hg_{i}\Omega\cap Hg_{i'}\Omega=\emptyset$ for distinct $i,i'\in k_{o}$, that is, $Hg_{i},Hg_{i'}$ are in different orbits under the action $H\backslash G\noa \Omega$.
 	\end{enumerate}
 	\end{definition}
 	
 	\begin{definition}\label{weakly admissible}
   	Given $H,\Omega$ and $\mathcal{S}$ as in \cref{d:pairs of tuples}, an element $g\in G$ is said to be \emph{$(H,\Omega,\mathcal{S})$-linearly admissible} if the corresponding length one tuple satisfies the following conditions: 
  \begin{enumerate}
  	\myitem{$(a')$} \label{weak hyp gi not F} $H^{g}\cap G_{\Omega}=\{1\}$;
 		\myitem{$(b')$} \label{weak hyp away from fixed points} $g\notin H\cdot \mathcal{S}\cdot G_{\Omega}$.
  \end{enumerate} 
   We will also refer to $(g)$ as a \label{key}$(H,\Omega,\mathcal{S})$-linearly admissible $1$-tuple.
   We say that a tuple is \emph{$(H,\Omega,\mathcal{S})$-admissible} if it is either $(H,\Omega,\mathcal{S})$-generously or linearly admissible.  
 	\end{definition}

 	
 	\begin{remark}
 		In the applications, $\Omega\in\D(\Theta)$, identified with a subgroup of the ambient group $G$ as in \cref{d:nice datum}, and the integer $k_{o}$ will be the number of free orbits of the action $\pi:k\noa\Omega$, hence the notation. However, for the time being we will avoid to burden the discussion with this additional information.
 	\end{remark}
 	
 	
 	\begin{definition}\label{d:basic extension setting}
 		Suppose that we are given $H,\Omega\leq G$, with $\Omega$ finite, $p\in X$, and let $\underline{g}_{-1},\underline{g}_{1}\in G^{k_{o}}$. We say that $H'$, $H\leq H'\leq G$ is an \emph{  $(\Omega,\mathcal{S},\underline{g}_{-1},\underline{g}_{1},\alpha)$-extension of $H$} for $\alpha\in G$ if $H'=\subg{H,\alpha_{i}}_{i\in k_{o}}\leq G$, where for each $i\in k_{o}$ we define
 		\begin{equation}\label{eq:extending tuple}
 			\alpha_{i}:=g_{i,-1}\alpha g_{i,1}^{-1}\in G, 
 		\end{equation}
 		We say that the extension is \emph{(generously/linearly) admissible} if the tuples $\underline{g}_{-1},\underline{g}_{1}$
 		are (linearly) admissible in the sense of \cref{d:pairs of tuples} (resp. \cref{weakly admissible}) and the following conditions are satisfied: 
 		\begin{enumerate}[label=(\roman*)]
     \item \label{hyp_Kg condition} $\K(\alpha)=\Omega$;
     \item \label{hyp commutation condition} $[\alpha,\Omega]=\{1\}$;%
     \item \label{hyp alt} $\alpha\in \bigcap_{s\in\mathcal{S}}H^{s}$. 
    \end{enumerate}
 		In this situation we refer to $(\alpha_{i})_{i\in k_{o}}$ as the \emph{extending tuple}.
  \end{definition}
  
  We would rather not provide the full context to which the definition above will be applied yet. For the time being, just note the following.
  \begin{observation}\label{algebraic consequence of basic extension}
  	If $H'$ is an $(\Omega,\mathcal{S},\underline{g}_{-1},\underline{g}_{1},\alpha)$-admissible extension of $H\leq G$, then the action of $\alpha$ on the right 
  	\begin{itemize}
  		\item $H s\alpha=H s$ and thus $H's\alpha=H's$ for all $s\in\mathcal{S}$; 
  		\item $H'g_{i,-1}\omega\alpha=H'g_{i,1}\omega$ for all $i\in k_{o}$ and $\omega\in\Omega$.
  	\end{itemize} 
  \end{observation}
  \begin{proof}
  	 Indeed, for  $i\in k$ and $\omega\in\Omega$ we can have
    \begin{equation*}
    	H'g_{i,-1}\omega\alpha=H'g_{i,-1}\alpha\omega=H'(g_{i,-1}\alpha g_{i,1}^{-1})g_{i,1}\omega=H'\alpha_{i}g_{i,1}\omega= H'g_{i,1}\omega,
    \end{equation*}
    and the first statement is likewise a trivial consequence of \cref{hyp alt} that holds already with $H$ in place of $H'$  
  \end{proof}
  
  The next definition establishes a set of conditions, parametrized by $p\in X$ and $\nu>0$, which control the geometry of the group $H'$ (more precisely, of the orbit of $p$ under the action of $H'$). 
  
  \begin{definition}\label{d:good extension setting}    
  	 Consider a $(\Omega,\mathcal{S},\underline{g}_{-1},\underline{g}_{1},\alpha)$-extension of $H$ with extending tuple $(\alpha_{i})_{i\in k_{o}}$, as defined in \cref{d:basic extension setting}, and write $G_{\Omega}=N_{G}(\Omega)$. 
  	 For $p\in X$ and $\nu>0$ consider the condition 
  	 \begin{enumerate}
  	 	\myitem{$(T)$} \label{hyp_tp} the element $\alpha$ is $(p,\nu)$-transverse to $H$ relative to $\mathcal{S}$ (and thus $(p,\nu)$-neatly transverse).
  	 \end{enumerate}
  	 We say that the extension is \emph{$(p,\nu)$-geometric} if either of\ref{strong case} or \ref{weak case} below holds:
  	\begin{enumerate}[label=(\Roman*)]
  		\item \label{strong case}It is \emph{ generously $(p,\nu)$-geometric}, meaning it satisfies condition \ref{hyp_tp}, as well as
    \begin{enumerate}[label=(\roman*)]
    	\myitem{$(Ad)$} \label{hyp admissible} the extension is generously admissible;
  	 \myitem{$(SC)$} \label{hyp_sc}$\alpha$ is $(p,\nu)$-small cancellation over $\{g_{i,\epsilon}\,|\,i\in k_{o},\epsilon\in\{\pm 1\}\}$.
    \end{enumerate}
  		\item\label{weak case} It is \emph{linearly $(p,\nu)$-geometric}, by which we mean that condition \ref{hyp_tp} is satisfied and:
    	\begin{enumerate}[label=(\roman*)]
    		\myitem{$(Ad')$} \label{hyp admissible2}the extension is linearly admissible (in particular, $k_{o}=1$); 
    		\myitem{$(St)$}  \label{hyp_stab} $\alpha$ is positively $(p,\nu)$-stable over $\{g_{0,-1},g_{0,1}\}$ in the sense of \cref{d:stable elements}.
    		\myitem{$(Sep)$} \label{hyp_sep} $g_{0,1}\notin Hg_{0,-1}G_{\Omega}$.
    		\myitem{$(Pos)$} \label{hyp_pos} $\E(\alpha)=\E^{+}(\alpha)=G_{\Omega}$.
    	\end{enumerate}
  	\end{enumerate}
    In contexts in which a $(p,\nu)$-geometric extension is given, we will use the term \emph{generous (resp. linear) case} to refer to the case in which the extension is generously (resp. linearly) $(p,\nu)$-geometric. 
  \end{definition}

  \begin{remark}
  	As we shall see later, condition \ref{hyp_sep} is equivalent to the statement that the pair of sets in $[H\backslash G]^{k}$ whose image in $H'\backslash G$ is forced to belong to the same orbit by the way $H'$ is constructed (as in \cref{algebraic consequence of basic extension}) were not already in the same orbit to begin with. It is thus quite natural as an assumption and we could have added a similar condition to case \ref{strong case} above at no loss. However, in case \ref{strong case} this hypothesis is not needed to control the basic geometry of the extension and comes into play rather late in the course of the proof, a fundamental difference between the two cases which we consider worth emphasizing.
  \end{remark}   
     
  \begin{assumption}\label{assumption extension}
  	For the rest of this subsection assume $H'$ is a $(p,\nu)$-geometric $(\Omega,\mathcal{S},\underline{g}_{-1},\underline{g}_{1},\alpha)$-extension of $H$ with extending tuple $(\alpha_{i})_{i\in k_{o}}$, as given by \cref{d:basic extension setting}. 
  \end{assumption}

  \begin{definition}\label{reduced word}
  	 By a \emph{reduced word} we mean a group word of the form
  	\begin{equation}
  		\label{eq:word}w:=\eta_{0}\alpha_{i_{0}}^{\epsilon_{0}}\eta_{1}\dots\alpha_{i_{r}}^{\epsilon_{r-1}}\eta_{r},
  	\end{equation}
  	where $\epsilon_{l}\in\{\pm 1\}$, $i_{l}\in k_{o}$ and $\eta_{l}\in H$, such that 
  	for any $1\leq l<r$ we have 
  	\begin{equation*}
  	  i_{l}=i_{l+1},\,\,\epsilon_{l}=-\epsilon_{l+1}\quad\Rightarrow\quad\eta_{l}\neq 1 .
  	\end{equation*} 
  	Note that there is a natural identification between the collection of such words and elements of $H\ast\F(\alpha_{0},\dots, \alpha_{k_{o}-1})$. We say that $w$ is \emph{cyclically reduced} if moreover $\eta_{r}=1$ and
  	\begin{equation*}
  		i_{1}=i_{r},\,\,\epsilon_{0}=-\epsilon_{r-1}\quad\Rightarrow\quad \eta_{0}\neq 1.
  	\end{equation*}
  \end{definition}
  
  \begin{observation}\label{o:words}
  	Note that reduced words are in correspondence with elements of the group $H\ast\F(\alpha_{j})_{j\in k_{o}}$.
    Every element in $H$ is represented by some reduced word in the sense above, i.e., there is a surjective 
    homomorphism $ev:H\ast\F(\alpha_{j})_{j\in k_{o}}\onto H'$. Every element of $H\ast\F(\alpha_{i})_{i\in k_{o}}$ which is not conjugate into $H$ is conjugate in $H\ast\F(\alpha_{i})_{i\in k_{o}}\setminus H$ to a cyclically reduced word.
  \end{observation}
  
  \begin{definition}
	  \label{d:x-paths}
    Let $w$ be a reduced word as in \eqref{eq:word}.
  	For $0\leq l\leq r$ write $\bar{h}_{l}$ for the element of $H'$ that is the image of the initial subword of $w$ containing precisely $l$-letters. We also write $p_{l}:=\bar{h}_{l} p$ 
   	We define the \emph{$(p,w)$-extension arc}, denoted by $\pa(w)$, 
   	as the arc $\tau$ obtained by concatenating arcs $J_{t}$, $0\leq t\leq m$ given as follows
  	\begin{enumerate}[label=(\alph*)]
  		\item $J_{2l}(w)$ is the arc $\bar{h}_{2l}[p,\eta_{l}p]$ between $p_{2l}$ and $p_{2l+1}$,
  		\item $J_{2l+1}(w)$ is the arc from $p_{2l+1}$ to $p_{2l+2}$ given by the natural concatenation of the following geodesic arcs
  		\begin{itemize}
          \item $I_{2l+1}(w):=\bar{h}_{2l+1}[p,g_{i_{l},-\epsilon_{l}}^{\epsilon_{l}}p]$,
          \item $J^{*}_{2l+1}(w):=\bar{h}_{2l+1}g_{i_{l},-\epsilon_{l}}\cdot [p,\alpha^{\epsilon_{l}}\cdot p]$,
          \item $I'_{2l+1}(w):=\bar{h}_{2l+1}g_{i_{l},-\epsilon_{l}}\alpha_{i_{l}}^{\epsilon_{l}}[p,g_{i_{l},\epsilon_{l}}^{-\epsilon_{l}}\cdot p]$.
  		\end{itemize}
  	\end{enumerate}
    We will refer to $J_{2l}$ as an \emph{$H$-subarc} of $\tau$ and to $\bar{h}_{2l}$ and $\eta_{l}$ as the corresponding \emph{translating element and label}, respectively. We will also refer to $J^{*}_{2l+1}$ as an \emph{$\alpha$-subarc} of $\tau$, to $\bar{h}_{2l+1}g_{i_{l},-\epsilon_{l}}$ as the \emph{translating element} and to $i_{l}$ and $\epsilon_{l}$ as the \emph{associated index and exponent}, respectively. We will refer to the triple $(\bar{h}_{2l+1}g_{i_{l},-\epsilon_{l}},i_{l},\epsilon_{l})$ as the \emph{signature} of the arc.
    The arcs $I_{2l+1}(w)$ and $I_{2l+1}'(w)$ will be called \emph{remainder subarcs}. 
   	
  	If $w$ is cyclically reduced, then we may also consider the path obtained by concatenating in the natural order the arcs of the form $h^{l}\tau$, $l\in\mathbb{Z}$, where $\tau$ is the $(p,w)$-extension arc and $h$ the image of $w$ in the group $H'$. We refer to this as the \emph{$p$-extension axis of $w$}, denoted by $\ax(w)$.
  	
  	For a general $w\in H\ast\F(\alpha_{i})_{i\in k_{o}}$ which is not conjugate in $H\ast\F(\alpha_{i})_{i\in k_{o}}$ to an element in $H$ choose $w_{0}$ and $u$ such that $w_{0}$ is cyclically reduced and $w=w_{0}^{u}$. Then we define the $(p,w)$-extension axis as the translate by $ev(u)$ in $H'$ of the $(w_{0},p)$-extension axis defined above. It is not hard to see that up to parametrization, the resulting line does not depend on the choice of $w_{0}$ and $u$.
  	
  	Let $\tau$ be the $(p,w)$-extension arc and for $\gamma\in G$ consider the translate $\gamma\tau$ (by postcomposition). We generalize the notation above as follows. If $J$ is an $\alpha$-subarc or $H$-subarc of $\tau$ with translating element $h$, then we say that $\gamma J$ is an $\alpha$ or $H$-subarc of $\gamma\tau$ with translating element $\gamma h$.  
  	
  	We will use the terms \emph{$p$-extension arc} to refer to $\pa(w)$ for some $w\in H\ast\mathbb{F}(\alpha_{i})_{i\in k_{o}}$. 
  \end{definition}

  The next lemma shows that if the extension is $(p,\nu)$-geometric for a sufficiently small $\nu$, then the geometry of the action of $H'$ on 
  $(X,d)$ reflects the algebraic structure of the extension. We focus on the orbit under $H'$ of a fixed point $p$. 
   
   \begin{lemma}
  	\label{l:p-paths} We continue working under \cref{assumption extension}. For any $\mu>0$ there is some $\nu_{path}(\mu)>0$ such that if our extension datum is $(p,\nu)$-geometric with $\nu\leq\nu_{path}(\mu)$, then the following holds, where $L=d(p,\alpha\cdot p)$:
  	 \begin{enumerate}[label=(\Alph*)]
  	 	\item \label{is free product} the evaluation map is an isomorphism $ev:H\ast\bigast_{i=0}^{k_{o}-1}\subg{\alpha_{i}}\cong H'$;
  	 	\item \label{is quasigeodesic} $\pa(w)$ is a $(1+\mu,\mu L)$-quasigeodesic arc for every $w\in H\ast\bigast_{i=0}^{k_{o}-1}\subg{\alpha_{i}}$; 
  	 	\item \label{new classes are hyperbolic} any element $h\in H'$ which is not conjugate in $H'$ to an element of $H$ is loxodromic, and the $(p,w)$-extension axis is an $(1+\mu,\mu L)$-quasigeodesic line.
  	 \end{enumerate}
    \end{lemma}
  \begin{proof}
  	Pick some reduced $w\in H\ast\bigast_{i=0}^{k_{o}-1}\subg{\alpha_{i}}$ as in \eqref{eq:word}. We will deduce a series of inequalities from conditions \ref{hyp_sc} and \ref{hyp_tp} in \cref{d:basic extension setting}, and then argue that said inequalities guarantee item \ref{is quasigeodesic} for $\nu$ sufficiently small. Item \ref{is free product} is an easy consequence of \ref{is quasigeodesic}.

  	First, we would like to apply \cref{l:p-paths}, for which we use the following auxiliary sublemma.
  	\begin{lemma}
  	 \label{l:projections}Assume that we are given 
  	 $0\leq i,i'\leq k_{o}-1$, $\epsilon,\epsilon'\in\{\pm 1\}$, $\eta\in H$ such that at least one of the following conditions holds: 
  	\begin{itemize}
  		\item $\eta\in H\setminus\{1\}$, 
  		\item $\epsilon\neq-\epsilon'$,
  		\item $i\neq i'$.
  	\end{itemize}
    Write $\zeta=g_{i,\epsilon}^{-1}$ and $\zeta'=g_{i',-\epsilon'}$.    
  	Consider the points
  	\begin{align*}
  		q_{0}&=p, \quad q_{1}=\alpha^{\epsilon}  p,\quad
  		q_{2}=\alpha^{\epsilon}\zeta  p, \quad q_{3}=\alpha^{\epsilon}\zeta\eta  p,\\
  		q_{4}&=\alpha^{\epsilon}\zeta\eta\zeta'  p,
  		\quad q_{5}=\alpha^{\epsilon}\zeta\eta\zeta'\alpha_{i'}^{\epsilon'}  p.
  	\end{align*}
  	Write $J=[q_{4},q_{5}]$, $J'=[q_{0},q_{1}]$. Then for any $\nu>0$ small enough we have
  	\begin{enumerate}[label=(\alph*)]
  			\item \label{aaaah} $d_{J}(q_{l},q_{4})\leq \nu L$ for $0\leq l\leq 4$,
  			\item \label{bbbbh} $d_{J'}(q_{l},q_{1})\leq \nu L$ for $2\leq l\leq 5$.
  	\end{enumerate} 	
  	\end{lemma}
  	\begin{subproof}
      It suffices to prove \ref{aaaah}, since \ref{bbbbh} follows from an identical argument. 
      We start with the following observations:
  	\begin{enumerate}[label=(\roman*)]
  		\item \label{preliminary 1} $d_{J}(q_{3},q_{4}),d_{J}(q_{1},q_{2})\leq \nu L+20\delta\leq 2\nu L$, by \cref{l:projections} and the inequalities 
  		\begin{equation*}
  			d_{J}(p,\zeta' p)\leq\nu L,\quad d_{J}(p,\zeta p)\leq\nu L,
  		\end{equation*}
  		 which are part of assumption \ref{hyp_sc} in \cref{d:good extension setting};
  		\item \label{preliminary 2} $d_{J}(q_{2},q_{3})\leq \nu L$ by the transversality condition \ref{hyp_tp} in \cref{d:basic extension setting}, since 
  		\begin{equation*}
  			d_{J}(q_{2},q_{3})=d_{[p,\alpha^{\epsilon'}p]}((\zeta')^{-1}\eta^{-1}p,(\zeta')^{-1}p)
  		\end{equation*}
  		and $\zeta'=g_{i',-\epsilon'}\notin H\cdot\mathcal{S}$.   
  	\end{enumerate} 
    
     Let $\chi:=\zeta\eta\zeta'$. 
     From \cref{sc and KG}, we conclude the following:  
     \begin{observation}\label{alternative in strong case}
     	 	In the generous case \ref{strong case} in \ref{d:good extension setting} either $\chi\notin G_{\Omega}$ or $\chi\in\K(\alpha)=\Omega$. 
     \end{observation}
   
     \begin{observation}
     	 \label{o:not in centralizer} In the generous case \ref{strong case} of \ref{d:good extension setting} the equality $\epsilon=-\epsilon'$ implies $\chi \notin G_{\Omega}$.
     \end{observation}
     \begin{subsubproof}
     	   Assume, for the sake of contradiction, that $\epsilon=-\epsilon'$ and $\chi\in G_{\Omega}$, which by the remarks above implies that $\chi\in\Omega$. Then 
     	   \begin{equation*}
     	   g_{i,\epsilon}\chi= \eta g_{i',\epsilon}\in g_{i,\epsilon}\Omega\cap Hg_{i',\epsilon}\neq\emptyset. 
     	   \end{equation*}
     	    By item \ref{hyp_dist_cos} of \cref{d:pairs of tuples} this can only be the case if $i=i'$ and $\chi\in\Omega$, so
     	    that 
     	    \begin{equation*}
     	    	\eta \in H\cap g_{i,\epsilon}\Omega g_{i,\epsilon}^{-1}=\{1\},
     	    \end{equation*}
     	    where the last equality is condition \ref{hyp gi not F} in \cref{d:pairs of tuples}.
     	\end{subsubproof}   
 
     The linear case is handled by the following lemma.
     \begin{lemma}\label{out of Ealpha}
     	 In the linear case \ref{weak case} of \ref{d:good extension setting} we have $\chi\notin G_{\Omega}$. 
     \end{lemma}
     \begin{proof}
     	 Here $k_{o}=1$, so there are two distinct possibilities:
     	\begin{enumerate}[label=(\alph*)]
     		\item If $\epsilon=-\epsilon'$, then $\chi=\eta^{g}$, for $g\in\{g_{0,-1},g_{0,1}\}$. 
         By condition \ref{weak hyp gi not F} of \cref{weakly admissible} we conclude that $\eta=1=\chi$. 
        However, it is part of the original assumption of \cref{l:projections} that $\epsilon=-\epsilon'$ do not hold simultaneously. 
     		 
     		\item If $\epsilon=\epsilon'$, then $\chi=g_{0,\epsilon}^{-1}\eta g_{0,-\epsilon}$ and so if $\chi\in G_{\Omega}$ we must have 
     		 \begin{equation*}
     		 	g_{0,1}\in Hg_{0,-1}G_{\Omega},
     		 \end{equation*}
     		 contradicting \ref{hyp_sep} in \cref{d:basic extension setting}.
     	\end{enumerate}
     \end{proof}

      If $\chi\notin G_{\Omega}$, then the small cancellation condition \ref{hyp_sc} in \cref{d:good extension setting}, together with \cref{sc misc}, or condition \ref{hyp_stab} directly in the linear case \ref{weak case} implies
      \begin{equation*}
      	d_{J}(q_{0},q_{1})=d_{[p,\alpha^{\epsilon} p]}(\chi^{-1}\alpha^{\frac{\epsilon-1}{2}}[p,\alpha])\leq\nu L.
      \end{equation*}

      If $\chi\in G_{\Omega}$, then by \cref{out of Ealpha} we must be in the generous case \ref{strong case} of \cref{d:good extension setting}, so that by \cref{o:not in centralizer} in fact $\epsilon=\epsilon'$ and by \cref{sc and KG} we have $\chi\in\Omega=\K(\alpha)$. 
      Writing $I=[p,\alpha^{\epsilon'} p]=[p,\alpha^{\epsilon}p]$, we get 
      \begin{align*}
      	&d_{J}(q_{0},q_{1})=d_{I}(\chi^{-1}\alpha^{-\epsilon}q_{0},\chi^{-1}\alpha^{-\epsilon}q_{1})=
      	d_{I}(\chi^{-1}\alpha^{-\epsilon}p,\chi^{-1}p)\leq\\
      	&d_{I}(\alpha^{-\epsilon}p,p)+d(p,\chi^{-1} p)+d(\alpha^{-\epsilon}p,\chi^{-1}\alpha^{-\epsilon}p)+8\delta.
      	\end{align*}
      	Since $\chi\in\Omega$, and in particular commutes with $\alpha$, all terms in the final expression can be made arbitrary small relative to $L$ by taking $\nu$ to be small enough.
        This concludes the proof of \cref{l:projections}.
  	\end{subproof}
  
   Item \ref{is quasigeodesic} in \cref{l:p-paths} now follows for the arc $\tau=\pa(w)$ by applying \cref{constructing quasigeodesics} to the arcs:  
   \begin{itemize}
   	\item	$\tau_{0}:=J_{0}(w)*I'_{1}(w)$,
   	\item $\tau_{2l+1}:=J^{*}_{2l+1}(w)$ for $0\leq l\leq m-1$, 
   	\item $\tau_{2l}:=I'_{2l-1}(w)\ast J_{2l}(w)*I_{2l+1}(w)$, $1\leq 2l+1\leq m-1$,
    \item $\tau_{2m}:=I_{2m-1}(w)\ast J_{2m}(w)$,
   \end{itemize}
   where $\ast$ stands for arc concatenation. 
   The same argument shows that when $w$ is cyclically reduced with $\eta_{r}=1$, then the $(p,w)$-extension arc is a $(1+\mu,\mu L)$-quasigeodesic. Item \ref{new classes are hyperbolic} follows, thereby concluding the proof of \cref{l:p-paths}.
   
  \end{proof}
 
   \begin{definition}
	 In view of the conclusion of \cref{l:p-paths}, whenever this conclusion holds we will write $\pa(h)$ and talk about $(p,h)$-extension arcs for $h\in H'$, and similarly for $p$-extension axes. 
  \end{definition}

  The following set of assumptions will be used in the next lemma and carried over to the next section.
  \begin{assumption}\label{constant assumptions}
  	We are given constants $\nu,\mu_{0},\lambda_{0}$ satisfying
  	\begin{align*}
  	  0<&\lambda_{0}<10^{-7}(1+\delta)^{-1},\\ 
  	 0<&\mu_{0}<\min\{\,\mu_{quad}(\lambda_{0}),\,10^{-4}\lambda_{0}\,\},\\
  	0<&\nu<\min\{\,\nu_{path}(\mu_{0}),\,\nu_{K}(\mu_{0}),\,10^{-4}\mu_{0}\,\},
  	\end{align*}
  	where $\mu_{quad}$ is the function given in \cref{quadrilaterals general}, $\nu_{path}$ the one given in \cref{l:p-paths},
  	and $\nu_{K}$ the function in \cref{sc misc}. 
  \end{assumption}

  \begin{corollary}\label{c:weak transversality of a set of elements}
  	 Under assumptions \ref{assumption extension} and \ref{constant assumptions}, there is some constant $\nu_{0}>0$ such that if $\nu<\nu_{0}$, then for any loxodromic element $h_{0}$ such that $h_{0}$ is transverse to $H$ relative to some finite set $\mathcal{S}_{0}$, we also have that $h_{0}$ is transverse to $H'$ relative to $\mathcal{S}_{0}$.   
  \end{corollary}
  \begin{proof}
    We keep the notation $L=d(p,\alpha p)$. Take $m$ large enough that $h_{0}^{m}$ is $(p,10^{-1})$-transverse to $H$ relative to $\mathcal{S}_{0}$ and $d(p,h_{0}^{m}p)=L_{0}\geq 10L$. By virtue of \cref{small cancellation and transversality}, if $\nu$ is small enough, then the element $\alpha$ must be $(p,10^{-1})$-transverse to $\subg{h_{0}}$. Let us verify that $h_{0}^{m}$ is also  $(p,\frac{8}{10})$-transverse to $H'$.  
       	
  	Write $I=[p,h_{0}^{m}p]$ and assume that we are given $g\in G$ and $\eta\in H'$ with $\diam_{I}(g^{-1}[p,\eta p])\geq \frac{L_{0}}{10}$. We need to show that this implies $g\in H's$ for some $s\in\mathcal{S}_{0}$. 
  	Let $\tau=\pa(\eta)$. By \cref{quadrilateral projection} and our assumption on $\nu$, there must be some subarc $\tau_{0}\subseteq\tau$ such that $g^{-1}\tau_{0}$ is contained in $\nb_{\lambda_{0}L}(I)$ and $\diam(\partial\tau_{0})\geq \frac{8L_{0}}{10}-\lambda_{0}L\geq\frac{7L_{0}}{10} $.  
  	 
  	If two points in $\tau_{0}$ at distance at least $\frac{L_{0}}{9}$ belong to the same $H$-subarc of $g\tau$, then 
  	our assumption that $h_{0}^{m}$ is $(p,10^{-1})$-transverse to $\subg{h_{0}}$, as well as \cref{quadrilaterals general} implies that $\eta_{0}^{-1}g\in H s$ for some $s\in\mathcal{S}_{0}$, where $\eta_{0}$ is the translating element of the $H$-subarc of $\tau$ in question, so that $g\in H' s$ and we are done.
  	
  	If no two such points exist, then, since remainder subarcs have length at most $\nu_{0}L<10^{-4}L$, each maximal subarc of an
  	$H$-subarc contained in $\tau_{0}$, together with any adjacent remainder subarc will have length at most $\frac{2L}{10}$.
  	This implies the existence of two points in $\tau_{0}$ belonging to the same $\alpha$-subarc of $g\tau_{0}$ and which are at distance larger than $10^{-1}L_{0}\geq L$ from each other. This easily contradicts our assumption that $\alpha$ is $(p,10^{-1})$-transverse to $\subg{h_{0}}$, thereby concluding the proof of the lemma.
  \end{proof}

 \section{Fellow-traveling extension arcs}
\label{s:fellow traveling paths}
 
 \newcommand{\li}[0]{\hat{\mathcal{I}}}
 \renewcommand{\aa}[0]{\mathrm{A}}  
 \newcommand{\bb}[0]{\mathrm{B}}
 \newcommand{\uh}[0]{H\backslash G}  
 \newcommand{\uhh}[0]{H'\backslash G} 
 
   \newcommand{\EH}[0]{\mathcal{E}(H)}
  \newcommand{\EHH}[0]{\mathcal{E}(H')}
  \newcommand{\ut}[0]{\t^{H}_{\Theta}}
  \newcommand{\utt}[0]{\t^{H'}_{\Theta}}
  \newcommand{\ko}[0]{k_{o}}
  \newcommand{\go}[0]{G_{\Omega_{0}}}

  In this section we study under which conditions translates of different $p$-extension arcs may fellow travel sufficiently close for a definite distance. As it turns out, this imposes very tight algebraic restrictions on the translating elements of the $\alpha$-arcs in question. 

  \begin{assumption}\label{assumption section parallel paths}
   For the remainder of this section, we assume we are given $H<H'\leq G$, as well as finite $\Omega\leq G$ and $\mathcal{S}\subseteq G$ such that $H'$ is a  $(\Omega,\mathcal{S},\underline{g}_{-1},\underline{g}_{1},\alpha)$-extension of $H$ with extension tuple $(\alpha_{i})_{i\in k_{o}}$, as per \cref{d:basic extension setting} and \cref{d:good extension setting}. Let constants $\lambda_{0},\mu_{0}>0$ as in \cref{constant assumptions} be given and suppose that for some $p\in X$ the extension is $(p,\nu)$-geometric. We write $L=d(p,\alpha p)$. 
  \end{assumption}
  
   Recall that in \cref{d:good extension setting} there are two cases: a generous case, described in 
  \ref{strong case}, and a linear case, described in \ref{weak case}.
  \begin{notation}\label{Gamma}
  	We let $\Gamma_{\Omega}$ be the subgroup of $G$ given by: 
  	\begin{itemize}
  		\item $\Gamma_{\Omega}=\Omega$ in the generous case,
  		\item  $\Gamma_{\Omega}=G_{\Omega}=N_{G}(\Omega)$ in the linear case. 
  	\end{itemize}
  \end{notation}
  
  \begin{remark}
  	Note that in the applications, it will turn out that the group $\Gamma_{\Omega}$ will depend exclusively on $\Omega$. 
  \end{remark}
  
  \begin{definition}\label{d:sets of subarcs}
  		Given a translate $\tau$ of a $p$-extension arc we write
  	\begin{itemize}
  		\item $\mathcal{A}_{H}(\tau)$ for the collection of $H$-subarcs of $\tau$,
  		\item $\mathcal{A}_{\alpha}(\tau)$ for the collection of $\alpha$-subarcs of $\tau$, 
  		\item $\mathcal{A}_{H\alpha}(\tau)=\mathcal{A}_{H}(\tau)\cup\mathcal{A}_{\alpha}(\tau)$. 
  	\end{itemize}
  	The set $\mathcal{A}_{H\alpha}(\tau)$ has a natural ordering $<$ induced by the parametrization of $\tau$. We will thus speak of \emph{consecutive} arcs in $\mathcal{A}_{H\alpha}$, \emph{intervals of subarcs of $\tau$} and so forth.
  \end{definition}
  
  \begin{definition}\label{congruence}
  	Given arcs $J_{0}$ and $J_{1}$, as well as a constant $\kappa\in(0,1)$  we write $J_{0}\simeq_{\kappa} J_{1}$ if there are points $p_{i},q_{i}\in J_{i}$, $i\in 2$ satisfying
  	\begin{itemize}
  		\item $d(p_{i},q_{i})>\kappa L$ for at least one $i\in 2$;
  		\item $d(p_{0},p_{1}),d(q_{0},q_{1})\leq\lambda_{0}L$;
  		\item $p_{i}$ appears before $q_{i}$ along $J_{i}$ for each $i\in 2$. 
  	\end{itemize}    
   We will use the following abbreviations: 
   \begin{itemize}
   	\item We write $\simeq$ for $\simeq_{5\lambda_{0}}$, and if $J\simeq J'$ we say that $J$ and $J'$ \emph{match}.
   	\item We write $\simeq^{s}$ for $\simeq_{(1-10^{-2})}$, and if $J\simeq^{s}J'$ we say that $J$ and $J'$ \emph{strongly match}.
   \end{itemize}
   We will also write $J\simeq^{\pm}_{\kappa}J'$ and related expressions as an abbreviation of the disjunction of $J^{-1}\simeq_{\kappa}J'$ and $J\simeq_{\kappa}J'$.  
 \end{definition}
  In this section this definition will always be applied to $H$-subarcs and $\alpha$-subarcs of translates of $p$-extension arcs.  
 
  We begin our discussion of the application of the definitions above with the following technical result.
 \begin{lemma}\label{simeq consequence}
 	Let $\rho>0$ and assume that $J_{i}\in\mathcal{A}_{\alpha}(\tau_{i})$ for translates of $p$-extension arcs $\tau_{0},\tau_{1}$. If $J_{0}\simeq_{1-\rho}J_{1}$, then
 	\begin{equation*}
 		d(J_{0}^{-},J_{1}^{-}),d(J_{0}^{+},J_{1}^{+})\leq (2\rho+\lambda_{0})L.
 	\end{equation*}
 	If $J_{0}\simeq_{1-\rho}J_{1}^{H}$ for $J_{1}^{H}\in\mathcal{A}_{H}(\tau_{1})$, then there are points
 	$r_{1},r'_{1}\in J_{1}^{H}$, appearing in that order, such that 
 	\begin{equation*}
 		d(J_{0}^{-},r_{1}),d(J_{0}^{+},r_{1}')\leq (\rho+\lambda_{0})L.
 	\end{equation*}
 	
 	Moreover, if $\rho\leq 10^{-4}$, and there are subarcs $\tilde{\tau}_{i}\subseteq\tau_{i}$ such that $\tilde{\tau}_{0}$ and $\tilde{\tau}_{1}$ are $\lambda_{0} L$-fellow traveling, then
 	\begin{equation*}
 		|J_{0}\cap\tilde{\tau}_{0}|\geq(1-10^{-4})L,\quad\Rightarrow \quad |J_{1}\cap\tilde{\tau}_{1}|\geq(1-10^{-3})L.
 	\end{equation*}
 \end{lemma}
 \begin{proof}
 	Note that the two points in $J_{0}$ witnessing each of the two relations $\simeq_{\rho}$ in the statement must be each at distance at most $\rho L$ from an endpoint of $J_{0}$. In the first case the same is true for the two corresponding points in $J_{1}$. The result now easily follows from the triangular inequality.
 	
 	For the moreover part, note that if the conclusion fails, then there must be $q\in\tau_{1}$ 
 	with 
 	\begin{equation*}
 		d(q,\tilde{\tau}_{1})\geq (\frac{1}{4\cdot 10^{3} }+3\lambda_{0})L,\quad \quad d(q,\partial \tau_{1})\geq (\frac{1}{4\cdot 10^{3} }-3\lambda_{0})L.
 	\end{equation*}
 	On the other hand, it follows from our assumption that $J_{0}\simeq_{1-\rho}J_{1}$ and the bound on $\rho$ that the point $q$ lies in the subarc of $J_{1}$ between the two points in $J_{1}$ witnessing $J_{0}\simeq J_{1}$, and at distance at least $\lambda_{0}L$ from each. Therefore $d(q,r)\leq\lambda_{0}L$ for some $r\in\tau_{0}$ and it is not hard to show using the triangular inequality that $d(r,\partial \tau_{0})\geq 10^{-3}L$ and thus $r\in\tilde{\tau}_{0}$. On the other hand, 
 	\begin{equation*}
 		d(r,\tilde{\tau}_{1})\geq d(q,\tilde{\tau}_{1})-\lambda_{0}L\geq (\frac{1}{4\cdot 10^{3} }-4\lambda_{0})L>\lambda_{0}L,
 	\end{equation*}
 	in contradiction with $\dih(\tilde{\tau}_{0},\tilde{\tau}_{1})\leq\lambda_{0}L$. 
 \end{proof}
 
 Clearly the relation $\simeq^{s}$ is finer than $\simeq$. Still working under \cref{assumption section parallel paths}, we also have the following, which follows from \cref{d:transverse precise}, \cref{d:small cancellation pointed} and \cref{transversality is inherited}, as well as the equality $H\cdot \mathcal{S}\cdot\G=H\cdot\mathcal{S}$.  
  \begin{observation}\label{consequence translating elements}
    For $i\in 2$ let $J_{i}=\zeta_{i}[p,\rho_{i} p]$, where 
    \begin{itemize}
    	\item $\rho_{i}\in \Gamma_{\Omega}$ in the linear case,  
    	\item $\rho=\alpha$ in the generous case.
    \end{itemize}
    Let also $J_{1}^{H}=\xi_{1}[p,\eta p]$ for some $\eta\in H$. 
    Then we have the following: 
  	\begin{enumerate}[label=(\roman*)]
  		\item \label{aa case} If $J_{0}\simeq J_{1}$, then $\zeta_{1}^{-1}\zeta_{0}\in\Gamma_{\Omega}$, for $\Gamma_{\Omega}$ as in \cref{Gamma}.
  		\item \label{aa case+} In the generous case \ref{strong case}, the conclusion of \ref{aa case} implies 
  		\begin{equation*}
  			\dih(J_{0},J_{1})\leq\nu L.
  		\end{equation*}
  		\item \label{H case}If $J_{1}\simeq J_{1}^{H}$, then $\xi_{1}^{-1}\zeta_{0}\in H\cdot \mathcal{S}$.   
  	\end{enumerate}
    In particular, the result holds for 
    $J_{i}\in\mathcal{A}_{\alpha}(\tau_{i})$, $i\in 2$, and $J_{1}^{H}\in\mathcal{A}_{H}(\tau_{1})$, where $\tau_{0},\tau_{1}$ are translates of $p$-extension arcs. 
    In this situation we also have the following. 
    \begin{enumerate}[label=(\roman*)] \setcounter{enumi}{3}
    	\item \label{aa case exponent} In \ref{aa case} we also have that the exponent associated to $J_{0}$ and $J_{1}$ is the same.
    \end{enumerate}
  \end{observation}
  
  In the case in which $J_{i}$ are $\alpha$-subarcs of some translate of $p$-extension arcs we will always apply this with $\zeta_{i}$ the translating element of $J_{i}$. 
  \begin{definition}
  	In the linear case for any arc of the form $J=\zeta[p,\rho p]$, where $\rho\in\Gamma_{\Omega}$ and $\zeta\in G$  
  	we will refer to $\zeta$ as \emph{a translating element of $J$}.
  \end{definition}

  \begin{proof}
  	Indeed, by the triangular inequality, the relation $J\simeq_{\nu+4\lambda_{0}}J'$ implies 
    \begin{equation*}
    	\diam_{J}(J')>\nu L, 
    \end{equation*}  
    in view of which the conclusion is immediate from condition \ref{hyp_tp} and either \ref{hyp_sc} or \ref{hyp_stab} of \cref{d:good extension setting}.
  \end{proof}

  \subsection{Pairs of (partially) fellow traveling extension arcs}

  We begin by showing that if an $\alpha$-subarc $J$ of a translate of a $p$-extension arc $\lambda_{0}L$-fellow travels with some sufficiently long subarc of another translate of a $p$-extension arc, then the latter must contain some $H$ or $\alpha$-subarc that matches $J$ in the sense of \cref{congruence}, and that in fact the match is strong under uniqueness.
  
  \begin{definition}
  	Given $\rho\in(0,1)$, a subarc $\tilde{\tau}\subseteq\tau$ and some $\alpha$-subarc $J$ of $\tau$,  write 
  	$J\sqsubseteq_{\rho} \tilde{\tau}$ if 
  	\begin{equation*}
  		|J\cap\tilde{\tau}|\geq(1-\rho)L.
  	\end{equation*}
  \end{definition}

  \begin{lemma}\label{existence of close arc}
 	 	Let $\tau,\tau'$ be translates of $p$-extension arcs, and assume that $\tau_{0}\subseteq\tau$ and $\tau'_{0}\subseteq\tau'$ are subarcs which are $\lambda_{0}L$-fellow traveling. 
 	 	Let $J\in\mathcal{A}_{\alpha}(\tau)$, and assume that $J\sqsubseteq_{10^{-3}}\tau_{0}$. Then the following holds.
 	 	\begin{enumerate}[label=(\alph*)]
 	 		\item \label{claim1} There exists $J'\in\mathcal{A}_{H\alpha}(\tau')$ such that $J\simeq J'$. 
 	 		\item \label{claim2} Assume that the $J'$ above is unique. Then in fact in \ref{claim1} the stronger conclusion $J\simeq^{s}J'$ holds. 
 	 	\end{enumerate}
 \end{lemma} 
  \begin{proof}
  	 Write $J_{0}=J\cap\tau_{0}$. Let $q_{\pm}=J_{0}^{\pm}$ be the two endpoints of $J_{0}$ and pick $r_{\pm}\in \tau'$ with $d(q_{\pm},r_{\pm})\leq\lambda_{0}L$. Since $\tau$ and $\tau'$ are $(1+\mu_{0},\mu_{0}L)$-quasigeodesic arcs and \cref{constant assumptions} holds, we know from \cref{quadrilaterals general} that there is some subarc 
  	 $\tau'_{1}$ of the subarc of $\tau'$ between  $r_{-}$ and $r_{+}$ such that 
  	 \begin{itemize}
  	 	\item $\tau_{0}'\subseteq\nb_{\lambda_{0}L}(J_{0})$,
  	 	\item $M :=\diam(\partial \tau_{0}')\geq|J_{0}|-3\lambda_{0}L\geq (1-10^{-3}-3\lambda_{0})L$.
  	 \end{itemize}
     
     The following is now immediate from the definition of $\simeq_{\kappa}$:
     \begin{observation}\label{implication diameter}
     	For all $\kappa\in(0,1)$ and all $J'\in\mathcal{A}_{H\alpha}(\tau')$ if $J\nsimeq_{\kappa} J'$, then $\diam(J'\cap\tau'_{0})\leq\kappa L$.
     \end{observation}
     
     Note that by the triangular inequality and the continuity of $\tau'_{0}$ we know that $\diam(\partial \tau_{0}')$ is bounded from above by the sum of the diameters of all the non-degenerate intersections of $\tau'_{0}$ with $H$-subarcs, $\alpha$-subarcs or remainder subarcs of $\tau'$. Recall also that by definition remainder subarcs have diameter bounded by $\nu L$. 
          
     If $J\nsimeq J'$ for all $J'\in\mathcal{A}_{\alpha}(\tau')$, then by \eqref{implication diameter} we have  $\diam(J'\cap\tau'_{0})\leq 5\lambda_{0}L$ for any such $J'$, which thereby cannot be contained in $\tau'_{0}$. It follows that there is at most one $J'_{H}\in\mathcal{A}_{H}(\tau')$ and at most two remainder subarcs of $\tau'$ whose intersection with $\tau_{1}'$ is non-degenerate.
     
     From this we conclude that there is $J'_{H}\in\mathcal{A}_{H}(\tau')$ with 
     \begin{equation*}
     	\diam(J'_{H}\cap\tau'_{0})\geq M-(2\nu+10\lambda_{0})L>(1-10^{-2})L,
     \end{equation*}
     so $J\simeq^{s}J'_{H}$ by \cref{implication diameter}. This is enough to establish claim \ref{claim1}. 
 
     To establish \ref{claim2} assume that $J\simeq J'\in\mathcal{A}_{\alpha}(\tau')$. If $J\nsimeq^{s} J'$, then 
     by \cref{implication diameter} we must have $\diam(J'\cap\tau'_{0})\leq(1-10^{-2})L$. In particular, $J'\nsubseteq \tau'_{0}$. 
     
     Our uniqueness hypothesis and \ref{implication diameter} thus imply that any other $J''\in\mathcal{A}_{\alpha}(\tau')$ must satisfy $\diam(J''\cap\tau'_{0})\leq 5\lambda_{0}L$ and thus $J''\nsubseteq \tau_{0}'$. Therefore, $\tau'_{0}$ can have non-degenerate intersection with: 
     \begin{itemize}
     	\item at most three remainder subarcs, 
     	\item at most one $J'_{H}\in\mathcal{A}_{H}(\tau')$, 
     	\item and at most one $J''\in\mathcal{A}_{\alpha}(\tau')$ distinct from $J'$,  
     \end{itemize}
     where $J',J'_{H},J''$ are consecutive in $\mathcal{A}_{H\alpha}(\tau')$. 
     
     This implies the existence of $J'_{H}\in\mathcal{A}_{H}(\tau')$ such that
     \begin{equation*}
     	 \diam(J'_{H}\cap\tau'_{0})\geq M-\left[(1-10^{-2})L+3\nu L+5\lambda_{0}L\right]\geq 5\lambda_{0}L,
     \end{equation*}
     so that $J\simeq J'_{H}$ against our uniqueness assumption, concluding the proof of claim \ref{claim2} of \cref{existence of close arc}.
  \end{proof}

 Our next goal is to prove the uniqueness of an arc $J'$ matching with a given $\alpha$-subarc $J$. We use an algebraic argument to exclude neighbouring double matches, and a geometric one to exclude any gaps in the collection of arcs matching with $J'$. 
 
 	 \begin{lemma}\label{uniqueness of matching}
  	Let $\tau$ be a translates of $p$-extension arc and consider some arc of the form
  		$J'=\zeta[p,\rho p]$,
  	where 
  	\begin{itemize}
  		\item $\rho=\alpha$ in the generous case (in this case we are only interested in the case in which $J$ is an $\alpha$ subarc of some translate of a $p$-extension arc),
  		\item $\rho\in N_{G}(\Omega)$ in the linear case.
  	\end{itemize}
    Then there can be at most one $J\in\mathcal{A}_{H\alpha}(\tau)$ such that $J\simeq J'$. 
    In particular, the above result holds for $J'\in\mathcal{A}_{\alpha}(\tau')$ for some other translate $\tau'$ of a $p$-extension arc. 
  \end{lemma}
  \begin{proof}
  	Let $J_{0}<J_{1}$ be arcs in $\mathcal{A}_{H\alpha}(\tau)$ such that $J'\simeq J_{0}$, $J'\simeq J_{1}$, and let $\xi_{i}$ be the translating element of $J_{i}$. We begin by ruling out the possibility that the $J_{i}$ are strictly consecutive non-remainder arcs. 
    \begin{lemma}\label{no consecutive Halpha}
    	The arcs $J_{0}$ and $J_{1}$ cannot be consecutive in $\mathcal{A}_{H\alpha}(\tau)$.
    \end{lemma}
     \begin{subproof}
     	Of two consecutive subarcs one must be an $\alpha$-subarc and the other an $H$-subarc. 
     	Without loss of generality, we may assume $J_{0}\in\mathcal{A}_{H}(\tau)$ and $J_{1}\in\mathcal{A}_{\alpha}(\tau)$. By \cref{consequence translating elements} we thus have
     	\begin{equation*}
     		\xi_{0}^{-1}\zeta\in H\cdot \mathcal{S},\quad\quad \beta=\xi_{1}^{-1}\zeta\in\Gamma_{\Omega}. 
     	\end{equation*}   
     	
     	On the other hand, the fact that $J_{0}<J_{1}$ are consecutive implies that $\xi_{1}=\xi_{0}g_{i,\epsilon}^{-1}$ for some $i\in k_{o}$ and $\epsilon\in\{\pm1\}$. We thus have
     	\begin{equation*}
     		H\cdot \mathcal{S} \ni\xi_{0}^{-1}\zeta=g_{i,\epsilon}^{-1}\xi_{1}^{-1}\zeta\in g_{i,\epsilon}\Gamma_{\Omega},
     	\end{equation*}
     	We thus have $g_{\epsilon}\in H\cdot\mathcal{S}$, contradicting either 
     	item \ref{hyp away from fixed points} in \cref{d:pairs of tuples} or \ref{weak hyp away from fixed points}
     	in \cref{weakly admissible}.
     	The case in which $J_{0}\in\mathcal{A}_{\alpha}(\tau)$, $J_{1}\in\mathcal{A}_{H}(\tau)$ is entirely symmetrical.
     \end{subproof}
     
     Next, we rule out the possibility of $J_{0}$ and $J_{1}$ being $\alpha$-subarcs separated by a single $H$-subarc. In the generous case the argument is geometric and yields a stronger conclusion right away.
     \begin{observation}
     	In the generous case \ref{strong case} the arcs $J_{0}$ and $J_{1}$ cannot be two distinct $\alpha$-subarcs of $\tau$. 
     \end{observation}
     \begin{proof}
     	Indeed, if they are, then it follows from part \ref{aa case+} of \ref{consequence translating elements} that 
     	$\dih(J_{0},J_{1})\leq 2\nu L<\mu_{0}L$, which is easily seen to contradict the fact that $\tau$ is $(1+\mu_{0},\mu_{0}L)$-quadigeodesic.    
     \end{proof}
     
     \begin{lemma}\label{no consecutive alphaalpha}
    	In the linear case \ref{weak case} the arcs $J_{0}$ and $J_{1}$ cannot be consecutive arcs in $\mathcal{A}_{\alpha}(\tau)$. 
     \end{lemma}
     \begin{subproof}
     	We abbreviate $g_{0,\epsilon}$ by $g_{\epsilon}$, and writing $\epsilon_{i}$ for the exponent associated with the $\alpha$-subarc $J_{i}$, and letting $\eta\in H$ be the label of the $H$-subarc between $J_{0}$ and $J_{1}$ we have: 
      \begin{equation*}
     		\beta_{i}=\xi_{i}^{-1}\zeta\in G_{\Omega},\;\,i\in 2;\quad\quad \xi_{0}^{-1}\xi_{1}=g_{\epsilon_{0}}^{-1}\eta g_{-\epsilon_{1}},
     	\end{equation*}  
     	so that
     	\begin{equation*}
     	 \eta g_{-\epsilon_{1}}= g_{\epsilon_{0}}\beta_{0}\beta_{1}^{-1}.
     	\end{equation*}
     	If $\epsilon_{0}=\epsilon_{1}$. Then the equation above implies that $g_{1}\in H\cdot g_{-1} \cdot G_{\Omega}$, contrary to property \ref{hyp_sep} in \cref{d:good extension setting}.
     	
     	If $\epsilon_{0}=-\epsilon_{1}$, then $\eta^{g_{\epsilon_{0}}}\in G_{\Omega}$. 
     	However, the element $g_{\epsilon_{0}}$ is assumed to be linearly admissible in the sense of \cref{weakly admissible}, and thus by \ref{weak hyp away from fixed points} therein we must have $\eta=1$, contradicting the fact that $p$-extension arcs are constructed from reduced normal words.
     \end{subproof} 
     The conclusion of the main lemma is now a consequence of the following observation, which follows easily from \cref{quadrilaterals general} and the fact that $\tau$ is a $(1+\mu_{0},\mu_{0}L)$-quasigeodesics.
     \begin{observation}
     	 Let $J\in\mathcal{A}_{\alpha}(\tau)$ and assume that there are points $q_{-},q_{+}\in \tau\cap\nb_{\lambda_{0}L}(J')$ such that $q_{-}$ appears before any point in $J$ and $q_{+}$ appears after any point in $J$. Then $J\simeq J'$.     
     \end{observation}
     Indeed, the observation above applies to any $\alpha$-subarc of $\tau$ that appears strictly between $J_{0}$ and $J_{1}$. This, together with $J'\simeq J_{i}$ for $i\in 2$, easily implies the existence of a pair of subarcs in 
     $\mathcal{A}_{H\alpha}(\tau)$ matching with $J'$ and either consecutive in $\mathcal{A}_{H\alpha}(\tau)$ or in
     $\mathcal{A}_{\alpha}(\tau)$, contradicting either \cref{no consecutive Halpha} or \cref{no consecutive alphaalpha}.   
      
     This finishes the proof of \cref{uniqueness of matching}. 
  \end{proof}
   
   The following is an important consequence of \cref{uniqueness of matching}.
   \begin{corollary}\label{transverse to homega0}
   	In the linear case of \cref{d:good extension setting} any loxodromic element in $\Gamma_{\Omega}$ is transverse to the extended group $H'=\subg{H,\alpha_{0}}$ relative to $\mathcal{S}$, where recall that here $\mathcal{S}$ is the finite set such that $\alpha\in\Gamma_{\Omega}$ is $(p,\nu)$-transverse to $H$ relative to $\mathcal{S}$.
   \end{corollary}
   \begin{proof}
   	Indeed, assume that there are $\rho\in\Gamma_{\Omega}$, $g\in G$, $h\in H'$ such that $\diam_{[p,\rho p]}(g^{-1}[p,hp])\geq 3L$, and let $\tau=g^{-1}\pa(h)$. Then it is easy to see using \cref{quadrilaterals general} that either 
   	\begin{itemize}
   		\item there is $J\in\mathcal{A}_{H}(\tau)$  such that $J\simeq^{\pm}[p,\rho p]$;
   		\item there are two distinct $J_{0},J_{1}\in\mathcal{A}_{\alpha}(\tau)$ such that $J_{l}\simeq^{\pm}[p,\rho p]>\nu L$ for each $l\in 2$. 
   	\end{itemize}
    The second possibility is ruled out by \cref{uniqueness of matching}, so the first one must hold.
    Writing $g^{-1}h_{0}$ for the translating element of $J$, the assumption that $\alpha$ is $(p,\nu)$-transverse to $H$ relative to $\mathcal{S}$ implies the existence of $\eta\in H$, $\rho'\in\Gamma_{\Omega}$ and $s\in\mathcal{S}$ such that $g^{-1}h_{0}\eta s=\rho'$ and thus that $g\in H'\cdot\mathcal{S}\cdot\Gamma_{\Omega}=H'\cdot \mathcal{S}$.  
   \end{proof}

  \begin{remark}
  	It is worth considering whether in the generous case the conclusion of \ref{no consecutive alphaalpha} can be established using a purely algebraic argument similar to the one used in the linear case, with the role of $G_{\Omega}$ being played by the smaller subgroup $\Omega$. In fact, it cannot. Indeed, the algebraic property that 
  	$\underline{g}_{-1},\underline{g}_{1}$ would need to satisfy for this to work would be 
  	\begin{equation*}
  		\forall i,j\in k_{o}\,\,\,g_{i,1}\notin Hg_{j,1}\Omega .
  	\end{equation*}
  	All we can require here is that the projections in $H\backslash G$ of the sets $\mathcal{S}\cup\{g_{i,-1}\omega\,|\,i\in k_{o},\omega\in\Omega\}$ and $\mathcal{S}\cup\{g_{i,1}\omega\,|\,i\in k_{o},\omega\in\Omega\}$ are not in the same orbit under the action of $G$ on the right, but for $k_{o}>1$ this is not enough to draw any conclusion regarding pairs of cosets from each side, $Hg_{i,-1},Hg_{j,1}$. So it seems that indeed the problem demands a stronger geometric assumption in this case. 
  \end{remark}

   \subsection{Finite collections of partially fellow traveling paths}
     
     In this subsection we move on from studying pairs of (partially) closely fellow traveling $p$-extension arcs to the study of finite sets of jointly (partially) closely fellow traveling $p$-arcs.

  \newcommand{\sa}[0]{\tilde{\mathcal{I}}}
  
    For the sake of using \cref{p:intersection_many3} below in future work, together with the statement needed in the proof of \cref{t: main}, we include a different one with a stronger assumption and a conclusion, centered around the following notion, inspired by/generalizing the notion of geometric separation in \cite{hull2017small}. 
    
    We emphasize that none of this is needed for the understanding any of the main results in this paper. Neither these definitions, nor part \ref{Delta} of \cref{many fellow traveling paths basic result} below play any role in understanding the main results of this work. 
   \begin{definition}\label{delta}
   	 Given geodesic arcs $(J_{i})_{i\in k}$, we write 
   	 \begin{equation*}
   	 \Delta(J_{0},\dots, J_{k-1})=\sup_{i\in k}\diam(\bigcap_{j\in k}\overline{\pi_{J_{i}}(J_{j})}),
   	 \end{equation*}
   	 where $\overline{K}$ for $K\subseteq J_{0}$ denotes the convex closure of $K$ in $J_{0}$, identifying the latter with a real interval. 
   \end{definition}
  	 \begin{definition}
	 	\label{d:geometrically separated point}
	 	Let $H\leq G$, $k>1$ some integer and $p\in X$. We say that $H$ is $(k,p,C)$-geometrically separated if for all $h_{0},\dots, h_{k-1}\in H$ and $g_{0},\dots, g_{k-1}\in G$ at least one of the following holds
	 		\begin{itemize}
	 			\item $Hg_{i}=Hg_{j}$ for some choice of distinct $i,j\in k$,
	 			\item $\Delta((g_{j}[p,h_{j}p])_{j\in k})\leq C$.  
	 		\end{itemize} 
 		 We say that $H$ is geometrically $k$-separated if it is $(k,p,C)$-geometrically separated for some (equivalently, for all) $p\in X$ and some $C=C(p)>0$.
	 \end{definition}

  \begin{lemma}
  	\label{many fellow traveling paths basic result} 
    Let $k$ be a fixed positive integer, and let $(\tau_{j})_{j\in k}$ be a tuple of translates of $p$-extension arcs. \footnote{Note that for now no assumptions are made on the relationship between $k$ and $k_{o}=|\underline{g}|_{\pm1}$.}
    Assume that we are given subarcs $\tilde{\tau}_{j}\subseteq\tau_{j}$ such that $\tilde{\tau}_{j}$ and $\tilde{\tau}_{j'}$ are $\lambda_{0}L$-fellow traveling for any pair $j,j'\in k$. 
    For $j\in k$ write 
    \begin{equation*}
    	\mathcal{I}_{j}=\{J\in\mathcal{A}_{\alpha}(\tau_{j})\,|\,J\sqsubseteq_{10^{-4}}\tilde{\tau}_{j}\},\quad\quad\sa_{j}=\{J\in\mathcal{I}_{j}\,|\,\,\exists j'\in k\,\exists J'\in\mathcal{I}_{j'}\,\,J\simeq J'\}.
    \end{equation*}
  	Then there is $m>0$ and for every $1\leq l\leq m$
  	some subset $\emptyset \neq M_{l}\subseteq k$, as well as the following: 
  	\begin{itemize}
  		\item for every $j\in M_{l}$ some $I_{j,l}\in\tilde{\mathcal{I}}_{j}$;
  		\item for every $j\notin M_{l}$ some  $I_{j,l}\in\mathcal{A}_{H}(\tau_{j})$,
  	\end{itemize}
   such that the following holds. 
  	\begin{enumerate}[label=(\alph*)]
  		\item \label{intervals1} $(I_{j,l})_{1\leq l\leq m}$ is a sequence of intervals in $\mathcal{A}_{H\alpha}(\tau_{j})$ in their natural order for every $j\in k$;
  		\item for fixed $j\in k$ we have \label{intervalsalpha} $\{I_{j,l}\,|\,1\leq l\leq m, \,j\in M_{l}\}=\sa_{j}$;
  		\item  for fixed $j\in k$ we have \label{intervalsH} $\{I_{j,l}\,|\,1\leq l\leq m, \,j\in k\setminus M_{l}\}\subseteq\mathcal{A}_{H}(\tau_{j})$;
  		\item \label{intervalsH2} for fixed $j\in k$ the value of $I_{j,l}$ is constant on each maximal integer subinterval in the complement of $\{1\leq l\leq m\,|\,j\in M_{l}\}$;
  		\item \label{equivalence of arcs} for $1\leq l ,l'\leq m$, $j\in M_{l}$ and $j'\in k$ the following are equivalent:
  		\begin{itemize}
  			\item $I_{j,l}\simeq I_{j',l'}$,
  			\item $I_{j,l}\simeq^{s} I_{j',l'}$,
  			\item $l=l'$.  
  		\end{itemize}
  	  \item \label{Delta} Let $1\leq l<m$ and take $i_{0}\in M_{l}$, $i_{1}\in M_{l+1}$. 
  	  Then there are $J^{H}_{j}\in\mathcal{A}_{H}(\tilde{\tau}_{j})$, $j\in k$ such that 
  	  \begin{equation*}
  	  	\Delta(J^{H}_{0},\dots, J^{H}_{k-1})\geq d(I_{i_{0},l},I_{i_{1},l+1})-\frac{1}{3}L. 
  	  \end{equation*} 
  	  \item \label{Delta'} Let $i_{0}\in M_{1},i_{1}\in M_{m}$. And let  \[ M=\max\{d(\tau_{i_{0}}^{-},I_{i_{0},1}),d(\tau_{i_{1}}^{+},I_{i_{1},m})\}.\] 
  	  If $M\geq \frac{4L}{3}$, then we can find $J_{j}\in\mathcal{A}_{H}(\tau_{j})$ for each $j\in k$ such that
  	  \begin{equation*}
  	  	\Delta(J_{0},\dots, J_{k-1})\geq M-\frac{4L}{3}.
  	  \end{equation*}
  	\end{enumerate}
  \end{lemma}
  \begin{proof}
  	Note that by the moreover part of \cref{simeq consequence} we must have 
    \begin{equation}\label{mostly inside}
     \forall J\in\sa_{j}\,\,\;(J\sqsubseteq_{10^{-3}}\tilde{\tau}_{j}).
    \end{equation}
  	We know from \cref{existence of close arc} that for any distinct $j,j'\in k$ and any 
  	$J\in\mathcal{I}_{j}$ there is $J'\in\mathcal{A}_{H\alpha}(\tau_{j})$ such that 
  	$J\simeq^{s}J'$, and that this $J'$ is in fact the unique arc $J'\in\mathcal{A}_{\alpha H}(\tau_{j})$ with $J\simeq J'$, by \cref{uniqueness of matching}. 
  	By definition, if $J'\in\mathcal{A}_{\alpha}(\tau_{j})$ then in fact $J'\in\sa_{j'}$. For simplicity, write $\mathcal{J}^{H}_{j}=\mathcal{A}_{H}(\tau_{j})$.  
    
  	\begin{lemma}\label{pseudotransitivity}
  		Let  $j,j',j''\in k$ and let the arcs $J,J',J''$ be as in one of the following cases: 
  	  \begin{enumerate}[label=(\alph*)]
  		\item \label{casea} $J\in\sa_{j}$, $J'\in\mathcal{I}_{j'}$ and $J''\in\mathcal{J}^{H}_{j''}$;
  		\item \label{caseb} $J\in\sa_{j}$, $J'\in\mathcal{I}_{j'}$ and $J''\in\sa_{j''}$.
  	 \end{enumerate} 
     Then 
     \begin{equation*}
  		J\simeq^{s}J',\,\,J'\simeq^{s} J''\quad\Rightarrow\quad\,\,J\simeq^{s} J'' 
  	\end{equation*}
  	\end{lemma}
  	\begin{subproof}
  		Using \cref{simeq consequence} and the triangular inequality, we easily derive in either case the existence of points $q^{-},q^{+}\in J''$, appearing in that order, such that: 
  		\begin{equation*}
  			d(J^{\pm},q^{\pm})\leq(4\cdot 10^{-3}+2\lambda_{0})L.
  		\end{equation*} 
  		By examining a geodesic quadrilateral between the four points above, and recalling the inequalities $\nu L>\delta$, and  $\lambda_{0}<10^{-4}$, we easily conclude that 
  		$J\simeq J''$. On the other hand, in case \ref{casea} we also have
  		$J\simeq^{s}J'$, by \eqref{mostly inside}, the final claim of \cref{existence of close arc} and \cref{uniqueness of matching}.
  	\end{subproof}
  	 
  	 Write $\mathcal{I}=\coprod_{j\in k}\sa_{j}$ and note that by \cref{pseudotransitivity} the relation 
  	 $\simeq$ determines an equivalence relation on $\mathcal{I}$. 
  	 
  	 \begin{definition}
  	 	  We define binary relations $\prec_{1}$ and $\prec_{2}$ on $\mathcal{I}$ as follows.
  	 	  \begin{itemize}
  	 	  	\item We let $J\prec_{1} J'$ for $J\in\sa_{j}$, $J'\in\sa_{j'}$ if  $J\simeq J''$ for some 
  	    $J''\in\mathcal{A}_{H\alpha}(\tau_{j'})$ such that $J''<J'$ along $\tau_{j'}$. 
  	 	  	\item We let $J\prec_{2} J'$ for $J\in\sa_{j}$, $J'\in\sa_{j'}$ if  $J'\simeq J''$ for some $J''\in\mathcal{A}_{H\alpha}(\tau_{j})$ such that $J<J''$ along $\tau_{j}$.
  	 	  \end{itemize}  
   	   For $l\in \{1,2\}$, we also write $J\preceq_{l}J'$ for the disjunction of $J\simeq^{s}J'$ and $J\prec_{l}J'$. 
  	 \end{definition}
    
    Note that in particular, if $J,J'\in\sa_{j}$, then $J\prec_{1}J'$ if and only if $J<J'$ in the sense of the parametrization of $\tau_{j}$.  
    Given a geodesic arc $J$, denote its midpoint by $q_{J}$. 
    We note the following.
    \begin{observation}\label{simple observation}
    	For any $J\in\sa_{j}$
    \begin{itemize}
    	\item $q_{J}\in J\cap\tilde{\tau}_{j}$ ,
    	\item if $J\simeq^{s} J'$ for $J'\in\sa$, then there is $q'\in J'\cap\tilde{\tau}_{j'}$ such that 
    	\begin{equation*}
    	   d(q_{J},q')\leq\lambda_{0}L,\quad \quad d(q',\partial J')\geq(\frac{1}{2}-\frac{3}{10^{2}}-\lambda_{0})L.    	
    	\end{equation*}
    	
    \end{itemize}  
    \end{observation}
      
    \begin{lemma}\label{equivalence prec}
      
      The following are equivalent:
      \begin{enumerate}[label=(\Alph*)]
      	\item \label{1 prec}$J\prec_{1}J'$; 
      	\item \label{2 prec}there is $q\in J$ such that $d(q,q')\leq\lambda_{0}L$ for some $q'\in \tau_{j'}$ such that:
      	\begin{itemize}
      		\item $q'$ appears before $J'$ along $\tau_{j'}$,
      		\item $d(q',J')\geq 10^{-2}L$;
      	\end{itemize}
         \item \label{3 prec} $d(q_{J},q')\leq\lambda_{0}L$ for some $q'\in \tau_{j'}$ such that:
      	\begin{itemize}
      		\item $q'$ appears before $J'$ along $\tau_{j'}$,
      		\item $d(q'',r')\geq\frac{L}{3}$ for all $q'',r'\in\tau_{j'}$ with $q''$ appearing no later than $q'$ and $r'$ appearing in $J'$ or after;
      	\end{itemize}
        \item \label{4 prec} $J\prec_{2}J'$. 
      \end{enumerate}
    \end{lemma}
    \begin{subproof}
    $\mathbf{\ref{2 prec}\Rightarrow\ref{1 prec}}$\emph{.} Assume that \ref{2 prec} holds, but  $J\prec_{1}J'$ fails.  We know that $J\simeq^{s}J''$ for some $J''\in\mathcal{A}_{H\alpha}(\tau_{j'})$, so $J'\leq J''$ and by the triangular inequality $J\subseteq\nb_{(10^{-2}+\lambda_{0})L}(I)$, where $I$ is some subarc of $\tau_{j'}$ all whose points appear no later than $J'$. 
    On the other hand, by \ref{2 prec}, there must be $q\in J$ and $q'\in\tau_{j'}$ such that $q'$ appears before $J'$ along $\tau_{j'}$ and $d(q',J')\geq 10^{-2}L$. 
    By the triangular inequality $d(q,q'')\leq 2\lambda_{0}L$ for some $q''$ occurring in or after $J'$ along $\tau_{j'}$. This is easily seen to violate the   $(1+\mu_{0},\mu_{0}L)$-quasigeodesic condition for $\tau_{j'}$, hence we are done. 
    
     $\mathbf{\ref{1 prec}\Rightarrow\ref{3 prec}}$\emph{.} Assume $J\prec_{1}J'$.
     Let $\tilde{J}$ be the unique subarc in $\mathcal{A}_{H\alpha}(\tau_{j'})$ such that 
     $J\simeq_{s} \tilde{J}$. By \cref{simple observation} we have $\tilde{J}<J'$ and there is 
     $\tilde{q}\in\tilde{J}\cap\tilde{\tau}_{j'}\cap\nb_{\lambda_{0}L}(q_{J})$ such that
      $d(\tilde{q},\tilde{J}^{+})\geq\kappa L$ for  
     \begin{equation*}
     	 \kappa=\frac{1}{2}-\frac{3}{10^{2}}-\lambda_{0}.
     \end{equation*}
     This implies that the distance between a parameter for $\tilde{q}$ in $\tau_{j'}$ and the parameter for any point in  $J'$ or larger is at least $\kappa L$, which together with the $(1+\mu_{0},\mu_{0}L)$-quasigeodesic condition for $\tau_{j'}$ and the inqualities in \cref{constant assumptions} easily implies \ref{3 prec}.   
    	
    $\mathbf{\ref{3 prec}\Rightarrow\ref{2 prec}}$\emph{.} This is trivial,   
    
    $\mathbf{\ref{1 prec}\Leftrightarrow\ref{4 prec}}$\emph{.} Clearly, $\prec_{2}$ is equivalent to statements analogous to \ref{2 prec}, \ref{3 prec}, with the role of the two sides exchanged. By the symmetry of the problem, it suffices to show that $J\prec_{1}J'$ implies that  $J'\preceq_{1}J$ does not hold, since the latter implies that $J<\tilde{J}$ for the unique $\tilde{J}\in\mathcal{A}_{\alpha H}(\tau_{j})$ with $\tilde{J}\simeq J'$, which means $J\prec_{2}J'$.  
    Now, $J\prec_{1}J'$ implies by \ref{3 prec} the existence of $q'\in\tilde{\tau}_{j'}\cap\nb_{\lambda_{0}L}(J)$ such that $d(\tilde{q},J')\geq\frac{L}{3}$ for any point $\tilde{q}\in\tau_{j'}$ appearing before $q'$. This rules out the possibility that $J'\preceq_{1}J$. 
    \end{subproof}
  
    In view of the above result, we will henceforth write $\prec$ for the equivalent relations $\prec_{1}$ or $\prec_{2}$.
    Next, we show: 
    \begin{lemma}
    	The following implications hold for $J\in\sa_{j},J'\in\sa_{j'},J''\in\sa_{j''}$:
    	\begin{itemize}
    		\item $J\prec J'$, and $J'\preceq J''$ implies $J\prec J''$;
    		\item $J\preceq J'$ and $J'\prec J''$ implies $J\prec J''$. 
    	\end{itemize}
    \end{lemma}
    \begin{subproof}
      It suffices to show the first assertion. Since $J\prec_{1} J'$, there is $q'\in\tilde{\tau}_{j'}\cap\nb_{\lambda_{0}L}(q_{J})$, which appears before $J'$ and satisfies $d(q',q^{*})\geq\frac{L}{3}$ for any point $q^{*}$ in the subarc $\rho\subseteq\tau_{j'}$ consisting of all points appearing in $J'$ or after. 
      
      The midpoint $q_{J}$ is also at distance at most $\lambda_{0}L$ from some point $\tilde{q}\in\tilde{\tau}_{j}\cap \tilde{J}$ for some $\tilde{J}\in\mathcal{A}_{H\alpha}(\tau_{j''})$. By the triangular inequality we have 
      $d(q',\tilde{q})\leq 2\lambda_{0}L$ and  
      $d(\tilde{q},\rho)\geq(\frac{1}{3}-2\lambda_{0})L>\frac{1}{4}L$. To conclude that $J\prec J''$, it suffices to show (by \ref{2 prec} in \cref{equivalence prec}) that $\tilde{q}$ appears before $J''$ along $\tau_{j''}$.
      Let $\zeta\subseteq\tilde{\tau}_{j''}$ be the subarc consisting of all points appearing in $J''$ or later.   
      Since $J'\prec_{2}J''$, we know that $\zeta\subseteq\nb_{10^{-1}L}(\rho)$. An easy application of the triangular inequality yields $d(\tilde{q},\zeta)>\frac{1}{5}L$, hence the result.
    \end{subproof}
    
    Let now $m$ be the number of $\simeq$-equivalence classes in $\mathcal{I}$, which we enumerate as 
    \begin{equation}
    	\mathcal{C}_{1}\prec \mathcal{C}_{2}\prec\dots \mathcal{C}_{m}. 
    \end{equation}
    For each $1\leq l\leq m$ we let 
    \begin{equation*}
    	M_{l}=\{j\in k\,|\,\sa_{j}\cap\mathcal{C}_{l}\neq\emptyset\},
    \end{equation*}
    and for each $j\in k$ we write $I_{j,l}$ for the subarc $J\in\mathcal{A}_{H\alpha}(\tau_{j})$ such that $J\simeq J'$ for some $J'\in\mathcal{C}_{l}$, which is unique by \cref{pseudotransitivity}.
    Conditions \ref{intervalsalpha}, \ref{intervalsH}, and \ref{equivalence of arcs} in the statement follow by construction and the fact that $\simeq$  is an equivalence relation and $|\mathcal{C}_{l}\cap\mathcal{I}_{j}|\leq 1$ for all $1\leq l\leq m$, $j\in k$, whereas \ref{intervals1} is immediate from the definition of $\prec_{1}$. 
      
    We claim that for $j\in k$ and $1\leq l<l'\leq m$ such that $j\in k\setminus (M_{l}\cup M_{l'})$ the equality $I_{j,l}=I_{j,l'}$ holds. Indeed, otherwise there must be some $J'\in\mathcal{A}_{\alpha}(\tau_{j})$ with $I_{j,l}<J'<I_{j,l'}$.
    However, by \cref{simple observation}, we have $J'\subseteq\tilde{\tau}_{j}$ and thus $J'\in\mathcal{C}_{s}$ for some $l<s<l'$, reaching a contradiction. 
    This establishes condition \ref{intervalsH2}.
    
    Condition \ref{Delta} follows easily from the following sublemma.
    \begin{lemma}
    	Let $j_{0}\in M_{l}$, $j_{1}\in M_{l+1}$ for some $1\leq l<m$. 
    	Let $q_{0}=I_{j_{0},l}^{+}$ and $q_{1}=I_{j_{1},l+1}^{-}$. 
    	For every $j\in k$ let $J_{j}$ be the unique arc $J\in\mathcal{A}_{H}(\tau_{j})$ such that at least one of the following holds:
    	\begin{itemize}
    		\item $I_{j_{s},l}\simeq J$ for some $s\in 2$,
    		\item  $I_{j_{0},l}<J<I_{j_{1},l+1}$ if $j\in M_{l}\cap M_{l+1}$. 
    	\end{itemize} 
    	
    	Then $d(q_{s},J_{j})\leq \frac{1}{20}L$ for every $s\in 2$, $j\in k$.
    \end{lemma}
    \begin{subproof}
    	By symmetry, it suffices to prove the result for $s=0$. 
    	Recall that remainder subarcs have diameter at most $\nu L$. Thus the result is immediate for $j=j_{0}$ and, similarly for a general $j\in M_{l}$ the first part \cref{simeq consequence} implies
    	$d(q_{0},I_{j,l}^{+})\leq(\frac{2}{10^{2}}+\lambda_{0})L$, whereas $d(I_{j,l}^{+},q)\leq\nu L$ where $q$ is a point in the $H$-subarc of $\tau_{j}$ appearing immediately after $I_{j,l}$ and the result follows from the triangular inequality. 
    	If $j\notin M_{l}$, then an even tighter inequality can be derived from the second part of \cref{simeq consequence}, since in this case $J_{j}\simeq I_{j_{0},l}$. 
    \end{subproof}
    Claim \ref{Delta'} follows by a very similar argument, which is left to the reader.
    This concludes the proof of \cref{many fellow traveling paths basic result}.
  \end{proof}

  The next corollary translates \cref{many fellow traveling paths basic result}, which is geometric in nature, into an algebraic statement, which is then refined in \cref{p:intersection_many3} under somewhat stronger assumptions, lying down the foundation for the induction step of \cref{t: main}. 
  Note that in these two results we make explicit the $p$-extension arc $\tau_{j}$ of which $\hat{\tau}_{j}$ was a translate, since the element mapping $\tau_{j}$ to $\hat{\tau}_{j}$ becomes relevant to the statement. 
  
  \begin{corollary}\label{p:many parallel paths}
  	  Under the hypothesis and notation of \cref{many fellow traveling paths basic result}, assume furthermore that 
  	  \begin{equation*}
  	  	\tau_{j}=\gamma_{j}^{-1}\cdot\pa(h_{j})
  	  \end{equation*}
  	  for $\gamma_{j}\in G$ and $h_{j}\in H'$.
  	  For $j\in k$ and $1\leq l\leq m$ write $i(j,l)\in k_{o}$ for the index associated with $I_{j,l}$,  $\epsilon_{l}$ for the associated exponent if $j\in M_{l}$, and $\zeta_{j,l}$ for the translating element of $I_{j,l}$.
     Then for $(j,l)\in k\times\{1,\dots, m\}$ we can find elements $a_{j,l},\theta(j,l)\in G$ satisfying the following:
  	\begin{enumerate}[label=(\roman*)]  
  			\item \label{h translating element1} $a_{j,l}=\gamma_{j}\zeta_{j,l}\in H'g_{i(j,l),-\epsilon_{l}}$  if $j\in M_{l}$. 
  			\item \label{h translating element2} $a_{j,l}\in\gamma_{j}\zeta_{j,l}H\subset H'$ if $j\notin M_{l}$.
     	\item\label{theta in X} $\theta(j,l)\in\mathcal{X}(j,l)$, where 
  		\begin{equation*}
  			\mathcal{X}(j,l)=
  			\begin{cases*}
  			  \Gamma_{\Omega} & if $j\in M_{l}$ \\
  			  \mathcal{S}  & if $j\in k\setminus M_{l}$.
  			\end{cases*}
  		\end{equation*}
  	 Moreover, if $j\in M_{l}$, then we have $d(p,\theta(j,l)p)<10^{-1}L$ even if the extension is linear.
  		\item \label{translating equal2} For fixed $1\leq l\leq m$ the value of the element \begin{equation*}
  	  	\gamma_{j}^{-1}a_{j,l}\theta(j,l)
  	  \end{equation*}
  	  does not depend on $j\ik $.
     \end{enumerate}
  \end{corollary}
  \begin{proof}
   First, note that the fact that the exponent associated with $I_{j,l}$ for $j\in M_{l}$ does not depend on $j$
   is a consequence \ref{aa case exponent} in \cref{consequence translating elements}. 
   
   For $1\leq l\leq m$ we fix some $j_{0}\in M_{l}$. Let $\theta(j_{0},l)=1$. 
   For $j\in M_{l}$, the element is already defined by the statement as $a_{j,l}=\gamma_{j}\zeta_{j,l}$. The existence of
   \begin{itemize}
   	\item  $a_{j,l}\in \gamma_{j}\zeta_{j,l}H$ for $j\in k\setminus M_{l}$,
    \item $\theta(j,l)$ for $j\in k\setminus\{j_{0}\}$,
   \end{itemize}
   such that \ref{theta in X} and \ref{translating equal2} hold follows from \ref{aa case} and \ref{H case} in \cref{consequence translating elements}, applied to $\zeta_{j_{0},l}^{-1}\zeta_{j,l}$ for $j\in k$.    
  \end{proof}

  The next lemma elaborates on the algebraic consequences of \cref{p:many parallel paths}. Once again, clause \ref{Delta part 2} at the end is only meant for use in for future work.
 
   \begin{lemma}
  	\label{p:intersection_many3} Under the hypotheses of \cref{p:many parallel paths}, assume furthermore that the following conditions hold:  
     \begin{enumerate}[label=(\alph*)]
     	\item \label{different gammas3}The cosets $H'\gamma_{j}$, $j\in k$ are all distinct. 
     	\item \label{distinct cosets} We have $H s\neq H s'$ for any two distinct $s,s'\in\mathcal{S}$.  
     \end{enumerate} 
     In addition to this, we consider the following two alternative hypotheses: 
      \begin{enumerate}[label=(\Alph*)]
      	\item \label{generous scenario}The extension is generous and the equality 
      	\begin{equation}\label{size of k}
      	 k=|\mathcal{S}|+k_{o}|\Omega|
      	\end{equation}
      	holds, where recall that $k_{o}=|\underline{g}_{\epsilon}|$ for any $\epsilon\in\{\pm1\}$.
      	\item \label{linear scenario} The extension is linear. 
      \end{enumerate} 
     
     Then in the conclusion of \cref{p:many parallel paths} we have $|M_{l}|=k-|\mathcal{S}|$ for all $1\leq l\leq m$ and there is a sequence of $k$-tuples: 
     \begin{equation*}
     	\underline{\xi}^{1},\underline{\chi}^{1},\dots, \underline{\xi}^{m},\underline{\chi}^{m} 
     \end{equation*}
     with the following properties:
     \begin{enumerate}[label=(\roman*)] 
     	\item \label{chain double action} For all $1\leq l\leq m$ there is $\beta_{l}\in G$ such that 
     	\begin{equation*}
     		(H\xi^{l+1}_{j})_{j\in k}=(H\chi^{l}_{j}\beta_{l})_{j\in k}.
     	\end{equation*}
     	\item \label{consecutive moving} If for some $j\ik$ and $1\leq l\leq m-1$ we have $j\in M_{l}\cap M_{l+1}$, then the element 
     	\begin{equation*}
     		\beta_{l}=\theta(j,l)^{-1}g_{i(j,l),\epsilon_{l}}^{-1}\eta g_{i(j,l+1),-\epsilon_{l+1}}\theta(j,l+1),
     	\end{equation*}
     	 where $\eta\in H$ is the letter in the normal word for $h_{j}$ in $H\ast\F(\alpha_{i})_{i\in k_{o}}$ sitting between occurrences of $\alpha_{i(j,l)}^{\epsilon_{l}}$ and of $\alpha_{i(j,l+1)}^{\epsilon_{l+1}}$ associated with the translating elements $a_{j,l}$ and $a_{j,l+1}$, respectively. 
     	\item \label{final end tuple} If for all $j\in k$ we have $\mathcal{A}_{\alpha}(\tau_{j})=\sa_{j}$, then the following pairs of tuples in $(H\backslash G)^{k}$ are also in the same orbit under $(H\backslash G)^{k}\noa G$:
     	\begin{itemize}
     		\item $(H\gamma_{j})_{j\in k}$ and $(H\xi^{1}_{j})_{j\in k}$;
     		\item  $(Hh_{j}^{-1}\gamma_{j})_{j\in k}$ and $(H\chi^{m}_{j})_{j\in k}$. 
     	\end{itemize}
      \item \label{chain intermediate tuples} For all $1\leq l\leq m$ the set of pairs of cosets	$\Phi=\{(H\xi^{l}_{j},H\chi^{l}_{j})\,\}_{j\in k}$ is a bijection between two sets in $[H\backslash G]^{k}$ as follows: 
      \begin{enumerate}[label=(\alph*),ref=(\roman{enumi}\alph*)]
      	\item  \label{cita}In scenario \ref{generous scenario} $\Phi$ is of the form
          \begin{equation*}
               \Phi=	\Id_{\{H s\,|\,s\in\mathcal{S}\,\}}\cup\{(Hg_{i,-\epsilon_{l}}\omega,Hg_{i,\epsilon_{l}}\omega)\,|\,i\in k_{o},\omega\in \Omega\}.
     	     \end{equation*}
     	  \item \label{citb} In scenario \ref{linear scenario} all we can say is that $\Phi$ is of the form 
        \begin{equation*}
     	  	\Id_{H\backslash(H\cdot \mathcal{S}_{0})}\cup\{\,(Hg_{0,-\epsilon_{l}}\rho,Hg_{0,\epsilon_{l}}\rho)\,|\,\rho\in T\,\}
     	  \end{equation*}
     	  for some $\mathcal{S}_{0}\subseteq\mathcal{S}$ and some 
     	  \begin{equation*}
     	  	T\subseteq\{\rho\in G_{\Omega}\,|\,d(p,\rho p)\leq 10^{-1}L\} .
     	  \end{equation*}
      \end{enumerate}
    \item \label{Delta part 2} If moreover $H$ is $(k,p,C)$-geometrically separated, then the following rough bound holds for $j\in k$:
     \begin{equation*}
     	\diam(\partial \tilde{\tau}_{j})\leq 4m(L+C).
     \end{equation*}
     \end{enumerate} 	 
    \end{lemma}
  \begin{proof}
    
   Let $m>0$, $\epsilon_{l}\in\{\pm 1\}$, $M_{l}\subseteq k$, $a_{j,l},h_{j,l}$, and $\theta(j,l)$ be as provided by \cref{p:many parallel paths}. We begin with the following sublemma:
   \begin{lemma}
   	\label{l:distinct indices3} For any fixed value of $1\leq l\leq m$ the following holds:
   	\begin{itemize}
   		\item For distinct $j,j'\in k\setminus  M_{l}$ we must have $H\theta(j,l)\neq H\theta(j',l)$.
   		\item For distinct $j,j'\in M_{l}$ we must have $(i(j,l),\theta(j,l))\neq (i(j',l),\theta(j',l))$. 
   		\item In setting \ref{generous scenario} the set $M_{l}$ has size $k-|\mathcal{S}|$ for each $1\leq l\leq m$.
   	\end{itemize}
   \end{lemma} 
   \begin{subproof}   	 
      If there are distinct $j_{1},j_{2}\in k\setminus M_{l}$ with $H\theta(j_{1},l)=H\theta(j_{2},l)$, that is, $\theta(j_{2},l)^{-1}\theta(j_{1},l)\in H$, then Item \ref{translating equal2} of \cref{p:many parallel paths} implies  $\gamma_{j_{1}}^{-1}a_{j_{1},l}\theta(j_{1},l)=\gamma_{j_{2}}^{-1}a_{j_{2},l}\theta(j_{2},l)$ and $a_{j_{1},l},a_{j_{2},l}\in H'$ by \ref{h translating element2} of the same lemma, which combined with the above two equalities implies $H'\gamma_{j_{1}}= H'\gamma_{j_{2}}$, contradicting condition \ref{different gammas3} in the statement. 
      
     Similarly, if $i(j_{1},l)=i(j_{2},l)$ and $\theta(j_{1},l)=\theta(j_{2},l)$ for distinct $j_{1},j_{2}\in M_{l}$ then the equality 
     $\gamma_{j_{1}}^{-1}a_{j_{1},l}\theta(j_{1},l)=\gamma_{j_{2}}^{-1}a_{j_{2},l}\theta(j_{2},l)$ provided by \ref{translating equal2} of \cref{p:many parallel paths}, combined with the two equalities above implies 
     \begin{equation*}
     	\gamma_{j_{1}}^{-1}a_{j_{1},l}g_{i(j_{1},l),-\epsilon_{l}}^{-1}=\gamma_{j_{2}}^{-1}a_{j_{2},l}g_{i(j_{2},l),-\epsilon_{l}}^{-1}.
     \end{equation*}
      Using \ref{h translating element1}  of \cref{p:many parallel paths} once again we conclude that $a_{j,l}g_{i(j,l),-\epsilon_{l}}^{-1}\in H'$ for any $j\in M_{l}$. Arguing as in the previous paragraph, we thus reach the contradictory conclusion $H'\gamma_{j_{1}}=H'\gamma_{j_{2}}$.
      
      In case \ref{generous scenario} the pigeonhole principle implies that $|M_{l}|\leq k_{o}|\Omega|$ for all $1\leq l\leq m$. In view of \ref{distinct cosets}, it also implies $|k\setminus M_{l}|\leq |\mathcal{S}|$. Thus in this case \eqref{size of k} allows us us to conclude that these inequalities are actually equalities, and the final claim in \cref{l:distinct indices3} is established.
   \end{subproof}
   
    For $1\leq l\leq m$, $j\ik$ we define
    \begin{align}
    \label{branch xi}
    	\xi_{j}^{l}=&
    	\begin{cases}
    	   g_{i(j,l),-\epsilon_{l}}\theta(j,l)& \text{if } j\in M_{l} \\
         \theta(j,l)   &  \text{if } j\nin M_{l},
    	\end{cases}
   \\
    \label{branch chi}
    	\chi_{j}^{l}=&
    	\begin{cases}
         g_{i(j,l),\epsilon_{l}}\theta(j,l)& \text{if } j\in M_{l} \\
         \theta(j,l)   &  \text{if } j\nin M_{l},
    	\end{cases}
    \end{align}
    For $j\in M_{l}$, note that the element $\alpha$ commutes with $\theta(j,l)\in\Omega$, so that in this case 
	  \begin{equation}\label{conjugation does nothing}
	    	\chi^{l}_{j}=g_{i(j,l),\epsilon_{l}}\theta(j,l)=g_{i(j,l),\epsilon_{l}}\alpha^{-\epsilon_{l}}\theta(j,l)\alpha^{\epsilon_{l}},
	   \end{equation}

    Condition \ref{chain intermediate tuples} in the statement of \cref{p:intersection_many3} follows from equations \eqref{branch chi} and \eqref{branch xi} above, as well as \cref{l:distinct indices3} and the observation in the previous paragraph.
    In particular, the cosets $\{H\xi_{j}^{l}\}_{j\in k}$ are all pairwise distinct, and the same applies to $\{H\chi_{j}^{l}\}_{j\in k}$. 
    
    Our next goal is to show part \ref{chain double action}. For that purpose, given $1\leq l\leq m-1$, $j\ik$, define first
    \begin{equation}
    	\label{bjl}\beta_{j,l}=\left(a_{j,l}\theta(j,l)\alpha^{\epsilon_{l}}\right)^{-1} a_{j,l+1}\theta(j,l+1).
    \end{equation} 
    
   \begin{observation}
    \label{o:connecting elements3} 
    For $1\leq l\leq m-1$ the element $\beta_{j,l}$ does not depend on $j$. 
   \end{observation}
   \begin{subproof}
   	Neither the value of $\gamma_{j}^{-1}a_{j,l}\theta(j,l)$ nor that of $\gamma_{j}^{-1}a_{j,l+1}\theta(j,l+1)$ depends on $j$, by \ref{translating equal2} in \cref{p:many parallel paths}, which together with the equality
   	\begin{align*}
   		a_{j,l}\theta(j,l)\alpha^{\epsilon_{l}}\beta_{j,l}=a_{j,l+1}\theta(j,l+1).
   	\end{align*}
   	implies the result.
   \end{subproof}
   In view of this, we will henceforth write $\beta_{l}$ instead of $\beta_{j,l}$. 
 

   Condition \ref{chain double action} is given by \cref{l:beta action} below, which is preceded by a short preliminary discussion. 
   
   Write $h_{j,l}\in H'$ for the translating element of $\gamma_{j}J$, where $J$ is the last $J\in\mathcal{A}_{H}(\tau_{j})$ such that $J\leq I_{j,l}$, i.e. the last such $J$ that is either equal  to $I_{j,l}$ (if $j\notin M_{l}$) or precedes it along $\gamma_{j}^{-1}\tau_{j}$. 
   We note the following.
   \begin{observation}
	   	\label{o:small step}For $1\leq l\leq m$ and $j\in k$ we have  
	   	\begin{enumerate}[label=(\roman*)]
	   		\item \label{part not in Ml}If $j\notin M_{l}$, then 
	   		\begin{itemize}
	   			\item $Ha_{j,l}^{-1}=Hh_{j,l}^{-1}$, 
	   			\item the two cosets above also equal  $Hh_{j,l+1}^{-1}$ if $l<m$.
	   		\end{itemize}
	   		\item \label{part both} $h_{j,l}^{-1}a_{j,l}\theta(j,l)\in H\xi_{j}^{l}$.
	   	\end{enumerate}
	   \end{observation}
	    \begin{subproof}
	    	For the first claim in \ref{part not in Ml} note that in the notation of \cref{p:many parallel paths} we have  $a_{j,l}^{-1}\in H\zeta_{j,l}^{-1}$ and $\zeta_{j,l}=h_{j,l}$ in this case, whereas if $l<m$, then the equality $Hh_{j,l}=Hh_{j,l+1}$ is a direct consequence of the definition of $h_{j,l},h_{j,l+1}$.  
	    	
	    	Part \ref{part both} note that if $j\nin M_{l}$, then $h_{j,l}^{-1}a_{j,l}\in H$ whereas if $j\in M_{l}$ then $h_{j,l}^{-1}a_{j,l}\in Hg_{i(j,l),-\epsilon_{l}}$, by comparing these expressions with \eqref{branch xi}.   
	    \end{subproof}

   \begin{lemma}
   	\label{l:beta action} For every $1\leq l\leq m-1$ and $j\in k$ we have 
   \begin{equation*}
   	H\chi_{j}^{l}\beta_{l}=H\xi_{j}^{l+1}.
   \end{equation*}
   \end{lemma}
   \begin{subproof}
   	Note that it suffices to show the inclusion $H\chi_{j}^{l}\beta_{l}\subseteq H\xi_{j}^{l+1}$.
   	
    \textbf{Case $\mathbf{j\in M_{l}}$.} 
   Applying \eqref{conjugation does nothing} and the definition of $\chi_{j}^{l}$ in the first equality below, we have
   \begin{align*}
   \label{eq: beta action3-Fl}
   	\begin{split}
   	H\chi_{j}^{l}\beta_{l}=&
   	Hg_{i(j,l),\epsilon_{l}}\alpha^{-\epsilon_{l}}\theta(j,l)\alpha^{\epsilon_{l}} \left(\left(a_{j,l}\theta(j,l)\alpha^{\epsilon_{l}}\right)^{-1}a_{j,l+1}\theta(j,l+1)\right)=\\
   	& H\left( a_{j,l}\alpha^{\epsilon_{l}}g_{i(j,l),\epsilon_{l}}^{-1} \right)^{-1}a_{j,l+1}\theta(j,l+1)=
   	Hh_{j,l+1}^{-1}a_{j,l+1}\theta(j,l+1)=H\xi_{j}^{l+1}
   	\end{split}
   \end{align*}	
   	where the last equality is just \ref{part both} of \cref{o:small step} applied to $(j,l+1)$.
   	\\
   	
   \textbf{Case $\mathbf{j\notin M_{l}}$.} Here $\xi_{j}^{l}=\theta(j,l)\in\mathcal{S}$, 
   and so 
   \begin{equation}\label{where is alpha}
   	\alpha^{\theta(j,l)^{-1}}=\theta(j,l)\alpha\theta(j,l)^{-1}\in H,
   \end{equation}
    by condition \ref{hyp alt} of \cref{d:basic extension setting}.
    We obtain: 
        \begin{align*}
   	\begin{split}
   	H\chi_{j}^{l}\beta_{l}=&H\theta(j,l)\beta_{l}= H\theta(j,l)\left(a_{j,l}\theta(j,l)\alpha^{\epsilon_{l}}\right)^{-1}a_{j,l+1}\theta(j,l+1)
   	=\\ 	
   	&H\left(\theta(j,l)\alpha^{-\epsilon_{l}}\theta(j,l)^{-1}\right)a_{j,l}^{-1}a_{j,l+1}\theta(j,l+1)=\\
   	&Ha_{j,l}^{-1}a_{j,l+1}\theta(j,l+1)=H\xi_{j}^{l+1},
   	\end{split}
   \end{align*}	
   where in the third equality we use  \eqref{where is alpha}, and in the last equality we combine the two parts of \cref{o:small step}.   
   
  \end{subproof}
   This finishes establishing property \ref{chain double action}, and property \ref{final end tuple} can be shown using a more straightforward version of the same argument.
 
   Claim \ref{consecutive moving} follows from a simple calculation using equation \eqref{bjl}, together with the fact that in this case $a_{j,l}$ and $a_{j,l+1}$ are the translating elements of consecutive $\alpha$-subarcs of $\tau_{j}$ (not $\gamma_{j}^{-1}\tau_{j}$).

   Claim \ref{Delta part 2} follows from part \ref{Delta} and \ref{Delta'} of \cref{many fellow traveling paths basic result} by a simple calculation involving the triangular inequality which is left to the reader.
   This concludes the proof of \ref{p:intersection_many3}.
   \end{proof}

  \newcommand{\hf}[0]{\sfrac{1}{2}}
  \newcommand{\HS}[0]{H'\cdot\mathcal{S}}
 \newcommand{\qc}[0]{(1+\mu_{0},\mu_{0}L)}
 \newcommand{\nbl}[0]{\nb_{\lambda_{0}L}}

 \section{Preservation of $k$-malnormality}
 \label{s:preservation of kmalnormality}
 
  For the rest of this section we fix $k\geq 2$ and a robust transitive subgroup $\Theta\leq\sym_{k}$, as well as a hyperbolic $\Theta$-seed (either ample or flat)
   \begin{equation}\label{datum2}
  	\datum.
   \end{equation}
  In the flat case we let $\Omega_{0}$ denote the distinguished group in $\D(\Theta)$, as in \ref{d:flat preseeds}. From previous sections we keep the notation $G_{\Omega}=N_{G}(\Omega)$. 
 
  In the course of this section we will introduce the full set of assumptions on the subgroup $H\leq G$ and the $(\Omega,\mathcal{S},\underline{g}_{-1},\underline{g}_{1},\alpha)$-extension $H'$ of $H$ we have been working with so far. We will use the additional assumptions to show the main preservation results needed for the induction step of \cref{t: main}. Our main ingredient will be \cref{p:intersection_many3} from the previous section, strengthened as in \cref{l:bounded sequence} under suitable assumptions.

 \begin{definition}\label{d:omega_set} 
   	Given $H\leq G$ and $\Omega\in\D(\Theta)$, we write $\mathcal{S}^{H}_{\Omega}=H\backslash(H\cdot \mathcal{S}_{\Omega})\subseteq \uh$. 
   	
   	By an \emph{$\Omega$-set in $\uh$} we mean any $A\in[\uh]^{k}$ which is invariant under the action of $\Omega$. We say that $A$ is a \emph{strict $\Omega$-set} if in addition to this we require $\Omega$ to be the full setwise stabilizer of $A$.
 \end{definition} 
  
  Write $k_{o}(\Omega)$ for the number of free orbits of $k\noa \Omega$, so that  
  \begin{equation*}
   	k_{o}(\Omega)=\frac{k-|\NFr_{\pi}(\Omega)|}{|\Omega|}=\frac{k-|\mathcal{S}_{\Omega}|}{|\Omega|}.
   \end{equation*}
  
  The following is easy to see: 
  \begin{observation}
  	  For any $(H,\Omega,\mathcal{S}_{\Omega})$-admissible $\underline{g}\in G^{k_{o}}$ 	the set 
   	\begin{equation}\label{defa}
   		A=\mathcal{S}_{\Omega}^{H}\cup\{H g_{i}\omega\,|\,\omega\in\Omega, i\in k_{o} \}\in[H\backslash G]^{k}
   	\end{equation}   
    is an $\Omega$-set. 	 
  \end{observation}
   
  \begin{definition}\label{omega tuple generate}
 	   We will refer to the set $A$ in \eqref{defa} as the \emph{$\Omega$-set generated by the tuple} $\underline{g}\in G^{k_{o}}$. 
 	   In the current context for $\epsilon\in\{\pm1\}$ we will write $A_{\epsilon}$
 	   for the $\Omega$-set in $[H\backslash G]^{k}$ generated by $\underline{g}_{\epsilon}$. 
  \end{definition}

   	\begin{definition}\label{d:promising} 
	   	We will say that a subgroup $H\leq G$ is \emph{promising} if $H_{0}\leq H$ and
	   	\begin{itemize}
	   		\item $H$ is geometrically taut with respect to the elements $(h_{\Omega})_{\oid}$ in the sense of \cref{d:nice datum},
	   		\item if the seed is flat, then we also require that $H$ is supported on $\G$. 
	   	\end{itemize} 
   	\end{definition}
   
     \begin{observation}\label{GOmega has no fixed points}
    	Let $G$ be a group admitting a flat hyperbolic $\Theta$-datum. If $H\leq G$ is geometrically taut and $g\in G_\setminus H\cdot \mathcal{S}_{\Omega_{0}}$, then $G_{\Omega_{0}}\cap H^{g}=\{1\}$ or, equivalently, the action $H\backslash (HgG_{\Omega_{0}})\noa G_{\Omega_{0}}$ is free.
      In particular, under this assumption condition \ref{weak hyp gi not F} of \cref{weakly admissible} follows from condition \ref{weak hyp away from fixed points}, and in this context  any set $C\in\mathcal{Z}^{k}_{H}$ projects to a set in $\mathcal{Y}^{k}_{H}$ (recall \cref{d:auxiliary sets}). 
     \end{observation}
   \begin{proof}
   	 Suppose for the sake of contradiction that $Hg\rho=Hg$ for some $\rho\in G_{\Omega_{0}}\setminus \{1\}$. On the other hand, we cannot have $\rho\in\Omega_{0}$, since the action of $\Omega_{0}$ on 
   	 $H\backslash G$ is free outside of $\mathcal{S}^{H}_{\Omega_{0}}$, as $H$ is geometrically taut.
   	 Since $G_{\Omega_{0}}$ fixes two points at infinity, $\beta$ must be loxodromic. On the other hand, $\rho$ preserves the set of cosets $A=\mathcal{S}_{\Omega_{0}}^{H}\cup\{Hg\omega\,|\,\omega\in\Omega_{0}\}$, which has size $k$ by our assumption that $\Omega_{0}$ has a single free orbit. It follows that for some $m>0$ the non-trivial element $\rho^{m}$ fixes the set 
   	 $A$ pointwise, contradicting the assumption that $(H\backslash G)\noa G$ is $k$-sharp or, equivalently, that $H$ is $k$-malnormal.
   \end{proof}
 

    In addition to the hyperbolic $\Theta$-seed structure on $G$, we start with the following working assumptions:
    \begin{assumption}\label{assumption section induction step}
    	The subgroup $H$ is promising, $\oid$, $k_{o}=k_{o}(\Omega)$, and $H'$ is a  $(\Omega,\mathcal{S}_{\Omega},\underline{g}_{-1},\underline{g}_{1},\alpha)$-extension of $H$ which is $(p,\nu)$-geometric in the sense of \cref{d:good extension setting}, where $\nu$ satisfies \cref{constant assumptions} and the extension is of linear type if  the seed in \eqref{datum2} is flat and $\Omega=\Omega_{0}$.
    \end{assumption}

   It follows from the above discussion, and in particular \cref{GOmega has no fixed points}, that all the results of \cref{s:fellow traveling paths} apply in this context.
 
   \begin{lemma} \label{l:bounded sequence}
   	Under \cref{assumption section induction step} and the hypotheses of $\cref{p:intersection_many3}$ in the generous case the following holds for the sequence of exponents $(\epsilon_{l})_{1\leq l\leq m}$  in the conclusion of \cref{p:intersection_many3}:
   	\begin{enumerate}[label=(\Alph*)]
   		\item \label{part different orbits}If $A_{-1},A_{1}$ are not in the same orbit under the action $[H\backslash G]^{k}\noa G$, then the sequence is alternating, that is, $\epsilon_{l}=-\epsilon_{l+1}$ for all $1\leq l< m$. 
   		\item \label{part strict}If in addition to this, one also assumes that $A_{1}$ is a strict $\Omega$-set,
   	 then the length $m$ of the sequence satisfies $m\leq 2$, and if $m=2$, then $\epsilon_{1}=-1$ and $\epsilon_{2}=1$. 
   	\end{enumerate}
  \end{lemma}
  \begin{proof} 
	 For part \ref{part different orbits} we prove the contrapositive. Assume that $\epsilon_{l}=\epsilon_{l+1}=\epsilon$ for some $1\leq l\leq m-1$. 
	 Item \ref{chain intermediate tuples} in \cref{p:intersection_many3} implies (in the notation of that same result):
	 \begin{equation*}
	 	\{H\chi^{l}_{j}\}_{j\in k}=A_{\epsilon},\quad\quad \{H\xi^{l+1}_{j}\}_{j\in k}=A_{-\epsilon}.
	 \end{equation*}
	 Item \ref{chain double action} of the conclusion of \cref{p:intersection_many3} then implies that $A_{-1}$ and $A_{1}$ are in same orbit under the diagonal action of $G$ on $[\uh]^{k}$.
 	  
	 For part \ref{part strict}, assume that the claim fails under the additional hypothesis on $A_{1}$. Then there is $1\leq l_{0}\leq m-1$ such that $\epsilon_{l_{0}}=1$ and $\epsilon_{l_{0}+1}=-1$. 
   \cref{p:intersection_many3} provides some element $\beta=\beta_{l_{0}}\in G$ such that 
   \begin{equation*}
  	(H\chi^{l_{0}}_{j}\beta)_{j\in k}=(H\xi^{l_{0}+1}_{j})_{j\in k}.
   \end{equation*}  
   Now, we know that
  \begin{equation}
   A_{1}=\{H\chi_{j}^{l_{0}}\}_{j\in k}=\{H\xi_{j}^{l_{0}+1}\}_{j\ik}=\mathcal{S}_{\Omega}^{H}\cup\{Hg_{i,1}\omega\,|\,i\in k,\omega\in\Omega\}.
  \end{equation}
  By assumption, the set $A_{1}$ is strict so it follows that $\beta\in\Omega$. In particular, the action of $\beta$ on the right on $A_{-1}$ fixes $\mathcal{S}^{H}_{\Omega}$. 
  
  In general, we know that $j\in M_{l}$ if and only if $\chi^{l}_{j}\in Hg_{i(j,l),\epsilon_{l}}\Omega$ and $\xi^{l}_{j}\in g_{i(j,l),-\epsilon_{l}}\Omega$. Therefore, in this case if $j_{0}\in M_{l_{0}}$, then also $j_{0}\in M_{l_{0}+1}$ and $i(j,l_{0})=i(j,l_{0}+1)$.
  Recall also that according to item \ref{consecutive moving} in \cref{p:intersection_many3} we have 
  \begin{equation*}
  	\beta=\theta(j_{0},l_{0})^{-1}g_{i(j_{0},l_{0}),1}^{-1}\eta g_{i(j_{0},l_{0}),1}\theta(j_{0},l_{0}+1),
  \end{equation*}
   where $\eta\in H$ is the letter in some reduced word in $H\ast \F(\alpha_{i})$ that appears between an occurrence of $\alpha_{i(j_{0},l)}^{-1}$ and an occurrence of $\alpha_{i(j_{0},l_{0})}$. 
    Since $j_{0}\in M_{l_{0}}\cap M_{l_{0}+1}$, we also know that $\theta(j_{0},l_{0}),\theta(j_{0},l_{0}+1)\in\Omega$. Combining this with the membership $\beta\in\Omega$ established above, we  have 
  \begin{equation*}
   \eta^{g_{i(j_{0},l),1}}=\theta(j_{0},l_{0})\beta\theta(j_{0},l_{0}+1)^{-1}\in\Omega.
  \end{equation*}
  However, our assumption on $H$ implies that $H\backslash G\noa\Omega$ is free outside of $\mathcal{S}^{H}_{\Omega}$ so that $\eta=1$. This implies that the word in the alphabet $H\cup\{\alpha_{i}^{\pm1}\}_{i\in k_{o}}$ from which $\tau_{j}$ was constructed contained a subword of the form $(\alpha_{i(j_{0},l_{0})},1,\alpha_{i(j_{0},l_{0})}^{-1})$. This is impossible, since such a word cannot be reduced.
 
  \end{proof}

   Recall the notation introduced in \cref{d:auxiliary sets}, as well as the graph $\mathtt{G}(H_{1})$, $H_{1}\in\{H,H'\}$ introduced in \cref{auxiliary graph}. Recall also \cref{GOmega has no fixed points}.

   \begin{observation} \label{l:sequence interpretation}
   	Under \cref{assumption section induction step} and the hypotheses of \cref{p:intersection_many3}. In the notation of the conclusion of \cref{p:intersection_many3} write 
   	\begin{equation*}
   		D_{l}=\{H\xi_{j}^{l}\}_{j\in k},\quad\quad F_{l}=\{H\chi_{j}^{l}\}_{j\in k}, \quad\quad O_{l}=\O_{H}(D_{l}), \quad\quad P_{l}=\O_{H}(F_{l}).
   	\end{equation*}
    
   	\begin{itemize}
   		\item In the linear case we have
   	\begin{equation}\label{type of intermediate sets}
   		D_{l}\in[\mathcal{Z}_{0}(g_{0,-\epsilon_{l}})]^{k},\quad\quad F_{l}\in[\mathcal{Z}_{0}(g_{0,\epsilon_{l}})]^{k},
   	\end{equation}
   	for finite subsets $\mathcal{Z}_{0}(g_{0,-\epsilon_{l}})\subseteq\mathcal{Z}(g_{0,-\epsilon_{l}})$ 
   	and $\mathcal{Z}_{0}(g_{0,\epsilon_{l}})\subseteq\mathcal{Z}(g_{0,\epsilon_{l}})$ depending only on the extension.
   	Thus $O_{l},P_{l}\in\mathtt{V}_{1}(H)$ for all $1\leq l\leq m$, and moreover
   	\begin{enumerate}[label=(\alph*)]
   		\item \label{connected by an edge} for all $1\leq l\leq m$ the vertex 
   		\begin{equation*}
   	     \T_{l}=\T_{H}(D_{l})=\T_{H}(F_{l})\in\mathtt{V}_{0}(H)
   		\end{equation*}
   		is connected with $O_{l}$ and $P_{l}$ by the edges 
   		$\Ed_{H}(D_{l})$ and $\Ed_{H}(F_{l})$. 
   		 
   		\item \label{path does not fold}  The edges $\Ed_{H}(D_{l})$ and $\Ed_{H}(F_{l})$ are distinct. 
       \item \label{PlequalOl+} $P_{l}=O_{l+1}$ for all $1\leq l<m$.  
       \item\label{if supported} If we assume $H$ to be supported on $\G$, then $O_{1}\notin\{P_{l}\}_{1\leq l\leq m}$. 
       \item \label{if full intervals} If $\sa_{j}=\mathcal{A}_{\alpha}(\hat{\tau}_{j})$ for all $j$, then
   	    \begin{equation*}
   		     \O_{H}(\{H\gamma_{j}\}_{j\in k})=O_{1}, \quad\quad \O_{H}(\{Hh_{j}^{-1}\gamma_{j}\}_{j\in k})=P_{m}.
   	    \end{equation*} 
   	\end{enumerate}
   \item In the generous case, assume that $A_{\epsilon}\in\mathcal{Y}_{H}^{k}$ for $\epsilon\in\{\pm1\}$. Then we have the following: 
   \begin{enumerate}[label=(\alph*)]
   	\item  For all $1\leq l\leq m$ we have 
   	\begin{equation*}
       O_{l}=\O_{H}(A_{-\epsilon_{l}}),\quad\quad P_{l}=\O_{H}(A_{\epsilon_{l}}). 
   	\end{equation*}
   	\item  If $\sa_{j}=\mathcal{A}_{\alpha}(\hat{\tau}_{j})$ for all $j$, then 
   \begin{equation*}
   		     \O_{H}(\{H\gamma_{j}\}_{j\in k})=\O_{H}(A_{-\epsilon_{1}}), \quad\quad \O_{H}(\{Hh_{j}^{-1}\gamma_{j}\}_{j\in k})=\O_{H}(A_{\epsilon_{m}}).
   \end{equation*} 
   \end{enumerate} 
   	\end{itemize}

   \end{observation}
  \begin{remark}\label{does not need to be in Y}
  	Note that in part \ref{if full intervals} we do not require that the sets 
  	$\{H\gamma_{j}\}_{j\in k}$ and $\{Hh_{j}^{-1}\gamma_{j}\}_{j\in k}$ themselves to be contained in sets of the form $\mathcal{Y}_{H}(g)$. 
  \end{remark}
 
    \begin{proof}[Proof of \cref{l:sequence interpretation}] 
  	Consider the linear case first. Assertions \eqref{type of intermediate sets} and \eqref{connected by an edge} follow directly from \ref{citb} in \cref{p:intersection_many3}. Assertion \ref{path does not fold} follows from 
  	the fact that $g_{0,-1}\notin Hg_{0,1}\G$, i.e. \ref{hyp_sep} in \cref{d:good extension setting}.
   On the other hand, \ref{chain double action} is equivalent to item \ref{PlequalOl+} in the statement.
   The meaning of the assumption that $H$ is supported on $\G$ is precisely that $\mathtt{G}(H)$ is a forest, so claim \ref{if supported} is immediate from the previous ones under this assumption. Finally, \ref{if full intervals} follows from \ref{final end tuple} in \cref{p:intersection_many3}.
   The proof of the statement pertaining the generous case follows analogously from part \ref{cita} and \ref{final end tuple} in \cref{p:intersection_many3}.
   \end{proof}

   Partially motivated by \cref{l:bounded sequence} and \cref{l:sequence interpretation}, let us now state our final set of conditions on the two tuples $\underline{g}_{-1},\underline{g}_{1}$ defining the extension, under which we will show that the extension $H'$ will itself be promising.
   Note that the second item of \ref{cg:ga} will not play a role until the next section. 
   
   Recall that the context is that of a promising $H\leq G$, $\Omega\in\mathcal{D}(\Theta)$, and $(H,\Omega,\mathcal{S}_{\Omega})$-admissible $k_{o}$-tuples $\underline{g}_{\epsilon}$, $\epsilon\in\{\pm1\}$ generating $\Omega$-sets $A_{\epsilon}\in[H\backslash G]^{k}$, where $k_{o}(\Omega)=k_{o}$. 
   
  \begin{assumption}\label{assumptions on the ksets}
   For some arbitrary $\theta_{0}\in\Theta\setminus(H\cdot \mathcal{S}_{\Omega_{0}})$ the following is satisfied by $(H,\Omega,\underline{g}_{-1},\underline{g}_{1})$:
   	\begin{enumerate}[label=(\roman*)]
   		\item \label{cg:A1where}$A_{-1}=H\backslash (H\cdot \Theta)\in\mathcal{Y}_{H}(\theta_{0})$.
   		\item \label{cg:generous case}In the generous case we assume the following:
   		\begin{enumerate}[label=(\alph*),ref=(\roman{enumi}\alph*)]
   				\item \label{cg:ga}If the seed is flat, then we assume:
   				  \begin{itemize}
   				  	\item $A_{-1},A_{1}\in\mathcal{Y}_{H}(\theta_{0})$ (although the orbit of $g_{i,\pm1}$ under $\Omega$ might not necessarily be contained in $\mathcal{Z}(\theta_{0})$);
   				  	\item $\O_{H}(A_{-1})$ and $\O_{H}(A_{1})$ belong to different connected components of $\mathtt{G}(H)$.
   				  \end{itemize}
   				\item \label{cg:gf}If the seed is ample, we only require that $A_{-1}$ and $A_{1}$ are not in the same orbit under $[H\backslash G]^{k}\noa G$. 
   			 \item \label{A1strict} $A_{1}$ is a strict $\Omega$-set.
   		\end{enumerate} 
   		\item \label{cg:linear case} In the linear case (which only occurs with flat seeds) we assume 
   		\begin{equation*}
   			g_{0,-1}=\theta_{0}\in\Theta\setminus(H\cdot\mathcal{S}_{\Omega}),
   		\end{equation*}
   		 and thus  
   		$g_{0,1}\in G\setminus (H\cdot \Theta\cdot\G)$.
   	\end{enumerate} 
   \end{assumption} 
   
   \begin{remark}
   	 One might wonder what happened with the assumption that $A_{1}$ is strict in the case in which the extension is linear. We do not use such an assumption as such in this case, but it does follow from the remaining hypothesis. Indeed, since the action $k\noa \Omega_{0}$ has a single free orbit,
   	 there does not exist any docile $\Omega_{0}<\Omega\leq\Theta$. As we shall see in \cref{l:invariant tuples}, this implies that if $A_{1}$ were not strict, then some translate of $A_{1}$ by an element of $\G$ would have to be equal to $A_{-1}=H\backslash(H\cdot \Theta)$. 
   \end{remark}
    
    \begin{lemma}\label{preservation of k-sharpness}
    	Under \cref{assumption section induction step} and \cref{assumptions on the ksets} on $\underline{g}_{\epsilon}$, the subgroup $H'\leq G$ is $k$-malnormal. In the generous case, if $H$ is geometrically $k$-separated, then so is $H'$.
    \end{lemma}
    \begin{proof}
        Assume for the sake of contradiction that there is some tuple of distinct cosets $(H'\gamma_{j})_{j\in k}$ and some tuple of elements $(h_{i})_{i\in k}\subseteq H'\setminus \{1\}$ such that the element $h_{i}^{\gamma_{i}}$ does not depend on $i$. 
        Let $F\subseteq k$ be the collection of indices for which $h_{j}$ is not conjugate in $H'$ to an element of $H$.
        Since $H$ is $k$-malnormal we must have $F\neq\emptyset$. By replacing each $h_{j}$ with its conjugate by some other element of $H'$, we may additionally assume:
     \begin{itemize}
  	     \item $h_{j}$ (i.e. the unique reduced word representing it) is cyclically reduced for $j\in  F$; 
  	     \item $h_{j}\in H$ for $j\in k\setminus F$. 
       \end{itemize} 
      From \cref{l:p-paths}, part \ref{new classes are hyperbolic} we know that $h_{j}$ is loxodromic for $j\in F$. Since all the $h_{j}$ are conjugate in $G$, in fact $h_{j}$ is loxodromic for all $j\in k$. For $j\in F$, let $\sigma_{j}=\ax(h_{j})$, as given by \cref{d:x-paths}. Note that $\tau_{j}$, $j\in F$ contains infinitely many $\alpha$-subarcs, since its fundamental domain must contain at least one. 
      For $j\in k\setminus F$ let $\sigma_{j}=[h_{j}^{-m_{j}}p,h_{j}^{m_{j}}p]$ for some large positive integer $m_{j}$. For $j\in k$ write $\tau_{j}=\gamma_{j}^{-1}\sigma_{j}$.  
      
      Let $n^{*}\in\mathbb{N}_{>0}$ be   
      \begin{itemize}
      	\item $3$ in the generous case,
      	\item $1+\max_{\epsilon\in\{\pm1\}}|[\mathcal{Z}_{0}(g_{0,\epsilon})]^{k}|$, where $\mathcal{Z}_{0}(g_{0,\epsilon})$ is the finite set in \cref{l:sequence interpretation}.    
      \end{itemize}
     
      It follows from \cref{l:p-paths} and \cref{quadrilaterals general}  that for any $\lambda_{0}>0$ $m_{j}$ large enough for $j\in k\setminus F$ there are subarcs $\tilde{\tau}_{j}\subseteq\tau_{j}$, for $j\in k$ such that:  
      \begin{enumerate}[label=(\alph*)]
      	\item \label{sufficiently long sequence} $\tilde{\tau}_{j_{0}}$ contains at least $n^{*}$-many $\alpha$-subarcs for some $j_{0}\in k$;
        \item $\tilde{\tau}_{j}$ and $\tilde{\tau}_{j'}$ are $\lambda_{0}L$-fellow traveling for all $j,j'\in k$. 
      \end{enumerate} 
       
       We can apply \cref{many fellow traveling paths basic result}, and the assumption \ref{sufficiently long sequence} above implies that the $m$ in the conclusion, i.e., the number of $\simeq$-classes in $\sa$  satisfies $m\geq n^{*}$.  However, under our assumptions we can also apply \cref{p:many parallel paths}, \cref{p:intersection_many3}.
       
       In the generous case we can then use \cref{l:bounded sequence} to conclude that $m<n^{*}$: a contradiction. Similarly, the definition of $n^{*}$ in the linear case, together with \cref{l:sequence interpretation} implies that if $m\geq n^{*}$, then the graph $\mathtt{G}(H)$ has to contain a cycle, contradicting our assumption that $H$ is supported on $\G$. 
  
      The same argument can be used to prove that $H'$ is geometrically $k$-separated if $H$ is, using points \ref{Delta} and \ref{Delta'} in the conclusion of \cref{p:many parallel paths}.
    \end{proof}
  
   \section{Preservation of $\G$-supportedness}
   \label{s:preservation of supportedness}
   \newcommand{\ff}[0]{\mathcal{F}}
   \renewcommand{\ll}[0]{\mathcal{L}}
   \newcommand{\bh}[0]{\hat{\underline{b}}}

  In this section we continue working with the hyperbolic $\Theta$-seed from \eqref{datum2} from the previous section, which here we assume to be flat with distinguished $\Omega_{0}\in\D(\Theta)$. The extension under consideration might be however either generous or linear and we assume that $H'$ is geometrically taut in the sense of \ref{existence of H0} in \cref{d:nice datum}.
   
  \subsection{The general relationship between $\mathtt{G}(H)$ and $\mathtt{G}(H')$ }
   
   Our goal is to be able to prove that under minor additional conditions on the extension the graph $\mathtt{G}(H')$ is a forest. As we shall state formally below, of the set of vertices of $\mathtt{G}$, only the side $\mathtt{V}_{1}$ of the partition changes when passing from $\mathtt{G}(H)$ to $\mathtt{G}(H')$ and most of the subsection consists of a long technical discussion around the relationship between the space of orbits of $[H\backslash G]^{k}\noa G$ that intersect $\mathcal{Y}_{H}^{k}$ and the space of orbits of $[H'\backslash G]^{k}\noa G$ that intersect $\mathcal{Y}_{H'}^{k}$. 
   
   At the end of the discussion we shall see that the graph $\mathtt{G}(H')$ can be obtained from $\mathtt{G}(H)$ by performing what amounts to a sequence of the following two types of operations, which do not create any cycles:
   \begin{itemize}
   	\item identifying two vertices in $\mathtt{V}_{1}$ that lie in different connected components;
   	\item folding two edges incident at the same vertex of $\mathtt{V}_{0}(H)$.
   \end{itemize} 
   
   In order to help motivate said technical discussion will take, let us be more precise. The following follows in a straightforward manner from the definition.
    \begin{observation}\label{graph collapse}
    	 The graph $\mathtt{G}(H')$ relates to the graph $\mathtt{G}(H)$ as follows. 
    	    There is a bijective correspondence between $\mathtt{V}_{0}(H)$ and $\mathtt{V}_{0}(H')$ given by
       \begin{equation*}
       	\T_{H}(C)\mapsto \Lambda(\T_{H}(C))=\T_{H'}(C).
       \end{equation*}
 
     As for $\mathtt{V}_{1}(H')$, we have the following.
      \begin{itemize}
      	\item Every $O\in\mathtt{V}_{1}(H)$ of the form $O=\O_{H}(C)$, for $C\in\mathcal{Y}_{H}(g)$ and $g\in G\setminus (H'\cdot\so)$ corresponds to a vertex 
      	\begin{equation*}
      		\Lambda(O)=\O_{H'}(C')\in\mathtt{V}_{1}(H'), 
      	\end{equation*}
      	where $C'$ is the projection of $C$ to $H'\backslash G'$, which still has size $k$ by \cref{GOmega has no fixed points} and our assumption that $H'$ is geometrically taut.
      	\item Every vertex in $\mathtt{V}_{1}(H')$ is of the form $\Lambda(O)$ for some $O\in\mathtt{V}_{1}(H)$ as above.
      \end{itemize}
      As for the set of edges $\mathtt{E}(H')$, we have: 
      \begin{itemize}
      	\item Any edge in $\mathtt{E}(H)$ of the form $e=\Ed_{H}(C)$, for $C\in\mathcal{Z}_{H'}^{k}$ corresponds to an edge 
      	\begin{equation*}
      		\Lambda(e)=\Ed_{H'}(C)\in\mathtt{E}(H'), 
      	\end{equation*}
      	whose endpoints are the image by $\Lambda$ of the endpoints of $e$. 
      	\item Every edge in $\mathtt{E}(H')$ is of the form above. 
      \end{itemize}
    \end{observation}
    
    In either the generous or the linear case above the map $\Lambda$ results in a number of foldings of edges. 
    The following observation is just a restatement of something already contained in \cref{d:auxiliary sets}, but worth emphasizing. 
    \begin{observation}\label{multifold}
	  	 Let $C\in[\mathcal{Z}(g)]^{k}$, $\hat{C}\in[\mathcal{Z}(\hat{g})]^{k}$ be two sets such that 
	  	 $C\cap\so=\hat{C}\cap\so$ and $(C\setminus\so)=h(\hat{C}\setminus\so)\rho$  for $h\in H'$ and $\rho\in\G$. 
	  	 Then the map $\Lambda$ folds the two edges $\Ed_{H}(C)$ and $\Ed_{H}(\hat{C})$, which connect 
	  	 $\T_{H}(C)=\T_{H}(\hat{C})\in\mathtt{V}_{0}(H)$ with $\O_{H}(C),\O_{H}(\hat{C})\in\mathtt{V}_{1}(H)$, respectively, into the single edge $\Ed_{H'}(C)=\Ed_{H'}(\hat{C})$ between 
	     $\T_{H'}(C)=\Lambda(\T_{H}(\hat{C}))$ and $\O_{H'}(C)=\O_{H'}(\hat{C})$.
	  \end{observation}
  
    This observation will allow us to reduce the study of the non-injectivity of the map $\Lambda$ to pairs of values of $g$ and $\hat{g}$ of a very specific form, namely ones which are close to ``shortest" in the appropriate double coset.
    
    In the context of \cref{l:sequence interpretation} we have already seen that the two sets $C=\{H\gamma_{j}\}_{j\in k}$ and $C'=\{Hh_{j}^{-1}\gamma_{j}\}_{j\in k}$ such that $\O_{H}(C)=\O_{H}(C')$ must satisfy 
    \begin{equation*}
    	\{\O_{H}(C),\O_{H}(C')\}\subseteq\{\O_{H}(A_{-1}),\O_{H}(A_{1})\},
    \end{equation*}
     with the latter lying in different connected components of $\mathtt{G}(H)$. 
    
    In the linear case, a second reading of the same result yields a similarly simple interpretation of \cref{l:sequence interpretation} in terms of the map $\Lambda$. 
    \begin{corollary}\label{folding in the graph}
    	In the linear case in the conclusion of \cref{l:sequence interpretation} we have:
    	\begin{itemize}
    		\item $\Lambda(O_{l})=\Lambda(P_{l})$ and does not depend on $1\leq l\leq m$;
    		\item for each $1\leq l\leq m$ the two edges 
    		$\Ed_{H}(D_{l})$ and $\Ed_{H}(F_{l})$ are folded together by $\Lambda$ into a single edge (which will in general depend on $l$).
    	\end{itemize}
    \end{corollary}
    
    Generally speaking, the goal of the discussion to follow will try to understand the failure for the map $\Lambda$ on $\mathtt{V}_{1}(H)$ to be injective to the three phenomena described above, with a number of manageable exceptions in the linear case.  
    
    

  \subsection{The internal geometry orbits of orbits of $[H'\backslash G]^{k}$ }

   \begin{definition}\label{connecting tuples}
  		Fix $H_{1}\in\{H,H'\}$. Given $g,g'\in G\setminus H_{1}\cdot \mathcal{S}_{\Omega}$, as well as $\underline{b}\in(\mathcal{Z}(g))^{k}$, $\underline{b}'\in(\mathcal{Z}(g'))^{k}$, we write.   
  	\begin{equation*}
  		\Conn_{H_{1}}(\underline{b},\underline{b}')=\{\,\underline{h}\in (H_{1})^{k}\,|\,\forall j\in k\,\,(b_{j}^{-1}h_{j}b'_{j}=b_{0}^{-1}h_{0}b'_{0})\,\}.
  	\end{equation*}  
  \end{definition}
  
    \begin{observation}\label{meaning of conn}
     The following properties hold:
   \begin{itemize}
   	\item For $\rho\in\G$ and $g\in G\setminus H\cdot \mathcal{S}$ we have 
  	$\mathcal{Y}_{H}(g\rho)=\mathcal{Y}_{H}(g)$.
  		\item There is a bijective correspondence between $\Conn_{H_{1}}(\underline{b},\underline{b}')$ and the set of elements $\gamma\in G$ such that $(H_{1}\backslash\underline{b})\gamma=(H_{1}\backslash\underline{b}')$, given by 
   	$\underline{h}\mapsto \gamma=b_{0}^{-1}h_{0}b'_{0}$.
  	\end{itemize}
  	\end{observation}

  \begin{definition}\label{equivalence relation bg}
   Given $g,g'\in G\setminus\HS$, $\underline{b}\in(\mathcal{Z}(g))^{k}$, $\underline{b}'\in\mathcal{Z}(g')$, we write $(g,\underline{b})\sim(g',\underline{b}')$ if there are $h\in H'$ and $\rho\in\G\cap \bigcap_{s\in\mathcal{S}}H^{s}$, $\rho'\in\G$, such that $g'=h g\rho'$, and  for all $i\in k$ we have
  	\begin{equation*}
  		b'_{i}=
  		\begin{cases*}
  		  hb_{i}\rho & if $b_{i}\in g\G$ \\
  		  b_{i}& otherwise.
  		\end{cases*}
  	\end{equation*}
  \end{definition}

  \begin{observation}\label{depends only on the class}
     Assume that $(g,\underline{b})\sim (g',\underline{b}')$.  
  	For any $\underline{h}\in\Conn_{H'}(\underline{b},\bh)$ there 
  	is $\underline{h}^{*}\in\Conn_{H'}(\underline{b}',\bh)$ such that
  	\begin{equation*}
  		b_{0}^{-1}h_{0}\hat{b}_{0}=\rho^{-1}(b_{0}')^{-1}h^{*}_{0}\hat{b}_{0},
  	\end{equation*}
  	and replacing $(g',\underline{b}')$ by another pair in the same $\sim$-class has a similar effect.
  \end{observation}
  \begin{proof}
   For the last assertion, note that if $b_{i}\in\so$, then 
  \begin{equation*}
  	b'_{i}=b_{i}=(\rho^{-1})^{b_{i}^{-1}}b_{i}\rho\in Hb_{i}\rho.
  \end{equation*}
  \end{proof}
    
      \begin{lemma}\label{implications of the equality of hi}
    	 	Assume that $H_{1}\leq G$ is geometrically taut and we are given $g,\hat{g}\in G\setminus (H_{1}\cdot \mathcal{S}_{\Omega})$, as well as $\underline{b}\in(\mathcal{Z}(g))^{k}$, $\bh\in(\mathcal{Z}(\hat{g}))^{k}$,
    	 	and $\underline{h}\in\Conn_{H_{1}}(\underline{b},\bh)$. Then we have the following: 
    	 	\begin{enumerate}[label=(\Alph*)]
    	 		\item \label{one direction} If $h_{j_{0}}=h_{j_{1}}$ for two distinct $j_{0},j_{1}\in k$ such that $b_{j_{l}}\in g\G$, $\hat{b}_{j_{l}}\in\hat{g}\G$ for each $l\in 2$.
      	  Then:
      	   \begin{itemize}
      		   \item $\hat{g}\in H_{1}g\G$,
      		   \item  $b_{i}^{-1}h_{i}\hat{b}_{i}\in\G$ for all $i\in k$,
      		   \item if $g=\hat{g}$, then $h_{i}=1$ for all $i\in k$ such that $b_{i}\in g\G$.
      	   \end{itemize}
    	 		\item \label{another direction} Conversely, if the constant $\chi=b_{j}^{-1}h_{j}\hat{b}_{j}$ is in $\G$, then 
    	 		\begin{itemize}
    	 			\item $\hat{g}\in H_{1}g\G$,
    	 			\item for all $i\in k$ we have $b_{i}\in g\G$ if and only if $\hat{b}_{i}\in\hat{g}\G$,
    	 			\item the value of $h_{i}$ is the same for all $i$ such that $ b_{i}\in g\G$.
    	 		\end{itemize}
    	 	\end{enumerate}
      \end{lemma}
      \begin{proof}
      	Write $\chi$ for the constant $b_{j}^{-1}h_{j}\hat{b}_{j}$, independent of $j\in k$. 
      	Let us start with \ref{one direction}. The equality $h_{j_{0}}=h_{j_{1}}=h$ implies that for some $\beta,\hat{\beta}\in\G\setminus\{1\}$ we have 
        $\beta g^{-1}h\hat{g}\hat{\beta}=g^{-1}h\hat{g}$. Writing  $\gamma=g^{-1}h\hat{g}$ we obtain: 
        \begin{equation*}
          \beta^{\gamma}=(\hat{\beta})^{-1}\in\G\cap\G^{\gamma},\quad\Rightarrow \quad \gamma\in\G,\quad\Rightarrow\quad \hat{g}\in hg\G\subseteq H_{1}g\G.
        \end{equation*} 
        It follows  that $\chi\in\G$ as well. From this it follows in turn that if $b_{j}\in g\G$, then $\hat{b}_{j}\in H_{1}b_{j}\G=H_{1}g\G$ and so in fact $\hat{b}_{j}\in \hat{g}\G$. Similarly, $\hat{b}_{j}\in \hat{g}\G$ implies that $b_{j}\in g\G$.
        If $g=\hat{g}$, then $h^{g}\in\G\cap (H_{1})^{g}=\{1\}$, and thus $h=1$ by \cref{GOmega has no fixed points}.
        
        For \ref{another direction}, note that the first item follows from the third, while the second follows from $(H_{1}\{g,\hat{g}\}\G)\cap (H_{1}\cdot\so)=\emptyset$. As for third item, 
        consider two distinct $i,j\in k$ such that $b_{i},b_{j}\in g\G$.
        Then 
        \begin{equation*}
        	b_{j}^{-1}b_{i}=(b_{j}^{-1}b_{i})^{\chi}=(\hat{b}_{j}^{-1}h_{j}^{-1}b_{j})b_{j}^{-1}b_{i}(b_{i}^{-1}h_{i}\hat{b}_{i}^{-1})=\hat{b}_{j}^{-1}h_{j}^{-1}h_{i}\hat{b}_{i}\in\G
        \end{equation*}
        and thus $(h_{j}^{-1}h_{i})^{\hat{g}}\in\G$, which implies that $h_{i}=h_{j}$ by \cref{GOmega has no fixed points}.
      \end{proof}

   We now will establish a set of notation that will allow us to describe the geometry of $\Conn_{H'}(\underline{b},\bh)$.
   
   \begin{definition}\label{system of paths}
   	 For $H_{1}\in\{H,H'\}$, we will assume that we are given $g ,g'\in G\setminus(H_{1}\cdot \so)$,
   	 as well as $k$-tuples $\underline{b}\in(\mathcal{Z}(g))^{k}$ and $\bh\in\mathcal{Z}(\bar{g})$.   	    
   	 Write $L_{g}=d(p,gp)$, $L_{\hat{g}}=d(p,\hat{g}p)$, as well as 
    \begin{equation*}
    	\ll=\{j\in k\,|\,b_{j}\in g\G\},\quad\quad \hat{\ll}=\{j\in k\,|\,\hat{b}_{j}\in \hat{g}\G\}.
    \end{equation*} 
     For $j\in k$, write 
     \begin{equation*}
     	\tau_{j}=b_{j}^{-1}\cdot\pa(h_{j}).
     \end{equation*}
     For every $j\in k$ write $\psi_{j}$ 
      \begin{align*}
     	\psi_{j}=
     	\begin{cases*}
     	g & if $j\in\ll$  \\
     	b_{j}\in\so & otherwise.
     	\end{cases*}
     	&
     		\quad\quad \hat{\psi}_{j}=
     	\begin{cases*}
     	\hat{g} & if $j\in\hat{\ll}$  \\
     	\hat{b}_{j}\in\so & otherwise.
     	\end{cases*}
     \end{align*} 
     For $i,j\in k$ write
     \begin{align*}
     	   I_{j}=&b_{j}^{-1}[p,\psi_{j}p] , \quad\quad	 &I_{i,j}=&b_{i}^{-1}\psi_{i}[p,\psi_{i}^{-1}b_{i}b_{j}^{-1}\psi_{j}p],\\
     \hat{I}_{j}=&b_{j}^{-1}h_{j}[p,\hat{\psi}_{j}p],\quad\quad &\hat{I}_{i,j}=&b_{i}^{-1}h_{i}\hat{\psi}_{i}[p,\hat{\psi}_{i}^{-1}\hat{b}_{i}\hat{b}_{j}^{-1}\hat{\psi}_{j}p].
     \end{align*}
     For any pair of $i,j\in k$ we write
     \begin{equation*}
     	 \beta_{i,j}=b_{j}^{-1}b_{i},\quad\quad \hat{\beta}_{i,j}=\hat{b}_{j}^{-1}\hat{b}_{i}.
     \end{equation*}
     We also consider the arcs
   	\begin{equation*}
     	\zeta_{i,j}=I_{i}\ast I_{i,j}\ast I_{j}^{-1}, \quad \quad \hat{\zeta}_{i,j}=\hat{I}_{i}\ast \hat{I}_{i,j}\ast \hat{I}_{j}^{-1}.
   	\end{equation*}
   \end{definition}

  We are interested in values of $g,\hat{g}$ with no geometric redundancy, in the  following sense.  
   
   \begin{definition}\label{d:short}
  	We say that an element $g\in G\setminus (H_{1}\cdot\mathcal{S})$ is \emph{$H_{1}$-short} for $H_{1}\in\{H,H'\}$ if 
  	\begin{equation*}
  		d(gp,hg\rho p)\geq d(p,gp)-\nu L  
  	\end{equation*}  
  	for all $h\in H_{1}$ and $\rho\in\G$. 
  \end{definition}
  
  \begin{observation}
  	In the case the extension is linear 
  	the two elements $g_{0,-1}$ and $g_{0,1}$ are $H_{1}$-short. 
  \end{observation}
 
  The following easy observation provides some motivation for the left side of \cref{d:short}.
   \begin{observation}\label{breakpoints}
     	In the situation described in \cref{system of paths} assume that $g$ and $\hat{g}$ are $H_{1}$-short and fix $j,j'\in k$ let $r\in\partial J$ for $J\in\mathcal{A}_{H\alpha}(\tau_{j})$. Then 
     	\begin{itemize}
     		\item if $j\in\ll$, then $d(r,I_{j'}^{+})\geq L_{g}-2\nu L $;
     		\item if $j\in\hat{\ll}$, then $d(r,\hat{I}_{j'}^{+})\geq L_{\hat{g}}-2\nu L$.
     	\end{itemize}
   \end{observation}
  
  The following lemma justifies the right side of \cref{d:short}. 
   \begin{lemma}\label{consequence of short}
  	For $\nu$ small enough, if $g$ is short, then the arcs $\zeta_{i,j}$ are $(1,\mu_{0}L)$-quasigeodesic for all pairs of distinct $i,j\in k$.  The same holds for $\hat{g}$ and the arcs $\hat{\zeta}_{i,j}$. 
  \end{lemma}
  \begin{proof}
  	The assumption that $g$ and $\hat{g}$ are short in the sense of \cref{d:short} implies that there is some constant $C$, depending only on $(G,\Theta)$ and $(G,\Theta,(X,d),\mu)$ such that
  	\begin{equation*}
  		(I_{i}^{-},I_{i,j}^{+})_{I_{i}^{+}}\leq C.
  	\end{equation*}
    
  	Since any element in $\G\setminus\Omega_{0}$ is loxodromic, there is $C'>0$, with the same dependency as $C$, such that if $b_{i}^{-1}b_{j}\notin\Omega_{0}$, then $\zeta_{i,j}$ is a $(1,C')$-quasigeodesic, 
    whereas the fact that the same property holds for some different constant $C''$ in the case in which  $b_{i}^{-1}b_{j}\in\Omega_{0}$ (necessarily non-trivial) follows from the assumption that $\Omega_{0}$ satisfies condition \ref{omega narrow} in \cref{d:flat preseeds}, as well as the fact that $g,\hat{g}$ are $H'$-short.
  	
  	Clearly, by taking $\nu$ small enough (which takes place after the choice of $\mu_{0}$) we may assume that $\frac{\max\{C,C',C''\}}{L}$ is arbitrarily small and the conclusion of \cref{consequence of short} follows.
  \end{proof}

  \begin{lemma}
  	\label{no-mirror} Assume that the conclusion of \cref{consequence of short} holds and consider distinct $i,j\in\ll$. Then there cannot be subarcs
  	 $K_{l}\subseteq I_{l}$, $l\in\{i,j\}$ such that 
  	\begin{itemize}
  		\item $I_{l}^{+}\in K_{l}$, $l\in\{i,j\}$, 
  		\item $|K_{l}|>6\lambda_{0}L$ 
  		\item $K_{l}\subseteq\nbl(J)$, where $J=\gamma[p,\rho p]$ for $\gamma\in G$, $\rho\in\G$.  
  	\end{itemize}
    The same result holds with the role of $\hat{I}_{\bigcdot}$ and $I_{\bigcdot}$ exchanged.  
  \end{lemma}
  \begin{proof}
    Recall that $\beta_{i,j}\cdot I_{i}=I_{j}$, and clearly we may as well assume that 
  	$K_{j}=\beta_{i,j}\cdot K_{i}$. Let $J'=\beta_{i,j}\cdot J$ and notice that
  	$K_{j}\subseteq\nbl(J')$. By considering the two endpoints of $K_{j}$ and using the triangular inequality, we conclude that there must be points $q_{0},q_{1}\in J$, 
  	$q_{0}',q_{1}'\in J'$ such that
  	\begin{itemize}
  		\item $d(q_{0},q_{1}),d(q_{0}',q'_{1})\geq|K_{j}|-2\lambda_{0}L\geq 6\lambda_{0}L$;
  		\item $d(q_{l},q'_{l})\leq 2\lambda_{0}L$ for $l\in\{i,j\}$;
  		\item the order in which $(\beta_{i,j}^{-1}q'_{0}$, $\beta_{i,j}^{-1}q'_{1})$ along $J$ is the opposite of that of $(q_{0},q_{1})$. 
  	\end{itemize}
  
  	It is not hard to see that $d_{J'}(q_{0},q_{1})>\nu L$, which is impossible, since by assumption
  	$h_{\Omega_{0}}$ is positively $\nu$-stable. 
  \end{proof}
  
  We are finally ready to introduce the full set of working assumptions for this particularly technical subsection.

   \begin{setting}\label{assumption quadrilaterals}
  	     For the rest of this subsections we assume that $H'$ is geometrically taut and 
  	     in the generous case also that $\alpha$ is transverse to $\G$ relative to $\emptyset$.
  	     We will also assume that the constant $\nu$ controlling the extension is taken small enough for the conclusion of \cref{consequence of short} to hold.

  	     We consider fixed $H_{1}\in\{H,H'\}$ as well as $g ,\hat{g}\in G\setminus (H_{1}\cdot \so)$ and tuples $\underline{b}\in(\mathcal{Z}(g))^{k}$, $\bh\in(\mathcal{Z}(\hat{g}))^{k}$, as in \cref{system of paths}.
  	     We will assume that $(g,\underline{b},\hat{g},\bh)$ is in \emph{simplified form}, by which we mean that the following holds:
  	     \begin{enumerate}[label=(\arabic*)]
  	     	\item The elements $g,\hat{g}$ are $H_{1}$-short in the sense of \cref{d:short}.
  	     	\item We have $b_{i}=g$ and $\hat{b}_{i'}=\hat{g}$ for some $i,i'\in k$.
          \item \label{if equal gepsilon} In the linear case, if $g\in H\cdot g_{0,\epsilon}$ for $\epsilon\in\{\pm1\}$, then $g=g_{0,\epsilon}$.
          Similarly, if $\hat{g}\in H\cdot g_{0,\epsilon}$, then $\hat{g}=g_{0,\epsilon}$.  
          
  \end{enumerate}
    
       We will also assume that we are given $\underline{h}\in\Conn_{H_{1}}(\underline{b},\bh)$ such that the followign holds:
  	     \begin{enumerate}[label=(\arabic*)]      \setcounter{enumi}{3}
  	     	  \item \label{chi not in GOmega} The constant $\chi=b_{j}^{-1}h_{j}\hat{b}_{j}$ (independent of $j$) does not belong to $\G$. 
  	     \end{enumerate}

  \end{setting}
  
  The following is an easy consequence of \cref{depends only on the class}. 
  \begin{observation}
  	Let $g ,\hat{g}\in G\setminus (H_{1}\cdot\so)$ and $\underline{b}\in\mathcal{Z}(g)$, $\bh\in\mathcal{Z}(\hat{g})$ .
  	Then there are $(g',\underline{b}')\sim (g,\underline{b})$ and
  	$(\hat{g}',\bh')\sim (\hat{g},\bh)$ such that $(g',\underline{b}',\hat{g},\hat{b}')$ is in simplified form. 
  \end{observation}

   We note the following easy consequences of the definition.
   \begin{observation}\label{properties of Ij}
   	The following holds:
   	\begin{enumerate}[label=(\arabic*)]
   		\item	\label{o0} We have 
   		\begin{equation*}
   			I_{i}^{+}=I_{i,j}^{-},\quad\quad I_{j}^{+}=I_{i,j}^{+},\quad\quad \hat{I}_{i}^{+}=\hat{I}_{i,j}^{-},\quad\quad \hat{I}_{j}^{+}=\hat{I}_{i,j}^{+}.
   		\end{equation*}
   	  \item \label{o1}We have $\beta_{i,j}I_{i}=I_{j}$ for $i,j\in\ll$ and $\hat{\beta}_{i,j}\hat{I}_{i}=\hat{I}_{j}$ for $i,j\in\hat{\ll}$.  
   	  \item \label{o2}For every distinct $i,j\in k$ both $I_{i,j}$ and $\hat{I}_{i,j}$ are translates of arcs of the form $g[p,\rho p]$, where:
   	  \begin{itemize}
   	  	\item $\rho\in\G$;
   	  	\item in the first case we can take $g$ to be either of $b_{i}^{-1}\psi_{i}$ and $b_{j}^{-1}\psi_{j}$, and in the second case we can take $g$ to be either 
   	  $\chi\hat{b}_{i}\hat{\psi}_{i}^{-1}$ or $\chi\hat{b}_{j}\hat{\psi}_{j}^{-1}$.
   	  \end{itemize}
     \item \label{o3}We cannot have $I_{i,j}\simeq^{\pm}_{5\lambda_{0}}\hat{I}_{i',j'}$ for $i,j,i',j'\in k$.  
     \item \label{o4}We cannot have either $I_{i,i'}\simeq^{\pm}_{5\lambda_{0}}J$ or $\hat{I}_{i,i'}\simeq^{\pm}_{5\lambda_{0}}J$ for $i,i',i''\in k$ and $J\in\mathcal{A}_{H}(\tau_{i''})$.
     \item \label{o5} If $I_{i,i'}\simeq^{\pm1}_{5\lambda_{0}}J$, where $J$ is the first $\alpha$-subarc of $\tau_{i''}$, then: 
      \begin{enumerate}[label=(\alph*),ref=(\arabic{enumi}\alph*)]
      	\item \label{special fold1} the extension falls in the linear case,
      	\item \label{special fold2} $g=g_{0,-\epsilon}$, where $\epsilon$ is the exponent associated with $J$,
      	\item \label{special fold3} the first $H$-subarc of $\tau_{i''}$ carries the identity element as label. 
      \end{enumerate}
      The same holds with $(\hat{I}_{i,i'},\tau_{i''}^{-1},\hat{g},g)$ in place of $(I_{i,i'},\tau_{i''},g,\hat{g})$. 
   	\end{enumerate}
   \end{observation}
   \begin{proof}
   	For \ref{o0}, note that the equalities $I_{i}^{+}=I_{i,j}^{-}$ and $I_{j}^{+}=I_{i,j}^{+}$, as well as $\hat{I}_{i}^{+}=\hat{I}_{i,j}^{-}$ are immediate from the definition, whereas by virtue of the fact that $\underline{h}\in\Conn(\underline{b},\bh)$ we have
   	\begin{equation*}
   		\hat{I}_{i,j}^{+}=b_{i}^{-1}h_{i}\hat{b}_{i}\hat{b}_{j}^{-1}\hat{\psi}_{j}p=\chi\hat{b}_{j}^{-1}\hat{\psi}_{j}p=b_{j}^{-1}h_{j}\hat{\psi}_{j}p=\hat{I}_{j}^{+}.
   	\end{equation*}
   	
   	Part \ref{o1} is clear from the definition, and \ref{o2} is a simple algebraic verification around the fact that $\psi_{j}^{-1}b_{j}\in\G$ for $j\in k$. For \ref{o3} note that by applying \cref{consequence translating elements} the relation $I_{i,j}\simeq^{\pm}_{5\lambda_{0}}\hat{I}_{i',j'}$ implies that $\chi\in\G$, contrary to our assumption.  
   	
   To show \ref{o4}, assume for the sake of contradiction that, for instance, $I_{i,i'}\simeq^{\pm}_{5\lambda_{0}}J$ for some $J\in\mathcal{A}_{H}(\tau_{i''})$. By \cref{consequence translating elements} again we conclude that $g\in H'\cdot\so$, contradicting our most basic assumption on $g$. 
   	   
   	As for \ref{o5}, note that the relation $I_{i,j}\simeq^{\pm}_{5\lambda_{0}}J$ for $J$ as in the statement can only take place in case the extension is linear (so that, in particular, $\Omega=\Omega_{0}$), since if it is generous then  $\alpha$ is $(p,\nu)$-transverse to $\G$. Part 
   	\ref{special fold1} follows. Moreover, in this case we must have have $g\in \eta g_{0,-\epsilon}\G$, where $\eta\in H$ is the label of 
   	the first $H$-subarc of $\tau_{i''}$ and  $\epsilon\in\{\pm 1\}$ the exponent of $J$. Our basic assumptions in \cref{assumption quadrilaterals}, as well as \cref{assumptions on the ksets} then implies that $g=g_{0,-\epsilon}$ and $\eta=1$, so that \ref{special fold2} and \ref{special fold3} hold. 
   	
   	The final claim is clear by the symmetry of the problem. 
   	
   \end{proof}

  \subsubsection{Bounding the cancellation between $\alpha$-subarcs of $\tau_{i}$ and $I_{i,j}$. }
  
  The most technical part of this section is the proof of  \cref{no mixed matching} below, which states that there is very little backtracking between any $\alpha$-subarc of $\tau_{i}$ and $I_{i,j}$. The main ingredient will be  \cref{backtracking copy}.

  \begin{definition}\label{folded escaping}
  	Let $\tau,\rho$ be two $\qc$-quasigeodesic arcs with $\tau^{-}=\rho^{-}$. 
  	We say that the pair $(\tau,\rho)$ is folded if 
  	\begin{equation*}
  		(\tau^{+},\rho^{+})_{\rho^{-}}^{-}>5\lambda_{0}L.
  	\end{equation*}
  \end{definition}
  
   The next two lemmas are auxiliary results involving the assumption that $g$ and $\hat{g}$ are $H'$-short. 
 
     \begin{lemma}\label{first breakpoint close to origin}
     	 Assume $H_{1}=H'$, fix $j_{0}\in k$ and suppose that there is $J\in\mathcal{A}_{\alpha}(\tau_{j_{0}})$ such that $d(q,q')\leq\lambda_{0}L$ for some $q'\in J$, $q\in I_{j_{0}}$. Then we have: 
     	 \begin{enumerate}[label=(\alph*)]
     	 	\item\label{first statement} $d(J^{-},\tau_{j_{0}}^{-})\leq 3\lambda_{0}L$;
     	 	\item\label{second statement} $J$ is the first $\alpha$-subarc of $\tau_{j_{0}}$;
     	 \end{enumerate}
    
     	  An analogous statement holds with the pair $(\hat{I}_{j_{0}},\tau_{j_{0}}^{-1})$ in place of $(I_{j_{0}},\tau_{j_{0}})$. 
     \end{lemma}
     \begin{proof}
       Since $\tau_{j_{0}}$ is a $\qc$-quasigeodesic arc, applying \cref{quadrilaterals general} to the subarc of $I_{j_{0}}$ between $I_{j_{0}}^{-}$ and $q$, and the subarc of $\tau_{j_{0}}$ between $I_{j_{0}}^{-}=\tau_{j_{0}}^{-}$ and $q'$ implies that: 
       \begin{itemize}
        \item $d(J^{-},\tau_{j_{0}}^{-})\leq 2\lambda_{0}L$, in which case we are done,
        \item or $d(J^{-},r)\leq \lambda_{0}L$ for some $r\in I_{j_{0}}$.
       \end{itemize}
         
       In the second case, by \cref{breakpoints} we have the inequality $L_{g}-\lambda_{0}L\leq d(J^{-},I_{j_{0}}^{+})$. On the other hand,
       \begin{equation*}
       	d(J^{-},I_{j_{0}}^{+})\leq d(r,I_{j_{0}}^{+})+\lambda_{0}L=L_{g}-d(r,I_{j_{0}}^{-})+\lambda_{0}L,
       \end{equation*}
       since $I_{j_{0}}$ is geodesic, so we conclude that $d(r,I_{j_{0}}^{-})\leq 2\lambda_{0}L$ and thus 
       $	d(J^{-},\tau_{j_{0}}^{-})\leq 3\lambda_{0}L$,
       establishing \ref{first statement}.      
        On the other hand, \ref{second statement} now follows from the fact that $\tau_{j_{0}}$ is a $\qc$-quasigeodesic arc.
     \end{proof}

     \begin{lemma}\label{no fold bound}
     	Let $j\in \ll$ and assume that $(I_{j},\tau_{j})$ is not folded and 
     	\begin{equation*}
     		d(I_{j}^{+},q)\leq 30\lambda_{0}L
     	\end{equation*}
     	 for some $q\in\hat{I}_{j}$. Then  $h_{j}\in H$ and 
     	 \begin{equation*}
     	 	d(p,h_{j}p)=d(I_{j}^{-},\hat{I}_{j}^{-})\leq 71\lambda_{0}L. 
     	 \end{equation*}
     \end{lemma}
     \begin{proof}
     	To begin with, note that by \cref{breakpoints} we must have 
     	\begin{equation*}
     		  L_{\hat{g}}-\lambda_{0}L\leq d(\hat{I}_{j}^{+},I_{j}^{-})\
     	\end{equation*}
     	On the other hand, 
     	\begin{equation*}
     		d(\hat{I}_{j}^{+},I_{j}^{-})\leq d(\hat{I}_{j}^{+},q)+30\lambda_{0}L+L_{g}=L_{\hat{g}}-d(q,\hat{I}_{j}^{-})+L_{g}+30\lambda_{0}L.
     	\end{equation*}
     	Combining the two inequalities above we get 
     	\begin{equation}\label{bound on qhatI}
     		d(q,\hat{I}_{j}^{-})\leq L_{g}+31\lambda_{0}L.
     	\end{equation}
     	Our assumption that $(I_{j},\tau_{j})$ is not folded means that 
     	$(I_{j}^{+},\hat{I}_{j}^{-})_{I_{j}^{-}}\leq 5\lambda_{0}L$ and thus 
     	$(q,\hat{I}_{j}^{-})_{I_{j}^{-}}\leq 5\lambda_{0}L$, since $q\in I_{j}$.   
     		Using this in the second inequality below, we obtain
     	\begin{equation*}
     		d(q,\hat{I}_{j}^{-})\geq d(q,I_{j}^{-})+d(I_{j}^{-},\hat{I}_{j}^{-})-10\lambda_{0}L\geq L_{g}+d(I_{j}^{-},\hat{I}_{j}^{-})-40\lambda_{0}L,
     	\end{equation*}
     	which combined with \eqref{bound on qhatI} implies $d(I_{j}^{-},\hat{I}_{j}^{-})\leq 71\lambda_{0}L$, as needed. That  $h_{j}\in H$ now follows from the fact that $\tau_{j}$ is $\qc$-quasigeodesic.  
     \end{proof}
     
  We use the notation $\mathcal{A}_{\alpha}(\tau),\mathcal{A}_{H}(\tau),\simeq$, etc. introduced in $\cref{congruence}$. 

\newcommand{\defecteq}[0]{
 d(I_{j}^{+},\hat{I}_{j})>d(I_{j},\hat{I}_{j})+10\lambda_{0}L  
}

  \begin{lemma}\label{backtracking copy} 
    Let $H_{1}=H$. Assume furthermore that for some $j_{0}\in k$  and $J_{j_{0}}\in\mathcal{A}_{\alpha}(\tau_{j_{0}})$  we have 
  	\begin{equation}\label{big fold}
  	 	I_{j_{0}}\simeq_{\kappa}J_{j_{0}}.
  	 \end{equation}
  	for some $\kappa\geq 10\lambda_{0}$. 

  	Then there is $\eta\in H$, $l\in k_{o}$ and $\epsilon\in\{\pm1\}$ such that $\eta\alpha_{l}^{\epsilon}$ is an initial subword of $h_{j_{0}}$ and the following holds:  	
  	\begin{enumerate}[label=(\alph*)]
  		\item \label{first folding claim} For every $j\in k$ for which $(I_{j},\tau_{j})$ is folded and every $j'\in k\setminus\{j\}$ we have the following.
  		\begin{enumerate}[label=(\arabic*),ref=(\alph{enumi}\arabic*)]
  			\item \label{if folded initial} The word $\eta\alpha_{l}^{\epsilon}$ is an initial subword of the normal form of $h_{j}$.
  			\item \label{projection folded0} The following inequality holds:  
  			\begin{equation*}
  				\sup \{\,d(r,I_{j})\,|\,r\in\zeta_{j,j'}\cap\nbl(\tau_{j})\,\}\leq 15\lambda_{0}L.
  			\end{equation*}
  			\item  \label{projection folded} For every $q\in\tau_{j}$ we have
  		\begin{equation*}  
     	 	d(q,\zeta_{j,j'})\geq d(q,I_{j}^{-})-(\hf +20\lambda_{0})L. 
     	 \end{equation*}
  		\end{enumerate}
   \end{enumerate}
   Assume furthermore that 
   \begin{equation*}
   	\kappa\geq 400\lambda_{0}L
   \end{equation*}
   and there is $j_{1}\in \ll$ such that:
   \begin{equation}\label{case defining inequality}
       d(I_{j_{1}}^{+},\hat{I}_{j_{1}})\leq 100\lambda_{0}L.
   \end{equation} 
   Then we have the following.
   
   \begin{enumerate}[label=(\alph*)]
    \setcounter{enumi}{1}
     \item \label{second part} If $(I_{j_{1}},\tau_{j_{1}})$ is folded, then there is $\eta'\in H$ such that for all $j\in\ll$ such that at least one of the two conditions 
     \begin{enumerate}[label=(\roman*)]
     	\item \label{not getting away0} $d(I_{j}^{+},\hat{I}_{j})\leq d(I_{j}^{+},\hat{I}_{j})+5\lambda_{0}L$;
     	\item \label{not on the back side0} $d_{\hat{I}_{j}}(\hat{I}_{j}^{-},I_{j}^{+})\geq 5\lambda_{0}L$.
     \end{enumerate}
    we have
       \begin{equation*}
       	h_{j}=\eta\alpha_{l}^{\epsilon}\eta'.
       \end{equation*}
   
    \item \label{third part} If $(I_{j_{1}},\tau_{j_{1}})$ is not folded, then for every $(j,j')\in (k)^{2}$ for which
    \begin{equation*}
	      \defecteq
	     \end{equation*}
     	 we have 
	     	 \begin{equation*}
	     	 	h_{j}=h_{j_{1}}\in H.
	     	 \end{equation*}
    \end{enumerate}
   The same result holds after exchanging the roles of $I_{\bigcdot}$ and $\hat{I}_{\bigcdot}$, and of $(\tau_{j})_{j\in k}$ and $(\tau_{j}^{-1})_{j\in k}$. 
  \end{lemma}
  \begin{proof}
  	 Write $\ff$ (resp. $\hat{\ff}$) for the collection of $j\in \L$ such that $(I_{j},\tau_{j})$ (resp. $(\hat{I}_{j},\tau_{j}^{-1})$) is folded.
  	 It is a simple observation that $(I_{j_{0}},\tau_{j_{0}})$ under the assumptions of the lemma $(I_{j_{0}},\tau_{j_{0}})$ must be folded. From \cref{first breakpoint close to origin} that $J_{j_{0}}$ is the first $\alpha$-subarc appearing along $\tau_{j_{0}}$, and that $d(J_{j_{0}}^{-},\tau_{j_{0}}^{-})\leq 3\lambda_{0}L$. Let $\eta\in H$, $l\in k_{o}$ and $\epsilon\in\{\pm1\}$ be such that $\eta\alpha_{l}^{\epsilon}$ is an initial subword of the normal form of $h_{j_{0}}$. We complement this preliminary discussion with the lemma below.
 
  	 \begin{lemma}\label{M0 ineq}
  	 	Write $M_{0}=(I_{j_{0}}^{+},J_{j_{0}}^{+})_{I_{j_{0}}^{-}}$. The following inequality holds:
  	 	\begin{equation*}
     	   M_{0}\leq\frac{(1+5\lambda_{0})L}{2}.
      \end{equation*}
  	 \end{lemma}
  	 \begin{subproof}
  	 By \cref{breakpoints} we must have 
	   \begin{equation*}
	   	d(I_{j_{0}}^{+},J_{j_{0}}^{+})\geq L_{g}-2\nu L, 
	   \end{equation*}
	   Using the inequality $d(p,\eta p)\leq 3\lambda_{0}L$ and the fact that remainder subarcs have length at most $\nu L$, we easily obtain $d(I_{j_{0}}^{-},J_{j_{0}}^{+})=d(\tau_{j_{0}}^{-},J_{j_{0}}^{+})\leq L+4\lambda_{0}L$. We now have
     \begin{equation*}
   	   L_{g}=|I_{j_{0}}|=2M_{0}-d(I_{j_{0}}^{-},J_{j_{0}}^{+})+d(I_{j_{0}}^{+},J_{j_{0}}^{+})\geq 2M_{0}-(L+4\lambda_{0}L)+(L_{g}-2\nu L),
     \end{equation*}
      from which the equality in the sublemma follows by a simple calculation. 
  	 \end{subproof}
     
     It is also can be easily deduced from the inequality $d(I_{j_{1}}^{-},J_{j_{0}}^{-})\leq 3\lambda_{0}L$ that 
     \begin{equation*}
     	M_{0}\geq 390\lambda_{0}L.
     \end{equation*}
     
      \emph{Part \ref{if folded initial}.} 
      Assume for the sake of contradiction that for some $j\in \ll$ the word $\eta\alpha_{l}^{\epsilon}$ is not an initial subword of the normal form for $h_{j}$, but that $(I_{j},\tau_{j})$ is folded, that is:
     \begin{equation*}
     	(\tau_{j}^{+},I_{j}^{+})_{I_{j}^{-}}\geq 5\lambda_{0}L.
     \end{equation*}
      Consider the arc 
     $\sigma_{j}=\beta_{j_{0},j}\cdot \tau_{j_{0}}$. We also know that $(I_{j_{0}},\tau_{j_{0}})$ is folded, so by equivariance the pair $(I_{j},\sigma_{j})$ must also be folded, i.e.:  
     \begin{equation*}
     	(I_{j}^{+},\sigma_{j}^{+})_{I_{j}^{-}}\geq 5\lambda_{0}L.
     \end{equation*}
     Since $X$ is $\delta$-hyperbolic, we conclude that 
     \begin{equation}\label{estimate gromov product}
     	(\tau_{j}^{+},\sigma_{j}^{+})_{I_{j}^{-}}\geq 5\lambda_{0}L-\delta.
     \end{equation}
     Our assumption that $\eta\alpha_{l}^{\epsilon}$ is not an initial subword of the normal form of $h_{j}$ implies that the concatenation $\sigma_{j}^{-1}\ast\tau_{j}$ is very close to being a translate of a $p$-extension arc. More precisely, a $p$-extension arc $\tau^{*}$ between $\sigma_{j}^{+}$ and $\tau_{j}^{+}$ can be obtained by replacing the concatenation of the last $H$-subarc $K_{0}$ of $\sigma_{j}^{-1}$, and the first $H$-subarc $K_{1}$ of $\tau_{j}$ by some geodesic arc between $K_{0}^{-}$ and $K_{1}^{+}$. 
     
     By \cref{constant assumptions} and \cref{quadrilaterals general} we have $K_{0}^{-}\in\tau^{*}\subseteq\nb_{\lambda_{0}L}([\sigma_{j}^{+},\tau_{j}^{+}])$. 
     But 
     \begin{equation*}
      d(K_{0}^{-},I_{j}^{-})=d(\tau_{j_{0}}^{-},J_{j_{0}}^{-})\leq 3\lambda_{0}L.
     \end{equation*}
      We thus conclude that $I_{j}^{-}\in\nb_{3\lambda_{0}L}([\sigma_{j}^{+},\tau_{j}^{+}])$, and since the dependency of the Gromov-product on the basepoint is $1$-Lipschitz, we get
     \begin{equation*}
     	(\sigma_{j}^{+},\tau_{j}^{+})_{I_{j}^{-}}\leq 4\lambda_{0}L,
     \end{equation*}
     in contradiction with \eqref{estimate gromov product}.
      
     For $j\in \ff$ let $\tau_{j,0}$ be the initial subarc of $\tau_{j}$ corresponding to the initial subword $\eta\alpha_{l}^{\epsilon}$, and $\tau_{j,1}$ for the rest of $\tau_{j}$, so that $\tau_{j}=\tau_{j,0}\ast\tau_{j,1}$. Then for any such $j$ we have $(\beta_{j_{0},j}\cdot I_{j_{0}},\beta_{j_{0},j}\cdot \tau_{j_{0},0})=(I_{j_{0}},\tau_{j,0})$. We let $J_{j}$ stand for the $\alpha$-subarc corresponding to said instance of $\alpha_{l}^{\epsilon}$. 
       
       \emph{Part \ref{projection folded0}.} Write
       \begin{equation*}
       	D=\sup \{\,d(r,I_{j})\,|\,r\in\zeta_{j,j'}\cap\nbl(\tau_{j})\,\}.
       \end{equation*}
       By \cref{M0 ineq}, any point in $(I_{j'}\cup I_{i,j'})\cap\nbl(\tau_{j})$ at distance at most $\sfrac{1}{3}L$ from $I_{j}$ must belong to $\nbl(J_{j})$. If $D>15\lambda_{0}L$, then by \cref{quadrilaterals general} and the fact that remainder arcs have length bounded by $\nu L$ we get a violation of either \cref{no-mirror} or the fact that $I_{j,j'}\nsimeq_{5\lambda_{0}} J_{j}$, which follows from $g\notin H\cdot\{g_{0,-1},g_{0,1}\}\cdot\G$ (since $L_{g}>\nu L$) and \cref{properties of Ij}, part \ref{o5}.
       
      \emph{Part \ref{projection folded}.} Consider any point $q\in\tau_{j}$. Let $\hat{\tau}_{j}$ the subarc of $\tau_{j}$ between $\tau_{j}^{-}$ and $q$. From \cref{quadrilaterals general}, as well as the fact that
       $\zeta_{j,j'}$ and $\tau_{j}$ are $1$-Lipschitz $\qc$-quasigeodesic arcs, we know that 
       \begin{equation*}
       	d(q,\zeta_{j,j'})\geq d(q,\tau_{j}^{-})-M_{1}-\lambda_{0}L,
       \end{equation*}
       where $M_{1}=d(r,I_{j}^{-})$ for some $r\in\zeta_{j,j'}$ such that the subarc $\tilde{\zeta}_{j,j'}\subseteq\zeta_{j,j'}$ between $I_{j}^{-}$ and $r$ satisfies  $\tilde{\zeta}_{j,j'}\subseteq\nb_{\lambda_{0}L}(\hat{\tau}_{j})$. 
       On the other hand, it follows from \cref{projection folded0} that 
       $M_{1}\leq M_{0}+15\lambda_{0}L$, and so part \ref{projection folded} follows from \cref{M0 ineq}. 
    
      \emph{Part \ref{second part}.} Assume now that for some $j_{1}\in\ff$ inequality \eqref{case defining inequality} holds. Part \ref{projection folded} easily implies that the pair $(\hat{I}_{j_{1}},\tau_{j_{1}}^{-1})$ is folded as well, and that $J_{j_{1}}\simeq^{\pm}_{400\lambda_{0}}\hat{I}_{j_{1}}$. This allows us to apply part \ref{first folding claim} symmetrically to the pair $(\hat{I}_{j_{1}},\tau_{j_{1}}^{-1})$ to conclude that
      the arc $J_{j_{1}}$ is the last, and hence the unique $\alpha$-subarc of $\tau_{j_{1}}$, as well as the following.  
      \begin{enumerate}
      	\myitem{$(a1')$} \label{endswith}There is some $\eta'\in H$ such that for any $j\in\hat{\ff}$ the normal form of $h_{j}$ ends with the subword $\alpha_{l}^{\epsilon}\eta'$. 
         \myitem{$(a2')$} \label{sticks out} For every $q\in\tau_{j}$ and $j'\in k\setminus\{j\}$ we have
  	    	\begin{equation*}  
     	      	d(q,\hat{\zeta}_{j,j'})\geq d(q,\hat{I}_{j}^{-})-(\hf +20\lambda_{0})L.
     	   \end{equation*}
    \end{enumerate} 
    
      In particular, 
      	\begin{equation*}
      		h_{j_{1}}=\eta\alpha_{l}^{\epsilon}\eta' .
      	\end{equation*} 

      Consider now some arbitrary $j\in\hat{\ll}$ and assume that at least one of the conditions
      \begin{enumerate}[label=(\roman*)]
     	\item \label{not getting away} $d(I_{j}^{+},\hat{I}_{j})\leq d(I_{j}^{+},\hat{I}_{j})+5\lambda_{0}L$;
     	\item \label{not on the back side} $d_{\hat{I}_{j}}(\hat{I}_{j}^{-},I_{j}^{+})\geq 5\lambda_{0}L$.
     \end{enumerate}
      holds for some $j\in k\setminus\{j_{1}\}$. Roughly speaking, the failure of this condition means that some significant final stretch of the geodesic $I_{j}$ finds itself traveling way from $\hat{I}_{j}$ from the side of $\hat{I}_{j}^{-}$.
       
      \begin{lemma}\label{not in H}
      	Under the assumptions above, $h_{j}\notin H$. 
      \end{lemma}
      \begin{subproof}
      	For the sake of contradiction assume $h_{j}\in H$ and write $\phi_{0}=\beta_{j_{1},j}\cdot \tau_{j_{1}}^{-1}$ and $\phi_{1}=\hat{\beta}_{j_{1},j}\cdot \tau_{j_{1}}^{-1}$. 
        Note that by \ref{endswith} the unique $\alpha$-subarc of each of $\phi_{0}$, $\phi_{1}$ is associated with an occurrence of the letter $\alpha_{l}^{-\epsilon}$. This implies that a translate $\phi$ of a $p$-extension arc can be obtained from the concatenation
        \begin{equation*}
      	   \phi'=\phi_{0}\ast \tau_{j}\ast\phi_{1} 
        \end{equation*}
        by replacing the subarc of $\phi'$ consisting of the concatenation of the terminal $H$-subarc of $\phi_{0}$, $\tau_{j}$, and the initial $H$-subarc of $\phi_{1}$ by a geodesic between its two endpoints. 
        
        If we denote by $K_{0}\in\mathcal{A}_{\alpha}(\phi_{0})$ and $K_{1}\in\mathcal{A}_{\alpha}(\phi_{1})$ the two $\alpha$-subarcs of $\phi$ in the order in which they appear along $\phi$, then by equivariance we know that
        \begin{align*}
        	d(K_{0}^{+},I_{j}^{-}),\,\,\,d(K_{1}^{-},\hat{I}_{j}^{-})&\leq 3\lambda_{0}L
        \end{align*}
        and there are points $q_{0}\in I_{j}\cap\nbl(K_{0})$, $\hat{q}_{0}\in\hat{I}_{j}\cap\nbl(K_{1})$ such that 
        \begin{equation*}
        	d(q_{0},I_{j}^{-}),\,\,d(\hat{q}_{0},\hat{I}_{j}^{-})\geq 400\lambda_{0}L.
        \end{equation*}
       We thus conclude that each of conditions \ref{not getting away} and \ref{not on the back side} fail, by part \ref{distance estimate projection to qc} of \cref{quadrilaterals general}. This is a contradiction.
      \end{subproof}
       
       If either one of $(I_{j},\tau_{j})$ or $(\hat{I}_{j},\tau_{j}^{-1})$ is not folded, or if $|\mathcal{A}_{\alpha}(\tau_{j})|\geq 2$ we likewise conclude, using part \ref{distance estimate projection to qc} of \cref{quadrilaterals general}, that conditions \ref{not getting away} and \ref{not on the back side} must both fail. Therefore,  $j\in\ff\cap\hat{\ff}$ and $|\mathcal{A}_{\alpha}(\tau_{j})|=1$. It follows that $\tau_{j}$ contains a unique $\alpha$-subarc to which \ref{if folded initial} and \ref{endswith} apply simultaneously resulting in the equality
       \begin{equation*}
       	 h_{j}=\eta\alpha_{l}^{\epsilon}\eta',
       \end{equation*}
       as needed. 
        
      \emph{Part \ref{third part}.} Assume $j_{1}\nin\ff$. Note that by \cref{no fold bound} we must have $h_{j_{1}}\in H$ and 
     \begin{equation*}
     	d(I_{j_{1}}^{-},\hat{I}_{j_{1}}^{-})\leq 71\lambda_{0}L.
     \end{equation*}
          
     Let now $j$ be such that either of \ref{not getting away} or \eqref{not on the back side} is satisfied. Our goal is to show that $h_{j}=h_{j_{1}}$, so assume this is not the case. Write 
     \begin{equation*}
     	\phi_{0}=\beta_{j_{0},j} \cdot\tau_{j_{0},0}^{-1},\quad\quad\phi_{1}=\hat{\beta}_{j_{1},j}\cdot\tau_{j_{1}}^{-1},\quad\quad\phi_{2}=\hat{\beta}_{j_{1},j}\beta_{j_{0},j_{1}}\cdot \tau_{j_{0}}.
     \end{equation*}
     It is simple to verify the concatenation 
     \begin{equation*}
     	\phi'=\phi_{0}\ast\tau_{j}\ast\phi_{1} \ast\phi_{2}
     \end{equation*}
     is well-defined.      
     Since $h_{j}\neq h_{j_{1}}$, we also have $\eta^{-1}h_{j}h_{j_{1}}^{-1}\eta\neq 1$, from which it follows that we can obtain a translate $\phi$ of a $p$-extension arc by replacing the subarc of $\phi'$ consisting of the concatenation of 
     \begin{itemize}
     	\item the last $H$-subarc of $\phi_{0}$, with label $\eta^{-1}$,
     	\item $\tau_{j}$, 
     	\item the first $H$-subarc of $\phi_{2}$, which we know has label $\eta$,
     \end{itemize}
     by some geodesic arc between the two endpoints of said subarc, corresponding to an $H$-subarc of the result with 
     label $\eta^{-1}h_{j}h_{j_{1}}^{-1}\eta\in H\setminus\{1\}$. 
     
     Denote by $K_{0}$ and $K_{1}$ the two $\alpha$-subarcs of $\phi$ in the order in which they appear along $\phi$, so that $K_{0}\subseteq\phi_{0}$ and $K_{1}\subseteq\phi_{2}$. Contrary to what was the case in the proof of \cref{not in H}, $K_{0}$ and $K_{1}$ are associated to opposite exponents in this case. By equivariance, we know that 
     \begin{equation*}
     	d(K_{0}^{+},\phi_{0}^{+}),\,\,d(K_{1}^{-},\phi_{2}^{-})\leq 3\lambda_{0}L.
     \end{equation*}
     We also know know that $d(\hat{\beta}_{j_{1},j}I_{j}^{-},\hat{I}_{j}^{-})=d(I_{j_{1}}^{-},\hat{I}_{1}^{-})\leq 71\lambda_{0}L$. 
     
     Note also that $\phi_{0}^{+}=I_{j}^{-}=\beta_{j_{0},j}I_{j_{0}}^{-}$. 
     By equivariance we know that there are points 
     \begin{align*}
     	q\in I_{j}\cap\nbl(K_{0}),\quad&\text{with}\quad d(q,I_{j}^{-})\geq 400\lambda_{0}L \\
     	q'\in\hat{\beta}_{j_{1},j}I_{j_{1}}\cap\nbl(K_{1}) ,\quad&\text{with}\quad d(q',\hat{\beta}_{j_{1},j}I_{j_{1}}^{-})\geq 400\lambda_{0}L.
     \end{align*}
     On the other hand, we know that 
     \begin{equation*}
     	d(\hat{\beta}_{j_{1},j}\cdot I_{j_{1}}^{+},\hat{I}_{j})=d(\hat{\beta}_{j_{1},j}I_{j_{1}}^{+},\hat{\beta}_{j_{1},j}\hat{I}_{j_{1}})=d(I_{j_{1}}^{+},\hat{I}_{j_{1}})\leq 100\lambda_{0}L.
     \end{equation*}
      
      The desired contradiction follows once again from part \ref{distance estimate projection to qc} of \cref{quadrilaterals general} .  
     The assertion at the end of the statement of the lemma is clear from the symmetry of the problem, so this concludes the proof of \cref{backtracking copy}.
  \end{proof}

  \begin{lemma}\label{no mixed matching} 
    Assume $H_{1}=H'$. For any $i\in k$ 
    any $J\in\mathcal{A}_{\alpha}(\tau_{i})$ we have:
  	 	\begin{equation*}
  	 		\forall J^{*}\in\{I_{j},\hat{I}_{j}\}_{j\in k}\,\,\,\,\,J\nsimeq^{\pm}_{400\lambda_{0}} J^{*}   %
  	 	\end{equation*}
  \end{lemma}
  \begin{proof}
     Assume for the sake of contradiction that the claim fails.
     Without loss of generality, we may assume that $2\subseteq\ff$ and that 
     \begin{equation*}
     	J\simeq_{400\lambda_{0}}I_{0}
     \end{equation*}
     for some $J\in\mathcal{A}_{\alpha}(\tau_{i_{0}})$ and $i_{0}\in k$. Note that in particular this implies that 
     $L_{g}\geq 398\lambda_{0}L$. 
       
    We also have: 
    \begin{lemma}
    	  The equality $i_{0}=0$ holds.
    \end{lemma}
    \begin{subproof}
    	Assume for the sake of contradiction that $i_{0}\neq 0$. Then by \cref{quadrilaterals general}
    	there are points $q\in I_{0}$, $q'\in\tau_{i_{0}}$ such that
    	\begin{itemize}
    		\item  $d(q,I_{0}^{+})<L_{g}-397\lambda_{0}L$;
    		\item the subarc $\zeta'\subseteq\zeta_{i_{0},0}$ between $I_{i_{0}}^{-}$ and $q$ is $\lambda_{0}L$-fellow traveling with the subarc $\tau'$ of $\tau_{i_{0}}$ between $I_{i_{0}}^{-}=\tau_{i_{0}}^{-}$ and $q'$.
    	\end{itemize}
      Note that all points $r\in\tau'$ with $d(r,\tau_{i_{0}}^{-})>10\lambda_{0}L$ must
      satisfy $d(r,I_{i_{0},0})<L_{g}-\lambda_{0}L$. It follows from \cref{breakpoints} that $J$ must be the first $\alpha$-subarc of $\tau_{i_{0}}$ and $d(J^{-},I_{i_{0}}^{+})>L_{g}-\lambda_{0}L$.  
    	Since $L_{g}=|I_{i_{0}}|\geq 398\lambda_{0}L$, we easily get a contradiction to \cref{no-mirror}. 
    \end{subproof}
     
       We thus fall under the hypotheses of \cref{backtracking copy}, part \ref{first folding claim}.
       Let $\eta\in H$, $l\in k_{o}$ and $\epsilon\in\{\pm1\}$ be such that the normal form of $h_{0}$ begins with $\eta\alpha_{l}^{\epsilon}$.
     
    \begin{lemma}\label{some index folds}
     	 For some $j_{1}\in 2$ it must be the case that 
     \begin{equation*}\label{hIj comes close}
     		d(I_{j_{1}}^{+},\hat{I}_{j_{1}})\leq 100\lambda_{0}L. 
     \end{equation*}
     \end{lemma}
     \begin{subproof}
     	Note that by \cref{quadrilaterals general}, we can decompose $\zeta_{0,1}$ as a concatenation $\zeta_{0,1}=\sigma_{0}\ast\sigma\ast \sigma_{1}$. where
     	\begin{equation*}
     		 \sigma_{l}\subseteq\nb_{\lambda_{0}L}(\tau_{l}),\,\,l\in 2,\quad\quad\sigma\subseteq\nb_{\lambda_{0}L}(\hat{\zeta}_{0,1}). 
     	\end{equation*}
     Note that in the linear case this implies that $g\notin H\cdot(\{g_{0,-1},g_{0,1}\})\cdot\G$, since otherwise we would have $g=g_{0,\epsilon}$ for some $\epsilon\in\{\pm1\}$ and thus $L_{g}\leq\nu L$. 
     By \cref{properties of Ij}  one easily deduces that
     	\begin{equation}\label{soon away}
     		\diam(\nbl(\tau_{l})\cap I_{0,1})\leq 11\lambda_{0}L. 
     	\end{equation}
     	
     	We distinguish two cases.
     	 
     	\emph{Case $1$: $|I_{0,1}|>30\lambda_{0}L$.} Let $K_{0,1}\subseteq I_{0,1}$ consist of those points at distance at least $11\lambda_{0}L$ from $\partial I_{0,1}$.  
     	It follow from \cref{properties of Ij} that $|I_{0,1}\cap\nbl(\hat{I}_{0,1})|\leq 5\lambda_{0}L$. Therefore one of the endpoints of $K_{0,1}$ does not belong to $\nbl(\hat{I}_{0,1})$. Without loss of generality assume that 
      $d(K_{0,1}^{+},\hat{I}_{0,1})>\lambda_{0}L$. We then know by \cref{quadrilaterals general}
      that $d(K_{0,1}^{+},q)\leq\lambda_{0}L$ for some $q\in \hat{I}_{0}\cup\hat{I}_{1}$. 
      If $q\in\hat{I}_{1}$, then $d(I_{1}^{+},q)\leq 12\lambda_{0}L$, by the triangular inequality, and so we are done. Otherwise, $q\in\hat{I}_{0}$. Applying \cref{quadrilaterals general} we easily conclude using \eqref{soon away} that 
      $K_{0,1}^{-}\in\nbl(\hat{I}_{0})$ and thus that $d(I_{0}^{+},\hat{I}_{0})\leq 12\lambda_{0}L$, as needed. 
     	
     	\emph{Case $2$: $|I_{0,1}|\leq 30\lambda_{0}L$.} For each $l\in 2$ let 
     	$q_{l}\in I_{l}$ be the unique point at distance $8\lambda_{0}L$ from $I_{l}^{+}$,
     	It follows from \cref{quadrilaterals general} and \cref{no-mirror} that 
     	$\{q_{0},q_{1}\}\nsubseteq\nbl(\hat{I}_{0,1})$, so pick $l_{0}\in 2$ such that $q_{l_{0}}\notin \nbl(\hat{I}_{0,1})$. 
     	It follows from \cref{quadrilaterals general} that at least one of the following holds: 
     	\begin{itemize}
     		\item $q_{l_{0}}\in\nbl(\hat{I}_{l})$ for some $l\in 2$,
     		\item $q_{l_{0}}\subseteq\nbl(\tau_{l})$ for some $l\in 2$.
     	\end{itemize}  
       In the first case it easily follows from the triangular inequality, the definition of $q_{0},q_{1}$ and the assumption $|I_{0,1}|\leq 30\lambda_{0}L$ that $d(I_{l}^{+},\hat{I}_{l})\leq 47\lambda_{0}L$, so in this case we are done.
       Assume now that $d(q_{l_{0}},q')\leq\lambda_{0}L$ for some $q'\in\tau_{l}$.
      Note that this implies that $(I_{l},\tau_{l})$ is folded.
      By part \ref{if folded initial} of \cref{backtracking copy} we have
     \begin{observation}\label{sticking out obs}
       Any point in $r\in\tau_{l}$ such that $d(r,\zeta_{0,1})\in(\lambda_{0}L,\sfrac{1}{3}L)$ must belong to an $\alpha$-subarc $J_{l}$ of $\tau_{l}$ with 
       \begin{equation*}
       	J_{l}\simeq_{400\lambda_{0}}I_{l}.
       \end{equation*}
     \end{observation}

      We know that we can decompose $\tau_{l}$ as  
     \begin{equation*}
     	\tau_{l}=\sigma_{l,0}\ast\sigma_{l,1}\ast\sigma_{l,2},
     \end{equation*}
     where some of the arcs in question might be degenerate or empty, and 
     \begin{equation*}
     	\sigma_{l,0}\subseteq\nbl(\zeta_{0,1}),\quad\quad \sigma_{l,1}\subseteq\nbl(\tau_{1-l}),\quad\quad
     	\sigma_{l,2}\subseteq\nbl(\hat{\zeta}_{0,1}).
     \end{equation*}
     
     We may assume that the $\sigma_{l,0}$ in this decomposition is taken to be as long as possible. We begin by observing that  
     \begin{equation*}
     	d(q',I_{l}^{-})\geq L_{g}-9\lambda_{0}L 
     \end{equation*}
     By part \ref{projection folded0} in \cref{backtracking copy} we have the bound 
     \begin{equation}\label{bound first part}
     	d(q',\sigma_{0,0}^{+})\leq 30\lambda_{0}L.
     \end{equation}

     \begin{lemma}\label{bound second part}
     	The inequality $\diam(\sigma_{l,1})\leq 6\lambda_{0}L$ holds. 
     \end{lemma}
     \begin{subsubproof}
     	 We may assume that $\sigma_{l,1}\neq\emptyset$, in which case it is a simple calculation that both $(I_{0},\tau_{0})$ and $(I_{1},\tau_{1})$ must be folded.
     	 By \cref{sticking out obs}, this implies the existence of three grossly misaligned points in some translate of $\G\cdot p$. More precisely, it implies the existence of $\beta_{1},\beta_{2}\in\G$ such that for any enumeration $(r_{1},r_{2},r_{3})$ of the triple $\{p,\beta_{1}p,\beta_{2}p\}$ we have $(r_{1},r_{2})_{r_{3}}\geq \lambda_{0}L$, whereas for at least one such enumeration we should see $(r_{1},r_{2})_{r_{3}}\leq 5\nu L$, by the assumption that $\alpha$ is $(p,\nu)$-stable in the sense of \cref{d:stable elements}.
     \end{subsubproof}
     
     Combining \cref{bound second part}, \cref{sticking out obs} and the assumption that 
     $\hat{I}_{0,1}\nsimeq^{\pm}_{5\lambda_{0}L}J_{l}$ we also get the inequality:  
     \begin{equation}\label{bound third part}
     	\diam(\sigma_{l,2}\cap\nbl(\hat{I}_{0,1}))\leq 5\lambda_{0}L
     \end{equation} 
     Finally, from \eqref{bound first part}, \cref{bound second part}, and \eqref{bound third part}, we conclude the existence of some $j_{1}\in 2$ and some $r\in\sigma_{l,2}\cap\nbl(\hat{I}_{j_{1}})$ such that $d(q',r)\leq 36\lambda_{0}L+1$. 
     In combination with the bound $d(q',I_{j_{1}}^{+})\leq 47\lambda_{0}L$, this results in the crude (rounded up) bound
     \begin{equation*}
     	d(I_{j_{1}}^{+},\hat{I}_{j_{1}})\leq 100\lambda_{0}L,
     \end{equation*}
     thereby concluding the proof of \cref{some index folds}. 
     \end{subproof}
  
     Clearly, $|\ll|\geq k-|\so|\geq 2$. Without loss of generality, assume that $1\in\ff\setminus\{0\}$. 
     The fact that $\zeta_{j_{2},j_{1}},\hat{\zeta}_{j_{2},j_{1}}$ are $(1,\mu_{0}L)$-quasigeodesics and $L_{g}\geq 398\lambda_{0}L$, leads to the following observation: 
     \begin{observation}\label{j1 situation}
      	For any $j_{2}\in \hat{\ll}\setminus \{j_{1}\}$ at least one of the following holds: 
      	\begin{itemize}
      		\item $d_{\hat{I}_{j_{2}}}(\hat{I}_{j_{2}}^{-},I_{j_{2}}^{+})\geq 5\lambda_{0}L$; 
      		\item $d(I_{j_{2}}^{+},\hat{I}_{j_{2}})\leq d(I_{j_{2}},\hat{I}_{j_{2}})+5\lambda_{0}L$. 
      	\end{itemize}
      Recall that this means that the final stretch of $I_{j_{2}}$ does not find itself moving away from  $\hat{I}_{j_{2}}$ on the side of $\hat{I}_{j_{2}}^{-}$.
     \end{observation}

     We are now ready to finish the proof of \cref{no mixed matching}. The key claim is that the equality
     	$h_{j_{1}}=h_{j_{2}}$ holds, contradicting \cref{implications of the equality of hi} and our assumption that $\chi\notin\G$, and thus it suffices to conclude the proof of \cref{no mixed matching}. 
     
     To show the claim note that by \ref{some index folds} and \ref{j1 situation} either part \ref{second part} or part \ref{third part} of \cref{backtracking copy} applies with $j=j_{2}$.
     There are thus two cases.
      
     If $(I_{j_{1}},\tau_{j_{1}})$ is folded, then  
     \begin{equation*}
       	h_{j_{1}}=\eta\alpha_{l}^{\epsilon}\eta'=h_{j_{2}},
      \end{equation*}
      for some $\eta'\in H$. 
     	  
      If $(I_{j_{1}},\tau_{j_{1}})$ is not folded, then 
      \begin{equation*}
     	  	h_{j_{1}}=h_{j_{2}}\in H.
      \end{equation*}
      This concludes the proof of \cref{no mixed matching}.
  \end{proof}

       \begin{corollary}\label{forced matching}
       	  In the case $H_{1}=H'$ of \cref{assumption quadrilaterals}, if for any $i,j\in k$ and $J\in\mathcal{A}_{\alpha}(\tau_{j})$ 
          we have $I_{i,i'}\nsimeq^{\pm}_{5\lambda_{0}}J$ for any $i,i',j\in k$ and $J\in\mathcal{A}_{\alpha}(\tau_{j})$, then there are subarcs $\tilde{\tau}_{j}\subseteq\tau_{j}$, $j\in k$ such that:
          \begin{itemize}
       	    \item $J\sqsubseteq_{900\lambda_{0}}\tilde{\tau}_{j}$ for all $j\in k$,
       	    \item $\tilde{\tau}_{i}$ and $\tilde{\tau}_{j}$ $\lambda_{0}L$-fellow travel for every pair of distinct $i,j\in k$.
           \end{itemize}
          In particular, \cref{p:intersection_many3} applies. 
       \end{corollary}

     \subsubsection{Controlling the failure of $\alpha$-subarcs in $\tau_{i}$ to match with arcs in $\tau_{j}$  }
 
  Recall assumption \ref{if equal gepsilon} from \cref{assumption quadrilaterals}.
  \begin{lemma}\label{middle short}
      Assume that one of the following situations holds: 
	  	\begin{enumerate}[label=(\roman*)]
	  		\item \label{case H} $H_{1}=H$;
	  		\item \label{case H'} $H_{1}=H'$,  $g=g_{0,-\epsilon_{0}}$ for some $\epsilon_{0}\in\{\pm1\}$ and neither the normal form of $h_{i}$ nor that of $h_{j}$ begins with the word $\alpha_{0}^{\epsilon_{0}}$ (or more formally $(1_{H},\alpha_{0}^{\epsilon_{0}})$).
	  	\end{enumerate}
  	Then for $i,j\in\ll$ we have    
  	\begin{equation*}
  	|I_{i,j}|=d(b_{i}^{-1}gp,b_{j}^{-1}gp)\leq 200\lambda_{0}L.
  	\end{equation*}
  \end{lemma}
  The same holds after exchanging the role of $(I_{\bigcdot},\tau_{\bigcdot},g)$ and $(\hat{I}_{\bigcdot},\tau_{\bigcdot}^{-1},\hat{g})$. 
  \begin{proof}
  	Without loss of generality we can take $(i,j)=(0,1)$ and assume, for the sake of contradiction, that $|I_{0,1}|>200\lambda_{0}L$.
  	Combining our assumptions in either case with \ref{o4} and \ref{o5} of \cref{properties of Ij} we know that:
  	\begin{enumerate}[label=(\alph*)]
  		\item \label{noneq tau}if $J$ is either the first $\alpha$-subarc or the first $H$-subarc of $\tau_{0}$ or $\tau_{1}$, then  $I_{0,1}\nsimeq^{\pm}_{5\lambda_{0}L}J$;
  		\item \label{noneq hat} $I_{0,1}\nsimeq^{\pm}_{5\lambda_{0}L}\hat{I}_{0,1}$. 
  	\end{enumerate}
  	
  	Using \cref{quadrilaterals general} we find a decomposition
  	\begin{equation*}
  		I_{0,1}=\xi_{0}\ast\xi\ast\xi_{1},
  	\end{equation*}
  	where $\xi_{l}\subseteq\nb_{\lambda_{0}L}(\tau_{l})$, and $\xi\subseteq\nbl(\hat{\zeta}_{0,1})$. 
  	It follows from \ref{noneq tau} that for $l\in 2$ we have
  	\begin{equation}\label{hI does not get in}
  		\diam(\xi_{l})\leq\diam(I_{0,1}\cap\nbl(\tau_{l}))\leq 11\lambda_{0}L. 
  	\end{equation} 
   	\newcommand{\arc}[0]{\varphi}
  	Inequality \eqref{hI does not get in} implies that $|\xi|\geq 178\lambda_{0}L$. 
  	By \ref{noneq hat} above, we know that $\diam(\xi\cap \nbl(\hat{I}_{0,1}))\leq 5\lambda_{0}L$, whereas by \cref{no-mirror} and \cref{quadrilaterals general} we know that  $\diam(\xi\cap \nbl(\hat{I}_{l}))\leq 8\lambda_{0}L$ for some $l\in 2$. Without loss of generality, assume that this happens for $l=0$. 
	  Using \cref{quadrilaterals general} one easily concludes the following.
	  \begin{observation}
		  \label{large arc parallel to Iij} There is some subarc $\sigma\subseteq\hat{I}_{1}$ such that
		  \begin{itemize}
		  	\item $\diam(\partial \sigma)\geq 165\lambda_{0}L$; 
		  	\item $\lambda_{0}L$-fellow travels with a subarc of $\tilde{\xi}\subseteq \xi\subseteq I_{1,0}=I_{0,1}^{-1}$;
		  	\item $d(\sigma^{\pm},\xi^{\pm})\leq\lambda_{0}L$;
		  	\item $d(\tilde{\xi}^{-},I_{1}^{+})=d(\tilde{\xi}^{-},I_{1,0}^{-})\leq 12\lambda_{0}L$.
		  \end{itemize}
		  In particular, $L_{\hat{g}}\geq 165\lambda_{0}L$. 
	  \end{observation}
    Write 
    \begin{equation*}
    		q_{0}=\tilde{\xi}^{-},\quad\quad\hat{q}_{0}=\sigma^{-},\quad\quad q_{1}=\tilde{\xi}^{+},\quad\quad \hat{q}_{1}=\sigma^{+}.
    \end{equation*}
    The above observation has the following important consequence.     	
			  \begin{lemma}\label{hj is short}
			  	The inquality 
		  		$L_{\hat{g}}\geq L_{g}+150\lambda_{0}L$ holds.  
			  \end{lemma}
			  \begin{subproof}
			  	We have the following:
			  	\begin{equation*}
			  		d(\hat{I}_{1}^{-},I_{1}^{+})\leq d(\hat{I}_{1}^{-},\hat{q}_{0})+\lambda_{0}L+d(q_{0},I_{1}^{+}) 
			  		\leq d(\hat{I}_{1}^{-},\hat{q}_{0})+13\lambda_{0}L
			  	\end{equation*}
			  	We also know from \cref{breakpoints} that 
			  	\begin{equation*}
			  		L_{g}-\mu_{0}L\leq d(\hat{I}_{1}^{-},I_{1}^{+}).
			  	\end{equation*}
			  Combining the two inequalities with we obtain: 
			  \begin{equation*}
			    L_{g}\leq d(q_{0},I_{1}^{+})+14\lambda_{0}L\leq \left( L_{\hat{g}}-165\lambda_{0}L \right)+14\lambda_{0}L<L_{\hat{g}}-150\lambda_{0}L,
			  \end{equation*}
			  as needed. 
			  \end{subproof}

			  We also observe the following.
			  \begin{lemma}\label{h1 is in H}
			  	We have $h_{1}\in H$.
			  \end{lemma}
			  \begin{subproof}
			  	Note that the statement is only non-trivial in case \ref{case H'} of \cref{middle short}, in which we have 
			  	$g=g_{0,-\epsilon_{0}}$ for some $\epsilon_{0}\in\{\pm1\}$, which implies that $L_{g}\leq\nu L$. 
			  	Assume for the sake of contradiction that $h_{1}\nin H$. Since $\tau_{1}$ is a $\qc$-quasigeodesic, it follows that $d(\tau_{1}^{-},\tau_{1}^{+})\geq\frac{9}{10}L$. 
			  	
			  	Since $\diam(\xi_{1})\leq 11\lambda_{0}L$, and $L_{g}\leq\nu L$, it easily follows from this and \cref{quadrilaterals general} (this is a coarse bound) that 
			  	$d(\hat{I}_{1}^{-},q_{0})\geq d(\tau_{1}^{+},q_{0})\geq (\frac{9}{10}-13\lambda_{0})L$. 
			  	By the triangular inequality,
			  	\begin{align*}
			  		d(I_{1}^{-},\hat{I}_{1}^{+})\leq &d(I_{1}^{-},q_{0})+d(q_{0},\hat{q}_{0})+d(\hat{q}_{0},\hat{I}_{1}^{+})\leq\\
			  		&\nu L+\lambda_{0}L+d(\hat{q}_{0},\hat{I}_{1}^{+})\leq 2\lambda_{0}L+\left(L_{\hat{g}}-165\lambda_{0}L \right)<L_{\hat{g}}-163\lambda_{0}L,
			  	\end{align*}
			  	contradicting \cref{breakpoints}. The result follows.
			  \end{subproof}

    Take now any $l_{0}\in \hat{\ll}\setminus\{1\}$.  Write $\mathcal{A}=\G\cdot I_{0}^{+}$ so that in particular for all $l\in k$ we have $I_{l}^{+}\in\mathcal{A}$. From \cref{no-mirror}  and \cref{quadrilaterals general} one easily deduces the following.
    \begin{observation}\label{measuring distance to A}
    	The following holds:
    	\begin{enumerate}[label=(\roman*)]
    	\item \label{away from axis}$\pi_{\hat{I}_{l_{0}}}(\mathcal{A})\subseteq\nb_{8\lambda_{0}L}(\hat{I}_{l_{0}}^{+})$; 
      \item \label{distance estimate} for any $q\in X$ such that $d_{\hat{I}_{l_{0}}}(q,\hat{I}_{_{l_{0}}}^{+})>10\lambda_{0}L$ we have 
       \begin{equation*}
       	d(q,\mathcal{A})\geq d(q,\hat{I}_{l_{0}}^{+})-15\lambda_{0}L.
       \end{equation*}
    \end{enumerate}
    \end{observation}
     
    Note that the subarc $\hat{\sigma}=\hat{\beta}_{1,l_{0}}\cdot \sigma\subseteq\hat{I}_{l_{0}}$ is $\lambda_{0}L$-fellow traveling with $\hat{\xi}=\hat{\beta}_{1,l_{0}}\cdot \tilde{\xi}$. 
    
    \begin{lemma}\label{translate also stays away}
    	We have $\hat{\sigma}\cap\nbl(\tau_{l_{0}})\subseteq\nb_{15\lambda_{0}L}(\hat{\sigma}^{-})$. 
    \end{lemma}
    \begin{subproof}
    	Assume that the inclusion fails. We show that this leads to a contradiction. First, note that under this assumption we must have $\hat{\xi}\simeq_{11\lambda_{0}L}\tau_{l_{0}}$ and thus 
    	$\tilde{\xi}\simeq_{11\lambda_{0}L}\hat{\beta}_{l_{0},1}\cdot\tau_{l_{0}}$. It follows that 
    	$J\simeq_{5\lambda_{0}L}I_{1,0}$, where $J$ is either the first $H$-subarc of $\hat{\beta}_{l_{0},1}\cdot\tau_{l_{0}}$ or the first $\alpha$-subarc of $\hat{\beta}_{l_{0},1}\cdot \tau_{l_{0}}$. 
    	
    	Note by \cref{h1 is in H} we have $h_{1}\in H$ regardless of which case in the statement of \cref{middle short} we find ourselves in. In particular, if $J$ is the first $H$-subarc of $\hat{\beta}_{l_{0},1}\cdot\tau_{l_{0}}$, then a translating element $\kappa$ of $I_{1,0}$ is obtained by multiplying on the right the translating  element of $\hat{\beta}_{l_{0},1}\cdot J$ by $h_{1}^{-1}$. It follows from \ref{consequence translating elements} that $\hat{g}\in H\cdot\so$.
    	
    	If $J$ is an $\alpha$-subarc, then for some $\eta\in H$ and $\epsilon\in\{\pm1\}$ the translating element of $J$ is of the form $\bar{\beta}\hat{g}^{-1}\eta g_{0,\epsilon}$, for $\kappa$ as above. However, this is impossible, since under the assumption \ref{if equal gepsilon} of \cref{assumption quadrilaterals} this would imply that $\hat{g}=g_{0,\epsilon}$ and thus $L_{\hat{g}}\leq\nu L$. 
    \end{subproof}
    
    We are now ready to conclude the proof of \cref{middle short}. It follows from \cref{translate also stays away}
    that $\pi_{\hat{I}_{l_{0}}}(\tau_{l_{0}}^{+})$ is far enough from 
    $\hat{I}_{l_{0}}^{+}$ for \cref{measuring distance to A} to apply with $q=\tau_{l_{0}}^{+}=I_{l_{0}}^{-}$. 
    We conclude that 
    \begin{equation*}
    	L_{g}=d(I_{l_{0}}^{-},I_{l_{0}}^{+})\geq d(I_{l_{0}}^{-},\mathcal{A})\geq d(I_{l_{0}}^{-},\hat{I}_{l_{0}}^{+})-15\lambda_{0}L.
    \end{equation*}
    On the other hand, by \cref{breakpoints} we must have $ d(I_{l_{0}}^{-},\hat{I}_{l_{0}}^{+})\geq L_{\hat{g}}-\lambda_{0}L$. Combining the two inequalities above, we obtain a contradiction with \cref{hj is short}.  
\end{proof}

    \begin{lemma}\label{long shift between long and short}
   		Under	\cref{assumption quadrilaterals}, assume that $H_{1}=H'$, the extension is linear case and that the following holds for $i,j\in \ll$: 
   		\begin{itemize}
   			\item $g=g_{0,-\epsilon_{0}}$ for some $\epsilon_{0}\in\{\pm 1\}$, 
   			\item $\alpha_{0}^{\epsilon_{0}}$ is an initial subword of $h_{i}$,
   			\item $\alpha_{0}^{\epsilon_{0}}$ is not an initial subword of $h_{j}$.
   		\end{itemize}
   		Then
  	    \begin{equation*}
  	  	d(b_{i}^{-1}g\alpha^{\epsilon_{0}}p,b_{j}^{-1}gp)\leq 200\lambda_{0}L.
  	  \end{equation*} 
    The same holds after exchanging the role of $(I_{\bigcdot},\tau_{\bigcdot})$ and $(\hat{I}_{\bigcdot},\tau_{\bigcdot}^{-1})$. 
  \end{lemma}
  \begin{proof}
    Note that in this case $L_{g}\leq\nu L$.  Assume for the sake of contradiction that the first inequality does not hold. Without loss of generality, we may assume $(i,j)=(0,1)$. 
    
    Write $J_{0}$ for the first $\alpha$-subarc of $\tau_{0}$. There is a geodesic arc 
    \begin{equation}\label{def xi}
    	\xi=b_{0}^{-1}g_{0,-\epsilon_{0}}\alpha^{\epsilon_{0}}[p,\rho p],
    \end{equation}
     where $\rho\in\G$, between the endpoints 
     \begin{equation*}
     \xi^{-}=J_{0}^{+}=b_{0}^{-1}g_{0,\epsilon_{0}}\alpha^{\epsilon_{0}}p,\quad\quad \xi^{+}=I_{1}^{+}=b_{1}^{-1}gp.	
     \end{equation*}
     
     Denote by $\sigma_{0}$ the subarc of $\tau_{0}$ consisting of all those points appearing after $J_{0}$.

   	\begin{lemma}\label{short sides}
   		The following inclusions hold:
   		\begin{equation*}
   		   \xi\cap\nbl(\tau_{1})\subseteq\nb_{(10\lambda_{0}+2\nu)L}(\xi^{+}),\quad\quad 		\xi\cap\nbl(\sigma_{0})\subseteq\nb_{(10\lambda_{0}+2\nu)L}(\xi^{-}).
   		\end{equation*}
   	\end{lemma}
   	\begin{subproof}
   		  The translating element of the first $H$-subarc $K_{0}$ of $\tau_{1}$ is of the form $b_{1}^{-1}$, whereas the translating element of the first $\alpha$-subarc $K_{1}$ of $\tau_{1}$ (if it exists) is of the form $b_{1}^{-1}\eta g_{0,\epsilon}$ for some $\epsilon\in\{\pm1\}$. 
   	  	On the other hand, $\xi^{-1}$ is of the form $b_{1}^{-1}[p,\rho'p]$ for some $\rho'\in\G$. 
   	  	In view of this, we have: 
   	  	\begin{itemize}
   			\item If $\xi\simeq^{\pm}_{5\lambda_{0}}K_{0}$, then  $g_{0,\epsilon_{0}}\in H\cdot \so$, which contradicting our choice of $g_{0,\epsilon_{0}}$. 
   			\item If $\xi\simeq^{\pm}_{5\lambda_{0}}K_{1}$, then  $g_{0,\epsilon_{0}}\in \eta\cdot g_{0,\epsilon}\G$, where $\eta$ is the label of $K_{0}$, 
   			which may only take place if $\epsilon=\epsilon_{0}$ and $\eta=1$, contrary to the assumption on $h_{1}=h_{j}$ in the statement.
   		 \end{itemize}
   			This establishes the first inclusion.  
   			
   		  The second $H$-subarc $K'_{0}$ of $\tau_{0}$ has translating element equal to 
   		  $b_{0}^{-1}g_{0,-\epsilon_{0}}\alpha^{\epsilon_{0}} g_{0,\epsilon_{0}}^{-1}$, whereas the second $\alpha$-subarc $K'_{1}$ of $\tau_{0}$ (if it exists) has translating element of the form $b_{0}^{-1}g_{0,-\epsilon_{0}}\alpha^{\epsilon_{0}} g_{0,\epsilon_{0}}^{-1}\eta g_{0,-\epsilon}$ where $\epsilon$ is the exponent associated with $K'_{1}$. 
   		  
   		  In view of this: 
   		  \begin{itemize}
   		  	\item if $\xi\simeq^{\pm}_{5\lambda_{0}}K'_{0}$, then $g_{0,\epsilon_{0}}\in H\cdot\so$, contrary to the choice of $g_{0,\epsilon_{0}}$;
   		  	\item if $\xi\simeq^{\pm}_{5\lambda_{0}}K'_{1}$, then  $g_{0,\epsilon_{0}}^{-1}\eta g_{0,-\epsilon}\in\G$, which can only hold if $\epsilon=-\epsilon_{0}$ and $\eta=1$, contrary to the fact that $\tau_{0}$ is built from a reduced word. 
   		  \end{itemize}
   		  The conclusion of \cref{short sides} follows easily from the discussion above. 
   	\end{subproof}
   
    	If we denote by $\xi'$ the subarc of $\xi$ consisting of points at distance at least $11\lambda_{0}L$ from $\partial \xi$, then it follows from \cref{quadrilaterals general} and \cref{short sides} that $\xi'\subseteq\nbl(\hat{\zeta}_{0,1})$. 
    	Note that by part \ref{o4} of \cref{properties of Ij} we must have
    	\begin{itemize}
    		\item  $\diam(\xi'\cap\nbl(\hat{I}_{0,1}))\leq 5\lambda_{0}L$,
    		\item $\diam(\xi'\cap\nbl(\hat{I}_{l}))\leq 7\lambda_{0}L$ for at least one value of $l\in 2$.
    	\end{itemize}
    	Without loss of generality, we will assume $l=1$, since although the hypothesis of the statement is not entirely symmetrical, it does behave symmetrically with regards to the argument below. 
    	
    	Then there are points in 
    	$\xi\cap\nbl(\hat{I}_{0})$ at distance at most $12\lambda_{0}L$ from $\xi^{-}$, from which it easily follows that 
    	\begin{equation*}
    		\diam(\xi\cap\nbl(\hat{I}_{0}))\geq 166\lambda_{0}L
    	\end{equation*}
    	and $L_{\hat{g}}\geq 160\lambda_{0}L$. It also follows from \cref{first breakpoint close to origin} and the fact that remainder subarcs have length at most $\nu L$  that 
    	\begin{equation}\label{tail short}
    		d(\sigma_{0}^{-},\sigma_{0}^{+})\leq 4\lambda_{0}L,\quad\quad h_{0}\in \alpha_{0}^{\epsilon_{0}}H.
    	\end{equation}
    	
    	Fix now $l_{0}\in \hat{\ll}\setminus\{0\}$ and write $\mathcal{A}=\G\cdot I_{0}^{+}$, so that 
    	$I_{l}^{+}\in\mathcal{A}$ for all $l\in k$. It follows from \cref{no-mirror} that 
    	\begin{equation}\label{away from main axis2}
        d(\hat{I}_{l_{0}},\mathcal{A})\geq L_{\hat{g}}-10\lambda_{0}L \geq 150\lambda_{0}L
    	\end{equation}

    \begin{lemma}\label{first alpha subarc2}
   	  The element $h_{l_{0}}$ does not belong to $H$ . Writing $J$ for the first $\alpha$-subarc of $\tau_{l_{0}}^{-1}$ we have 
   	  \begin{itemize}
   	  	\item $\hat{I}_{l_{0}}\simeq_{5\lambda_{0}}J$, 
   	  	\item $d(J^{-},\tau_{l_{0}}^{+})\leq 3\lambda_{0}L$.
   	  \end{itemize}
    \end{lemma}
    \begin{subproof}
    	It follows from \ref{away from main axis2} that the subarc $\tilde{I}_{l_{0}}\subseteq\hat{I}_{l_{0}}$ consisting of all points at distance at most $145\lambda_{0}L$ from $\hat{I}_{l_{0}}^{-}=\tau_{l_{0}}^{-}$ is $\lambda_{0}L$-fellow traveling with some subarc $\tau^{*}$ of $\tau_{l_{0}}^{-1}$. 	
   	
   	\newcommand{\arc}[0]{\varphi}
   	 One easily checks using \eqref{tail short} that $\xi^{-1}\simeq_{5\lambda_{0}L}\hat{\beta}_{l_{0},0}\cdot J$, where $J$ is either the first $H$-subarc of $\tau_{l_{0}}^{-1}$ or the first $\alpha$-subarc of $\tau_{l_{0}}^{-1}$. 
   	 Note that if $J$ is the first $H$-subarc of $\tau_{l_{1}}^{-1}$, then the translating element of 
   	 $\hat{\beta}_{l_{0},0}$ is $b_{0}^{-1}h_{0}$. Using \cref{consequence translating elements}, and comparing the former with the expression \eqref{def xi} we obtain the existence of $s\in\so$, $\rho'\in\G$ such that
   	 $b_{0}^{-1}h_{0}Hs=b_{0}^{-1}g_{0,-\epsilon_{0}}\alpha\rho$. Note that  $h_{0}^{-1}g_{0,-\epsilon_{0}}\alpha\in Hg_{0,\epsilon_{0}}$ by \eqref{tail short}, hence we would deduce
   	 \begin{equation*}
   	 	g_{0,\epsilon_{0}}\in H\cdot\so,
   	 \end{equation*}
   	 contrary to our most basic assumption on $g_{0,\epsilon_{0}}$.
   	 
   	 Thus $J$ must be the first $\alpha$-subarc of $\tau_{l_{0}}^{-1}$. The inequality $d(J^{-},\hat{I}_{l_{0}}^{-})\leq 3\lambda_{0}L$ follows from \cref{first breakpoint close to origin}.
    \end{subproof}
    We may apply part \cref{first folding claim} of \cref{backtracking copy} to the folded pair $(\hat{I}_{l_{0}},\tau_{l_{0}}^{-1})$. In particular, it follows from \ref{projection folded} that
    \begin{equation}\label{J separates2}
    	d(J^{+},\hat{\zeta}_{l_{0},0})>\sfrac{1}{3}L.
    \end{equation}
    Apply now \cref{quadrilaterals general} to the $\qc$-quasigeodesic quadrilateral with sides $\zeta_{l_{0},0}$, $\hat{\zeta}_{l_{0},0}$, $\tau_{l_{0}}$, $\tau_{0}$. It follows from the inequality above, together with the fact that $L_{g}\leq\nu L$ and $\diam(\sigma_{0})\leq 4\lambda_{0}L$ (from \eqref{tail short}), that  $J\simeq_{\sfrac{1}{4}L}I_{0,l_{0}}$. But then $\partial J\subseteq\mathcal{A}$, by \ref{consequence translating elements} and thus $d(\mathcal{A},\tau_{l_{0}}^{+})\leq 3\lambda_{0}L$, in contradiction with inequality \eqref{away from main axis2}. This contradiction concludes the proof of \cref{long shift between long and short}. 
  \end{proof}
   
   The following is the main takeaway of \cref{middle short} and \cref{long shift between long and short}.
   
      \begin{corollary}\label{two clusters}
       Assume that the extension falls into the linear case and $g=g_{0,-\epsilon_{0}}$ for some $\epsilon_{0}\in\pm1$. Write $\ll_{0}$ for the collection of indices in $\ll$ for which the normal form of $h_{j}$ starts with $\alpha_{0}^{\epsilon_{0}}$, and $\ll_{1}=\ll\setminus\ll_{0}$, and assume that $\ll_{1}\neq\emptyset$. 
       For any $l\in\ll$ write $\rho_{l}=b_{l}^{-1}g$. 
       Then the following holds:
        \begin{enumerate}[label=(\roman*)]
        \item \label{short distance same} For all $i,j\in\ll_{1}$ we have $d(\rho_{i}p,\rho_{j}p)\leq 200\lambda_{0}L$.
       \item \label{short distance mixed} For all $i\in\ll_{0}$, $j\in\ll_{1}$ we have
       \begin{equation*}\label{inequality mixed}
        	d(\rho_{i} \alpha^{\epsilon_{0}}p,\rho_{j}p)\leq 200\lambda_{0}L,\quad\quad d(\rho_{i}p,\rho_{j}p)\geq(1-200\lambda_{0})L.
       \end{equation*} 
        \item \label{short distance same2} For all $i,j\in\ll_{0}$ we have $d(\rho_{i}p,\rho_{j}p)\leq 400\lambda_{0}L$.
        	\item \label{uniqueness partition} There is no other partition $\ll=\ll'_{0}\coprod\ll'_{1}$, $\ll'_{1}\neq\emptyset$ for which \ref{short distance same}, \ref{short distance mixed}, \ref{short distance same2} are satisfied.
        \end{enumerate}
       The same holds with $(I_{\bigcdot},\hat{I}_{\bigcdot},\tau_{\bigcdot}^{-1},g,\hat{g})$ in place of $(\hat{I}_{\bigcdot},I_{\bigcdot},\tau_{\bigcdot},\hat{g},g)$. 
      \end{corollary}
      \begin{proof}
      	  Part \ref{short distance same} follows from \cref{middle short}, whereas part \ref{short distance mixed} follows from \cref{long shift between long and short}, the second inequality being a consequence of the first, the fact that $d(\rho_{i}\alpha^{\epsilon_{0}}p,\rho_{i}p)=L$ and the triangular inequality.
        	Part \ref{short distance same2} follows from \ref{short distance mixed}, since we may pick $l\in\ll_{1}$ and then the fact that $\alpha$ commutes with $\rho_{i}$ and $\rho_{j}$, 
        	together with the triangular inequality yields
        	\begin{equation*}
        		d(\rho_{i}p,\rho_{j}p)=d(\alpha^{\epsilon_{0}}\rho_{i}p,\alpha^{\epsilon_{0}}\rho_{j}p)\leq
        		d(\rho_{i}\alpha^{\epsilon_{0}}p,\rho_{l}p)+d(\rho_{j}\alpha^{\epsilon_{0}}p,\rho_{l}p)\leq 400\lambda_{0}L.        		 
        	\end{equation*}

          As for \ref{uniqueness partition}, note that the second inequality in \ref{short distance mixed}, as well as the inequalities \ref{short distance same} and \ref{short distance same2} easily imply that $\{\ll_{0},\ll_{1}\}=\{\ll'_{0},\ll'_{1}\}$. If $\ll_{0}=\emptyset$, then we are done. 
          If not, then observe that we may not have $i\in\ll_{0}\cap\ll'_{1}$, $j\in\ll_{1}\cap\ll'_{0}$ for any $i,j\in\ll$, since then the relations $[\rho_{i},\alpha]=[\rho_{j},\alpha]=1$ we get: 
        	 \begin{align*}
        	 	d(p,\alpha^{2\epsilon_{0}}p)=&d(\rho_{i}\alpha^{\epsilon_{0}}p,\rho_{i}\alpha^{-\epsilon_{0}}p)\leq d(\rho_{i}\alpha^{\epsilon_{0}}p,\rho_{j}p)+d(\rho_{i}\alpha^{-\epsilon_{0}}p,\rho_{j}p)=\\
        	 	&d(\rho_{i}\alpha^{\epsilon_{0}}p,\rho_{j}p)+d(\rho_{j}p,\rho_{i}\alpha^{\epsilon_{0}}p)\leq 400\lambda_{0}L,
        	 \end{align*}
        	 in contradiction with \ref{concatenation quasigeodesic} of \cref{d:stable elements}
        	  So in fact $\ll_{0}=\ll'_{0}$ and $\ll_{1}=\ll'_{1}$, as needed.
         \end{proof}

	  \subsection{Preservation of supportedness on $\G$}
	  	  
	  	  It is now time to give an interpretation in terms of the graph $\mathtt{G}(H_{1)}$, $H_{1}\in\{H,H'\}$ to the geometric information gathered in the previous subsection.
	  	  
	  	  \begin{definition}
	  	  	In a graph $\mathtt{G}$ we write $\Lk(v)$ for the collection of vertices connected to $v$ by an edge and $\Cc(v)$ for the connected component containing $v$. 
	  	  \end{definition}
	  	  
	  	  \begin{definition}\label{d:special vertex}
        	We will say that a vertex $\T_{0}\in\mathtt{V}_{0}(H)$ is \emph{special} if there is $\epsilon\in\{\pm1\}$ such that some (equivalently, any) $B\in \T_{0}\cap[\mathcal{Z}(g_{0,\epsilon})]^{k}$ such that 
        	any ordering $\underline{b}\in G^{k}$ of $B$ satisfies the conditions of \cref{two clusters} with $\ll_{0},\ll_{1} \neq\emptyset$. We say that a connected component of the graph is \emph{special} if it contains some special vertex. Otherwise, we say it is \emph{regular}. 
        	
        	By the \emph{gateway to $\T_{0}$}, denoted by $\Gate(\T_{0})$ we will mean the collection of all vertices (necessarily in $\Lk(\T_{0})$) of the form $\O_{H}(C)$, where $C\in \T_{0}\cap\mathcal{Z}(g)$ for some 
        	$g\in H'g_{0,-1}\G=H'\cdot \Theta \cdot\G$. 
        \end{definition}
	  	  
	  	  The following is an immediate consequence of \cref{multifold}.
	  	   \begin{observation}\label{gateway folds}
	  	   	For any special vertex  $\T_{0}$ and $O_{0},O_{1}\in\Gate(\T_{0})$ connected to $\T_{0}$ by edges $e_{0}$ and $e_{1}$, respectively, we must have $\Lambda(e_{0})=\Lambda(e_{1})$ and $\Lambda(O_{0})=\Lambda(O_{1})$.  
	  	   \end{observation}
	  	   
	  	   The following lemma is the main motivation behind \cref{middle short} in the previous section. 
	       \begin{lemma}\label{l:special vertices}
	       	Under the assumptions on the extension made in \cref{assumption quadrilaterals}, 
        	the following properties hold for any special $\T_{0}\in\mathtt{V}_{0}(H)$. 
          \begin{enumerate}[label=(\roman*)] 
       	 	  \item \label{isolated vertex} We have	$\{\T_{0}\}=\Cc(\T_{0})\cap\mathtt{V}_{0}(H)$. 
        	  \item \label{only one attaching tentacle}  Any $O\in\Lk(\T_{0})$ such that $\Lambda^{-1}(\Lambda(O))\nsubseteq\Lk(\T_{0})$ must belong to $\Gate(\T_{0})$. 
        	  \item \label{indentifications fold} Let $O_{0},O_{1}\in \Lk(\T_{0})$ be connected to $\T_{0}$ by edges $e_{0}$ and $e_{1}$, respectively. If $\Lambda(O_{0})=\Lambda(O_{1})$, then $\Lambda(e_{0})=\Lambda(e_{1})$.
          \end{enumerate}
          Moreover, the following holds. 
          \begin{enumerate}[label=(\roman*)]\setcounter{enumi}{3}
          	\item \label{non-special do not merge} If $\T_{1},\T_{2}\in\mathtt{V}_{0}(H)$ are not special and $\Cc(T_{1})\neq\Cc(T_{2})$, then 
          \end{enumerate}
        \end{lemma}
        \begin{proof}
        	 \emph{Part \ref{isolated vertex}.} Let $T_{0}=\T_{H}(g_{0,\epsilon})$ be special and assume for the sake of contradiction that there is $\hat{T}\in\mathtt{V}_{0}(H)\setminus \{T_{0}\}$ which is connected to $T_{0}$ by a path of length $2$ in $\mathtt{G}(H)$. This means that there are  $B\in[\mathcal{Z}(g_{0,\epsilon})]$, $\hat{B}\in[\mathcal{Z}(\hat{g})]$ and $\hat{g}\in G\setminus(H\cdot\so)$ such that 
        	 \begin{itemize}
        	 	\item $B\in T_{0}$, $\hat{B}\in\hat{T}$  
        	 	\item $C=H\backslash B$ and $\hat{C}=H\backslash \hat{B}$ are in the same orbit under the action $[H\backslash G]^{k}\noa G$. 
        	 \end{itemize}
           The second point above implies the existence of permutations $\underline{b}$ and $\bh$ of $B$ and $\hat{B}$ respectively, as well as $\underline{\eta}\in\Conn_{H}(\underline{b},\bh)$.
           Since $T_{0}\neq \hat{T}$, we may assume that the constant $b_{j}^{-1}\eta_{j}\hat{b}_{j}$ is not in $\G$ for all $j\in k$. We may also assume that $g,\hat{g}$ are $H$-short, since this can be done without altering $\O_{H}(C)$ and $\O_{H}(\hat{C})$, so that \cref{assumption quadrilaterals} applies with $H_{1}=H$. Then part \ref{case H} of \cref{middle short} leads to a contradiction with the existence, by \ref{short distance mixed} of \cref{two clusters}, of $i,j\in k$ for which the inequality 
           \begin{equation}\label{long distance}
           	d(b_{i}^{-1}gp,b_{j}^{-1}gp)\geq(1-200\lambda_{0})L
           \end{equation}
           holds.
           
           \emph{Part \ref{only one attaching tentacle}.} Let say we have $O=\O_{H}(B)$ for some $B\in T_{0}\cap[\mathcal{Z}(g)]^{k}$ for $g\in G\setminus(H'\cdot\so)$ such that
           $\O_{H'}(B)=\O_{H'}(\hat{B})$ for some $\hat{B}\in[\mathcal{Z}(\hat{g})]^{k}$ and
           $\hat{g}\in G\setminus(H'\cdot\so)$ with $T\neq \T_{H}(\hat{B})$.     
           Then, we can find $\underline{b}$, $\bh$ just as was done for \ref{isolated vertex} with $H_{1}=H'$. \cref{multifold} allows us to assume that $g$ and $\hat{g}$ are $H'$-short, and there is a witness $\underline{h}\in\Conn_{H'}(\underline{b},\bh)$, which we know cannot be in $H^{k}$, since $\O_{H}(B)\neq \O_{H}(\hat{B})$, as a consequence of part \ref{isolated vertex}. Part \ref{case H'} of \cref{middle short}, as well as the inequality \eqref{long distance} leads to a contradiction unless $g\in H\{g_{0,-1},g_{0,1}\}\G$, which int turn implies that $O\in\Gate(T_{0})$.

           \emph{Part \ref{indentifications fold}.} Let us say that $O_{l}=\O_{H}(B_{l})$ for some $l\in 2$ and $B_{l}\in T_{0}\cap[\mathcal{Z}(g_{l})]^{k}$. Then $\Lambda(B_{0})=\Lambda(B_{1})$, as witnessed by $\underline{h}\in\Conn_{H'}(\underline{b}^{0},\underline{b}^{1})$ for enumerations $\underline{b}^{l}$ of $B_{l}$. \cref{multifold} allows us to assume that $g_{l}$ is $H$-short in the sense of \cref{d:short} for $l\in 2$. If the constant $(b^{0}_{j})^{-1}h_{j}b^{1}_{j}$ is in $\G$, then it follows from part \ref{another direction} of \cref{implications of the equality of hi} and $\cref{multifold}$ that 
           $\Lambda(e_{0})=\Ed_{H'}(B_{0})=\Ed_{H'}(B_{1})=\Lambda(e_{1})$. Otherwise, it follows from inequality \eqref{long distance} and part \ref{case H'} of \cref{middle short} that for every $l\in 2$ we must have $O_{l}\in\Gate(T_{0})$ and thus 
           $\Lambda(e_{0})=\Lambda(e_{1})$ by \cref{gateway folds}. 
           
        \end{proof}

	   \begin{lemma}\label{preservation of supportedness}
	    	Under \cref{assumption section induction step}, and \cref{assumptions on the ksets}, suppose  that the given seed is flat. If the extension is generous, assume furthermore that $\alpha$ is $(p,\nu)$-transverse to $\G$.  
	    	
	    	Then $H'$ is also supported on $\G$.  
	    \end{lemma}
	    \begin{proof}
	    	By assumption, $\mathtt{G}(H)$ is a forest and our goal is to prove that so is $\mathtt{G}(H')$ under the given conditions. Recall the general description of the relationship between the two graphs given in \cref{graph collapse}.
	    	
	    	\emph{Linear case.} To obtain the desired conclusion, it suffices to check the following claim:
	       \begin{claim}\label{claim linear case}
	       	 In the linear case, let $O,\hat{O}\in\dom(\Lambda)\subseteq\mathtt{V}_{1}(H)$ be such that $\Lambda(O)=\Lambda(\hat{O})$. Then one of the following holds: 
	    	  \begin{itemize}
	    	 	  \item at least one of $O,\hat{O}$ belongs to $\Lk(T_{0})$ for a special $T_{0}\in\mathtt{V}_{0}(H)$:
	    		  \item there is a path in $\mathtt{G}(H)$ between $O$ and $\hat{O}$ such that every two edges along the path incident on the same vertex $T\in\mathtt{V}_{0}(H)$ are folded by $\Lambda$ into a single edge.
	    	   \end{itemize}
	       \end{claim}
        To see how the claim implies \cref{preservation of supportedness} in the case in which the extension is linear, observe that
        \begin{itemize}
        	\item the image in $\mathtt{G}(H')$ of every connected component of $\mathtt{G}(H)$ is a tree;
        	\item the union of the images under $\Lambda$ of all the regular connected components of $\mathtt{G}(H)$ is a forest;
        	\item by \cref{l:special vertices} for each special connected component $\mathtt{C}$ of $\mathtt{G}(H)$ the image $\Lambda(\mathtt{C})$ has at most one vertex in common with  $\Lambda(\mathtt{G}(H)\setminus\mathtt{C})$. 
        \end{itemize}
         
        To prove the claim assume that we are given $B\in[\mathcal{Z}(g)]^{k},\hat{B}\in[\mathcal{Z}(\hat{g})]^{k}$ for $g,\hat{g}\in G\setminus(H'\cdot\so)$ such that neither $T=\T_{H}(B)$ nor $\hat{T}=\T_{H}(\hat{B})$ are special and 
        $\O_{H'}(B)=\O_{H'}(\hat{B})$, and let $\underline{b},\bh$ be enumerations of $B$ and $\hat{B}$, respectively, such that $\Conn_{H'}(\underline{b},\bh)\neq\emptyset$.  
        
        By virtue of \cref{multifold} we may assume
        \begin{enumerate}[label=(\arabic*)]
        	\item \label{short}$g$ and $\hat{g}$ are $H'$-short in the sense of \cref{d:short},
        	\item if $g\in Hg_{0,\epsilon}\G$ for some $\epsilon\in\{\pm1\}$, then $g=g_{0,\epsilon}$ and the same property holds for $\hat{g}$;
        	\item if $g=g_{0,-\epsilon_{0}}$, then the normal form of one $h_{j}$, $j\in k$ does not begin with $\alpha_{0}^{-\epsilon_{0}}$, and the same holds for $\hat{g}$;
        	\item by taking \cref{implications of the equality of hi} into account, we may also assume that $b_{j}^{-1}h_{j}\hat{b}_{j}\notin \G$. 
        \end{enumerate}
       Recall the system of arcs defined in \cref{system of paths}. Since $T,\hat{T}$ are not special, it follows from \cref{properties of Ij} and \cref{two clusters} that $J\nsimeq^{\pm}_{5\lambda_{0}} I_{i,j}$ for any $i,j\in\ll$ and $J\in\mathcal{A}_{\alpha}(\tau_{i})$. 
       Therefore, in this case the desired conclusion thus follows from \cref{forced matching}, \cref{p:intersection_many3} and \cref{l:sequence interpretation}, which 
       yields some path in $\mathtt{G}(H)$ between $\O_{H}(B)$ and $\O_{H}(\hat{B})$, and  \cref{folding in the graph}, which implies that this path gets collapsed to a single edge by $\Lambda$ (by folding at every vertex in $\mathtt{V}_{0}(H)$).

       \emph{Generous case.} In the generous case recall the definition of the $\Omega$-sets 
       $A_{-1}$ and $A_{1}$, and let $R_{\epsilon}$ the unique vertex in $\mathtt{V}_{1}(H)$ corresponding to the orbit of $A_{\epsilon}$ for $\epsilon\in\{\pm 1\}$. Recall that it is part of the choice of $A_{-1}$ and $A_{1}$ that $\Cc(R_{-1})\neq\Cc(R_{1})$. In view of this, the proof of \cref{preservation of supportedness} in this case reduces to the following claim: 
       \begin{claim}\label{claim generous case}
        In the generous case, let $O,\hat{O}\in\dom(\Lambda)\subseteq\mathtt{V}_{1}(H)$ be such that $\Lambda(O)=\Lambda(\hat{O})$. Then there are $\epsilon,\hat{\epsilon}\in\{\pm1\}$ and paths $\gamma$  in $\mathtt{G}(H)$ between $O$ and $R_{\epsilon}$ and $\hat{\gamma}$ between $\hat{O}$ and $O_{\hat{\epsilon}}$ such that all edges along $\gamma$ (respectively, $\hat{\gamma}$) are folded by $\Lambda$ onto a single edge.
       \end{claim}
        The proof of the claim follows the same structure as that of \cref{claim linear case}. 
        We can choose $B,\hat{B}$, $g,\hat{g}$ and a similar application of \cref{multifold} implies that (up to an implicit replacement of $B$ and $\hat{B}$) we may assume that $g$ and $\hat{g}$ are $H'$-short. 
        Our assumption that $\alpha$ is $(p,\nu)$-transverse to $\G$ allows us to apply \cref{forced matching},
        which together with \cref{p:intersection_many3} and \cref{l:sequence interpretation} yields that 
        after the above reduction $\{O,\hat{O}\}\subseteq\{R_{-1},R_{1}\}$, as needed. 
	    \end{proof}

 \section{Proof of \cref{t: main} and \cref{c:hyperbolic case}}
   \label{s:proof of main}
   
  In this section we finish proving \cref{t: main}, by adapting the general strategy used in the proof of \cite[Prop. 5.2]{gonzalez2025two} with the geometric tools developed in previous sections. The induction step is handled by \cref{l:main induction step}, whose proof has little content beyond organizing results from previous sections. Right after this, we use \cref{t: main} to prove \cref{c:hyperbolic case}.
 
  \subsection{General seeds}

  \begin{lemma}
  	\label{l:main induction step}Let $\datum$ be a hyperbolic $\Theta$-seed, $H\leq G$ a promising subgroup in the sense of \cref{d:promising}, and $\Omega\in\D(\Theta)$. Let $k_{o}=k_{o}(\Omega)$, and let $\underline{g}_{-1},\underline{g}_{1}$ be $k_{o}$-tuples that are $(H,\Omega,\mathcal{S}_{\Omega})$-admissible and generate $\Omega$-sets $A_{-1}$ and $A_{1}$, respectively, satisfying \cref{assumptions on the ksets}.
    Then for any positive constant $M$ there is a promising subgroup $H'\leq G$ containing $H$ such that the following holds:
  	  \begin{enumerate}[label=(\alph*)]
  	  	\item \label{three} the quotient map $\pi:H\backslash G \onto H'\backslash G$ is injective on the set 
  	  	\begin{equation*}
  	  		\{\,Hg\,|\,d(p,gp)\leq M\,\}\cup A_{-1}\cup A_{1}\subseteq\uh;
  	  	\end{equation*}
  	  	\item \label{one}the sets $A_{-1},A_{1}$ map to sets $\bar{A}_{-1},\bar{A}_{1}\in[\uhh]^{k}$ such that $\bar{A}_{-1}\alpha=\bar{A}_{1}$ for some $\alpha\in C_{G}(\Omega)$. 
  	  \end{enumerate}
     If the seed is ample, and $H$ is geometrically $k$-separated, then we can take $H'$ to be geometrically $k$-separated.
  \end{lemma}
  \begin{proof} 
  	Fix $p\in X$.
  	We let $H'$ be a  
  	$(\Omega,\mathcal{S}_{\Omega},\underline{g}_{-1},\underline{g}_{1},\alpha)$-extension of $H$ in the sense of 
  	\cref{d:basic extension setting} which is $(p,\nu)$-geometric in the sense of \cref{d:good extension setting}.
    That such an extension exists is a consequence of \cref{existence of alpha}. 
    Property \ref{one} is just \cref{algebraic consequence of basic extension}. Our next goal is to show that if the extension is $(p,\nu)$-geometric in the sense of \cref{d:good extension setting} for $\nu$ small enough, then the rest of the conclusion is satisfied. We may take $\nu$ to be small enough as to satisfy \cref{constant assumptions}. In particular, \cref{l:p-paths}, \cref{p:many parallel paths}, and \cref{p:intersection_many3} apply.
  	  	
  	We know from \cref{l:p-paths} that $\pa(h)$ (see \cref{d:x-paths}) is $(1+\mu_{0},\mu_{0} L)$-quasigeodesic for any for $h\in H'$. Since $L$ goes to infinity as $\nu$ tends to $0$, this implies that the infimum of $d(p,hp)$ for $h\in H'\setminus H$ goes to infinity with $\frac{1}{\nu}$, ensuring item \ref{three} of the statement for $\nu$ small enough.  
  	
    Next, let us show that if we choose $\nu$ to be small enough, then the group $H'$ is geometrically taut in the sense of \cref{d:nice datum}. Let us start with property \ref{h0 set tautness}.  Note that for any $g\in G$ the statement that $\theta\in\Theta$ fixes $H'g\in\uhh$ is equivalent to  $\theta^{g^{-1}}\in H'$.
    In view of this, property \ref{h0 set tautness} for $H'$ is equivalent to the conjunction of the following two conditions:
    \begin{itemize}
    	\item $H'\cap\Theta=\Stab_{\pi}(0)$, 
    	\item for any $g\in G$ such that $\theta^{g^{-1}}\in H'$ we have $g\in H'\rho$ for some $\rho\in\Theta$ such that $0\cdot_{\pi}\rho$ is fixed by $\theta$. 
    \end{itemize}
    The algebraic clause \ref{is free product} in \cref{l:p-paths} is that $H'=H\ast\mathbb{F}(\alpha_{i})_{i\in k_{o}}$. 
    This implies the first condition listed above right away. For the second one, assume that $\theta\in\Theta$ fixes the point $H'g$. This is equivalent to $\theta^{g^{-1}}\in H'$. By a well-known property of free products, any finite subgroup of $H'$ must be conjugate in $H'$ to a finite subgroup of $H$. It follows that there is $h\in H'$ with $\theta^{g^{-1}h^{-1}}\in H$. The assumption that $H$ itself is geometrically taut and thus satisfies property \ref{h0 set tautness} in \cref{d:nice datum} implies the existence of $\rho\in\Theta$ such that $hg\in H\rho$ and $0\cdot_{\pi}\rho$ is fixed by $\theta$, so we conclude that $g\in H'\rho$, thereby establishing property \ref{h0 set tautness} for $H'$. The $k$-malnormality condition \ref{h0 malnormal} for $H'$ follows from \cref{preservation of k-sharpness} and \cref{assumptions on the ksets} on $A_{-1}$ and $A_{1}$.  
    
    That for $\nu$ small enough and any $\Omega'\in\D(\Theta)$ the element $h_{\Omega'}$ is transverse to the extended subgroups $H'$ (condition \ref{h0 transversality}) is a consequence of \cref{c:weak transversality of a set of elements} in the case in which the extension is generous, and of \cref{transverse to homega0} in the case in which the extension is linear.    
    
    To conclude, we need to show that if the given seed is flat, then the group $H'$ is also supported on $\G$. 
    This is a result follows from \cref{preservation of supportedness} and \cref{assumptions on the ksets} by ensuring that in the case that the extension is generous, that is, when $\Omega$ is not the distinguished group $\Omega_{0}\in\D(\Theta)$, the small cancellation element $\alpha$ is $(p,\nu)$-transverse to the elementary group $\G=N_{G}(\Omega_{0})$ relative to $\emptyset$. This is a consequence of the fact that $h_{\Omega}$ is transverse to $\G$ relative to $\emptyset$, which follows from the fact that no conjugate of $h_{\Omega}$ lies in $\G$. Since $\alpha$ is rootless, the latter can be ensured by choosing $\nu$ small enough that for any $q\in X$ and $m\in\mathbb{Z}\setminus\{0\}$ we have $d(q,h_{\Omega_{0}}q)<d(q,\alpha^{m}q)$. 
  \end{proof}
  
   The argument allowing for the use of \cref{l:main induction step} in a context in which we are given pair of $k$-sets in $H\backslash G$ of a more general form than described in \cref{assumptions on the ksets} is very general and already appeared as Lemma 2.10 in \cite{gonzalez2025two}.
   \begin{lemma} \label{l:invariant tuples}
  	Let $\Theta \leq \sym_k$ be robust. Let $\lambda : U \noa \Theta$ be an action consisting of infinitely many free orbits together with a $\Theta$-invariant set $A_0$ such that $(\lambda|_{A_0} : A_0 \noa \Theta) \simeq (\pi :  k \noa \Theta)$. Then:
  	\begin{enumerate}[label=(\roman*)]
  		\item \label{k-sets iso to std} for any $\Theta$-invariant set $A\in [U]^k$ we have $(\lambda|_A : A \noa \Theta) \simeq (\pi :  k \noa \Theta)$, and thus in particular $\NFr_\lambda(\Theta) \subset A$;
  		\item \label{set stab is docile} for any $A\in[U]^{k}$ with $A \neq A_0$, letting $\Omega=\SStab_{\lambda}(A)$, we have that $\Omega$ is docile and $\NFr_{\lambda}(\Omega) \subset A$.
  	\end{enumerate}	
  \end{lemma}  
  
  This result, together with \ref{h0 set tautness}, implies the following:
  \begin{corollary}\label{c:omega sets}
  	Let $H\leq G$ be promising and let $\lambda: \uh\noa G$. Then 
  	\begin{equation*}
  		|\mathcal{S}^{H}_{\Omega}|=|\mathcal{S}_{\Omega}|=|\NFr_{\pi}(\Omega)|=|\NFr_{\lambda}(\Omega)| 
  	\end{equation*}
  	and any $\Omega$-set $A\in[\uh]^{k}$ is generated by some $(H,\Omega,\mathcal{S}_{\Omega})$-admissible tuple $(g_{i})_{i\in k_{o}}$ in the sense of \cref{d:pairs of tuples}.
  \end{corollary}

   \begin{lemma}\label{strict tuples in the orbit}
   	Let $\datum$ be a hyperbolic $\Theta$-seed, $H\leq G$ a promising subgroup in the sense of \cref{d:promising}. 	
   	Then for any $A\in[H\backslash G]^{k}$ there is some translate $B$ of $A$ such that one of the following holds:
   	\begin{itemize}
   		\item $B$ is a strict $\Omega$-set for some $\oid$;
   		\item $\Theta$ is not docile and $B$ is the unique $k$-subset of $\uh$ invariant under $\Theta$.
   	\end{itemize}
    If $\Theta$ is not docile, then the two possibilities are mutually exclusive. 
   \end{lemma}
   \begin{proof}
   	 By $k$-sharpness, $\SStab_{\tilde{\lambda}}(A)$ must embed into $\sym_{k}$ and thus, by assumption \ref{conditions on finite subgruops} it must be conjugate into $\Theta$.  
   	 From this point one can argue as in \cite[Lem. 3.27]{gonzalez2025two} to obtain the desired conclusion. First, note that there must be $g\in G$ such that $A'=Ag$ satisfies $\SStab_{\lambda}(A')\leq\Theta$. 
   	 By \cref{l:invariant tuples} either $\Theta$ is not docile and $A'$ is the unique $k$-subset of $\uh$ invariant under $\Theta$ or the setwise stabilizer of $A'$ is docile. If the latter is the case, then up to replacing $A'$ by some translate under the action of $\Theta$, we may assume that $A'$ is a strict $\Omega$-set for some $\Omega\in\D(\Theta)$.  
   \end{proof}
   
  \setcounter{thmx}{3} 
  \begin{thmx}
	  \bodymainthm
  \end{thmx} 
  \begin{proof}
   \newcommand{\at}[0]{A_{\Theta}}
   Recall \cref{ass SOmega}. In addition to this, let $A_{\Theta}\subseteq\Theta$ be some collection of representatives of the cosets in $\Stab_{\pi}(0)\backslash\Theta$, where recall that $\pi$ is the standard action $k\noa\Theta$. 
   We may assume $\so\subseteq\Theta$ and fix some arbitrary $\theta_{0}\in A_{\Theta}\setminus\so$. 
   Let $\datum$ be our hyperbolic $\Theta$-seed. For every $H\leq G$ we will abbreviate $A_{\Theta}^{H}=H\backslash A_{\Theta}$.

   Fix some linear order $<$ on $G$ with the order type of the integers. Write $<^{k}$ for the resulting lexicographical order on $[G]^{k}$, i.e., $A<^{k}B$ holds for $A\neq B$ if and only if there is $C\subseteq A\cap B$ such that $C>A\setminus C,B\setminus C$ and $\max(A\setminus C)<\max(B\setminus C)$. Given $H\leq G$, we define an order $<^{k}_{H}$ on $\uh$ by letting $A<^{k}_{H}B$ if there is some lift $\tilde{A}\in[G]^{k}$ of $A$ such that $\tilde{A}<^{k}\tilde{B}$ for all preimages $\tilde{B}$ of $B$ in $[G]^{k}$. 
   Given $R>0$, write $\mathcal{N}(R)=\{g\in G\,|\,d(p,gp)<R\}$, and given $H\leq G$ write $\mathcal{N}_{H}(R)=\{\,Hg\,|\,g\in\mathcal{N}(R)\,\}$.
   
   We construct a chain of  subgroups $H_{0}\leq H_{1}\lneq \dots$ such that the following holds:
   \begin{enumerate}[label=(\alph*)]
   	\item \label{final promising} $H_{n}$ is promising;
   	\item \label{final quotients}the quotient map $H_{n}\backslash G\onto H_{n+1}\backslash G$ is injective on $\mathcal{N}_{H_{n}}(n+1)$; 
   	\item \label{final transitivity} for every $A\in[\mathcal{N}_{H_{n}}(n+1)]^{k}$ there is $m\geq n+1$ such that the image of $A$ in $H_{m}\backslash G$ and $\at^{H_{m}}$ are in the same orbit under the diagonal action of $G$ on $[H_{m}\backslash G]^{k}$. 
   \end{enumerate}
   Note that $H_{n}\backslash G$ will be infinite for all $n$ by \cref{taut implies infinite index}.
   It is easy to check that if all these conditions are satisfied and we write $H_{\infty}=\bigcup_{n\in\N}H_{n}$, then 
   $H_{\infty}$ is $k$-malnormal, since $H_{n}$ being $k$-malnormal is part of the hypothesis of $H_{n}$ being promising. In other words, the action $H_{\infty}\backslash G\noa\Theta$ is $k$-sharp. 
   
   To show that $H_{\infty}\backslash G\noa G$ is also $\Theta$-transitive, it suffices to observe that for any $A\in[H_{\infty}\backslash G]^{k}$ there must be some $n\geq 0$ such that $A$ is the image of some $k$-set in $\mathcal{N}_{H_{n}}(n+1)$.
   The existence of some $g\in G$ such that $Ag=\at^{H_{\infty}}$ now follows from condition \ref{final transitivity}.
   
   We start with the $H_{0}\leq G$ which is given as part of the information of which our seed consists of. Assume now that $(H_{m})_{m\leq n}$ are given. We distinguish two cases in the construction of $H_{n+1}$ according to whether the seed was ample or flat.
   
   \emph{Ample case.} If it is not possible to find $A\in[\mathcal{N}_{H_{n}}(n+1)]^{k}$ which is not in the same orbit as $\at^{H_{n}}$, then we let $H_{n+1}=H_{n}$. Otherwise we let $A$ as above be $<^{k}_{H_{n}}$-smallest for that property. It follows from \cref{strict tuples in the orbit} that there is $A'\in[H_{n}\backslash G]^{k}$ in the orbit of $A$ under $H_{n}\backslash G\noa G$ and some $\oid$ such that $A'$ is $\Omega$-strict. 
   \cref{l:main induction step} allow us to find some promising subgroup $H_{n+1}\leq G$ such that
   condition \ref{final quotients} is satisfied and $\at^{H_{n+1}}$ is in the orbit of the image of $A'$ and thus in the orbit of the image of $A$ in $H_{n+1}\backslash G$. The satisfaction of \ref{final transitivity} follows from the well-foundedness of $<^{k}$, and the choice of the $k$-set $A$ at each step. Notice that the construction ensures that every $A\in[\mathcal{N}_{H_{n}}(n+1)]^{k}$ maps injectively onto a $k$-subset of $H_{m}\backslash G$ for all $m\geq n$.
  
    \emph{Flat case.} Here we distinguish two cases according to the parity of $n$.  
    
     \begin{itemize}
     	\item  If $n$ is odd, then if the set $\mathcal{N}(n+1)\setminus(H_{n}\cdot\Theta\cdot\G)$ is empty, then we let $H_{n+1}=H_{n}$. Otherwise we let $g_{0,1}$ be the $<$-least element therein and write $A_{1}$ for the $\Omega_{0}$-set generated by $(g_{0,1})$. \cref{l:main induction step} allow us to find a promising subgroup $H_{n+1}\leq G$ such that for some $\alpha\in C_{G}(\Omega_{0})$ we have  
      \begin{equation*}
      		H_{n+1}\theta_{0}\alpha=H_{n+1}g_{0,-1}\alpha=H_{n+1}g_{0,1}.
      \end{equation*}
      
      \item  If $n$ is even, we let $\tilde{A}$ be the $<^{k}$ smallest set in $[\mathcal{N}_{H_{n}}(n+1)\cap\mathcal{Z}(\theta_{0})]^{k}$ such that $\O_{H_{n}}(\tilde{A})$ and $\O_{H_{n}}(\at)$ sit in different connected components of $\mathtt{G}(H_{n})$, provided such a set exists; if not, then once again we set $H_{n+1}=H_{n}$. Recall that this is stronger than asking that $A=H_{n}\backslash\tilde{A}$ and $\at^{H_{n}}$ are in different orbits under the action of $G$ on the right. 
      Once again, from \cref{strict tuples in the orbit} we deduce the existence of $A'\in[H_{n}\backslash G]^{k}$ (not necessarily in $\mathcal{Y}_{H_{n}}^{k}$) which is in the orbit of $A$ under $H_{n}\backslash G\noa G$ and some $\oid$ for which $A'$ is $\Omega$-strict. Yet another application of \cref{l:main induction step} allows us to find some promising subgroup $H_{n+1}\leq G$ such that $A'$ projects to a $k$-subset of $H_{n+1}\backslash G$ that lie in the same orbit as $\at^{H_{n+1}}$.
     \end{itemize}
      In both subcases \ref{l:main induction step} allows us to assume that condition \ref{final quotients} is satisfied.
      
      All that is left in the flat case is that \ref{final transitivity} is verified by the full run of the procedure above. Let $A$ be an arbitrary set in $[\mathcal{N}_{H_{n}}(n+1)]^{k}$. There must be $m_{0}\geq n+1$ for which, after a number of applications of the first subcase, we have that the image of $A$ in $H_{m_{0}}\backslash G$ is a $k$-set entirely contained in  $\mathcal{Y}_{H_{m_{0}}}(\theta_{0})$. Likewise, the second subcase must eventually apply to the image of $A$ after that point, yielding $m_{1}>m_{0}$ for which the image of $A$ in $H_{m_{1}}\backslash G$ is in the same orbit as $\at^{H_{m_{1}}}$.  
      This concludes the proof of \cref{t: main}.
\end{proof}

\subsection{Some remarks about the definition of flat hyperbolic $\Theta$-seeds} 
 \label{ss: remarks on definition}
 
 We begin with the following observation, which addresses the nature of condition \ref{self normalized} in \cref{d:flat preseeds}.

	\begin{observation}\label{non-elementary for free}
		Let $G$ be a group acting acylindrically on some Gromov-hyperbolic space, $\Omega\leq G$ a finite subgroup and $E\leq G$  a maximal non-elliptic virtually cyclic group. If $\Omega$ is normalized by $E$ and by some finite $\Omega'$, $\Omega<\Omega'<G$, such that $\Omega'\cap E=\Omega$, then $N_{G}(\Omega)$ is non-elementary.   
	\end{observation}
  \begin{proof}
  	The group generated by $\Omega'$ and $E$ cannot be elementary, since the general theory of acylindrical actions on $\delta$-hyperbolic spaces would imply that $\Omega'\leq E$, by \cref{l:Eg}.  
  \end{proof}
  
	On the other hand, the hypothesis that $\Omega_{0}$ has a single free orbit is an essential requirement, as illustrated by \cref{normalizer of Omega not virtually abelian} below, which is based on the following standard fact.
	 
	\begin{observation}\label{not virtually abelian 2}
		A virtually abelian group $B$ does not admit any action $V\noa B$ on a set which is $k_{o}$-transitive for some $k_{o}\geq 2$ and where the pointwise stabilizer of any $C\in[V]^{k_{o}}$ is finite.
	\end{observation}
   \begin{proof}
   		Let $K\leq B$ be a finite index abelian subgroup. The action of $K$ on $V$ must have only finitely many orbits, each of which must be infinite. The stabilizer $K_{u}$ of a given $u\in V$ in $K$ is infinite, since the action of $B$ is $k_{o}$-transitive and thus the stabilizer of $u$ in $B$ is infinite. On the other hand,  
   		$[K_{u},K]=\{1\}$, so $K_{u}$ must stabilize any other point in the orbit of $u$ under $K$. Since said orbit is infinite, we get a contradiction with our assumption that the pointwise stabilizer of any $C\in[V]^{k}$ is finite. 
   \end{proof}
 
   \begin{corollary}\label{normalizer of Omega not virtually abelian}
		Assume that $k\noa\Omega$ has more than one free orbit for $\Omega\in\D(\Theta)$, and let $G$ be a group admitting a $\Theta$-transitive action $U\noa G$. Let $\Omega\leq G$ be a copy of $\Omega$ whose action on some $A\in[U]^{k}$ is isomorphic to $k\noa \Omega$. Then the normalizer $N_{G}(\Omega)$ cannot be virtually abelian.  
   \end{corollary}
   \begin{proof}
   	Let $A_{0}\in[U]^{k}$ be a set on which the action of $\Omega\leq G$ is standard, i.e., isomorphic to $k\noa \Omega$, and let $A_{1}\subseteq A_{0}$ the union of all the non-free orbits, so that $A_{0}=A_{1}\cup A_{2}$, where $A_{2}$ is the union of $k_{o}$-many free orbits, $k_{o}\geq 2$. Let $V$ be the collection of free orbits of $U\noa \Omega\leq G$. There is a well-defined action $V\noa N_{G}(\Omega)$ and it is not hard to see that 
   	$V\noa N_{G}(\Omega)$ is $k_{o}$-transitive, and that the pointwise stabilizer of any $C\in[V]^{k_{o}}$ is finite, so that \cref{not virtually abelian 2} applies.
   \end{proof}
    
   In practice, we are interested in the definition of a flat hyperbolic $\Theta$-seed for a handful of values of $(\Theta,\Omega_{0})$, as per the following observation.
   
   \begin{observation}\label{examples of Omega0}
   	The subgroup $\Theta\leq\sym_{k}$ is transitive and robust and $\Omega_{0}\in\D(\Theta)$ a good docile subgroup in the following cases:
   \begin{enumerate}[label=(\roman*)]
   	\item\label{Theta0} If $\Theta=\Omega_{0}=\cyc_{p}\leq\sym_{p}$ for some prime $p$. 
   	\item \label{Theta1} If $\Theta=\sym_{3}$ and $\Omega_{0}$ is the group generated by an involution. 
   	\item \label{Theta2} If $\Theta=\alt_{4}$ and $\Omega_{0}$ is the group generated by an element of order $3$. 
   \end{enumerate}
   Any other subgroup of $\Theta$ with a single free orbit which is not conjugate to $\Omega_{0}$ fails to be self-normalized. For $\Theta=\alt_{5}$ every $\Omega\in\mathcal{D}(\Theta)$ has either more than one free orbit on $k$ or is not self-normalized.  
   \end{observation}
  \begin{proof}
  	 The fact that the $\Omega_{0}$ singled out in each of the cases above is good in the sense of \cref{d:flat preseeds}, is a trivial verification. As for the second part, it follows by inspecting the description of the docile subgroups of $\Theta$ provided in \cite[Lem. 2.9]{gonzalez2025two}.
  \end{proof}

  \subsection{Proof of \cref{c:hyperbolic case}}
  Given a hyperbolic group $G$, we can consider the action of $G$ on its Cayley graph $(X,d)$ over an arbitrary set of generators. It is a standard fact that the properness of $X$ and the freeness of the action implies that the action is acylindrical. We say that a group is \emph{elementary} if it is either finite or virtually cyclic. It also is well-known that the action of a subgroup $H\leq G$ on the Cayley graph $X$ is elementary if and only if $H$ is elementary in the sense above (see \cite{ghys2013groupes}). 
  
  In order to be able to apply the flat case of \cref{t: main} in the proof of \cref{c:hyperbolic case}, we make the following preliminary observation, which we have not been able to locate in the literature.
  
  \begin{lemma}\label{omega narrow for free}
  	Let $G$ be a hyperbolic group with Cayley graph $(X,d)$, and $G_{0}< G$ some quasiconvex subgroup. Then for any $\omega$ such that $C_{G}(\omega)=G_{0}$ there is $C>0$ such that for any $g\in G$ we have: 
  	\begin{equation*}
  		d(\omega g,g)\geq 2d(g,G_{0})-C.
  	\end{equation*}
  \end{lemma}
  \begin{proof}
    Assume that $G_{0}$ is $D_{0}$-quasiconvex for some $D_{0}> 0$. Pick an arbitrary $g\in G$ and let 
    \begin{equation*}
    	 L=d(g,G_{0})-\frac{1}{2}d(\omega g ,g). 
    \end{equation*}
    Our goal is to find some upper bound on $L$ independent of $g$.
    
    Note that $d(\omega \rho,\rho)=d(\omega,1)=D_{1}$ for all $\rho\in G_{0}$. Pick $\rho\in G_{0}$ be such that $d(g,\rho)=d(g,G_{0})$ and write $\tau=[\rho,g]$. Since $\omega$ and $\rho$ commute, we may translate the entire setting by $\rho^{-1}$ and thus assume that $\rho=1$. 
    It follows from the hyperbolicity condition that there is some initial subarc $\tau_{0}\subseteq\tau$ such that:
     \begin{enumerate}[label=(\alph*)]
     	\item $L'=|\tau_{0}|\geq L-2(D_{1}+\delta)$,
     	\item \label{second item above} for all $q\in\tau_{0}$ the inequality $d(\omega q,q)\leq D_{1}+2\delta$ holds. 
     \end{enumerate}
      Note that \ref{second item above} above implies that $h^{-1}\omega h\in\nb_{D_{1}+2\delta}(1)$ for any $h\in\tau_{0}\cap G$. Let $N$ be the number of elements in the ball of radius $D_{1}+2\delta$ around the identity.
     
     We need only bound $L'$ from above. Assume for the sake of contradiction that the following (non-optimal) bound holds:
     \begin{equation*}
     	L'> (N+2)(D_{0}+\delta+1).
     \end{equation*}
     
     The pigeonhole principle implies the existence of $g_{0},g_{1}\in\tau_{0}\cap G$ such that
     \begin{enumerate}[label=(\roman*)]
     	\item \label{i} $d(g_{1},1)\geq d(g_{0},1)+D_{0}+\delta+1$, 
     	\item \label{ii} $g_{0}^{-1}\omega g_{0}=g_{1}^{-1}\omega g_{1}$. 
     \end{enumerate}  
    Part \ref{ii} be rewritten as $g_{1}g_{0}^{-1}\in C_{G}(\omega)=G_{0}$. Let $\tau'=[1,g_{1}g_{0}^{-1}]$. Write $\tau_{1}$ for the subarc of $\tau_{0}$ between $1$ and $g_{1}$, and $\tau_{2}=g_{1}[1,g_{0}^{-1}]$. Consider the geodesic triangle given by $\tau_{1},\tau',\tau_{2}$. It follows from \ref{i} and the hyperbolicity condition that 
    any point in $\tau_{1}$ at distance at most $D_{0}+\delta+1$ from the identity is at distance at most $\delta$ from $\tau'$, whereas $\tau'\subseteq\nb_{D_{0}}(G_{0})$ by the quasigeodesic condition on $G_{0}$. 
    This implies the existence of $q\in\tau_{1}$ such that $d(q,G_{0})<d(q,1)$, contradicting the choice of $\tau$ as a shortest geodesic between $g$ and $G_{0}$.    
  \end{proof}
   
  The second part of \cref{c:hyperbolic case} relies on the fact that for quasiconvex subgroups of a hyperbolic group, the notion of transversality can be described in purely algebraic terms, as follows.
 \begin{lemma}\label{characterization of transverse in hyperbolic}
 	Let $G$ be a hyperbolic group, $(X,d)$ a Cayley graph of $G$, and $H\leq G$ some quasiconvex subgroup of $G$. Let $h$ be a loxodromic element and $\mathcal{S}$ some finite set such that $h\in H^{s}$ for all $s\in\mathcal{S}$.
 	Then, $h$ is transverse to $H$ relative to some set $\mathcal{S}\subseteq G$ if and only if 
 	for all $g\in G$ the existence of  $m\in \N\setminus \{0\}$ such that $h^{m}\in H^{g}$ implies that $g\in Hs$ for some $s\in\mathcal{S}$. 
 \end{lemma}
 \begin{proof}
 	Let us say that $X$ is $\delta$-hyperbolic. The only if direction is clear. For the if direction, note that quasiconvexity implies that if $h$ fails to be transverse relative to $\mathcal{S}$, then there must be some $g\in G\setminus\bigcup_{s\in\mathcal{S}}Hs$ and some constant $K>0$ such that for any $m\in\Z$ the element $h^{m}$ is at distance at most $K$ in the Cayley graph from an element in $g^{-1}H$. Indeed, for any $N>0$ there must be $g\in G\setminus\bigcup_{s\in\mathcal{S}} Hs$ and $m>0$ such that the geodesic arc $\tau=[1,h^{m}]$ contains some subarc $\tau_{0}$ of length at least $N$ which is contained in $\nb_{2\delta}(\gamma)$ from some geodesic $\gamma$ between two elements of $gH$. Since $\gamma\subseteq\nb_{C}(H)$ for some $C$ depending only on $H$, in fact $\gamma_{0}\subseteq\nb_{2\delta+C}$ for some $C$ depending only on $H$. 
  
  By the pigeonhole principle, as well as the fact that bounded sets of vertices in $X$ are finite, we deduce the existence of $g_{0}\in G$, as well as distinct $m,m'\in\Z$ such that $h^{m}g_{0},h^{m'}g_{0}\in gH$. It follows that $h^{m-m'}\in H^{g_{0}^{-1}}$, which by assumption implies that $g_{0}^{-1}\in Hs$ for some $s\in\mathcal{S}$. In particular, $h^{m-m'}\in H^{s^{-1}}$. We thus have $g\in Hs h^{m}$, and since $h\in H^{s}$, it follows that $g\in Hs$: a contradiction.
 \end{proof}
  
  \setcounter{thmx}{0} 
  
 We assume that the reader is familiar with the basics of Bass-Serre theory, as introduced in the seminal work \cite{serre2002trees}.
 \begin{thmx}
 	Let $G$ be a non-elementary hyperbolic group.
 	\begin{enumerate}[label=(\Alph*)]
 		\item \label{2str'}Assume that $G$ admits a unique infinite class of involutions $\sigma^{G}$, the centralizer $C_{G}(\sigma)$ is either
 		\begin{itemize}
 			\item non-elementary,
 			\item of the form $\subg{\sigma}\times \Z$,
 		\end{itemize}
 		 and $\subg{\sigma}$ is the largest finite group normalized by $C_{G}(\sigma)$.
 		 Then $G$ is sharply $2$-transitive.
 		\item \label{3str'}Assume that $G$ contains a copy of $\sym_{3}$, henceforth identified with $\sym_{3}$, such that any finite subgroup of $G$ that embeds into $\sym_{3}$ must be conjugate into it. 
 		Let $\sigma\in\sym_{3}\leq G$ be the involution fixing $0\in 3$, and write
 			 $G_{\sigma}=C_{G}(\sigma),\quad\quad G_{\tau}=N_{G}(\tau)$. 
 		Assume that the following holds:
 		\begin{enumerate}[label=(\roman*)]
		  	\item \label{normalizer condition} the normalizer of any $\{1\}\leq\Omega<\sym_{3}$ is infinite and normalizes no finite group other than $\Omega$;
 		  	\item \label{Gsigma malnormal}$G_{\sigma}$ is malnormal.
 				\item \label{Gsigma conjugates} if $G_{\sigma}^{g}\cap\Theta\neq\{1\}$, then $g\in G_{\sigma}\theta$ for some $\theta\in\sym_{3}$; 
 				\item \label{Gsigma convex}$G_{\sigma}$ is a quasiconvex subgroup of $G$;
 				\item \label{Gsigma alternative}one of the following holds: 
 				\begin{itemize}
 					\item  $G_{\sigma}$ is non-elementary
 					\item $G_{\sigma}\cong \cyc_{2}\times\mathbb{Z}$ and $G$  splits as an amalgamated free product
 				\begin{equation}\label{gogd}
 				G=G_{\sigma}\ast_{\subg{\sigma}}G'.
 				\end{equation}
 				\end{itemize} 
 		
 			\end{enumerate}
 		\end{enumerate}
 		  Then $G$ admits a sharply $3$-transitive action on an infinite set. 
 \end{thmx}
 \begin{proof}
   \emph{Part \ref{2str'}.} The result follows from \cref{c:simple cases}, using the fact that a subgroup of $G$ acts non-elementarily on the Cayley graph of $G$ if and only if it is itself non-elementary. 
 	
 	 \emph{Part \ref{3str'}.} Here the group $H_{0}$ in \cref{d:nice datum} must contain $\sigma$, and $H_{0}$ is the only point fixed by $\sigma$ in the action $H_{0}\backslash G\noa G$. It follows that the group $G_{\sigma}=C_{G}(\sigma)$ fixes that same coset, i.e., $G_{\sigma}\leq H_{0}$. The statement of the second part of the corollary reflects a set of sufficient conditions for $H_{0}=G_{\sigma}$ to satisfy the conditions required by  \cref{d:nice datum} and \cref{d:flat seed}.
 	 
 	 Let $\tau\in \sym_{3}$ be an element of order $3$. We may assume $\D(\Theta)=\{\{1\},\subg{\sigma},\subg{\tau}\}$. 
 	 For each $\Omega\in\D(\Theta)$ pick some loxodromic element $h_{\Omega}$ with $\K(h_{\Omega})=\Omega$ and $[h_{\Omega},\Omega]=1$. 
 	 
 	 Condition \ref{conditions on finite subgruops} appears in the statement of the theorem, and \ref{GOmega K}
   and \ref{hOmega loxodromic} are part of \ref{normalizer condition} and the choice of $h_{\Omega}$ above. 
 	 Condition \ref{h0 set tautness} in \cref{d:nice datum}, i.e. the assertion that $G_{\sigma}\backslash G\noa\sym_{3}$ has a single standard orbit, and that the remaining orbits are free is a restatement of condition \ref{Gsigma conjugates} in the statement. Condition \ref{h0 malnormal} follows immediately from \ref{Gsigma malnormal} in the statement. As for condition \ref{h0 transversality}, it follows from \cref{characterization of transverse in hyperbolic}, conditions \ref{normalizer condition} and \ref{Gsigma convex} in the statement, as well as the fact that if we write 
 	 $\bar{G}_{\tau}=C_{G}(\tau)$, then for any $g\in G$ we have $\bar{G}_{\tau}\cap G_{\sigma}^{g}=\{1\}$. For the last claim, note that if $h\in (G_{\tau}\cap H_{0}^{g})\setminus\sym_{3}$, then 
 	 $H_{0}g,H_{0}g\tau,H_{0}g\tau$ are all fixed points of $H_{0}\backslash G\noa G$ and $H_{0}$ cannot be $3$-malnormal. This concludes the proof of part \ref{existence of H0} in the statement.
 	 
 	 Next, note that $G_{\tau}$ is necessarily non-elementary by \cref{non-elementary for free}, and so is $G=G_{\{1\}}$ by assumption. There are now two cases, according to which of the two cases in condition \ref{Gsigma alternative} are satisfied. 
   
   If $G_{\sigma}$ is non-elementary, then condition \ref{GOmega non elementary} in the definition of an ample hyperbolic $\sym_{3}$-seed is satisfied and we are done, so let us consider the case in which $G_{\sigma}\cong\subg{\sigma}\times \Z$ and $G$ splits as in \eqref{gogd}.  
   We claim that the conditions for a flat hyperbolic seed (see \cref{d:flat preseeds} and \cref{d:flat seed}) are satisfied with $\Omega_{0}=\subg{\sigma}$. We already know that $\subg{\sigma}$ satisfies all the conditions imposed on $\Omega_{0}$ by \cref{d:flat preseeds}, and \ref{no-dihedrals} is the isomorphism defining this case. 
   Condition \ref{omega narrow} in the current context is a consequence of \cref{omega narrow for free}. So the data given so far determines a flat preseed. 
   
   To conclude the proof, we need to ensure that $H_{0}$ is supported on $G_{\sigma}$ in the sense of \cref{auxiliary graph}. We will use our very strong (likely, far too strong) condition on the existence of the splitting \eqref{gogd} to check that $H_{0}$ is strongly supported on $G_{\sigma}$, which here satisfies $G_{\sigma}=H_{0}$. We take $\mathcal{S}_{\sigma}=\{1\}$.
   Assume that the property fails and let $g ,g'\in G\setminus G_{\sigma}$ and $\rho,\rho'\in G_{\sigma}\setminus\{1\}$, $\rho^{\pm}\neq\rho'$ be such that 
   \begin{equation*}
   	\{H_{0},H_{0}g,H_{0}g\rho\}\gamma=\{H_{0},H_{0}g',H_{0}g'\rho'\}
   \end{equation*}
   for some $\gamma\in G$. Without loss of generality $H_{0}g\gamma=H_{0}g'$ and there are two possibilities: 
   \begin{itemize}
   	\item $\gamma\in H_{0}$, $H_{0}g\rho\gamma=H_{0}g'\rho'$;
   	\item  $H_{0}g\rho\gamma=H_{0}$, $H_{0}\gamma=H_{0}g'\rho'$.  
   \end{itemize}
 	 In the first case $\gamma\in H_{0}$ implies that  
 	 \begin{equation*}
 	 	H_{0}g\gamma=H_{0}g'\gamma^{-1}\rho\gamma=H_{0}g'\rho',
 	 \end{equation*}
 	 and since $H_{0}=G_{\sigma}$ is commutative we get $\rho'\rho^{-1}\in H_{0}^{g'}$, which implies $\rho=\rho'$, by the malnormality of $H_{0}$ and the fact that $g'\nin H_{0}$.   
 	 Let us now consider the second case. Write $\gamma=\eta g'\rho'$ for some $\eta\in H_{0}$. We then have 
 	 \begin{equation*}
 	 	H_{0}=H_{0}g\rho\gamma=H_{0}g'\rho^{\gamma}=H_{0}g' \rho^{\eta g'\rho'}
 	 \end{equation*}
 	 and since $\eta,\rho,\rho'\in H_{0}=G_{\sigma}$ and the latter is commutative we obtain
   \begin{equation}\label{in H0}
   	g'(\rho')^{-1} \rho^{g'}\in H_{0}.
   \end{equation}
   Let us show that the inclusion above contradicts the graph of groups decomposition \eqref{gogd}. 
   For that sake, let $T$ be the simplicial $G$-tree dual to this decomposition and let $u$ be the vertex of $T$ fixed by $G_{\sigma}$. Since $G_{\sigma}$ is malnormal in $G$, no non-trivial group element of $G$ can fix two points in the orbit of $u$. 
   
   Let $u'=g'^{-1}u$. Since $g'\notin G_{\sigma}$ we know that $u'\neq u$. The element $\rho^{g'}$ fixes $u'$. 
   Let $u''=\rho^{g'}u$. By the above remarks must have $d_{T}(u'',u)=2(d_{T}(u,u')-\epsilon)$, where $\epsilon=1$ if 
   $\rho=\sigma$ and $\epsilon=0$ otherwise. Write $v$ for the first vertex after $u$ in a simple path between $u$ and $u'$ and note that $v$ will be contained in the unique simple path between $u'$ and $u''$ regardless of which of the two cases above is the case. 
   
   Since assume \eqref{in H0} holds, we have $g'(\rho')^{-1} u''=u$ and thus $\rho' u'=u''$. Since $\rho'$ fixes $u$, this implies
   \begin{equation*}
   	d_{T}(u,u')=d_{T}(u,u'')=2(d_{T}(u,u')-\epsilon)\quad\Rightarrow\quad \epsilon=1,\,\,\,d_{T}(u,u')=2. 
   \end{equation*}
 	 For the equality $\rho'u'=u''$ to hold it is necessary that the element $\rho'\in G_{\sigma}$ fixes the vertex $v$ and thus $\rho'=\sigma$. But then $\rho=\rho'=\sigma$, contrary to our assumption that $\rho\neq\rho'$.   
 \end{proof}
 
  \setcounter{thmx}{2} 
 We now move on to the proof of \cref{thm:virtually free examples}. Some of the cases considered can be deduced from \cref{c:hyperbolic case}, but one may also argue from scratch in a more straightforward manner. 
 
  \begin{thmx}\label{thm:virtually free examplesbis}
 	 	Assume that $\Theta$ is one of the following transitive robust permutation groups:
   \begin{enumerate}[label=(\roman*)]
   	\item  $\Theta=\Omega_{0}=\cyc_{p}\leq\sym_{p}$ for some prime $p$;
   	\item  $\Theta=\sym_{k}$ for $k\leq 3$;
   	\item  $\Theta=\alt_{k}\leq\sym_{k}$ for $k\in\{4,5\}$.
   \end{enumerate}
 		Then the group $G$ which is the result of taking the group $\Theta$ and amalgamating the group $\Omega\times\Z$ for each $\Omega\in\mathcal{D}(\Theta)$ admits a sharply $\Theta$-transitive action. 
 	\end{thmx}
   \begin{proof}
   	 We show that the group $G$ admits a $\Theta$-seed, which is flat with distinguished $\Omega_{0}\in\D(\Theta)$ in the following cases:
   	 \begin{itemize}
   	 	\item if $\Theta=\cyc_{p}\leq\sym_{p}$ for $p$ prime, with $\Omega_{0}=\cyc_{p}$;
   	 	\item if $\Theta=\sym_{3}$ with $\Omega_{0}\in\D(\Theta)$ generated by an involution;
   	 	\item if $\Theta=\alt_{4}$ with $\Omega_{0}\in\D(\Theta)$ generated by a subgroup of order $3$.   
   	 \end{itemize}
   	
     Consider the action $\mu$ of $G$ on a simplcial tree $(T,d)$ dual to the fundamental group of the graph of groups $\mathcal{G}$ consisting of a single vertex $\Theta$, with a self edge carrying the group $\Omega$ for each $\oid$, a generator $h_{\Omega}$ of the cyclic factor of $\Omega\times \Z$ can be identified with a Bass-Serre element of the resulting graph of groups. The action of $G$ on $T$ is clearly acylindrical and the $h_{\Omega}$ above loxodromic elements. The stabilizer $\Gamma$ of $0$ under the standard action $k\noa \Theta$ can be assumed to be in $\D(\Theta)$ and to agree with the distinguished $\Omega_{0}\in\D(\Theta)$ required for a flat seed in case it is non-trivial. We take $H_{0}=\{1\}$ if $\Gamma=\{1\}$ and $H_{0}=\Omega_{0}\times \subg{h_{\Omega_{0}}}$ otherwise. Using \cref{examples of Omega0}, it is easy to show that  
  	$(G,H_{0},h_{\Omega},(T,d),\mu)_{\oid}$ is a hyperbolic $\Theta$-seed. The fact that $H_{0}$ is geometrically taut with respect to $(h_{\Omega})_{\oid}$ is immediate from the descriptions of the orbits of edges in $T$, and so is the condition \eqref{omega narrow}. The fact that $H_{0}$ strongly supported on $G_{\Omega_{0}}=\Omega_{0}\times \subg{h_{\Omega_{0}}}$ is immediate if $H_{0}=\{1\}$ and otherwise follows from elementary combinatorial considerations along the lines of the end of the proof of \cref{c:hyperbolic case}.
  \end{proof}
%
 \appendix

 \section{An elementary proof that the group $(C_{2}\times \Z)\ast \Z$ is sharply $2$-transitive.} 
  \label{s: simple example}
  In this appendix we provide a purely combinatorial proof of the sharp $2$-transitivity of $(\cyc_{2}\times \Z)\ast\Z$, that can be used as a source of intuition for some aspects of the proof of the flat case of \cref{t: main}.
  
  \begin{theorem}\label{almost free case}
  	The group $(C_{2}\times \Z)\ast \Z$ is sharply $2$-transitive. 
  \end{theorem}

  For the rest of this section we write 
  \begin{equation}\label{easy group}
  	G=(C_{2}\times \Z)\ast \Z
  \end{equation}
  and fix the following additional notation:
  \begin{itemize}
  	\item $G_{0}$ for the free factor $C_{2}\times \Z$,
  	\item $\sigma$ for the unique involution in $G_{0}$, 
  	\item $s$ for a generator of the cyclic summand of $G_{0}$, 
  	\item $t$ for a generator of the complementary free factor $\Z$ in \eqref{easy group}. 
  \end{itemize}
  Any element in $G$ can be uniquely written as an alternating product of the form: 
  \begin{equation}
  	\beta_{0}t^{\epsilon_{0}}\beta_{1}\dots t^{\epsilon_{m-1}}\beta_{m},
  \end{equation}
  where $\epsilon_{i}\in\{\pm 1\}$ for $i\in m$, $\beta_{i}\in G_{0}$ for $i\in(m+1)$ and $\epsilon_{i}=-\epsilon_{i+1}$ only if $\beta_{i+1}\neq 1$. We will refer to this word as the normal form of the element. Let $\lambda^{0}:U\noa G_{0}$ be some regular action. 
   \begin{definition}
     A \emph{finite partial extension of $\lambda^{0}$ to $G$} is given by a bijection $\Phi_{t}$ between two finite subsets of $U$. We extend this notation by letting $\Phi_{t^{-1}}=\Phi_{t}^{-1}$, and letting $\Phi_{g}$ for $g\in G_{0}$ be the bijection of $U$ given by the action of $\lambda^{0}$ on $U$. By abuse of notation, we identify $\lambda^{0}$ with the finite partial extension of itself with $\Phi_{t}=\emptyset$. 
     
     Given a general element $g\in G$, we extend the notation further by letting $\Phi_{g}$ be the composition from left to right of $\Phi_{Z_{i}}$, where $Z_{0}Z_{1}\dots Z_{m}$ is a normal form representing $g$. Note that for any $g\notin G_{0}$ this is a (possibly empty) bijection between two finite subsets of $U$. We write $\dom(\Phi_{g}),\im(\Phi_{g})$ for the domain and image, respectively, of $\Phi_{g}$, and 
     $\supp(\Phi)=\dom(\Phi_{t})\cup\im(\Phi_{t})$.    
     
     We say that $\Phi$ is \emph{$2$-sharp} if there is no $g\in G\setminus\{1\}$ and $\{u,v\}\in [U]^{2}$ such that $\Id_{\{u,v\}}\subseteq\Phi_{g}$. We write $\mathcal{P}_{2sh}$ for the collection of all $2$-sharp finite partial extensions of $\lambda^{0}$ to $G$. 
     Given two finite partial extensions of $\lambda^{0}$ to $G$, $\Phi_{1},\Phi_{2}$, we write $\Phi^{1}\leq\Phi^{2}$ if 
     $\Phi^{1}_{t}\subseteq\Phi^{2}_{t}$.  
   \end{definition}
   
  \begin{notation}
  	We will henceforth identify $U$ with $G_{0}$ in a way that makes $\lambda_{0}$ agree with right multiplication.
   Given $\beta\in G_{0}$, write $|\beta|$ for the unique element $\gamma_{0}\in\{\beta,\beta^{-1}\}$ such that 
  	$\gamma_{0}\in \mathtt{V}$.
  \end{notation}
 
 \begin{definition}
 	Given a finite partial extension $\Phi$ of $\lambda_{0}$ to $G$, we define a directed graph $\mathtt{G}(\Phi)$ as follows: 
  	\begin{itemize}
  		\item  the set of vertices of $\mathtt{G}(\Phi)$ is the set
  		\begin{equation*}
  			\mathtt{V}=\{s^{m}\sigma\,|\,m\geq 0\}\cup\{s^{m}\,|\,m>0\}\subseteq G_{0}; 
  		\end{equation*}
  		\item the set of edges $\mathtt{E}(\Phi)$ of $\mathtt{G}(\Phi)$
  		comprises of an edge from $\beta$ to $\beta'$  for every $(\beta,\beta')$ such that
  		\begin{equation*}
  			\Phi_{t}((\gamma,\gamma\beta))=(\gamma',\gamma'\beta').
  		\end{equation*}
  	\end{itemize}
 \end{definition}
 
 Clearly, $\mathtt{G}(\Phi)$ is a subgraph of $\mathtt{G}(\Phi')$ for any $\Phi,\Phi'\in \mathcal{P}_{2sh}$ with $\Phi\leq\Phi'$. The following is also an easy observation.
 \begin{observation}\label{2-sharpness criterion}
 	The pairs $(1,\beta_{0})$ and $(1,\beta_{1})$ are in the same $\Phi$-orbit for $\beta_{0},\beta_{1}\in G_{0}\setminus\{1\}$ if and only if $\beta_{0}$ and $\beta_{1}$ are in the same connected component of $\mathtt{G}(\Phi)$. The extension	$\Phi$ is $2$-sharp if and only if $\mathtt{G}(\Phi)$ contains no non-directed cycles of any length, i.e., including double edges. 
 \end{observation}
 
  \begin{proof}[Proof of \cref{almost free case}]
  	We will construct a sharply $2$-transitive action $\lambda:U\noa G$ as the direct limit of a chain of partial actions in $\mathcal{P}_{2sh}$,
  	\begin{equation*}
  		\lambda^{0}=\Phi^{0}<\Phi^{1}<\Phi^{2}\dots ,
  	\end{equation*}     
  	by which we mean that the action of $t$ on $U$ will be given by the bijection $\bigcup_{n\in\N}\Phi^{n}_{t}$. 
  	It is easy to see that $\lambda$ is $2$-sharp if all $\Phi^{n}$ are $2$-sharp. Likewise, $\lambda$ is trivially $1$-transitive, so $2$-transitivity of $\lambda$ will follows from transitivity of the stabilizer of $1\in U$ on the complement $U\setminus \{1\}$. 
  	
  	Since the action of $\sigma$ exchanges the points $1$ and $\sigma$, the desired outcome reduces to the following two conditions: 
  	\begin{itemize}
  		\item $\bigcup_{n\in\N}\dom(\Phi^{n}_{t})=\bigcup_{n\in\N}\im(\Phi^{n}_{t})=U$;
  		\item for every $\beta\in \mathtt{V}$ there are $g\in G$ and $n\in N$ such that $\Phi^{n}_{g}$ maps $\{1,\beta\}$ to $\{1,\sigma\}$. 
  	\end{itemize}
  	This will be achieved by alternating the use of \cref{first claim} and \cref{second claim} below to construct $\Phi^{n+1}$ from $\Phi^{n}$. Given a subset $A\subseteq\Z$, we will use the notation 
  	\begin{equation*}
  		U_{A}=\{s^{m},s^{m}\sigma\,|\,m\in A\}. 
  	\end{equation*}

  	\begin{lemma}\label{first claim}
  		Let $\Phi\in\mathcal{P}_{2sh}$, $\beta_{0}\in U$. Then there is $M>0$ such that for any $m\in\Z$ with $|m|\geq M$ if $\Phi'$ is given by $\Phi'_{t}=\Phi_{t}\cup\{(\beta_{0},s^{m})\}$ or by $\Phi'_{t}=\Phi_{t}\cup\{( s^{m},\beta_{0})\}$, then $\Phi'\in\mathcal{P}_{2sh}$. 
  	\end{lemma}
    \begin{subproof} 
    	Without loss of generality, assume that $\Phi'$ as in the first case in the statement. 
    	Let $M_{0}$ be such that $\supp(\Phi)\cup\{\beta\}\subseteq U_{[-M_{0},M_{0}]}$ and write $M=3M_{0}+1$.
    	The choice of $M_{0}$ implies the following: 
    	\begin{enumerate}[label=(\roman*)]
    		\myitem{$(\dagger)$} \label{conc0} For every pair of the form $(\gamma,\gamma\beta)\in\dom(z^{\epsilon})$, $\epsilon\in\{\pm1\}$, we have $\beta\in U_{[-2M_{0},2M_{0}]}$. In particular, any vertex in $\mathtt{G}_{x}$ of valence at least $1$ must belong to 
    		\[\mathtt{V}\cap U_{[-2M_{0},2M_{0}]}=U_{[0,2M_{0}]}.\]
    	\end{enumerate}
    	
    	The second case can be handled in an entirely symmetric way. By virtue of \cref{2-sharpness criterion}, all we need to show is that the graph $\mathtt{G}(\Phi')$ contains no cycles. 
    	Any edge $(\gamma_{0},\gamma_{1})\in \mathtt{E}(\Phi')\setminus \mathtt{E}(\Phi)$ that is not in $\mathtt{G}(\Phi)$ must involve the pair in $\Phi_{t}'\setminus \Phi_{t}$, together with one of the pairs from  $\Phi_{t}$. In other words, it must result from a pair of relations of the form:
    	\begin{equation*}
    		\Phi_{t}'(\beta_{0})=s^{m},\quad\quad \Phi_{t}(\beta_{0}\gamma_{0})=s^{m}\gamma_{1}^{\epsilon},
    	\end{equation*}
    	for some suitable exponent $\epsilon\in\pm1$. 
    	In view of this, we derive the following conclusions:
    		
    	\begin{enumerate}[label=(\roman*)]
    		\setcounter{enumi}{1}
    		\item \label{conc1} $\gamma_{1}$ and $\epsilon$ are uniquely determined by $\gamma_{0}$; vice versa, the value of $\gamma_{0}$ is determined by that of $\gamma_{1}$;
    		\item \label{conc2} as a result of
    	\begin{equation*} 
    		\Phi_{t}(\beta_{0}\beta)\in\supp(\Phi)\subseteq U_{[-M_{0},M_{0}]},\quad\quad m>3M_{0} 
    	\end{equation*}
    	we have $\gamma_{1}\notin U_{[-2M_{0},2M_{0}]}$.
    	\end{enumerate}
    	
    	Part \ref{conc2} together with \ref{conc0} implies that each of the new edges is incident to one vertex that is not incident to any edge in $\mathtt{G}(\Phi)$. Part \ref{conc1} implies that this edge is the endpoint of only one of the new edges. \cref{first claim} follows.    	 
    \end{subproof}

   \begin{lemma}\label{second claim}
    Let $\Phi\in\mathcal{P}_{2sh}$, $\beta_{0}\in \mathtt{V}$ and $\epsilon\in\pm 1$, and assume that $\{1,\beta_{0}\}$ and $\{1,\sigma\}$ are not in the same $\Phi$-orbit. Then there is some $M>0$ such that if we write 
    \begin{equation*}
    	\Phi'_{t}=\Phi_{z}\cup\{(s^{3m},s^{5m}),(s^{3m}\beta_{0}, s^{5m}\sigma)\},
    \end{equation*}
    then $\Phi'\in\mathcal{P}_{2sh}$. 
   \end{lemma}
  \begin{subproof}
  	 Let $M$ be a positive integer such that $\supp(\Phi)\cup\{\beta\}\subseteq U_{(-M,M)}$, so that statement \ref{conc0} in the proof of \cref{first claim} holds. In particular all edges of $\mathtt{G}(\Phi)$ are between vertices in $U_{[0,2M)}$. 
  	Take any $m\geq M$.
   	
  	In this case the complement $P=\Phi'_{z}\setminus \Phi_{z}$ consists of exactly two pairs of elements in $U$ and an edge $e\in \mathtt{E}(\Phi')\setminus \mathtt{E}(\Phi)$ involves either one or two of the pairs in $P$, leading to three distinct possibilities:
  	\begin{enumerate}[label=(\Alph*)]
  		\item \label{case first pair} The edge $e$ results from relations of the form: 
  		 \begin{equation*}
    		\Phi_{t}(\gamma_{0})=\gamma_{1},\quad\quad \Phi'_{t}(s^{3m}\beta_{0})=s^{5m}\sigma.
       \end{equation*}
       for $\gamma_{0}\in\dom(\Phi_{z})$, $\gamma_{1}\in\im(\Phi_{z})$, so that 
       \begin{equation}\label{expA}
       	e=(\gamma_{0}^{-1}s^{3m}\beta_{0},\gamma_{1}^{-1}s^{5m}\sigma )\in U_{(2M,4M)}\times U_{(4M,6M)}
       \end{equation}
  		\item \label{case second pair} The edge $e$ results from relations of the form
  		  		 \begin{equation*}
    		        \Phi'_{t}(s^{3m})=s^{5m},\quad\quad \Phi_{t}(\gamma_{0})=\gamma_{1}.
            \end{equation*}
             for $\gamma_{0}\in\dom(\Phi_{z})$, $\gamma_{1}\in\im(\Phi_{z})$. In this case, $e$ is of the form 
             \begin{equation}\label{expB}
             	e=(\gamma_{0}^{-1}s^{3m},\gamma_{1}^{-1}s^{5m})\in U_{(4M,6M)}\times U_{(2M,4M)}.
             \end{equation}
     \item \label{case both} The edge $e$ results from relations of the form: 
  		 \begin{equation*}
    		\Phi'_{t}(s^{3m})=s^{5m},\quad\quad \Phi'_{t}(s^{3m}\beta_{0})=s^{5m}\sigma,
       \end{equation*}
       so that $e=(\beta_{0},\sigma)$ 
  	\end{enumerate}
    Write 
    \begin{equation*}
     	\mathtt{V}_{0}=U_{[0,2M)},\quad \quad \mathtt{V}_{1}=U_{(2M,6M)},\quad\quad \mathtt{V}_{1}^{-1}=U_{(2M,4M)},\quad\quad \mathtt{V}_{1}^{1}=U_{(4M,6M)}.
     \end{equation*}
     It follows from the discussion above that any edge in  $\mathtt{G}(\Phi')$ either has both ends contained in $\mathtt{V}_{0}$ or in $\mathtt{V}_{1}$. Moreover, 
     \begin{itemize}
     	\item the subgraph of $\mathtt{G}(\Phi')$ spanned by $\mathtt{V}_{0}$ differs from the one spanned by $\mathtt{G}(\Phi)$ by the addition of the single edge $(\beta_{0},\sigma)$;
     	\item any edge between two vertices of $\mathtt{V}_{1}$ is of one of the forms \ref{case first pair} or \ref{case second pair} above;
     	\item in the expression \eqref{expA} and the expression \eqref{expB} the value of $\gamma_{0}$ and $\gamma_{1}$ uniquely determine each other, and are uniquely determined by either of the endpoints of $e$; thus for any $\epsilon\in\{\pm1\}$, $u\in \mathtt{V}_{1}^{\epsilon}$ there is at most one edge as in \ref{case first pair} and at most one edge as in \ref{case second pair} incident on $u$.
     \end{itemize} 
      Our assumption that $(1,\sigma)$ and $(1,\beta_{0})$ where not in the same $\Phi$ orbit implies that the vertices $\beta_{0}$ and $\sigma$ of $\mathtt{V}$ are in different connected components of $\mathtt{G}(\Phi)$. Therefore, if the graph $\mathtt{G}(\Phi')$ is to contain any undirected cycle, then such cycle must consist of an alternating sequence of edges of types \ref{case first pair} and \ref{case second pair}, without loss of generality of the form:
      \begin{equation}\label{cycle}
      	e_{0},e_{1},\dots, e_{2r-1} 
      \end{equation}
      for some $r\geq1$, where for $l\in r$ we have:
      \begin{enumerate}[label=(\roman*)]
      	\item \label{a}$e_{2l}\in \mathtt{V}_{1}^{-1}\times \mathtt{V}_{1}^{1}$ is of type \ref{case first pair} 
      	\item \label{b}$e_{2l+1}\in \mathtt{V}_{1}^{1}\times \mathtt{V}_{1}^{-1}$ is of type \ref{case second pair} for $l\in r$;
      	\item \label{c}$e_{2l}$ and $e_{2l+1}$ have the same terminal vertex for $l\in r$;
      	\item \label{d}$e_{2l-1}$ and $e_{2l}$ have the same origin for $l\in r$.  
      \end{enumerate}
        Here for convenience we have extend the indexing cyclically modulo $2r$, writing $e_{-1}=e_{2r-1}$, $e_{2r}=e_{0}$ and so forth.
       
       For $-1\leq l\leq 2r$ we write $\gamma_{0,l}\in\dom(\Phi_{z})$ and $\gamma_{1,l}=\Phi_{z}(\gamma_{0,l})$ for the two coordinates of the pair in $\Phi_{z}$ involved in the definition of the edge $e_{l}$, as described in \ref{case first pair} and \ref{case second pair} above. Conditions \ref{c} and \ref{d} above translate into the following relationship between the $\gamma_{i,l}$:
       \begin{observation}\label{gammais}
       	For $l\in r$ the following holds
       	\begin{itemize}
       		\item $\gamma_{1,2l}\sigma=\gamma_{1,2l+1}$; 
       		\item $\gamma_{0,2l-1}=\gamma_{0,2l}\beta_{0}^{-1}$. 
       	\end{itemize}
       \end{observation}
       To conclude the proof of \cref{second claim}, we will use \cref{gammais} to show the existence of the cycle above implies that one of our assumptions on $\Phi$ must fail. We distinguish two cases. 
       
       \emph{Case $1$: $r=1$.} Here we have 
       \begin{equation*}
       	\Phi_{t}(\gamma_{0,1})=\gamma_{1,0}, \quad \quad \Phi_{t}(\gamma_{0,1}\beta_{0})=\Phi_{t}(\gamma_{0,1})=\gamma_{1,1}=\gamma_{1,0}\sigma,
       \end{equation*}
        contradicting our assumption that $(1,\beta_{0})$ and $(1,\sigma)$ are not in the same $\Phi$-orbit. 
       
       \emph{Case $2$: $r>1$.} Notice that for all $l\in r$ we have
       \begin{align*}
       	\Phi_{t\sigma t^{-1}\beta_{0}}(\gamma_{0,2l})=&\Phi_{\sigma t^{-1}\beta_{0}}(\gamma_{1,2l})=\Phi_{ t^{-1}\beta_{0}}(\gamma_{1,2l}\sigma)=\\
       	&\Phi_{t^{-1}\beta_{0}}(\gamma_{1,2l+1})=\Phi_{\beta_{0}}(\gamma_{0,2l+1})=\gamma_{0,2l+2}.
       \end{align*}
       So in this case the non-trivial element $(t\sigma t^{-1}\beta_{0})^{r}$ fixes the pair $(\gamma_{0,0},\gamma_{0,2})\in(U)^{2}$ (recall the cyclical notation), in contradiction with our assumption that $\Phi$ is $2$-sharp. 
  \end{subproof}	
     A standard back-and-forth concludes the proof. Fix some order $\prec$ on $U$ with the order type of the natural numbers. Assume that at some point in the construction we are given $\Phi^{n}$ for $n>0$. Then we construct  $\Phi^{n+1}\in\mathcal{P}_{2sh}$ extending $\Phi^{n}$ as follows: 
    \begin{itemize}
    	\item if $n\equiv 0 \mod 4$, then we let $\beta\in U$ be the $\prec$-smallest element not in $\dom(\Phi^{n}_{t})$ and we use \cref{first claim} to construct $\Phi^{n+1}$ such that $\beta\in\dom(\Phi^{n+1}_{t})$;
    	\item if $n\equiv 2 \mod 4$, then we let $\beta\in U$ be the $\prec$-smallest element not in $\im(\Phi^{n}_{t})$ and we use \cref{first claim} to construct $\Phi^{n+1}$ such that $\beta\in\im(\Phi^{n+1}_{t})$;
    	\item if $n$ is odd, then we choose the $\prec$-smallest $\beta\in \mathtt{V}$ such that $\{1,\beta\}$ and $\{1,\sigma\}$ are not in the same orbit under $\Phi^{n}$ and use \cref{second claim} to ensure that they are in the same orbit under $\Phi^{n+1}$.  
    \end{itemize}
  	This concludes the proof of \cref{almost free case}. 
  \end{proof}

\emph{AI statement.} A cheap LLM subscription was used exclusively to check the draft for grammatical issues, misprints, notational clashes.   

\bibliographystyle{alpha}
\bibliography{references}

\end{document}